\documentclass[11pt, letterpaper,final]{amsart}
\usepackage{amsfonts,amsmath, amsthm, amssymb, latexsym, epsfig}
\usepackage[english]{babel}
\usepackage[utf8]{inputenc}
\usepackage[all]{xy}
\usepackage{xspace}
\usepackage{comment}
\usepackage{setspace}
\usepackage{enumerate}
\usepackage{stmaryrd}
\usepackage[mathscr]{eucal}
\usepackage[notcite,notref]{showkeys}
\usepackage{verbatim}
\usepackage{fancyvrb}
\usepackage{mathrsfs}

	\newcommand{\Ext}{\ensuremath{\operatorname{Ext}}}
	
	\newcommand{\multialg}[1]{\mathcal{M}(#1)\xspace}
	\newcommand{\corona}[1]{\mathcal{Q}(#1)\xspace}	
	\newcommand{\Prim}{\ensuremath{\operatorname{Prim}}}
	\newcommand{\CPC}{\mathrm{\CPC}}
	\newcommand{\nuc}{\mathrm{nuc}}
	\newcommand{\Ad}{\ensuremath{\operatorname{Ad}}\xspace}
	
	\newcommand{\Hom}{\mathrm{Hom}}
	\newcommand{\as}{\mathrm{as}}
	\newcommand{\oh}{\mathrm{oh}}
	\newcommand{\h}{\mathrm{h}}
	\newcommand{\op}{\mathrm{op}}
	\newcommand{\Aut}{\mathrm{Aut}}
	\newcommand{\id}{\mathrm{id}}
	\newcommand{\CP}{\mathrm{CP}}
	\newcommand{\Cu}{\mathrm{Cu}}
	\newcommand{\Ann}{\mathrm{Ann}\,}
	\newcommand{\Annn}{\mathrm{Ann}}
	\newcommand{\ev}{\mathrm{ev}}
	\newcommand{\I}{\mathrm{I}}
	\newcommand{\aMvN}{\mathrm{aMvN}}
	\newcommand{\asMvN}{\mathrm{asMvN}}
	\newcommand{\au}{\mathrm{au}}
	\newcommand{\asu}{\mathrm{asu}}
	\newcommand{\st}{\mathrm{st}}

	\theoremstyle{plain}
	\newtheorem{thma}{Theorem}
	\newtheorem{introtheorem}[thma]{Theorem}
	\newtheorem{introcorollary}[thma]{Corollary}

	\theoremstyle{plain}
	\newtheorem{thm}{Theorem}[section]
	\newtheorem{lemma}[thm]{Lemma}
	\newtheorem{theorem}[thm]{Theorem}
	\newtheorem{proposition}[thm]{Proposition}
	\newtheorem{corollary}[thm]{Corollary}

	\theoremstyle{definition}
	\newtheorem{definition}[thm]{Definition}
	
	\newtheorem{remark}[thm]{Remark}
	\newtheorem{notation}[thm]{Notation}
	\newtheorem{example}[thm]{Example}

	\newtheorem{question}[thm]{Question}
	\newtheorem{introquestion}[thma]{Question}
	\newtheorem{observation}[thm]{Observation}
	
	\numberwithin{equation}{section}
	\numberwithin{figure}{section}

\VerbatimFootnotes

\begin{document}
	\title[Classification of $\mathcal O_\infty$-stable $C^\ast$-algebras]{Classification of $\mathcal O_\infty$-stable $C^\ast$-algebras}
	\author{James Gabe}
	\address{Current: Department of Mathematics and Computer Science, University of Southern Denmark, Campusvej 55, DK-5230 Odense M, Denmark}
	\address{Former: School of Mathematics and Applied Statistics, University of Wollongong, Wollongong NSW 2522, Australia}
        \email{gabe@imada.sdu.dk}
	\thanks{This work was funded by the Carlsberg Foundation through an Internationalisation Fellowship. }

\begin{abstract}
I present a proof of Kirchberg's classification theorem: two separable, nuclear, $\mathcal O_\infty$-stable $C^\ast$-algebras are stably isomorphic if and only if they are ideal-related $KK$-equivalent. In particular, this provides a more elementary proof of the Kirchberg--Phillips theorem which is isolated in the paper to increase readability of this important special case. 
\end{abstract}

\maketitle
\setcounter{tocdepth}{1}
\tableofcontents

\section{Introduction and main results}

The Kirchberg--Phillips theorem -- proved independently by Kirchberg \cite{Kirchberg-simple} and Phillips \cite{Phillips-classification} in the mid 90s -- was one of the first major breakthroughs in the Elliott classification programme, constituting a complete $K$-theoretic classification of separable, nuclear, simple, purely infinite $C^\ast$-algebras satisfying the universal coefficient theorem (UCT). The simple, stably finite side of the classification programme took 20 years longer but was finally resolved by the work of many hands -- a highly incomplete list of references being \cite{GongLinNiu-classZ-stable}, \cite{GongLinNiu-classZ-stable2}, \cite{ElliottGongLinNiu-classfindec}, \cite{TikuisisWhiteWinter-qdnuc} -- leading to a complete classification of all separable, nuclear, unital, simple, $\mathcal Z$-stable $C^\ast$-algebras satisfying the UCT using $K$-theory and traces.\footnote{Here the main result of \cite{CETWW-nucdim} was used to reduce from finite nuclear dimension to $\mathcal Z$-stability.}

The concept of pure infiniteness for simple $C^\ast$-algebras was first introduced by Cuntz in \cite{Cuntz-K-theoryI} as a $C^\ast$-algebraic analogue of Type $\mathrm{III}$ von Neumann algebra factors. In this paper it was proved that the Cuntz algebras $\mathcal O_n$ introduced in \cite{Cuntz-simple} enjoy this property. Purely infinite, simple $C^\ast$-algebras arise naturally from various constructions, such as simple Cuntz--Krieger algebras arising from symbolic dynamics \cite{CuntzKrieger-MarkovI}, and the reduced crossed product $C(\partial G)\rtimes_r G$ of a non-elementary hyperbolic group $G$ acting naturally on its Gromov  boundary $\partial G$ \cite{AD-purelyinfdyn}. In these cases, the $C^\ast$-algebras fall into the class covered by the Kirchberg--Phillips classification, and their isomorphism class is therefore determined by their $K$-theory.

A major leap forward in the study of purely infinite, simple $C^\ast$-algebras came from deep work of Kirchberg announced in \cite{Kirchberg-ICM}, see \cite{KirchbergPhillips-embedding} for published proofs. He proved his iconic $\mathcal O_2$-embedding theorem -- that all separable, exact $C^\ast$-algebras embed into $\mathcal O_2$ -- and showed that a separable, nuclear, simple $C^\ast$-algebra is purely infinite if and only if it is isomorphic to its tensor product with the Cuntz algebra $\mathcal O_\infty$. A $C^\ast$-algebra satisfying this latter property is called $\mathcal O_\infty$-stable. The class of separable, nuclear, simple, purely infinite $C^\ast$-algebras are now commonly known as Kirchberg algebras, while Kirchberg refers to the unital ones as pi-sun algebras (\underline{p}urely \underline{i}nfinite (simple),\footnote{Because when Kirchberg did this ground breaking work, pure infiniteness was not defined outside the simple setting.} \underline{s}eparable, \underline{u}nital, \underline{n}uclear). These results of Kirchberg were key ingredients in proving the Kirchberg--Phillips theorem. 

While Elliott's original classification results for stably finite $C^\ast$-algebras such as for AF algebras \cite{Elliott-AFclass} and certain AH algebras \cite{Elliott-classrr0} revolved around classifying not necessarily simple $C^\ast$-algebras, most subsequent classification results dealt with simple $C^\ast$-algebras.
On the purely infinite side, around the turn of the millennium, Kirchberg vastly generalised his and Phillips' classification results by classifying all separable, nuclear, $\mathcal O_\infty$-stable $C^\ast$-algebras up to stable isomorphism, using an ideal-related version of Kasparov's $KK$-theory as the classifying invariant. This resulted in a published manuscript \cite{Kirchberg-non-simple} outlining the strategy of the proof, and a complete proof is to appear as the headline result in a yet unpublished book by Kirchberg \cite{Kirchberg-book}. The goal of this paper is to present a proof of this highly influential work of Kirchberg thus closing this massive gap in the published literature. The proof presented here is very different from the one outlined by Kirchberg, and, dare I argue, quite a lot simpler.

In \cite{KirchbergRordam-purelyinf} and \cite{KirchbergRordam-absorbingOinfty}, Kirchberg and Rørdam generalised the notion of pure infiniteness for simple $C^\ast$-algebras to the not necessarily simple case. This lead to three notions, namely weak pure infiniteness, pure infiniteness, and strong pure infiniteness, which are successively stronger conditions. While these are all equivalent for simple $C^\ast$-algebras, it is an open problem whether any or all of these notions are equivalent in general. Kirchberg and Rørdam showed that a separable, nuclear $C^\ast$-algebra is $\mathcal O_\infty$-stable if and only if it is strongly purely infinite,\footnote{Kirchberg and Rørdam proved this for stable $C^\ast$-algebras. One should appeal to \cite[Proposition 4.4(4,5)]{Kirchberg-Abel} or \cite[Corollary 3.2]{TomsWinter-ssa} for the general statement.} and hence these are the $C^\ast$-algebras classified by the above mentioned result -- which is also the main theorem of this paper. 

Strong pure infiniteness is rather hard to verify in examples, although significant progress was made in \cite{BrownClarkSierakowski-purelyinfgroupoids}, \cite{KirchbergSierakowski-spicrossed} and \cite{KirchbergSierakowski-filling}. Weak pure infiniteness and pure infiniteness are however easier to check. For instance, any $C^\ast$-algebra with finite nuclear dimension is weakly purely infinite if and only if it is traceless, in the sense that it has no non-trivial lower semicontinuous $[0,\infty]$-valued 2-quasitraces \cite[Theorem 5.2]{WinterZacharias-nucdim}.\footnote{A much easier proof is implicitly presented in \cite[Proposition 2.3]{RobertTikuisis-nucdimnonsimple}. In fact, if $A$ has nuclear dimension at most $m$, then its Cuntz semigroup has $(m+1)$-comparison for some $m<\infty$. It follows from tracelessness that the sum $a\oplus \cdots \oplus a$ ($m+1$ summands) of any positive element $a$ is properly infinite. Hence the $C^\ast$-algebra is weakly purely infinite.} There are conditions under which the three notions of pure infiniteness are known to be equivalent, for instance when the primitive ideal space is zero dimensional \cite{PasnicuRordam-purelyinfrr0} or Hausdorff \cite{BlanchardKirchberg-Hausdorff}. These results have been used to verify strong pure infiniteness of many more examples, see for instance \cite{HongSzymanski-purelyinfgraph}, \cite{Kwasniewski-crossedC_0}, \cite{KwasniewskiSzymanski-pureinffell}, \cite{BonickeLi-purelyinfample}.

In \cite{Gabe-O2class}, I presented a new proof of a special case of the main theorem of this paper: a complete classification of separable, nuclear, stable/unital $\mathcal O_2$-stable $C^\ast$-algebras. While some of the machinery developed in that paper is crucial for proving the main results of this paper, these tools have already had other applications: they were used to answer exactly when traceless $C^\ast$-algebras are AF embeddable \cite[Corollary C]{Gabe-AFemb}, when Connes and Higson's $E$-theory \cite{ConnesHigson-deformations} can be unsuspended \cite[Corollary E]{Gabe-AFemb}, and lead to the exact computation of the nuclear dimension and decomposition rank of separable, nuclear, $\mathcal O_\infty$-stable $C^\ast$-algebras \cite[Theorems A and B]{BGSW-nucdim}. The techniques presented in this paper should reach far beyond those in \cite{Gabe-O2class} and should be applicable in a much broader context.

\subsection*{Classification of $C^\ast$-algebras through classification of maps}

The main bulk of this paper is to provide existence and uniqueness results for $\ast$-homomorphisms satisfying certain properties. The purpose of this section is to explain why this is a natural thing to aim for, how it ties in with classification, and why it exactly becomes the class of separable, nuclear, $\mathcal O_\infty$-stable $C^\ast$-algebras which are classified by these results.

A classical method for classifying $C^\ast$-algebras is through an Elliott intertwining argument, see \cite[Section 2.3]{Rordam-book-classification}. A version of Elliott intertwining states that if $A$ and $B$ are separable $C^\ast$-algebras, and $\phi\colon A \to B$ and $\psi \colon B\to A$ are $\ast$-homo\-morphisms for which $\psi \circ \phi \sim_\au \id_A$ and $\phi \circ \psi \sim_\au \id_B$, where $\sim_\au$ denotes approximate unitary equivalence, then $A\cong B$. In particular, suppose one is given a functorial invariant $\mathrm{Inv}$, and a class $\mathscr C$ of separable $C^\ast$-algebras, such that one has the following:
\begin{itemize}
\item[$\exists$] \emph{Existence:} For all $A,B\in \mathscr C$ and any homomorphism $\rho \colon \mathrm{Inv}(A) \to \mathrm{Inv}(B)$, there is a $\ast$-homomorphism $\phi \colon A \to B$ such that $\mathrm{Inv}(\phi) =\rho$;
\item[$!$] \emph{Uniqueness:} For all $A,B \in \mathscr C$ and any $\ast$-homomorphisms $\phi, \psi \colon A \to B$ such that $\mathrm{Inv}(\phi) = \mathrm{Inv}(\psi)$, one has $\phi\sim_\au \psi$.
\end{itemize}
Then for every $A,B\in \mathscr C$, $A\cong B$ if and only if $\mathrm{Inv}(A) \cong \mathrm{Inv}(B)$. So the objects in $\mathscr C$ are completely classified by the invariant $\mathrm{Inv}$. 

To see that one obtains classification, fix an isomorphism $\rho\colon \mathrm{Inv}(A) \to \mathrm{Inv}(B)$. By the existence part above one may lift $\rho$ and $\rho^{-1}$ to $\phi$ and $\psi$ respectively. As $\mathrm{Inv}(\psi \circ \phi) = \mathrm{Inv}(\id_A)$, uniqueness applied to $\psi\circ \phi$ and $\id_A$ entails that $\psi \circ \phi \sim_\au \id_A$, and similarly $\phi\circ \psi \sim_\au \id_B$. By Elliott intertwining one obtains $A\cong B$.

In general, it is too much to hope for that one can find an invariant which classifies \emph{all} $\ast$-homomorphisms in this fashion. For instance, if the class $\mathscr C$ only consists of unital $C^\ast$-algebras it seems quite natural to only consider \emph{unital} $\ast$-homomorphisms. Or perhaps one wants to consider injective $\ast$-homomorphism. Either way, it makes sense to only consider existence and uniqueness for $\ast$-homomorphisms which satisfy a property $\mathscr P$, for which $\mathscr P$ is closed under taking compositions. There is one major caveat though: in order to run the classification argument above, one needs to know that $\id_A$ and $\id_B$ satisfy the property $\mathscr P$. 

An instance where this plays a notable role is \emph{nuclearity}. A map between $C^\ast$-algebras is said to be nuclear if it has the completely positive approximation property, and a $C^\ast$-algebra $A$ is nuclear exactly when $\id_A$ is nuclear. Nuclearity of $\ast$-homomorphisms is a fundamental part of most known existence and uniqueness theorems, and this explains why classification should only be expected to hold for nuclear $C^\ast$-algebras.

A major part of this paper will be to produce existence and uniqueness results for $\ast$-homomorphisms satisfying mainly two properties: they should be nuclear and \emph{strongly $\mathcal O_\infty$-stable}. Strong $\mathcal O_\infty$-stability of $\ast$-homomorphisms was introduced in \cite{Gabe-O2class} where it was proved that a separable $C^\ast$-algebra $A$ is $\mathcal O_\infty$-stable -- i.e.~$A \cong A\otimes \mathcal O_\infty$ -- if and only if $\id_A$ is strongly $\mathcal O_\infty$-stable. In particular, the separable $C^\ast$-algebras $A$ classified by these existence and uniqueness results are exactly the ones for which $\id_A$ is nuclear and strongly $\mathcal O_\infty$-stable, or equivalently, $A$ should be separable, nuclear, and $\mathcal O_\infty$-stable.

\subsection*{The main results -- the simple case}

Before presenting the general (not necessarily simple) classification results, the focus will be on the classical case leading to the Kirchberg--Phillips theorem. In this case, the classifying invariant will be $KK_\nuc$ as defined by Skandalis \cite{Skandalis-KKnuc}; a variation on Kasparov's $KK$-theory \cite{Kasparov-KKExt}. While all $\ast$-homomorphisms $\phi \colon A\to B$ give rise to elements $KK(\phi) \in KK(A,B)$, only nuclear $\ast$-homomorphisms $\phi$ give rise to elements $KK_\nuc(\phi) \in KK_\nuc(A,B)$. If either $A$ or $B$ is nuclear, then the natural map $KK_\nuc(A,B) \to KK(A,B)$ is an isomorphism, and moreover, every $\ast$-homomorphism $\phi\colon A\to B$ is nuclear.

Let $A$ be a separable $C^\ast$-algebra and $B$ a $\sigma$-unital $C^\ast$-algebra, i.e.~a $C^\ast$-algebra containing a strictly positive element. The properties satisfied by the $\ast$-homo\-morphisms $\phi \colon A \to B$ for the existence and uniqueness theorems are the following:
\begin{itemize}
\item \emph{Nuclear:} $\phi$ is a point-norm limit of maps factoring via completely positive maps through matrix algebras;
\item \emph{Strongly $\mathcal O_\infty$-stable:}\footnote{If one replaces the paths with sequences $(s_n^{(i)})_{n\in \mathbb N}$ which satisfy the analogous conditions, then one obtains the definition of $\phi$ being \emph{$\mathcal O_\infty$-stable} from \cite{Gabe-O2class}. It is an open problem, studied in more detail in Section \ref{s:stronglyOinfty}, whether every $\mathcal O_\infty$-stable map is strongly $\mathcal O_\infty$-stable.} there are continuous, bounded paths $(s_t^{(i)})_{t\in[0,\infty)}$ in $B$ for $i=1,2$ such that
\begin{equation}
\lim_{t\to \infty}\| s_t^{(i)}  \phi(a) - \phi(a) s_t^{(i)}\| = 0, \qquad \lim_{t\to \infty} \| ((s_t^{(i)})^\ast s_t^{(j)} - \delta_{i,j}) \phi(a) \| = 0,
\end{equation}
for every $a\in A$ and $i,j=1,2$, where $\delta_{i,j} = 1$ if $i=j$ and $\delta_{i,j}=0$ if $i\neq j$; 
\item \emph{Full:} $\phi(a)$ generates all of $B$ as a two-sided, closed ideal, for every non-zero $a\in A$.
\end{itemize}

These properties of $\ast$-homomorphisms translate into properties of $C^\ast$-algebras, in the sense that if $A$ is a separable $C^\ast$-algebra, then $\id_A$ is nuclear, strongly $\mathcal O_\infty$-stable, or full respectively, exactly when $A$ is nuclear, $\mathcal O_\infty$-stable, or simple respectively. Moreover, if $\phi \colon A \to B$ is a $\ast$-homomorphism and $A$ or $B$ is nuclear or $\mathcal O_\infty$-stable respectively, then $\phi$ is nuclear or strongly $\mathcal O_\infty$-stable respectively. Moreover, if $B$ is simple and $\phi$ is injective, then it is also full, or if $A$ is simple, then $\phi$ is full when corestricted to the two-sided, closed ideal generated its image.

Let
\begin{equation}
\Gamma_i \colon KK_\nuc(A,B) \to \Hom(K_i(A), K_i(B)),
\end{equation}
for $i=0,1$ be the canonical homomorphisms induced by the Kasparov product under the canonical identifications $KK_\nuc(\mathbb C, B) \cong K_0(B)$ and $KK_\nuc(C_0(\mathbb R), B) \cong K_1(B)$, and let
\begin{equation}
\Gamma \colon KK_\nuc(A,B) \to \Hom(K_\ast(A),K_\ast(B))
\end{equation} 
be the induced homomorphism. If $\phi \colon A\to B$ is a nuclear $\ast$-homomorphism, then $\Gamma_i(KK_\nuc(\phi)) = \phi_i \colon K_i(A) \to K_i(B)$ for $i=0,1$ is the induced map in $K$-theory.
So if $A$ and $B$ are unital $C^\ast$-algebras, and if $\phi \colon A\to B$ is a unital, nuclear $\ast$-homomorphism, then
\begin{equation}
\Gamma_0(KK_\nuc(\phi)) ([1_A]_0) = \phi_0([1_A]_0) = [1_B]_0.
\end{equation}

The following is the main existence part in the Kirchberg--Phillips theorem.

\begin{introtheorem}[Existence -- full case]\label{t:existsimple}
Let $A$ be a separable, exact $C^\ast$-algebra, let $B$ be a $\sigma$-unital $C^\ast$-algebra containing a full, properly infinite projection, and let $x\in KK_\nuc(A,B)$. Then there exists a nuclear, strongly $\mathcal O_\infty$-stable, full $\ast$-homo\-morphism $\phi \colon A \to B$ such that $KK_\nuc(\phi) = x$.

Moreover, if $A$ and $B$ are unital, then there exists a unital, nuclear, strongly $\mathcal O_\infty$-stable, full $\ast$-homomorphism $\phi \colon A \to B$ for which $KK_\nuc(\phi) = x$ if and only if 
\begin{equation}
\Gamma_0(x)([1_A]_0) = [1_B]_0.
\end{equation}
\end{introtheorem}

The next goal is to present a uniqueness theorem to keep the above existence theorem company. In order to do so, one must first know which equivalence relation of $\ast$-homomorphisms to obtain uniqueness by. The obvious choice would be approximate unitary equivalence as this is what is needed for classification of $C^\ast$-algebras, but this does not imply agreement in $KK_\nuc$. The right notion should have more of a homotopic flavour.

Let $\phi ,\psi \colon A \to B$ be $\ast$-homomorphisms for which $A$ is separable. Say that $\phi$ and $\psi$ are \emph{asymptotically Murray--von Neumann equivalent}, written $\phi \sim_\asMvN \psi$, if there is a norm-continuous path $(v_t)_{t\in [0,\infty)}$ of contractions in the multiplier algebra $\multialg{B}$, such that
\begin{equation}
\lim_{t\to\infty} \| v_t^\ast \phi(a) v_t - \psi(a) \| = 0, \qquad \lim_{t\to \infty} \| v_t \psi(a) v_t^\ast - \phi(a)\| = 0
\end{equation}
for all $a\in A$. If each $v_t$ can be chosen to be a unitary, then $\phi$ and $\psi$ are said to be \emph{asymptotically unitarily equivalent}, written $\phi \sim_\asu \psi$.

In the special case where $A = \mathbb C$, then $\phi \sim_\asMvN \psi$ (respectively $\phi \sim_\asu \psi$) if and only if the projections $\phi(1_\mathbb{C})$ and $\psi(1_{\mathbb C})$ are Murray--von Neumann equivalent (respectively unitarily equivalent).
Asymptotic Murray--von Neumann equivalence and unitary equivalence are (for general $A$) connected in similar ways as Murray--von Neumann and unitary equivalence of projections are connected. For instance, $\phi \sim_\asMvN \psi$ if and only if $\phi \oplus 0 \sim_\asu \psi \oplus 0$.

It turns out that asympotitc Murray--von Neumann equivalence is exactly the equivalence relation of $\ast$-homomorphisms which is captured by $KK_\nuc$. However, this is not strong enough in general to obtain classification of $C^\ast$-algebras, see \cite[Remark 3.15]{Gabe-O2class} for an example. Luckily, asymptotic Murray--von Neumann equivalence and asymptotic unitary equivalence turn out to agree in many interesting cases, for instance whenever $B$ is stable, or whenever $A,B,\phi$ and $\psi$ are all unital. This explains why we obtain classification of stable and of unital $C^\ast$-algebras.
The following is the uniqueness theorem for full maps.

\begin{introtheorem}[Uniqueness -- full case]\label{t:uniquesimple}
Let $A$ be a separable $C^\ast$-algebra, and let $B$ be a $\sigma$-unital $C^\ast$-algebra. Suppose that $\phi, \psi\colon A \to B$ are nuclear, strongly $\mathcal O_\infty$-stable, full $\ast$-homomorphisms. The following are equivalent:
\begin{itemize}
\item[$(i)$] $KK_\nuc(\phi) = KK_\nuc(\psi)$;
\item[$(ii)$] $\phi$ and $\psi$ are asymptotically Murray--von Neumann equivalent.
\end{itemize}
Additionally, if either $B$ is stable, or if $A, B, \phi$ and $\psi$ are all unital, then $(i)$ and $(ii)$ are equivalent to
\begin{itemize}
\item[$(iii)$] $\phi$ and $\psi$ are asymptotically unitarily equivalent (with unitaries in the minimal unitisation).
\end{itemize}
\end{introtheorem}

A nuclear, strongly $\mathcal O_\infty$-stable, full $\ast$-homomorphism $A \to B$ as above exists only when $A$ is exact, and $B$ contains a properly infinite, full projection. So the $C^\ast$-algebras for which Theorem \ref{t:uniquesimple} applies, are the same as those considered in Theorem \ref{t:existsimple}. 

The classical approach to the Kirchberg--Phillips theorem in \cite{Kirchberg-simple} and \cite{Phillips-classification} was also through existence and uniqueness results similar to Theorems \ref{t:existsimple} and \ref{t:uniquesimple}, see also \cite[Theorems 8.2.1 and 8.3.3]{Rordam-book-classification}. In Section \ref{ss:classical} it is explained why these existence and uniqueness results are special cases of Theorems \ref{t:existsimple} and \ref{t:uniquesimple}.

In order to apply existence and uniqueness results as above for classification, the identity maps of the $C^\ast$-algebras in question must also satisfy the relevant conditions on the maps. In the above case, this means that $\id_A$ and $\id_B$ should be nuclear, strongly $\mathcal O_\infty$-stable and full. These properties translate into properties of the $C^\ast$-algebras: $A$ and $B$ should be nuclear, $\mathcal O_\infty$-stable, and simple. Adding separability assumptions to the mix, one exactly arrives at the class of \emph{Kirchberg algebras}: separable, nuclear, $\mathcal O_\infty$-stable, simple $C^\ast$-algebras.

While pure infiniteness is usually considered in the definition of Kirchberg algebras in place of $\mathcal O_\infty$-stability, a by-now-classical theorem of Kirchberg \cite{Kirchberg-ICM} (see also \cite{KirchbergPhillips-embedding}) says that a separable, nuclear, simple $C^\ast$-algebra is $\mathcal O_\infty$-stable if and only if it is purely infinite. While this is of course a beautiful and highly applicable characterisation, the $\mathcal O_\infty$-stability (of maps) is really what drives all the proofs in the classification results.

The main result of the first part of the paper is the following classification result due to Kirchberg \cite{Kirchberg-simple} and Phillips \cite{Phillips-classification}. Recall that an element $x\in KK(A,B)$ is \emph{invertible}, if there exists $y\in KK(B,A)$ such that $y\circ x = KK(\id_A)$ and $x\circ y = KK(\id_B)$, and that $A$ and $B$ are \emph{$KK$-equivalent} if there exists an invertible element in $KK(A,B)$.

\begin{introtheorem}[Classification of Kirchberg algebras]\label{t:KP}
Let $A$ and $B$ be Kirchberg algebras (separable, nuclear, purely infinite, simple $C^\ast$-algebras).
\begin{itemize}
\item[$(a)$] If $A$ and $B$ are stable, then $A\cong B$ if and only if $A$ and $B$ are $KK$-equivalent. Moreover, for any invertible $x\in KK(A,B)$ there exists an isomorphism $\phi \colon A \xrightarrow \cong B$, unique up to asymptotic unitary equivalence (with unitaries in the minimal unitisation), such that $KK(\phi) = x$.
\item[$(b)$] If $A$ and $B$ are unital, then $A\cong B$ if and only if there is an invertible $x\in KK(A,B)$ such that $\Gamma_0(x)([1_A]_0) = [1_B]_0$. Moreover, for any such $x$ there is an isomorphism $\phi \colon A \xrightarrow \cong B$, unique up to asymptotic unitary equivalence, such that $KK(\phi) = x$.
\end{itemize}
\end{introtheorem}

By a dichotomy result of Zhang \cite{Zhang-dichotomy}, any Kirchberg algebra is either stable or unital and hence the above theorem classifies all Kirchberg algebras.

Recall that a separable $C^\ast$-algebra $A$ \emph{satisfies the universal coefficients theorem (UCT)} of Rosenberg and Schochet \cite{RosenbergSchochet-UCT} (in $KK$-theory) if for any $\sigma$-unital $C^\ast$-algebra $B$ there is a (natural) short exact sequence
\begin{equation}\label{eq:UCT}
0 \to \Ext(K_{1-\ast}(A), K_\ast(B)) \to KK(A,B) \xrightarrow{\Gamma} \Hom(K_\ast(A), K_\ast(B)) \to 0.
\end{equation}
Rosenberg and Schochet proved that a separable $C^\ast$-algebra satisfies the UCT if and only if it is $KK$-equivalent to a separable $C^\ast$-algebra of Type I. A remarkable theorem of Tu \cite{Tu-BaumConnes} implies that the groupoid $C^\ast$-algebra of a locally compact, Hausdorff, second countable, amenable groupoid satisfies the UCT. It is an open problem, often referred to as the \emph{UCT problem}, whether every separable, nuclear $C^\ast$-algebra satisfies the UCT.

An elementary consequence of the UCT, see \cite[Proposition 7.3]{RosenbergSchochet-UCT}, is that the following holds whenever both $A$ and $B$ satisfy the UCT: if $\alpha_\ast \colon K_\ast(A) \to K_\ast(B)$ is an isomorphism and  $x\in KK(A,B)$ satisfies $\Gamma(x) = \alpha_\ast$ (such $x$ always exists by the UCT short exact sequence \eqref{eq:UCT}), then $x$ is invertible in $KK$. Hence $C^\ast$-algebras satisfying the UCT are strongly classified by $K$-theory up to $KK$-equiva\-lence.\footnote{\emph{Strong classification} means that one may lift an isomorphism on the invariant to an isomorphism.} The following is therefore immediately obtained from Theorem \ref{t:KP}.

\begin{introtheorem}[Classification of UCT Kirchberg algebras]\label{t:KPUCT}
Let $A$ and $B$ be Kirchberg algebras satisfying the UCT.
\begin{itemize}
\item[$(a)$] If $A$ and $B$ are stable, then $A\cong B$ if and only if $K_\ast(A) \cong K_\ast(B)$. Moreover, for any isomorphism $\alpha_\ast \colon K_\ast(A) \xrightarrow \cong K_\ast(B)$ there is an isomorphism $\phi \colon A \xrightarrow \cong B$ such that $\phi_\ast = \alpha_\ast$.
\item[$(b)$] If $A$ and $B$ are unital, then $A\cong B$ if and only if there exists an isomorphism $\alpha_\ast \colon K_\ast(A) \xrightarrow \cong K_\ast(B)$ such that $\alpha_0([1_A]_0) = [1_B]_0$. Moreover, for any such $\alpha_\ast$, there is an isomorphism $\phi \colon A \xrightarrow \cong B$ such that $\phi_\ast = \alpha_\ast$.
\end{itemize}
\end{introtheorem}

Note that since $K$-theory is a weaker invariant than $KK$-theory, one no longer has uniqueness of the isomorphisms in Theorem \ref{t:KPUCT} as opposed to Theorem \ref{t:KP}.

\subsection*{The main results -- the general case}

In order to explain the general classification results, one must consider $C^\ast$-algebras with specified (two-sided, closed) ideal structure, and a version of $KK$-theory which takes this extra information into account. This is completely analogous to considering group equivariant $KK$-theory, when the $C^\ast$-algebras come equipped with a group action.

To be more precise, for a $C^\ast$-algebra $A$ let $\mathcal I(A)$ denote the partially ordered set of two-sided, closed ideals in $A$. Similarly, for a topological space $X$ let $\mathcal O(X)$ denote the partially ordered set of open subsets of $X$. Both $\mathcal I(A)$ and $\mathcal O(X)$ are complete lattices. An \emph{action of $X$ on $A$} is an order preserving map $\Phi_A \colon \mathcal O(X) \to \mathcal I(A)$, and an \emph{$X$-$C^\ast$-algebra} is a $C^\ast$-algebra $A$ together with an action $\Phi_A$ of $X$ on $A$. Usually $\Phi_A$ is eliminated from the notation by defining $A(U) := \Phi_A(U)$ for $U\in \mathcal O(X)$. A map $\phi \colon A\to B$ is $X$-equivariant if $\phi(A(U)) \subseteq B(U)$ for all $U\in \mathcal O(X)$.

A special -- and very important -- type of action is that of the primitive ideal space $\Prim A$ on $A$, given by the canonical order isomorphism $\I_A \colon \mathcal O(\Prim A) \to \mathcal I(A)$, see \cite[Theorem 4.1.3]{Pedersen-book-automorphism}. More generally, an $X$-$C^\ast$-algebra $A$ is called \emph{tight} if the action $\Phi_A \colon \mathcal O(X) \to \mathcal I(A)$ is an order isomorphism. Note that any isomorphism $\phi \colon A \xrightarrow \cong B$ of $C^\ast$-algebras induces a tight action $\Phi_B$ of $\Prim A$ on $B$, such that $\phi \colon (A,\I_A) \xrightarrow \cong (B,\Phi_B)$ is an isomorphism of $\Prim A$-$C^\ast$-algebras. Hence it suffices to classify all tight $X$-$C^\ast$-algebras for different spaces $X$. 

In the special case where $X$ is a one-point space, tightness of $A$ means that it is simple, so classifying simple $C^\ast$-algebras is the same as classifying tight $X$-$C^\ast$-algebras in the case where $X$ is a one-point space. 

Whenever $A$ is a separable $X$-$C^\ast$-algebra and $B$ is a $\sigma$-unital $X$-$C^\ast$-algebra, one may construct groups $KK(X; A, B)$ and $KK_\nuc(X; A, B)$ analogously as one constructs $KK(A,B)$ and $KK_\nuc(A,B)$. In fact, just as when constructing $KK^G(A, B)$ for $G$-$C^\ast$-algebras where $G$ is a group, one constructs $KK(X)$ and $KK_\nuc(X)$ by only considering Kasparov modules which remember the $X$-$C^\ast$-algebra structure. As in the classical case, the Kasparov product induces a bilinear product, and it therefore makes sense to talk about invertible $KK(X)$-elements. This is all done in full detail in Section \ref{s:KK}. The main classification result will be in terms of the existence of invertible $KK(X)$-elements, as in Theorem \ref{t:KP}.

 Every $X$-equivariant $\ast$-homomorphism $\phi \colon A \to B$ induces an element $KK(X;\phi) \in KK(X; A, B)$, and if $A$ is exact and $\phi$ is nuclear, then one obtains $KK_\nuc(X;\phi) \in KK_\nuc(X; A, B)$.

There are canonical forgetful maps 
\begin{equation}
KK_\nuc(X; A, B) \to KK(X; A, B) \to KK(A,B),
\end{equation}
and thus every element $x\in KK_\nuc(X; A,B)$ induces a homomorphism $\Gamma_i(x) \colon K_i(A) \to K_i(B)$ for $i=0,1$ which is compatible with taking compositions of (nuclear) $X$-equivariant $\ast$-homomorphisms in the obvious way.

Just as in Theorems \ref{t:existsimple} and \ref{t:uniquesimple}, the focus will be on nuclear, strongly $\mathcal O_\infty$-stable $\ast$-homomorphisms. However, fullness in the ideal-related setting is a bit more delicate. The idea is the following: let $\phi \colon A \to B$ be an $X$-equivariant $\ast$-homomorphism between $X$-$C^\ast$-algebras. Suppose that for every $a\in A$ that there is a smallest open subset $U_a$ of $X$ such that $a\in A(U_a)$. Then $\phi(a)$ will have to be contained in $B(U_a)$ by $X$-equivariance of $\phi$. The map $\phi$ is called \emph{$X$-full} if $\phi(a)$ is full in $B(U_a)$ for every $a\in A$.

In the case that $X$ is a one-point set, and thus only have open subsets $\emptyset$ and $X$, and the action on $A$ (similarly $B$) is $A(\emptyset) = 0$ and $A(X) = A$, then for every $a\in A$, one has $U_a = X$ if $a\neq 0$ and $U_a = \emptyset$ if $a=0$. Hence one arrives at the usual definition of fullness: $\phi \colon A \to B$ is $X$-full exactly when $\phi(a)$ is full in $B = B(X)$ for every non-zero $a\in A$.

In the case where $A$ and $B$ are tight $X$-$C^\ast$-algebras, $\phi \colon A \to B$ being $X$-full means that if $U\in \mathcal O(X)$ and if $a\in A(U)$ is a full element, then $\phi(a)$ is full in $B(U)$. In particular, any isomorphism of tight $X$-$C^\ast$-algebras is $X$-full, analogously as to how isomorphisms of simple $C^\ast$-algebras are full.

It turns out that the existence of such a minimal open set $U_a$ for every $a\in A$ is equivalent to the action $\Phi_A \colon \mathcal O(X) \to \mathcal I(A)$ preserving all infima. In this case, the $X$-$C^\ast$-algebra $A$ is said to be \emph{lower semicontinuous}. Hence $X$-fullness of maps is only defined when the domain is lower semicontinuous.
Additionally, an $X$-$C^\ast$-algebra is said to be
\begin{itemize}
\item \emph{monotone continuous} if the action preserves infima and increasing suprema;
\item \emph{upper semicontinuous} if the action preserves suprema;
\item \emph{$X$-compact} if the action preserves compact containment (see Definition \ref{d:lattice}).
\end{itemize}

Note that a tight $X$-$C^\ast$-algebra satisfies all the above conditions, i.e.~it is lower semicontinuous, monotone continuous, upper semicontinuous, and $X$-compact.

\begin{introtheorem}[Existence]\label{t:irexistence}
Let $X$ be a topological space, let $A$ be a separable, exact, monotone continuous $X$-$C^\ast$-algebra, and let $B$ be an $\mathcal O_\infty$-stable, upper semicontinuous, $X$-compact $X$-$C^\ast$-algebra. For every element $x\in KK_\nuc(X; A,B)$ there exists a nuclear, $X$-full $\ast$-homomorphism $\phi \colon A \to B$ such that $KK_\nuc(X; \phi) = x$.

Moreover, if $A$ and $B$ are unital, then we may pick $\phi$ to also be unital if and only if the following conditions hold:
\begin{itemize}
\item[(1)] $B(U) = B$ for every $U\in \mathcal O(X)$ satisfying $A(U) = A$, and 
\item[(2)] $\Gamma_0(x)([1_A]_0) = [1_B]_0$ in $K_0(B)$.
\end{itemize}
\end{introtheorem}

In Theorem \ref{t:existsimple} the target $C^\ast$-algebra $B$ was not assumed to be $\mathcal O_\infty$-stable, only to contain a properly infinite, full projection. By using the $\mathcal O_2$-embedding theorem, one may equivalently ask that there exists a full, nuclear, $\mathcal O_\infty$-stable $\ast$-homomorphism $A\to B$.  There is also a more general version of Theorem \ref{t:irexistence} with an analogous assumption -- the existence of an $X$-full, nuclear, $\mathcal O_\infty$-stable $\ast$-homomorphism $A\to B$ -- see Proposition \ref{p:existence}. One obtains the existence result above by combining that proposition with an ideal-related version of the $\mathcal O_2$-embedding theorem, Corollary \ref{c:irO2X}.

The uniqueness below is almost as general as Theorem \ref{t:uniquesimple}, although the asymptotic unitary equivalence is with multiplier unitaries instead of unitaries in the unitisation.

\begin{introtheorem}[Uniqueness]\label{t:iruniqueness}
Let $X$ be a topological space, let $A$ be a separable, exact, lower semicontinuous $X$-$C^\ast$-algebra, and let $B$ be a $\sigma$-unital $X$-$C^\ast$-algebra. Suppose that $\phi, \psi \colon A \to B$ are $X$-full, nuclear, strongly $\mathcal O_\infty$-stable $\ast$-homo\-morphisms. The following are equivalent:
\begin{itemize}
\item[$(i)$] $KK_\nuc(X; \phi) = KK_\nuc(X; \psi)$;
\item[$(ii)$] $\phi$ and $\psi$ are asymptotically Murray--von Neumann equivalent.
\end{itemize}
Additionally, if either $B$ is stable, or if $A,B,\phi$ and $\psi$ are all unital, then $(i)$ and $(ii)$ are equivalent to
\begin{itemize}
\item[$(iii)$] $\phi$ and $\psi$ are asymptotically unitarily equivalent (with multiplier unitaries in the stable case).
\end{itemize} 
\end{introtheorem}

As when classifying Kirchberg algebras, one can apply the above existence and uniqueness theorem for classification provided the identity maps satisfy the conditions on the maps. So $\id_A$ and $\id_B$ should be nuclear, strongly $\mathcal O_\infty$-stable and $X$-full. This translates into properties of the $X$-$C^\ast$-algebras: $A$ and $B$ must be nuclear, $\mathcal O_\infty$-stable, and the action of $X$ should be \emph{tight}, i.e.~the actions $\Phi_A \colon \mathcal O(X) \to \mathcal I(A)$ and $\Phi_B \colon \mathcal O(X) \to \mathcal I(B)$ should be order isomorphisms. Note that tightness is for $X$-$C^\ast$-algebras what simplicity is for $C^\ast$-algebras without an action of $X$.

As with usual $KK$, two separable $X$-$C^\ast$-algebras $A$ and $B$ are \emph{$KK(X)$-equivalent} if $KK(X;A, B)$ contains an invertible element $x$, i.e.~an element for which there exists $y\in KK(X;B,A)$ such that $y\circ x = KK(X;\id_A)$ and $x\circ y = KK(X; \id_B)$.

\begin{introtheorem}[Classification]\label{t:nonsimpleclass}
Let $X$ be a topological space, and suppose that $A$ and $B$ are separable, nuclear, $\mathcal O_\infty$-stable, tight $X$-$C^\ast$-algebras.
\begin{itemize}
\item[$(a)$] If $A$ and $B$ are stable, then $A$ and $B$ are isomorphic as $X$-$C^\ast$-algebras if and only if they are $KK(X)$-equivalent. Moreover, for any invertible $x\in KK(X; A, B)$ there exists an $X$-equivariant isomorphism $\phi \colon A \xrightarrow \cong B$, unique up to asymptotic unitary equivalence (with multiplier unitaries), such that $KK(X; \phi) = x$.
\item[$(b)$] If $A$ and $B$ are unital, then $A$ and $B$ are isomorphic as $X$-$C^\ast$-algebras if and only if there is an invertible $x\in KK(X; A, B)$ such that $\Gamma_0(x)([1_A]_0) = [1_B]_0$. Moreover, for any such $x$ there is an $X$-equivariant isomorphism $\phi \colon A \xrightarrow \cong B$, unique up to asymptotic unitary equivalence, such that $KK(X; \phi) =x$.
\end{itemize}
\end{introtheorem}

One can avoid the actions of topological spaces in the statement of Theorem \ref{t:nonsimpleclass} by introducing the following notation. Say that two separable $C^\ast$-algebras $A$ and $B$ are \emph{ideal-related $KK$-equivalent} if there exists an order isomorphism $\Phi \colon \mathcal I(A) \xrightarrow \cong \mathcal I(B)$ such that the induced tight $\Prim A$-$C^\ast$-algebras $(A, \I_A)$ and $(B, \Phi \circ \I_A)$ are $KK(\Prim A)$-equivalent. One immediately obtains the following corollary.

\begin{introcorollary}\label{c:class}
Let $A$ and $B$ be separable, nuclear, $\mathcal O_\infty$-stable $C^\ast$-algebras. Then $A$ and $B$ are stably isomorphic if and only if they are ideal-related $KK$-equivalent.
\end{introcorollary}

Of course the above corollary can be formulated in such a way that the classification is strong, i.e.~so that the ideal-related $KK$-equivalence can be lifted to an isomorphism of the stabilised $C^\ast$-algebras, and a unital version if $A$ and $B$ are unital.

\subsection*{$K$-theoretic classification}

One thing that makes the Kirchberg--Phillips theorem highly applicable is Theorem \ref{t:KPUCT}; that UCT Kirchberg algebras are classified not just by $KK$-theory, but even by $K$-theory. The first similar result in the non-simple, purely infinite case was due to Rørdam \cite{Rordam-classsixterm} who showed that separable, nuclear, stable, purely infinite $C^\ast$-algebras $A$ satisfying the UCT, which contain \emph{exactly one} non-zero, proper (two-sided, closed) ideal $I$, also satisfying the UCT, are classified by the six-term exact sequence
\begin{equation}
\xymatrix{
K_0(I) \ar[r]^{\iota_0} & K_0(A) \ar[r]^{\pi_0\,\,\,} & K_0(A/I) \ar[d]^{\partial_0} \\
K_1(A/I) \ar[u]^{\partial_1} & K_1(A) \ar[l]_{\quad \pi_1} & K_1(I). \ar[l]_{\,\,\, \iota_1}
}
\end{equation}

So essentially this classification applies when $\mathcal I(A) = \{ 0, I , A\}$  and all ideals satisfy the UCT, using the induced six-term exact sequence in $K$-theory as the classifying invariant.
The following very natural question arises.

\begin{introquestion}
Are separable, nuclear, $\mathcal O_\infty$-stable $C^\ast$-algebras for which every two-sided, closed ideal satisfies the UCT classified (up to stable isomorphism) by a $K$-theoretic invariant?
\end{introquestion}

By ``a $K$-theoretic invariant" one would hope for an invariant for which $K$-theory plays a dominating part, in the same spirit as the six-term exact sequence above.
By Corollary \ref{c:class}, it suffices to show that any isomorphism of the $K$-theoretic invariant in question lifts to an ideal-related $KK$-equivalence.  One does have the following partial solution to the question, which is immediately obtained by combining Theorem \ref{t:nonsimpleclass} with \cite[Theorem 6.2]{Gabe-cplifting} and \cite[Theorem 4.6]{DadarlatMeyer-E-theory}. This paper does not contain new proofs of the results from \cite{Gabe-cplifting} and \cite{DadarlatMeyer-E-theory}.

\begin{introtheorem}
Let $X$ be a topological space, suppose that $A$ and $B$ are separable, nuclear, $\mathcal O_\infty$-stable, stable, tight $X$-$C^\ast$-algebras for which all ideals satisfy the UCT, and let $\alpha \in KK(X; A, B)$. If $\alpha$ induces an isomorphism in $K$-theory $\alpha(U) \colon K_\ast (A(U)) \xrightarrow \cong K_\ast(B(U))$ for every $U\in \mathcal O(X)$, then $\alpha$ lifts to an $X$-equivariant isomorphism $A\xrightarrow \cong B$.
\end{introtheorem}

This is unfortunately not an optimal solution to the above question since one needs to know that such an $\alpha$ exists before one can apply the theorem. However, one would only need to lift isomorphisms on the $K$-theoretic invariants to some $KK(X)$-elements without worrying about whether the lift is invertible or not, since invertibility comes for free by the above theorem.

The most general and systematic attack on the above problem is due to Meyer and Nest \cite{MeyerNest-homalg}, \cite{Meyer-homalg2} where they study the category of separable $X$-$C^\ast$-algebras (plus extra assumptions on the actions) with morphism sets $KK(X; A, B)$. This category turns out to be triangulated, and the question above reduces to doing homological algebra in such triangulated categories. This allowed for partial solutions to the above question assuming the $C^\ast$-algebras have finitely many ideals, see \cite{MeyerNest-filtrated}, \cite{BentmannKohler-UCTfinite}, \cite{Bentmann-vanishingbdry}, \cite{ArklintRestorffRuiz-classrr0}, \cite{BentmannMeyer-generalclass}, and \cite{Meyer-generalclassII}. While most applications of the Meyer--Nest techniques are for $C^\ast$-algebras with finitely many ideals, the general machinery is applicable much more broadly, and was for instance used for classification of certain continuous fields of Kirchberg algebras in \cite{DadarlatMeyer-E-theory} and \cite{BentmannDadarlat-oneparKirchberg}.

A few other approaches have also been used for classification of non-simple, purely infinite $C^\ast$-algebras. Restorff showed \cite{Restorff-classCK}, using symbolic dynamics instead of the Meyer--Nest framework and Theorem \ref{t:nonsimpleclass}, that purely infinite Cuntz-Krieger algebras are classified by an invariant which essentially consists of the $K$-theory of all subquotients $J/I$ of the $C^\ast$-algebra, where $I\subseteq J$ are ideals. This classification is however internal in the sense that it can only be used if one knows that \emph{both} $C^\ast$-algebras in question are purely infinite Cuntz--Krieger algebras. A much more general, but still internal, classification result was recently obtained by Eilers, Restorff, Ruiz, and Sørensen \cite{ERRS-classunitalgraph} where all unital graph $C^\ast$-algebras are classified by a $K$-theoretic invariant. This remarkable result reaches far beyond the purely infinite case, and does not use Theorem \ref{t:nonsimpleclass} to obtain classification.

A different approach is due to Dadarlat and Pennig \cite{DadarlatPennig-DixmierDouady}, where they prove a Dixmier--Douady type classification for certain continuous fields over compact, metrisable spaces $X$ for which the fibres are $\mathcal D\otimes \mathcal K$ for a fixed strongly self-absorbing $C^\ast$-algebra $\mathcal D$. When $\mathcal D$ is purely infinite and $X$ is finite dimensional, the $C^\ast$-algebras classified are separable, nuclear, and $\mathcal O_\infty$-stable by \cite{Dadarlat-findim} and are thus covered by Theorem \ref{t:nonsimpleclass}, although Dadarlat and Pennig prove the classification via other methods. The classifying invariant in this case is an induced element in a generalised cohomology group $\overline{E}^1_{\mathcal D}(X)$, see \cite{DadarlatPennig-DixmierDouady} and \cite{DadarlatPennig-DixmierDouadyII} for more details. These results might suggest that a complete $K$-theoretic invariant for classification should also depend on cohomological data of the space $X$ when $X$ is not zero-dimensional.

\subsection*{Notation}

Let $\mathcal K := \mathcal K(\ell^2(\mathbb N))$ denote the compact operators, and let $e_{i,j}$ denote the standard $(i,j)$'th matrix unit in $\mathcal K$ for $i,j\in \mathbb N$. Similarly $e_{i,j}$ denotes the standard matrix unit in the matrix algebra $M_n(\mathbb C)$. A $C^\ast$-algebra $B$ is \emph{stable} if $B\cong B \otimes \mathcal K$. 

The multiplier algebra of a $C^\ast$-algebra $B$ is denoted $\multialg{B}$, and the induced corona algebra $\multialg{B}/B$ is denoted by $\corona{B}$. 

There are two types of unitisations which play role in the paper: the \emph{minimal unitisation} $\widetilde B$, which is $B$ if $B$ was already unital; and the \emph{forced unitisation} $A^\dagger$ which is $A\oplus \mathbb C$ if $A$ was already unital. When $\rho \colon A \to B$ is a contractive completely positive map, then $\rho^\dagger \colon A^\dagger \to \widetilde B$ (or $\rho^\dagger \colon A^\dagger \to \multialg{B}$) is the induced unital map given by $\rho^{\dagger}(a+ \mu 1_{A^\dagger}) = \rho(a) + \mu 1_{\widetilde{B}}$ for $a\in A$ and $\mu \in \mathbb C$. Then $\rho^\dagger$ is completely positive by \cite[Proposition 2.2.1 and Lemma 2.2.3]{BrownOzawa-book-approx}, and is a $\ast$-homomorphism whenever $\rho$ is a $\ast$-homomorphism.  As a rule of thumb, the forced unitisation $A^\dagger$ is used for domains of maps, whereas minimal unitisations $\widetilde B$ and multiplier algebras $\multialg{B}$  are used for codomains.

\subsection*{Acknowledgement}
This is the product of many years of work, and has greatly benefited from conversations with a lot of people. To this extent I thank Joan Bosa, Jorge Castillejos, Marius Dadarlat, Efren Ruiz, Chris Schafhauser, Aidan Sims, Gábor Szabó, Simon Wassermann, Stuart White, and Wilhelm Winter for fruitful discussions. A lot more people have definitely indirectly played an important role, and I thank you all.

Parts of this project were completed during a research visit at the
Mittag–Leffler Institute during the programme Classification of Operator Algebras: Complexity, Rigidity, and Dynamics, and during a visit at the CRM institute during the programme IRP Operator Algebras: Dynamics and Interactions. I am thankful for their hospitality during these visits.

While I was a PhD student I completed certain crucial steps for the overall proof, at which time I was supported by the Danish National Research Foundation through the Centre for Symmetry and Deformation (DNRF92). This research was also funded by the Carlsberg Foundation through an internationalisation fellowship.

Finally, I would like to thank the referee for many helpful comments and suggestions.

\section{Equivalence of $\ast$-homomorphisms}\label{s:hom}

Approximate and asymptotic unitary equivalence of $\ast$-homomorphisms is often too strong of an equivalence relation for obtaining classification (or uniqueness results) of $\ast$-homomorphisms, at least when working with non-unital $\ast$-homo\-morphisms. This motivated the notion of approximate and asymptotic Murray--von Neumann equivalence. 
In the following $\mathbb R_+ := [0,\infty)$.

\begin{definition}[{\cite[Definition 3.4]{Gabe-O2class}}]\label{d:MvN}
Let $A$ and $B$ be $C^\ast$-algebras, and let $\phi,\psi \colon A \to B$ be $\ast$-homomorphisms. Say that $\phi$ and $\psi$ are \emph{approximately Murray--von Neumann equivalent}, written $\phi \sim_{\aMvN} \psi$, if for any finite set $\mathcal F \subset A$ and any $\epsilon >0$, there exists a contraction $w\in \multialg{B}$ such that
\begin{equation}
\| w^\ast \phi(a) w - \psi(a)\| < \epsilon, \qquad \| w \psi(a) w^\ast - \phi(a)\| < \epsilon
\end{equation}
for all $a\in \mathcal F$. If one may always pick $w$ to be a unitary, then $\phi$ and $\psi$ are said to be \emph{approximately unitarily equivalent}, written $\phi \sim_{\au} \psi$.

If $A$ is separable, say that $\phi$ and $\psi$ are \emph{asymptotically Murray--von Neumann equivalent}, written $\phi \sim_{\asMvN} \psi$, if there is a norm-continuous path $(v_t)_{t\in \mathbb R_+}$ of contractions in $\multialg{B}$, such that
\begin{equation}
\lim_{t\to \infty} \| v_t^\ast \phi(a) v_t - \psi (a) \| = 0 , \qquad \lim_{t\to \infty} \| v_t \psi(a) v_t^\ast - \phi (a) \| = 0
\end{equation}
for all $a\in A$. If one may pick each $v_t$ to be a unitary, then $\phi$ and $\psi$ are said to be \emph{asymptotically unitarily equivalent}, written $\phi \sim_{\asu} \psi$.
\end{definition}

\begin{remark}
The definition of approximate and asymptotic Murray--von Neumann equivalence in \cite[Definition 3.4]{Gabe-O2class} was slightly different than the above definition, but the definition above is equivalent. 

In \cite[Definition 3.4]{Gabe-O2class} it was required that each $w$ and $v_t$ was in $B$ instead of in $\multialg{B}$. Clearly that definition implies the condition in Definition \ref{d:MvN}. If $v_t\in \multialg{B}$ as in the above definition, let $(e_t)_{t\in \mathbb R_+}$ be a continuous approximate identity in $A$.\footnote{Such an approximate identity exists in any $\sigma$-unital $C^\ast$-algebra. For instance, one could fix an approximate identity $(e_n)_{n\in \mathbb N}$ and let $e_{n+r} = (1-r) e_n + r e_{n+1}$ for $r\in [0,1]$ and $n\in \mathbb N$.} Then $w_t := \phi(e_t) v_t \psi(e_t) \in B$, and
\begin{equation}
\lim_{t\to \infty} \| w_t^\ast \phi(a) w_t - \psi (a) \| = 0 , \qquad \lim_{t\to \infty} \| w_t \psi(a) w_t^\ast - \phi (a) \| = 0
\end{equation}
for all $a\in A$, so the two definitions are equivalent. The same holds in the approximate case.
\end{remark}

By \cite[Lemma 3.5]{Gabe-O2class} one does not need to assume that $w$ in Definition \ref{d:MvN} is contractive. However, it is convenient to assume that $w$ and all $v_t$ are contractions so that Remark \ref{r:MvNelement} below is more readily applicable.

The following was essentially contained in \cite[Proposition 3.12]{Gabe-O2class}. An addition has been made in a special case for obtaining asymptotic and approximate unitary equivalence with unitaries in the minimal unitisation instead of the multiplier algebra.

\begin{proposition}\label{p:MvNvsue}
Let $A$ and $B$ be $C^\ast$-algebras with $A$ separable, and let $\phi, \psi \colon A \to B$ be $\ast$-homomorphisms. If either
\begin{itemize}
\item[$(a)$] $A,B,\phi$ and $\psi$ are all unital, or
\item[$(b)$] $B$ is stable,
\end{itemize}
then $\phi \sim_{\aMvN} \psi$ if and only if $\phi \sim_{\au} \psi$, and $\phi \sim_{\asMvN} \psi$ if and only if $\phi \sim_{\asu} \psi$.

Additionally, in case $(b)$, if $B$ is $\sigma$-unital and contains a full projection, then the unitaries implementing the approximate and asymptotic unitary equivalences may be taken in the minimal unitisation $\widetilde B$ of $B$.
\end{proposition}
\begin{proof}
The first part is \cite[Proposition 3.12]{Gabe-O2class}. For the additional part, suppose that $B$ is $\sigma$-unital, stable with a full projection $p$. By Brown's stable isomorphism theorem \cite{Brown-stableiso}, $B\cong pBp \otimes \mathcal K(H)$ where $H$ is a separable, infinite dimensional Hilbert space. We assume without loss of generality that $B = pBp \otimes \mathcal K(H)$. Let $(\xi_n)_{n\in \mathbb N}$ be an orthonormal basis for $H$, and let $T_1,T_2\in \mathcal B(H)$ be given by $T_1 \xi_n = \xi_{2n-1}$ and $T_2 \xi_n = \xi_{2n}$ for $n\in \mathbb N$. In the proof of \cite[Lemma 3.12]{Gabe-O2class}, a norm-continuous unitary path $(U_t)_{t\in \mathbb R_+}$ in $\mathcal B(H)$ is constructed as follows: Let $(V_{k, t})_{t\in [k-1,k]}$ for $k\geq 2$ be a norm-continuous unitary path with $V_{k,k-1} = 1_H$, which restricts point-wise to the identity on $\mathrm{span}\{ \xi_k,\xi_{2k-1}\}^\perp$ and such that $V_{k,k}$ flips $\xi_k$ and $\xi_{2k-1}$. Defining $U_t := V_{k,t} V_{k-1,k-1} \cdots V_{2,2}$ for $t\in [k-1,k]$ and $k\geq 2$ gives a continuous unitary path, and it is shown in the proof of \cite[Lemma 3.12]{Gabe-O2class} that the path of endomorphisms $U_t T_1(-) T_1^\ast U_t^\ast$ on $\mathcal K(H)$ converges point-norm to $\id_{\mathcal K(H)}$. The construction shows that if $k\geq 2$ and $t\in [k-1,k]$, then $U_t$ decomposes as $W_t \oplus 1_{H_k^\perp}$ on $H_k \oplus H_k^\perp = H$, where $H_k = \mathrm{span}\{\xi_1, \dots ,\xi_{2k-1}\}$. Hence, since $\mathcal B(H_k \oplus 0) \subseteq \mathcal K(H)$, it follows for $t\in [k-1,k]$ that
\begin{equation}
u_t := p \otimes U_t = p \otimes ((W_t - 1_{H_k})\oplus 0) + p \otimes (1_{H_k}\oplus 1_{H_k^\perp}) \in (pBp \otimes \mathcal K(H))^\sim = \widetilde B.
\end{equation}
So $(u_t)_{t\in [1,\infty)}$ is a norm-continuous unitary path in $\widetilde B$. Let $s_i := p \otimes T_i \in \multialg{B}$ for $i=1,2$. Then $s_1,s_2$ are isometries for which $s_1s_1^\ast + s_2 s_2^\ast =1$. Note that $u_t s_1(-) s_1^\ast u_t^\ast\colon B \to B$ converges point-norm to $\id_B$. Hence $\phi$ and $\psi$ are asymptotically unitarily equivalent to $\phi_1 := s_1\phi(-)s_1^\ast$ and $\psi_1 := s_1 \psi(-) s_1^\ast$ respectively, with unitaries in $\widetilde B$. The isomorphism $\theta \colon B \to M_2(B)$ given by $\theta(b) = (s_i^\ast b s_j)_{i,j=1,2}$ satisfies $\theta(s_1 bs_1^\ast) = b \oplus 0 \in M_2(B)$ for every $b\in B$. Hence $\theta \circ \phi_1 = \phi \oplus 0$ and $\theta \circ \psi_1 = \psi \oplus 0$. By \cite[Proposition 3.10]{Gabe-O2class}, if $\phi \sim_\mathrm{a(s)MvN} \psi$ then $\phi \oplus 0 \sim_\mathrm{a(s)u} \psi \oplus 0$ with unitaries in the minimal unitisation of $M_2(B)$. Applying $\theta^{-1}$ (extended to the minimal unitisations) it follows that $\phi \sim_\asu \phi_1\sim_\mathrm{a(s)u} \psi_1\sim_\asu \psi$ with unitaries in $\widetilde B$.
\end{proof}

\begin{remark}\label{r:MvNelement}
For any $C^\ast$-algebra $B$, let 
\begin{equation}
B_\infty := \prod_\mathbb{N} B /\bigoplus_\mathbb N B, \quad \textrm{and} \quad B_\as := C_b(\mathbb R_+, B)/C_0(\mathbb R_+, B)
\end{equation}
 be the \emph{sequence algebra} and the \emph{path algebra} of $B$ respectively. Clearly $B$ embeds into $B_\infty$ and $B_\as$ as constant sequences and constant paths respectively. 

The following was observed in \cite[Observation 3.7]{Gabe-O2class} and will be used without reference: If $A$ is a separable $C^\ast$-algebra and $\phi, \psi \colon A \to B$ are $\ast$-homomorphisms then $\phi \sim_\aMvN \psi$ (respectively $\phi \sim_\asMvN \psi$) if and only if there is a contraction $v\in B_\infty$ (respectively $v\in B_\as$) such that
\begin{equation}
v^\ast \phi(a) v = \psi(a) , \quad \textrm{and} \quad v \psi(a) v^\ast = \phi(a)
\end{equation}
for all $a\in A$.

Such a contraction $v\in B_\infty$ (respectively $v\in B_\as$) will be said to implement the approximate (respectively asymptotic) Murray--von Neumann equivalence.
\end{remark}

The following lemma illustrates how Remark \ref{r:MvNelement} will be applied (with $D = B_\infty$ or $D = B_\as$).

\begin{lemma}[{\cite[Lemma 3.8]{Gabe-O2class}}]\label{l:conjhom}
Let $A$ and $D$ be $C^\ast$-algebras and let $\phi, \psi \colon A \to D$ be $\ast$-homomorphisms. Suppose that there is a contraction $v\in D$ such that $v^\ast \phi(-) v = \psi$. Then
\begin{itemize}
\item[$(a)$] $vv^\ast \in D \cap \phi(A)'$,
\item[$(b)$] $v^\ast v \psi(a) = \psi(a)$ for all $a\in A$,
\item[$(c)$] $\phi(a) v = v \psi(a)$ for all $a\in A$.
\end{itemize}
\end{lemma}

\begin{remark}
For any $\ast$-homomorphism $\phi \colon A \to B$ let $B_\infty \cap \phi(A)'$ denote the commutant of $\phi(A)\subseteq B \subseteq B_\infty$, and let
\begin{equation}
\mathrm{Ann}_{B_\infty} \phi(A) := \{ x \in B_\infty : x \phi(A) + \phi(A) x \subseteq \{0\} \}
\end{equation}
be the annihilator of $\phi(A)$ in $B_\infty$. When there is no cause of confusion, $\Ann \phi(A)$ will be written instead of $\mathrm{Ann}_{B_\infty}\phi(A)$. Clearly $\Ann\phi(A)$ is an ideal in $B_\infty \cap \phi(A)'$, and the quotient $(B_\infty \cap \phi(A)')/\Ann \phi(A)$ will play an important role in what follows. 

Similarly one obtains an ideal $\mathrm{Ann}_{B_\as} \phi(A) = \Ann\phi(A)$ in $B_\as \cap \phi(A)'$. 
\end{remark}

Relative commutants of the form $\frac{B_\infty \cap \phi(A)'}{\Ann \phi(A)}$ were studied extensively by Kirchberg in \cite{Kirchberg-Abel}. The following sums up some elementary properties of such relative commutants.

\begin{lemma}\label{l:relcombasic}
Let $A$ and $B$ be $C^\ast$-algebras for which $A$ is separable, and let $\phi, \psi \colon A \to B$ be $\ast$-homomorphisms.
\begin{itemize}
\item[$(a)$] If $C$ is a $C^\ast$-algebra containing $B$ as a hereditary $C^\ast$-subalgebra, then the inclusion $B \hookrightarrow C$ induces an isomorphism
\begin{equation}\label{eq:relcomBvsM(B)}
\frac{B_\infty \cap \phi(A)'}{\Ann\phi(A)} \xrightarrow{\quad\cong\quad} \frac{C_\infty \cap \phi(A)'}{\Ann\phi(A)}
\end{equation}
\item[$(b)$] The embedding $B \to M_2(B)$ into the (1,1)-corner, induces an isomorphism
\begin{equation}\label{eq:relcomcorner}
\frac{B_\infty \cap \phi(A)'}{\Ann\phi(A)} \xrightarrow{\quad \cong\quad} (1_{\tilde B}\oplus 0) \frac{M_2(B)_\infty \cap (\phi \oplus \psi)(A)'}{\Annn(\phi \oplus \psi)(A)} (1_{\tilde B}\oplus 0).
\end{equation}
\item[$(c)$] Suppose that $v\in B_\infty$ implements an approximate Murray--von Neumann equivalence $\phi \sim_\aMvN \psi$, i.e.~$v\in B_\infty$ is contraction for which $v^\ast \phi(-) v = \psi$ and $v\psi(-) v^\ast = \phi$. Then the c.p.~map $v^\ast(-) v \colon B_\infty \to B_\infty$ induces a (multiplicative) isomorphism
\begin{equation}
\frac{B_\infty \cap \phi(A)'}{\Ann\phi(A)} \xrightarrow{\quad \cong\quad} \frac{B_\infty \cap \psi(A)'}{\Ann\psi(A)}.
\end{equation}
The inverse is induced by $v(-)v^\ast$.
\end{itemize}
Moreover, the same statements hold if one replaces all sequence algebras with path algebras, and ``$\sim_\aMvN$'' with ``$\sim_\asMvN$''.
\end{lemma}

Note that using part $(a)$ above to replace $B$ with its minimal unitisation $\widetilde B$, it follows that the projection $1_{\tilde B}\oplus 0 \in \frac{M_2(B)_\infty \cap (\phi \oplus \psi)(A)'}{\Annn(\phi \oplus \psi)(A)}$ in part $(b)$ above is well-defined. 

\begin{proof}[Proof of Lemma \ref{l:relcombasic}]
We only prove the approximate version as the asymptotic version is virtually identical.

$(a)$: It is obvious that one gets a well-defined $\ast$-homo\-morphism in \eqref{eq:relcomBvsM(B)}, and it is injective since
\begin{equation}
B_\infty \cap \mathrm{Ann}_{C_\infty}\phi(A) = \mathrm{Ann}_{B_\infty}\phi(A).
\end{equation}
Let $x \in C_\infty \cap \phi(A)'$, and let $(e_n)_{n\in \mathbb N}$ be a countable approximate identity in $A$. Let $f\in B_\infty$ be the element induced by $\phi(e_n)_{n\in \mathbb N}$. Clearly $f \in B_\infty \cap \phi(A)'$, and $f \phi(a) = \phi(a) f = \phi(a)$ for all $a\in A$, so $1_{\widetilde C} - f \in \Ann\phi(A)$. Hence $f x f + \Ann\phi(A) = x + \Ann\phi(A)$. As $B_\infty$ is a hereditary $C^\ast$-subalgebra of $C_\infty$, it follows that $fxf \in B_\infty \cap \phi(A)'$ so the map \eqref{eq:relcomBvsM(B)} is surjective.

$(b)$: By part $(a)$ we may assume that $B$ is unital (otherwise replace $B$ by its unitisation $\widetilde B$). Clearly the $\ast$-homo\-morphism \eqref{eq:relcomcorner} is well-defined. If $b\in B_\infty \cap \phi(A)'$ satisfies that $b \oplus 0 \in \Annn(\phi\oplus \psi)(A)$, then $b\in \Ann\phi(A)$, so the map is injective. Finally, if $b = (b_{i,j})_{i,j=1,2} \in M_2(B)_\infty \cap (\phi \oplus \psi)(A)'$, then $(1_{\tilde B}\oplus 0) b (1_{\tilde B}\oplus 0) = b_{1,1} \oplus 0$, and clearly $b_{1,1} \in B_\infty \cap \phi(A)'$, so the map is surjective.

$(c)$: We first check that $v^\ast (B_\infty\cap  \phi(A)') v \subseteq B_\infty \cap \psi(A)'$, so let $b\in B_\infty \cap \phi(A)'$ and $a\in A$. By Lemma \ref{l:conjhom} we have $v\psi(a) = \phi(a) v$ and $v^\ast \phi(a) = \psi(a) v^\ast$. Hence
\begin{equation}
v^\ast b v \psi(a) = v^\ast b \phi(a) v = v^\ast \phi(a) b v = \psi(a) v^\ast b v.
\end{equation}
It follows that $v^\ast (-) v \colon B_\infty\cap  \phi(A)' \to B_\infty \cap \psi(A)'$ is a well-defined c.p.~map. If $x\in \Ann \phi(A)$, then
\begin{equation}
v^\ast x v \psi(a) = v^\ast x \phi(a) v = 0
\end{equation}
for any $a\in A$, and thus $v^\ast x v\in \Ann \psi(A)$. Hence $v^\ast(-) v$ induces a c.p.~map
\begin{equation}
\eta \colon \frac{B_\infty \cap \phi(A)'}{\Ann\phi(A)} \to \frac{B_\infty \cap \psi(A)'}{\Ann\psi(A)}.
\end{equation}
By Lemma \ref{l:conjhom}$(b)$ we have $vv^\ast \phi(a) = \phi(a)$ and $v^\ast v \psi(a) = \psi(a)$ for all $a\in A$. Hence if $b,c \in B_\infty \cap \phi(A)'$ and $a\in A$, then
\begin{eqnarray}
(v^\ast b v)(v^\ast c v) \psi(a) &=& v^\ast b vv^\ast c \phi(a) v  \nonumber\\
&=& v^\ast b v v^\ast\phi(a) c v  \nonumber\\
&=& v^\ast b \phi(a) c v  \nonumber\\
&=& (v^\ast b c v) \psi(a).
\end{eqnarray}
In particular, $(v^\ast b v)(v^\ast c v) - (v^\ast bc v) \in \Ann\phi(A)$, so $\eta$ is a $\ast$-homomorphism.

The exact same arguments as above show that $v(-)v^\ast$ induces a $\ast$-homo\-morphism
\begin{equation}
\rho \colon \frac{B_\infty \cap \psi(A)'}{\Ann\psi(A)} \to \frac{B_\infty \cap \phi(A)'}{\Ann\phi(A)}.
\end{equation}
Again by Lemma \ref{l:conjhom}$(b)$ we have $vv^\ast \phi(a) = \phi(a)$ for all $a\in A$. Thus, if $b\in B_\infty \cap \phi(A)'$ and $a\in A$ then
\begin{equation}
vv^\ast b vv^\ast \phi(a) = vv^\ast b \phi(a) = vv^\ast \phi(a) b = \phi(a) b = b \phi(a).
\end{equation}
Thus $vv^\ast b vv^\ast - b \in \Ann\phi(A)$, so the composition $\rho \circ \eta$ is the identity map. Similarly, the composition $\eta \circ \rho$ is the identity map, so the maps $\eta$ and $\rho$, which are induced by $v^\ast(-) v$ and $v(-)v^\ast$ respectively, are isomorphisms and each others inverses.
\end{proof}

The following is an extension of \cite[Proposition 3.10]{Gabe-O2class}.

\begin{proposition}\label{p:MvNeq}
Let $A$ and $B$ be $C^\ast$-algebras with $A$ separable, and let $\phi, \psi \colon A \to B$ be $\ast$-homomorphisms. The following are equivalent:
\begin{itemize}
\item[$(i)$] $\phi$ and $\psi$ are asymptotically Murray--von Neumann equivalent;
\item[$(ii)$] $\phi \oplus 0, \psi \oplus 0 \colon A \to M_2(B)$ are asymptotically unitarily equivalent;
\item[$(ii')$] $\phi \otimes e_{1,1} , \psi \otimes e_{1,1} \colon A \to B \otimes \mathcal K$ are asymptotically Murray--von Neumann equivalent;
\item[$(ii'')$] $\phi \otimes e_{1,1} , \psi \otimes e_{1,1} \colon A \to B \otimes \mathcal K$ are asymptotically unitarily equivalent;
\item[$(iii)$] The projections
\begin{equation}
\left( \begin{array}{cc} 1_{\widetilde B} & 0 \\ 0 & 0 \end{array} \right), \left( \begin{array}{cc} 0 & 0 \\ 0 & 1_{\widetilde B} \end{array} \right) \quad \in  \quad \frac{M_2(B)_\as \cap (\phi \oplus \psi)(A)'}{\Annn(\phi \oplus \psi)(A)}
\end{equation}
are Murray--von Neumann equivalent.\footnote{Here $\widetilde B$ is the minimal unitisation of $B$. These projections are well-defined by considering $C= \widetilde B$ in Lemma \ref{l:relcombasic}(a)}
\end{itemize}
The similar statement where one replaces ``asymptotically'' with ``approximately'', and ``$M_2(B)_\as$'' with ``$M_2(B)_\infty$'' also holds.
\end{proposition}
\begin{proof}
The proofs in the ``asymptotic'' and the ``approximate'' are virtually identical, so we only do the ``asymptotic'' version.

By \cite[Proposition 3.10]{Gabe-O2class}, $(i), (ii)$ and $(iii)$ are equivalent. By Proposition \ref{p:MvNvsue}$(b)$, $(ii')$ and $(ii'')$ are equivalent.

$(i) \Rightarrow (ii')$: If $v_t$ implements $\phi \sim_{\asMvN} \psi$, then $v_t \otimes e_{1,1}$ implements $\phi \otimes e_{1,1} \sim_{\asMvN} \psi \otimes e_{1,1}$.

$(ii') \Rightarrow (i)$: Note that $(1_{\multialg{B}} \otimes e_{1,1}) \multialg{B \otimes \mathcal K} (1_{\multialg{B}} \otimes e_{1,1}) = \multialg{B} \otimes e_{1,1} \cong \multialg{B}$. Thus, if $v_t\in \multialg{B \otimes \mathcal K}$ implements $\phi \otimes e_{1,1} \sim_{\asMvN} \psi \otimes e_{1,1}$, and $w_t \in \multialg{B}$ is such that $w_t \otimes e_{1,1} = (1_{\multialg{B}} \otimes e_{1,1}) v_t (1_{\multialg{B}} \otimes e_{1,1})$, then clearly $w_t$ implements $\phi \sim_{\asMvN} \psi$.
\end{proof}


\section{Approximate domination and nuclearity}

While one's aim might be to classify $\ast$-homomorphisms up to approximate or asymptotic Murray--von Neumann (or unitary) equivalence, one might want to aim lower to begin with. This is where approximate domination enters the picture.

\begin{definition}\label{d:approxdom}
Let $A$ and $B$ be $C^\ast$-algebras, let $\phi \colon A \to B$ be a $\ast$-homo\-morphism, and let $\rho \colon A \to B$ be a c.p.~map. Say that $\phi$ \emph{approximately dominates} $\rho$ if for any finite subset $\mathcal F \subset A$ and $\epsilon >0$ there are $n\in \mathbb N$ and $b_1, \dots, b_n \in B$ such that
\begin{equation}
\| \rho(a) - \sum_{i=1}^n b_i^\ast \phi(a) b_i \| < \epsilon, \qquad a\in \mathcal F.
\end{equation}
Say that $\phi$ \emph{approximately $1$-dominates} $\rho$ if the above holds with $n=1$.
\end{definition}

It is clear that if two $\ast$-homomorphisms are approximately Murray--von Neumann equivalent then they approximately (1-)dominate each other. So when one wants to prove that two $\ast$-homomorphisms are approximately Murray--von Neumann equivalent, a good starting point would be to prove that they approximately dominate each other. At first glance, even this seems like a very non-trivial task. However, this is exactly where nuclearity plays a fundamental role, see Corollary \ref{c:fulldom}.

\begin{remark}\label{r:approxdom}
Let $\mathscr C$ denote the set of all c.p.~maps which are approximately dominated by the $\ast$-homo\-morphism $\phi \colon A \to B$. It is clear that $\mathscr C$ is closed in the point-norm topology. Moreover, it follows immediately from the definition that if $\rho_1, \rho_2 \in \mathscr C$ then the c.p.~map $\rho_1 + \rho_2 $ is in $\mathscr C$. An easy consequence is that $\mathscr C$ is a point-norm closed, convex cone of completely positive maps. 
\end{remark}

Recall the following fundamental definition. 

\begin{definition}
A  map  between $C^\ast$-algebras is called \emph{nuclear} if it is a point-norm limit of maps factoring via c.p.~maps through matrix algebras. 
\end{definition}

Note that nuclear maps are automatically completely positive, so whenever a map is referred to as being nuclear it is implicit that it is completely positive. However, it is not necessarily assumed to be contractive, see Remark \ref{r:nuccontractive}.

\begin{observation}\label{o:nuccomp}
Clearly the composition of a nuclear map and a c.p.~map will again be nuclear. This elementary observation will be applied frequently without reference.
\end{observation}

\begin{remark}
Similarly as in Remark \ref{r:approxdom}, the set of nuclear maps between two $C^\ast$-algebras is a point-norm closed, convex cone. This observation is folklore (or an elementary exercise left to the reader).
\end{remark}

\begin{remark}\label{r:nuccontractive}
The above definition of nuclear maps is slightly different than the one often found in the literature (e.g.~\cite[Definition 2.1.1]{BrownOzawa-book-approx}) where all maps in question -- including the ones going in and out of the matrix algebras -- are assumed to be contractive. The following folklore lemma implies that the two definitions in question are equivalent for contractive maps. As I have not been able to find a reference in the literature, I have included a proof for the readers convenience.
\end{remark}

If $\rho \colon A \to B$ is a c.p.~map and $(e_\lambda)$ is an approximate identity in $A$, then $\| \rho \| = \lim_\lambda \| \rho(e_\lambda)\|$. Hence if $a\in A$, then $\| \rho(a^\ast(-)a) \| = \| \rho(a^\ast a)\|$. This elementary fact will be used somewhat frequently throughout the paper without mentioning.

\begin{lemma}\label{l:nuccontractive}
Any contractive nuclear map is a point-norm limit of maps factoring via \emph{contractive} c.p.~maps through matrix algebras.
\end{lemma}
\begin{proof}
Let $\rho \colon A \to B$ be a contractive nuclear map and let $(e_\lambda)_{\lambda \in \Lambda}$ be an approximate identity in $A$. Then $\rho(e_\lambda(-)e_\lambda)$ defines a net of contractive nuclear maps\footnote{Each map $\rho(e_\lambda(-) e_\lambda)$ is nuclear by Observation \ref{o:nuccomp}.} converging point-norm to $\rho$, so it suffices to show that each $\rho(e_\lambda(-) e_\lambda)$ is a point-norm limit of maps factoring via contractive c.p.~maps through matrix algebras.

Given a finite set of contractions $\mathcal F \subset A$, and $\epsilon >0$, pick c.p.~maps $\psi \colon A \to M_n(\mathbb C)$ and $\eta \colon M_n(\mathbb C) \to B$ such that $\eta(\psi(e_\lambda a e_\lambda)) \approx_{\epsilon/2} \rho(e_\lambda a e_\lambda)$ for $a\in \mathcal F \cup \{ 1_{\widetilde{A}}\}$. 
Let $y$ be the inverse of $\psi(e_\lambda^2)$ in the hereditary $C^\ast$-subalgebra of $M_n(\mathbb C)$ that $\psi(e_\lambda^2)$ generates. Define c.p.~maps 
\begin{equation}
\psi_0 = y^{1/2} \psi(e_\lambda(-) e_\lambda) y^{1/2} \colon A \to M_n(\mathbb C)
\end{equation}
and
\begin{equation}
\eta_0 = \tfrac{1}{1+\epsilon/2} \eta(\psi(e_\lambda^2)^{1/2} (-) \psi(e_\lambda^2)^{1/2}) \colon M_n(\mathbb C) \to B.
\end{equation}
Then
\begin{equation}
\| \psi_0 \| = \| y^{1/2} \psi(e_\lambda^2) y^{1/2} \| = 1, \qquad \| \eta_0\| =  \| \tfrac{1}{1+\epsilon/2} \eta(\psi(e_\lambda^2))\| \leq 1,
\end{equation}
where the latter holds since $\| \eta(\psi(e_\lambda^2)) \| \leq \| \rho(e_\lambda^2)\| + \epsilon/2 \leq 1+\tfrac{\epsilon}{2}$. Moreover,
\begin{equation}
\eta_0 \circ \psi_0 (a) \approx_{\epsilon/2} (1+\tfrac{\epsilon}{2}) \eta_0 \circ \psi_0(a) = \eta \circ \psi(e_\lambda a e_\lambda) \approx_{\epsilon/2} \rho(e_\lambda a e_\lambda), \quad a\in \mathcal F. \qedhere
\end{equation}
\end{proof}

It was proved independently by Choi--Effros \cite{ChoiEffros-CPAP} and Kirchberg \cite{Kirchberg-CPAP} that a $C^\ast$-algebra $A$ is nuclear (i.e.~the canonical $\ast$-epimorphism $A \otimes_{\max{}} C \to A \otimes C$ is an isomorphism for any $C^\ast$-algebra $C$) if and only if $\id_A$ is nuclear. This is an immediate consequence of the following tensor product characterisation of nuclear maps.

\begin{proposition}[Cf.~Choi--Effros \cite{ChoiEffros-CPAP}, Kirchberg \cite{Kirchberg-CPAP}]\label{p:nuctensor}
Let $A$ and $B$ be $C^\ast$-algebras and let $\phi \colon A \to B$ be a c.p.~map. Then $\phi$ is nuclear if and only if for any $C^\ast$-algebra $C$ the induced c.p.~map $\phi \otimes \id_C \colon A \otimes_{\max{}} C \to B \otimes_{\max{}} C$ factors through the spatial tensor product $A \otimes C$.
\end{proposition}
\begin{proof}
\cite[Corollary 3.8.8]{BrownOzawa-book-approx} is the result in the case where $A,B$ and $\phi$ are all unital. The general case can be obtained by normalising $\phi$ and unitising everything.
\end{proof}

\begin{remark}\label{r:nucemb}[Nuclear embeddability and exactness]
The symbol $\otimes$ denotes the \emph{spatial} (which is the minimal) tensor product of $C^\ast$-algebras. Recall that a $C^\ast$-algebra $A$ is called \emph{exact} if the functor $A \otimes -$ takes short exact sequences to short exact sequences. Say that $A$ is \emph{nuclearly embeddable} if there exists an injective, nuclear $\ast$-homomorphism $A \to B$ into some $C^\ast$-algebra $B$ (which may obviously be chosen to be $B = \mathcal B(\mathcal H)$).

By a remarkable result of Kirchberg \cite{Kirchberg-CAR} it follows that a $C^\ast$-algebra is exact if and only if it is nuclearly embeddable. See \cite[Section 3.9]{BrownOzawa-book-approx} for a more elementary proof. Also, if $A$ is a separable, exact $C^\ast$-algebra, one may always find an embedding $A \to \mathcal O_2$ by Kirchberg's $\mathcal O_2$-embedding theorem \cite{Kirchberg-ICM} (see also \cite{KirchbergPhillips-embedding}). Such an embedding is automatically nuclear by nuclearity of $\mathcal O_2$.

As the focus of this paper will mainly be on nuclear $\ast$-homomorphisms $A\to B$, it will essentially not be any loss of generality to assume that $A$ is exact, which will be done in most major results.
\end{remark}

Recall that an element $b\in B$ is called \emph{full} if it is not contained in any proper two-sided, closed ideal in $B$.

\begin{definition}
A $\ast$-homomorphism $\phi \colon A \to B$ between $C^\ast$-algebras is said to be \emph{full} if $\phi(a)$ is full in $B$ for every non-zero $a\in A$. 
\end{definition}

\begin{remark}\label{r:fullvssimple}
Fullness is for $\ast$-homomorphisms what simplicity is for $C^\ast$-algebras. In fact, if $A$ is a $C^\ast$-algebra then the identity map $\id_A$ is full if and only if $A$ is simple.
\end{remark}

The goal of this section is to prove the following proposition. 

\begin{proposition}\label{p:fulldom}
Let $A$ and $B$ be $C^\ast$-algebras and suppose that $\phi \colon A \to B$ is a full $\ast$-homomorphism. Then $\phi$ approximately dominates any nuclear map $\rho \colon A \to B$.
\end{proposition}

While a slight weakening of the above proposition (where $\phi$ is assumed to also be nuclear) would suffice for the purpose of proving the Kirchberg--Phillips theorem, and can be deduced almost immediately from \cite[Theorem 3.3]{Gabe-O2class}, I have chosen to take a longer yet more elementary approach for proving this. This approach should look familiar to those who are familiar with the completely positive approximation property for nuclear $C^\ast$-algebras.

Before proving the above result,  the following immediate corollary will be recorded for later use.

\begin{corollary}\label{c:fulldom}
Any two full, nuclear $\ast$-homomorphisms $\phi, \psi \colon A \to B$ approximately dominate each other.
\end{corollary}

In the following, $A^\ast_+$ denotes the convex cone of positive linear functionals on the $C^\ast$-algebra $A$ equipped with the weak$^\ast$-topology. Similarly, let $\CP(A,B)$ denote the convex cone of all c.p.~maps $A\to B$ equipped with the point-norm topology.

The following is well-known and can be proved essentially exactly like \cite[Proposition 1.5.14]{BrownOzawa-book-approx}.  The details are left for the reader.

\begin{proposition}\label{p:*+CP}
Let $A$ be a $C^\ast$-algebra and let $n\in \mathbb N$. Then there is an affine homeomorphism
\begin{equation}
(A \otimes M_n(\mathbb C))^\ast_+ \to \CP( A , M_n(\mathbb C) )
\end{equation} 
given by
\begin{equation}\label{eq:checkf}
f \mapsto \check f := \sum_{i,j=1}^n f(- \otimes e_{i,j}) e_{i,j}.
\end{equation}
\end{proposition}

A \emph{pure positive linear functional} on a $C^\ast$-algebra $A$ is a positive linear functional which is either zero or for which its normalisation is a pure state.

\begin{lemma}\label{l:nucsimpleapprox}
Let $A$ and $B$ be $C^\ast$-algebras and let $\rho \colon A \to B$ be a nuclear map. For any finite subset $\mathcal F \subset A$ and any $\epsilon >0$ there are $n,m\in \mathbb N$, pure positive linear functionals $f_1,\dots,f_m \in (A\otimes M_n(\mathbb C))^\ast_+$, and elements $b_{i,j} \in B$ for $i,j=1,\dots, n$, such that
\begin{equation}
\| \rho(a) - \sum_{l=1}^m \sum_{i,j,k=1}^n f_l(a\otimes e_{i,j}) b_{k,i}^\ast b_{k,j} \| < \epsilon , \qquad a\in \mathcal F.
\end{equation}
\end{lemma}
\begin{proof}
We may find  $n\in \mathbb N$ and c.p.~maps $\psi \colon A \to M_n(\mathbb C)$ and $\eta \colon M_n(\mathbb C) \to B$ such that 
\begin{equation}\label{eq:rhoetapsihalf}
\| \rho (a) - \eta \circ \psi(a) \| < \epsilon /2  ,\qquad a\in \mathcal F.
\end{equation}
 By (the proof of) \cite[Proposition 1.5.12]{BrownOzawa-book-approx} there are $b_{i,j} \in B$ such that 
\begin{equation}
\eta(e_{i,j}) = \sum_{k=1}^n b_{k,i}^\ast b_{k,j}, \qquad i,j=1,\dots,n.\footnote{One proves that $(\eta(e_{i,j}))_{i,j=1}^n \in M_n(B)$ is positive, and $b_{i,j}$ is the $(i,j)$'th entry of the square root of $(\eta(e_{i,j}))_{i,j=1}^n$.}
\end{equation}
By Proposition \ref{p:*+CP} there is a positive linear functional $f$ on $A\otimes M_n(\mathbb C)$ such that $\check f = \psi$ (see \eqref{eq:checkf}). In particular,
\begin{equation}\label{eq:etapsicheckf}
\eta \circ \psi(a) = \eta\circ \check f (a) = \sum_{i,j=1}^n f(a\otimes e_{i,j}) \eta(e_{i,j}) = \sum_{i,j,k=1}^n f(a\otimes e_{i,j}) b_{k,i}^\ast b_{k,j}
\end{equation}
for all $a\in A$.
We may find $m\in \mathbb N$ and pure positive linear functionals $f_1,\dots, f_m$ on $A\otimes M_n(\mathbb C)$ such that $\sum_{l=1}^m f_i$ approximates $f$ well enough on $\{ a \otimes e_{i,j} : a\in \mathcal F, \; i,j=1,\dots, n\}$ such that
\begin{equation}\label{eq:sumijkfaeij}
\|  \sum_{i,j,k=1}^n f(a\otimes e_{i,j}) b_{k,i}^\ast b_{k,j} - \sum_{l=1}^m \sum_{i,j,k=1}^n f_l(a\otimes e_{i,j}) b_{k,i}^\ast b_{k,j} \| < \epsilon/2
\end{equation}
for $a\in \mathcal F$. The result now follows by combining \eqref{eq:rhoetapsihalf}, \eqref{eq:etapsicheckf}, and \eqref{eq:sumijkfaeij}.
\end{proof}

It is well-known that if $A$ and $B$ are simple $C^\ast$-algebras then the spatial tensor product $A\otimes B$ is also simple. The following is a generalisation for $\ast$-homomorphisms, cf.~Remark \ref{r:fullvssimple}.

\begin{lemma}\label{l:fulltensor}
Let $A,B,C$ and $D$ be $C^\ast$-algebras and suppose that $\phi \colon A \to B$ and $\psi \colon C \to D$ are full $\ast$-homomorphisms. Then $\phi \otimes \psi \colon A \otimes C \to B \otimes D$ is a full $\ast$-homomorphism.

In particular, if $\phi \colon A \to B$ is a full $\ast$-homomorphism and $D$ is a simple $C^\ast$-algebra, then $\phi \otimes \id_D \colon A \otimes D \to B  \otimes D$ is a full $\ast$-homomorphism.
\end{lemma}
\begin{proof}
Let $x\in A \otimes C$ be a non-zero element. The (two-sided, closed) ideal in $A\otimes C$ generated by $x$ contains a non-zero elementary tensor $a\otimes c$ by Kirchberg's slice lemma \cite[Lemma 4.19]{Rordam-book-classification}, alternatively see \cite[Lemma 2.12$(ii)$]{BlanchardKirchberg-Hausdorff}. In particular, $(\phi \otimes \psi)(a\otimes c)$ is contained in the ideal generated by $(\phi \otimes \psi)(x)$. However, as $a$ and $c$ are non-zero elements, and as $\phi$ and $\psi$ are full, it follows that $(\phi \otimes \psi)(a\otimes c) = \phi(a) \otimes \psi(c)$ is a full element in $B\otimes D$. Hence $(\phi \otimes \psi)(x)$ is also full in $B\otimes D$, so $\phi\otimes \psi$ is a full $\ast$-homomorphism.

The ``in particular'' part follows since $\id_D$ is full provided $D$ is simple.
\end{proof}

A few elementary lemmas will be recorded for convenience. The main point of the following lemma is to obtain a norm estimate of a sum which does not dependent on the number of summands.

\begin{lemma}\label{l:sumconjest}
Let $B$ be a $C^\ast$-algebra and let $z,c_1,\dots,c_m\in B$. Then
\begin{equation}
\| \sum_{k=1}^m c_k^\ast z c_k \| \leq 4 \| z\| \| \sum_{k=1}^m c_k^\ast c_k\|.
\end{equation}
\end{lemma}
\begin{proof}
We may find positive $z_0,\dots,z_3 \in B$ such that $\| z_l \| \leq \| z\|$ for $l=0,\dots,3$ and $z = \sum_{l=0}^3 i^l z_l$. Hence
\begin{equation}
\| \sum_{k=1}^m c_k^\ast z c_k \| \leq \sum_{l=0}^3 \| \sum_{k=1}^m c_k^\ast z_l c_k \| \leq 4 \| z \| \|\sum_{k=1}^m c_k^\ast c_k \|. \qedhere
\end{equation}
\end{proof}

It is well-known (e.g.~\cite[Corollary II.5.2.13]{Blackadar-book-opalg}) that if $b,b_0 \in B$ are positive elements and $b_0 \in \overline{B b B}$,\footnote{As is customary, $\overline{BbB}$ denotes the \emph{closed linear span} of $\{ xby : x,y\in B\}$, i.e.~it is the two-sided, closed ideal generated by $b$.} then for any $\epsilon >0$ there are $m\in \mathbb N$ and $c_1,\dots, c_m \in B$ such that 
\begin{equation}
\| \sum_{k=1}^m c_k^\ast b c_k - b_0 \| < \epsilon.
\end{equation}
While one can always arrange that $\| \sum_{k=1}^m c_k^\ast b c_k \| \leq \| b_0\|$ (by perhaps perturbing the $c_k$'s from above slightly), one will in general not have control of the norm of the element $\sum_{k=1}^m c_k^\ast c_k$. The following lemma gives a criterion for when this norm can be controlled.

For a positive element $b$ in a $C^\ast$-algebra and $\delta >0$, $(b-\delta)_+$ denotes the element obtained by applying the function $t\mapsto \max{}(t-\delta,0)$ to $b$ via functional calculus. Also, for elements $x,y$ in a $C^\ast$-algebra and $\gamma>0$, write $x\approx_\gamma y$ when $\| x-y\|<\gamma$.

\begin{lemma}\label{l:sumidealgen}
Let $B$ be a $C^\ast$-algebra and suppose that $b,b_0\in B$ are positive, non-zero elements such that $b_0 \in \overline{B (b-\delta)_+ B}$ for every $0< \delta < \| b\|$.\footnote{Note that the interesting information is when $\delta$ is close to $\|b\|$.} Then for any $\epsilon>0$ there exist $m\in \mathbb N$ and $c_1,\dots, c_m\in B$ such that
\begin{equation}
\| \sum_{k=1}^m c_k^\ast b c_k - b_0\| < \epsilon
\end{equation}
and $\| \sum_{k=1}^m c_k^\ast c_k \| \leq \|b_0\|/\| b\|$. 
\end{lemma}
\begin{proof}
It suffices to prove the case $\| b\| = \| b_0 \| =1$. Also, we may assume that $\epsilon < 1$. Let $\delta = 1-\epsilon/8$, and let $g \colon [0,1] \to [0,1]$ be the continuous function
\begin{equation}
g(t) = \left\{ \begin{array}{ll}
0 , & t=0 \\
1, & t \in [\delta , 1] \\
\textrm{affine}, & \textrm{otherwise}
\end{array} \right. 
\end{equation}
for $t\in [0,1]$. In particular, $\| g(b) - b \| \leq \epsilon/8$, and $(b-\delta)_+ g(b) =  (b-\delta)_+$. The element $b_0$ is contained in the two-sided, closed ideal generated by $(b-\delta)_+$ by assumption, so we may find $c_1',\dots, c_m' \in B$ such that
\begin{equation}
\| \sum_{k=1}^m c_k'^\ast (b-\delta)_+ c_k' - b_0 \| < \epsilon /2,
\end{equation}
and $\| \sum_{k=1}^m c_k'^\ast(b-\delta)_+ c_k'\| \leq \| b_0\| = 1$. Let $c_k = (b-\delta)_+^{1/2} c_k'$ for $k=1,\dots, m$. Then
\begin{equation}\label{eq:sumckck}
\| \sum_{k=1}^m c_k^\ast c_k \| = \| \sum_{k=1}^m c_k'^\ast (b-\delta)_+ c_k'\|  \leq 1
\end{equation}
and
\begin{equation}
\sum_{k=1}^m c_k^\ast b c_k \stackrel{\eqref{eq:sumckck},~\textrm{Lem.~}\ref{l:sumconjest}}{\approx_{\epsilon/2}} \; \; \sum_{k=1}^m c_k^\ast g(b) c_k =  \sum_{k=1}^m c_k'^\ast (b-\delta)_+ c_k' \approx_{\epsilon/2} b_0
\end{equation}
as wanted.
\end{proof}

\begin{proof}[Proof of Proposition \ref{p:fulldom}]
By Remark \ref{r:approxdom} and Lemma \ref{l:nucsimpleapprox} it suffices to show that $\phi$ approximately dominates any map $\rho$ for which there is an $n\in \mathbb N$, a pure positive linear functional $f$ on $A\otimes M_n(\mathbb C)$, and $b_1,\dots, b_n\in B$ such that
\begin{equation}\label{eq:rhosumfaeij}
\rho(a) = \sum_{i,j=1}^n f(a \otimes e_{i,j}) b_{i}^\ast b_{j}, \qquad a\in A.
\end{equation}
If $f=0$ then $\rho = 0$ and the result obviously follows. Hence we may assume that $f$ is non-zero, and clearly also that $\| f\| =1$ so that $f$ is a pure state.

Let $\mathcal F \subset A$ be finite and let $\epsilon >0$. 
By \cite{AkemannAndersonPedersen-excision} we may excise $f$, so we may find a positive element $x \in A \otimes M_n(\mathbb C)$ of norm 1 such that
\begin{equation}\label{eq:excisionaeij}
\| x(a \otimes e_{i,j}) x -  f(a \otimes e_{i,j}) x^2\| < \frac{\epsilon}{8 n^2 \cdot \max\{ \| b_1\|^2, \dots, \| b_n\|^2\}}
\end{equation}
for $a\in \mathcal F$ and $i,j = 1,\dots, n$. Let 
\begin{equation}
\phi^{(n)} := \phi \otimes \id_{M_n(\mathbb C)} \colon A \otimes M_n(\mathbb C) \to B \otimes M_n(\mathbb C)
\end{equation}
which is full by Lemma \ref{l:fulltensor}. 
By fullness of $\phi^{(n)}$ Lemma \ref{l:sumidealgen} provides $c_1,\dots, c_m \in B\otimes M_n(\mathbb C)$ such that $\| \sum_{k=1}^m c_k^\ast c_k \| \leq 1$, and such that
\begin{align}\label{eq:sumfaeijbi}
& \bigg\| \sum_{i,j=1}^n f(a\otimes e_{i,j}) (b_i^\ast \otimes e_{1,1}) \left( \sum_{k=1}^m c_k^\ast \phi^{(n)} (x^2) c_k \right) (b_j \otimes e_{1,1}) \\
& \qquad - \sum_{i,j=1}^n f(a \otimes e_{i,j}) (b_i^\ast b_j) \otimes e_{1,1} \bigg\| < \epsilon /2 \nonumber
\end{align} 
for all $a\in \mathcal F$ (for instance by picking $\sum_{k=1}^m c_k^\ast \phi^{(n)} (x^2) c_k$ close to an approximate unit in $B\otimes M_n(\mathbb C)$).

Let $d_k\in B$ for $k=1,\dots, m$ be the unique element such that
\begin{equation}
d_k \otimes e_{1,1} = \sum_{j=1}^n \phi^{(n)}((1_{\widetilde A} \otimes e_{1,j} ) x) c_k (b_j \otimes e_{1,1}) \in B \otimes e_{1,1} \subseteq B \otimes M_n(\mathbb C).
\end{equation}

For $a\in \mathcal F$ we have (with computations in $B\otimes M_n(\mathbb C)$)
\begin{eqnarray}
&& \left(\sum_{k=1}^m d_k^\ast \phi(a) d_k \right) \otimes e_{1,1} \nonumber\\
&=& \sum_{k=1}^m (d_k \otimes e_{1,1})^\ast \phi^{(n)}(a \otimes e_{1,1}) (d_k \otimes e_{1,1}) \nonumber\\
&=&  \sum_{i,j=1}^n (b_i^\ast \otimes e_{1,1}) \left( \sum_{k=1}^m c_k^\ast \phi^{(n)} (x (a\otimes e_{i,j}) x) c_k \right) (b_j \otimes e_{1,1})\nonumber\\
&\stackrel{(\ast)}{\approx}_{\epsilon/2}& \sum_{i,j=1}^n (b_i^\ast \otimes e_{1,1}) \left( \sum_{k=1}^m c_k^\ast \phi^{(n)} ( f(a\otimes e_{i,j}) x^2) c_k \right) (b_j \otimes e_{1,1}) \nonumber\\
&\stackrel{\eqref{eq:sumfaeijbi} \; \;}{\approx_{\epsilon/2}}& \sum_{i,j=1}^n f(a\otimes e_{i,j})(b_i^\ast b_j) \otimes e_{1,1} \nonumber\\
&\stackrel{\eqref{eq:rhosumfaeij}}{=}& \rho(a) \otimes e_{1,1} 
\end{eqnarray}
where the estimate labeled $(\ast)$ above follows by applying Lemma \ref{l:sumconjest} with $z = x(a \otimes e_{i,j}) x -  f(a \otimes e_{i,j}) x^2$ and using \eqref{eq:excisionaeij}. Hence
\begin{equation}
\| \sum_{k=1}^m d_k^\ast \phi(a) d_k - \rho(a) \| < \epsilon
\end{equation}
for $a\in \mathcal F$ which finishes the proof.
\end{proof}


\section{$\mathcal O_2$-stable and $\mathcal O_\infty$-stable $\ast$-homomorphisms}\label{s:PIhom}

As a tool of getting from approximate domination to approximate 1-domination of maps, the following definition is very useful.

\begin{definition}[{\cite[Definition 3.16]{Gabe-O2class}}]
Let $\mathcal D$ be either $\mathcal O_2$ or $\mathcal O_\infty$. Let $A$ and $B$ be $C^\ast$-algebras with $A$ separable, and $\phi \colon A \to B$ be a $\ast$-homomorphism. Say that
\begin{itemize}
\item $\phi$ is \emph{$\mathcal D$-stable} if $\mathcal D$ embeds unitally in $B_\infty \cap \phi(A)' / \Ann \phi(A)$,
\item $\phi$ is \emph{strongly $\mathcal D$-stable} if $\mathcal D$ embeds unitally in $B_\as \cap \phi(A)' / \Ann \phi(A)$.
\end{itemize}
\end{definition}

\begin{remark}\label{r:Oinftypictures}
It is well-known that a unital $C^\ast$-algebra $D$ is properly infinite if and only if $\mathcal O_\infty$ embeds unitally into $D$, see \cite[Proposition 4.2.3]{Rordam-book-classification}.
Hence $\phi$ being $\mathcal O_\infty$-stable is equivalent to the induced sequential relative commutant being properly infinite. This means that there exist bounded sequences $(s_n^{(i)})_{n\in \mathbb N}$ in $B$ for $i=1,2$, such that
\begin{equation}
\lim_{n\to \infty} \| [s_n^{(i)}, \phi(a)] \| = 0, \qquad \lim_{n\to \infty} \| ((s_n^{(i)})^\ast s_n^{(j)} - \delta_{i,j}) \phi(a)\| = 0
\end{equation}
for all $i,j = 1,2$ and $a\in A$. Here $\delta_{i,j} = 1_{\widetilde{B}}$ if $i=j$ and $\delta_{i,j}=0$ if $i\neq j$. Similarly, $\phi$ is $\mathcal O_2$-stable exactly when one may additionally arrange that
\begin{equation}
\lim_{n\to \infty} \| (s_n^{(1)} (s_n^{(1)})^\ast + s_n^{(2)}(s_n^{(2)})^\ast - 1_{\widetilde{B}}) \phi(a) \| = 0
\end{equation}
for all $a\in A$. Similar characterisations hold for strongly $\mathcal O_2$-stable and strongly $\mathcal O_\infty$-stable $\ast$-homomorphisms, but with bounded, continuous paths $(s_t^{(i)})_{t\in \mathbb R_+}$. 
In particular, it follows that
\begin{equation}
\xymatrix{
\textrm{Strong $\mathcal O_2$-stability} \ar@{=>}[rr] \ar@{=>}[d] && \textrm{$\mathcal O_2$-stability} \ar@{=>}[d] \\
\textrm{Strong $\mathcal O_\infty$-stability} \ar@{=>}[rr] && \textrm{$\mathcal O_\infty$-stability}.
}
\end{equation}
\end{remark}

It was observed in \cite[Remark 3.24]{Gabe-O2class}, using the Kirchberg--Phillips theorem, that $\mathcal O_2$-stable $\ast$-homomorphisms are not necessarily strongly $\mathcal O_2$-stable. Also, the identity map on $\mathcal O_\infty$ is strongly $\mathcal O_\infty$-stable, but neither $\mathcal O_2$-stable nor strongly $\mathcal O_2$-stable. However, I do not know if $\mathcal O_\infty$-stability and strong $\mathcal O_\infty$-stability are equivalent. Thus, I emphasise the following question posed in \cite[Question 3.24]{Gabe-O2class}.

\begin{question}\label{q:Oinftyvsstrong}
Is every $\mathcal O_\infty$-stable $\ast$-homo\-mor\-phism also strong\-ly $\mathcal O_\infty$-stable?
\end{question}

In Section \ref{s:stronglyOinfty}, this question will be addressed. In particular, it is shown that if every unital, properly infinite $C^\ast$-algebra is $K_1$-injective -- which is another open problem, see \cite{BlanchardRohdeRordam-K1inj} -- then the question above has an affirmative answer.

\begin{observation}
Note that Lemma \ref{l:relcombasic}$(c)$ implies that (strong) $\mathcal D$-stabili\-ty is preserved by $\sim_{\mathrm{a(s)MvN}}$. This will be used without mentioning.
\end{observation}

The following says, that a $\ast$-homomorphism is strongly $\mathcal D$-stable if it factors through a $\mathcal D$-stable $C^\ast$-algebra, i.e.~a $C^\ast$-algebra $C$ for which $C \otimes \mathcal D \cong C$. A converse was shown for $\mathcal D$-stable maps in \cite[Corollary 4.5]{Gabe-O2class}, using what is Theorem \ref{t:uniquesimple} of this paper, by showing that $\mathcal D$-stable $\ast$-homomorphisms have a certain McDuff type property a la \cite{McDuff-centralseq}. A similar proof can be used to give a similar McDuff type characterisation for strongly $\mathcal D$-stable maps by applying \cite[Proposition 1.3.7]{Phillips-classification} instead of \cite[Theorem 4.3]{Gabe-O2class}.

The following is an immediate consequence of \cite[Propositions 3.17, 3.18 and Lemma 3.19]{Gabe-O2class}.

\begin{proposition}\label{p:Oinftyfactor}
Let $\mathcal D$ be either $\mathcal O_2$ or $\mathcal O_\infty$. Let $A,B$ and $C$ be $C^\ast$-algebras with $A$ separable, and let $\eta \colon A \to C$ and $\rho \colon C \to B$ be $\ast$-homo\-morphisms. If $C$ is $\mathcal D$-stable, then $\rho \circ \eta$ is strongly $\mathcal D$-stable.

In particular, if either $A$ or $B$ is $\mathcal D$-stable then any $\ast$-homomorphism $\phi \colon A \to B$ is strongly $\mathcal D$-stable.
\end{proposition}

\begin{remark}
It should be empahsised that the proof of \cite[Propositions 3.17 and 3.18]{Gabe-O2class}, and therefore Proposition \ref{p:Oinftyfactor} above, rely on the fact that $\mathcal O_2$ and $\mathcal O_\infty$ are strongly self-absorbing.\footnote{A separable, unital $C^\ast$-algebra $\mathcal D$ is \emph{strongly self-absorbing}, as defined in \cite{TomsWinter-ssa}, if $\mathcal D \not \cong \mathbb C$, and if there exists an isomorphism $\mathcal D \xrightarrow \cong \mathcal D \otimes \mathcal D$ which is approximately unitarily equivalent to the embedding $\id_\mathcal{D} \otimes 1_{\mathcal D} \colon \mathcal D \to \mathcal D \otimes \mathcal D$.} This highly non-elementary fact is crucial for the methods in this paper to work and cannot be deduced from the Kirchberg--Phillips Theorem -- Theorem \ref{t:KP} -- without some sort of circular argument. That $\mathcal O_2$ is strongly self-absorbing follows from \cite[Theorems 5.1.1 and 5.2.1]{Rordam-book-classification}, and that $\mathcal O_\infty$ is strongly self-absorbing follows from \cite[Proposition 7.2.5 and Theorem 7.2.6]{Rordam-book-classification}.
\end{remark}

\begin{remark}\label{r:fullpropinfproj}
If $\phi \colon A \to B$ is an $\mathcal O_\infty$-stable $\ast$-homomorphism, then every non-zero positive element in the image $\phi(A)$ is properly infinite in the sense of \cite[Definition 3.2]{KirchbergRordam-purelyinf}, see for instance \cite[Lemma 4.4]{BGSW-nucdim}. In particular, if $\phi$ is also full, then $B$ contains a properly infinite, full projection, even when there are no non-zero projections in $\phi(A)$. This follows from \cite[Lemma 7.2]{BGSW-nucdim}, but was essentially proved by Pasnicu and Rørdam in \cite[Proposition 2.7]{PasnicuRordam-purelyinfrr0} by using tricks of Blackadar and Cuntz from \cite{BlackadarCuntz-stablealgsimple}.
\end{remark}

The following is a Stinespring type theorem for (strongly) $\mathcal O_\infty$-stable $\ast$-homo\-morphisms which was essentially obtained in \cite{Gabe-O2class}.

\begin{theorem}[Cf.~{\cite[Theorem 3.22]{Gabe-O2class}}]\label{t:Stinespring}
Let $A$ and $B$ be $C^\ast$-algebras with $A$ separable, let $\phi \colon A \to B$ be a $\ast$-homomorphism, and let $\rho \colon A \to B$ be a c.p.~map which is approximately dominated by $\phi$. 
\begin{itemize}
\item[$(a)$] If $\phi$ is $\mathcal O_\infty$-stable there is an element $v\in B_\infty$ of norm $\| \rho\|^{1/2}$ such that $\rho = v^\ast \phi(-) v$.
\item[$(b)$] If $\phi$ is strongly $\mathcal O_\infty$-stable there is an element $v\in B_\as$ of norm $\| \rho\|^{1/2}$ such that $\rho = v^\ast \phi(-) v$.
\end{itemize}
\end{theorem}
\begin{proof}
This is exactly what is proved in \cite[Theorem 3.22]{Gabe-O2class}. The first sentence of said proof reduces the statement in \cite[Theorem 3.22]{Gabe-O2class} to the statement above, which is subsequently proved.
\end{proof}

\begin{lemma}\label{l:MvNsubeq}
Let $\phi , \psi \colon A \to B$ be $\ast$-homomorphisms between $C^\ast$-algebras for which $A$ is separable. Suppose that there is a contraction $v\in B_\as$ such that $v^\ast \phi(-) v = \psi$. Then the element
\begin{equation}
w := \left( \begin{array}{cc} 0 & v \\ 0 & 0 \end{array} \right) + \Annn(\phi\oplus \psi)(A) \in \frac{M_2(B)_\as \cap (\phi \oplus \psi)(A)'}{\Annn(\phi\oplus \psi)(A)},
\end{equation}
satisfies $ww^\ast \leq 1_{\widetilde B}\oplus 0$ and $w^\ast w = 0\oplus 1_{\widetilde B}$.

In particular, the projection 
\begin{equation}
\left( \begin{array}{cc} 1_{\widetilde B} & 0 \\ 0 & 0 \end{array} \right) \in \frac{M_2(B)_\as \cap (\phi \oplus \psi)(A)'}{\Annn(\phi\oplus \psi)(A)}
\end{equation}
is full.

The same result holds when replacing $M_2(B)_\as$ with $M_2(B)_\infty$.
\end{lemma}
\begin{proof}
Clearly the arguments given hold just as well for $M_2(B)_\infty$ as they do for $M_2(B)_\as$.

By Lemma \ref{l:conjhom}, $vv^\ast \in B_\as \cap \phi(A)'$, $v^\ast v \psi(a) = \psi(a)$ for $a\in A$, and $\phi(a) v = v \psi(a)$. Hence, it is straight forward to check, that $v\otimes e_{1,2}$ induces an element $w\in \frac{M_2(B)_\as \cap (\phi \oplus \psi)(A)'}{\Annn(\phi\oplus \psi)(A)}$, for which $ww^\ast \leq 1_{\widetilde B} \oplus 0$ and $w^\ast w = 0 \oplus 1_{\widetilde B}$.

``In particular'' is obvious, since $1_{\widetilde B}\oplus 1_{\widetilde B}= w^\ast (1_{\widetilde B}\oplus 0) w + 1_{\widetilde B}\oplus 0$ is the unit of $\frac{M_2(B)_\as \cap (\phi \oplus \psi)(A)'}{\Annn(\phi\oplus \psi)(A)}$.
\end{proof}

As a somewhat easy consequence, one obtains the following absorption result. This result is closely connected to Dadarlat's notion of \emph{absorbing} $\ast$-homomorphisms in \cite[Definition 2.2]{Dadarlat-htpygrpsKirchbergalg}.

Note that in part $(a)$ one wants $\theta$ to approximately dominate $\phi$, whereas $(b)$ requires that $\phi$ approximately dominates $\theta$. This is because part $(a)$ is used for existence theorems while part $(b)$ is used for uniqueness theorems.

\begin{proposition}\label{p:piabsorbing}
Let $A$ and $B$ be $C^\ast$-algebras with $A$ separable, and let $\phi, \theta \colon A \to B$ be $\ast$-homomorphisms.
\begin{itemize}
\item[$(a)$] If $\theta$ is (strongly) $\mathcal O_\infty$-stable and $\theta$ approximately dominates $\phi$, then $\phi \oplus \theta \colon A \to M_2(B)$ is (strongly) $\mathcal O_\infty$-stable,
\item[$(b)$] If $\phi$ is (strongly) $\mathcal O_\infty$-stable, if $\theta$ is (strongly) $\mathcal O_2$-stable, and if $\phi$ approximately dominates $\theta$, then $\phi\oplus 0 \sim_{\mathrm{a(s)MvN}} \phi \oplus \theta$ as maps $A \to M_2(B)$.
\end{itemize}
\end{proposition}
\begin{proof}
Only the ``non-strong'' statements  will be proved, since the ``strong'' statements are virtually identical.

$(a)$: To simplify notation, let
\begin{equation}
D := \frac{M_2(B)_\infty \cap (\phi \oplus \theta)(A)'}{\Annn (\phi\oplus \theta)(A)}.
\end{equation}
As $(0\oplus 1_{\widetilde B})D (0\oplus 1_{\widetilde B}) \cong B_\infty \cap \theta(A)' /\Ann\theta(A)$ by Lemma \ref{l:relcombasic}$(b)$, and since this $C^\ast$-algebra is properly infinite by $\mathcal O_\infty$-stability of $\theta$, it follows that the projection $0 \oplus 1_{\widetilde B} \in D$ is properly infinite. As $\theta$ is $\mathcal O_\infty$-stable and approximately dominates $\phi$, it follows from Theorem \ref{t:Stinespring} and Lemma \ref{l:MvNsubeq} that $0\oplus 1_{\widetilde B}$ is a full projection. Thus $1_{\widetilde B}\oplus 1_{\widetilde B}$ is properly infinite, which implies that $\phi \oplus \theta$ is $\mathcal O_\infty$-stable. 

$(b)$: Again, to simplify notation, let
\begin{equation}
E:=  \frac{M_4(B)_\infty \cap (\phi \oplus 0 \oplus \phi \oplus \theta)(A)'}{\Annn (\phi \oplus 0 \oplus \phi \oplus \theta)(A)}.
\end{equation}
By Proposition \ref{p:MvNeq} it suffices to show that $1_{\widetilde B}\oplus 1_{\widetilde B}\oplus 0 \oplus 0 \sim 0\oplus 0\oplus 1_{\widetilde B}\oplus 1_{\widetilde B}$ in $E$. First observe that
\begin{equation}
1_{\widetilde B}\oplus 1_{\widetilde B} \oplus 0 \oplus 0 = 1_{\widetilde B}\oplus 0 \oplus 0\oplus 0
\end{equation}
in $E$, and that $v= 1_{\widetilde B} \otimes e_{1,3}$ is a well-defined partial isometry in $E$ with
\begin{equation}
vv^\ast = 1_{\widetilde B}\oplus 0 \oplus 0 \oplus 0, \qquad v^\ast v = 0\oplus 0\oplus 1_{\widetilde B} \oplus 0.
\end{equation}
Hence it suffices to show that $0\oplus 0\oplus 1_{\widetilde B}\oplus 0\sim 0\oplus 0 \oplus 1_{\widetilde B} \oplus 1_{\widetilde B}$. 

Note that as in part $(a)$, $0\oplus 0 \oplus 1_{\widetilde B}\oplus 0$ is a properly infinite, full projection, and thus so is $0\oplus 0\oplus 1_{\widetilde B}\oplus 1_{\widetilde B}$. Hence by a result of Cuntz \cite{Cuntz-K-theoryI}, see also \cite[Proposition 4.1.4]{Rordam-book-classification}, it suffices to show that
\begin{equation}
[0\oplus 0\oplus 1_{\widetilde B}\oplus 0]_0 = [0\oplus 0 \oplus 1_{\widetilde B} \oplus 1_{\widetilde B}]_0 \in K_0(E),
\end{equation}
or equivalently, that $[0\oplus 0 \oplus 0 \oplus 1_{\widetilde B}]_0 = 0 \in K_0(E)$. As $\theta$ is $\mathcal O_2$-stable
\begin{equation}
(0\oplus 0 \oplus 0 \oplus 1_{\widetilde B})E (0\oplus 0 \oplus 0 \oplus 1_{\widetilde B}) \stackrel{\textrm{Lem.~\ref{l:relcombasic}}(b)}{\cong} \frac{B_\infty \cap \theta(A)'}{\Ann \theta(A)}
\end{equation}
contains a unital embedding of $\mathcal O_2$ and therefore $[0\oplus 0 \oplus 0 \oplus 1_{\widetilde B}]_0 = 0$.
\end{proof}


\section{Absorbing representations}\label{s:absrep}

In this section some (mostly well-known) results are presented on absorbing representations $A \to \multialg{B}$.

\subsection{Cuntz sums and infinite repeats}

This first subsection only contains well-known information on Cuntz sums and infinite repeats.

\begin{remark}[Cuntz sum]
Suppose $D$ is a unital $C^\ast$-algebra containing a pair $s_1,s_2$ of $\mathcal O_2$-isometries, i.e.~isometries such that $s_1s_1^\ast + s_2 s_2^\ast = 1_D$. The most important example to keep in mind is $D = \multialg{B}$  where $B$ is a stable $C^\ast$-algebra. For any elements $d,e\in D$ one forms the \emph{Cuntz sum} by
\begin{equation}
d\oplus_{s_1,s_2} e := s_1 d s_1^\ast + s_2 e s_2^\ast.
\end{equation}
Note that if $\Phi_{s_1,s_2} \colon M_2(D) \xrightarrow \cong D$ is the isomorphism 
\begin{equation}
\Phi_{s_1,s_2} ((d_{i,j})_{i,j=1,2}) = \sum_{i,j=1}^2 s_i d_{i,j} s_j^\ast
\end{equation}
then 
\begin{equation}
\Phi_{s_1,s_2}(d \oplus e) = d \oplus_{s_1,s_2} e. 
\end{equation}
Hence Cuntz sums are a variation of diagonal sums. Also, Cuntz sums are unique up to unitary equivalence. In fact, if $t_1,t_2$ are also $\mathcal O_2$-isometries in $D$ then $u := s_1 t_1^\ast + s_2 t_2^\ast$ is a unitary satisfying
\begin{equation}
u^\ast (d \oplus_{s_1,s_2} e) u = d \oplus_{t_1,t_2} e.
\end{equation}
If $\phi, \psi \colon A \to D$ are maps then one can form the (point-wise) Cuntz sum $\phi \oplus_{s_1,s_2} \psi \colon A \to D$ which is again unique up to unitary equivalence.
\end{remark}

\begin{remark}[Infinite repeats]\label{r:infrep}
Let $B$ be a stable $C^\ast$-algebra. There are several ``pictures'' of infinite repeats in $\multialg{B}$. These pictures will be recalled below, and it will be explained how they are connected. 

By stability of $B$ there exists a sequence $(t_k)_{k\in \mathbb N}$ of isometries in $\multialg{B}$ with orthogonal range projections, such that $\sum_{k=1}^\infty t_k t_k^\ast = 1_{\multialg{B}}$, with convergence in the strict topology. Fix such $(t_k)_{k\in \mathbb N}$ (this will be used in all pictures).

\emph{First picture}: If $\theta \colon A\to B$ is a map, then an \emph{infinite repeat} of $\theta$ is a map of the form $\theta_\infty := \sum_{k=1}^\infty t_k \theta(-) t_k^\ast$ with $(t_k)_{k\in \mathbb N}$ as above. Any two infinite repeats of $\theta$ are unitarily equivalent. In fact, if $(s_k)_{k\in \mathbb N}$ is another such sequence of isometries, then $u = \sum_{k=1}^\infty s_k t_k^\ast$ (strict convergence) is a unitary for which
\begin{equation}
u (\sum_{k=1}^\infty t_k \theta(-) t_k^\ast ) u^\ast = \sum_{k=1}^\infty s_k \theta(-) s_k^\ast.
\end{equation}

\emph{Second picture}: Consider $B$ as a right Hilbert $B$-module, so $\multialg{B} = \mathcal B(B)$. Let $B^{\oplus \infty} = B \oplus B \oplus \dots$ be the countable infinite direct sum of $B$ with itself and let $\theta^{\oplus \infty} \colon A \to \mathcal B(B^{\oplus \infty})$ denote the diagonal map
\begin{equation}
 \theta^{\oplus \infty} (a)(b_1,b_2,\dots) = (\theta(a) b_1, \theta(a) b_2, \dots), \qquad a\in A, (b_1,b_2,\dots) \in B^{\oplus \infty}.
\end{equation}
Then $u \in \mathcal B(B, B^{\oplus\infty})$ given by $u b = (t_1^\ast b , t_2^\ast b , \dots)$ for $b\in B$ is a unitary satisfying
\begin{equation}
u^\ast \theta^{\oplus\infty}(-) u = \sum_{k=1}^\infty t_k \theta(-) t_k^\ast \colon A \to \mathcal B(B) = \multialg{B}.
\end{equation}

\emph{Third picture}: There is an isomorphism $\Omega \colon B \otimes \mathcal K \xrightarrow \cong B$ given on elementary tensors by $b \otimes e_{i,j} = t_i b t_j^\ast$. Clearly this extends to an isomorphism $\multialg{\Omega} \colon \multialg{B\otimes \mathcal K} \xrightarrow \cong \multialg{B}$. The map $\theta \otimes 1_{\multialg{\mathcal K}} \colon A \to \multialg{B \otimes \mathcal K}$ can also be thought of as an infinite repeat since
\begin{equation}
\multialg{\Omega} \circ (\theta \otimes 1_{\multialg{\mathcal K}}) = \sum_{k=1}^\infty t_k \theta(-) t_k^\ast.
\end{equation}
\end{remark}

\begin{remark}[``$1+\infty= \infty$'']\label{r:infty+1}
Let $B$ be a stable $C^\ast$-algebra and $\phi \colon A \to \multialg{B}$ be a $\ast$-homomorphism. If $\phi_\infty$ is an infinite repeat of $\phi$, and if $s_1,s_2\in \multialg{B}$ are $\mathcal O_2$-isometries, then $\phi \oplus_{s_1,s_2} \phi_\infty$ is unitarily equivalent to $\phi_\infty$. Informally, this says, unsurprisingly, that ``$1+\infty=\infty$''.
\end{remark}

\subsection{General absorption}

\begin{definition}\label{d:abs}
Let $A$ be a separable $C^\ast$-algebra, let $B$ be a $\sigma$-unital, stable $C^\ast$-algebra, and let $\phi, \psi \colon A \to \multialg{B}$ be $\ast$-homomorphisms. Say that $\phi$ \emph{absorbs} $\psi$ if there is a sequence $(u_n)$ of unitaries in $\multialg{B}$ such that
\begin{itemize}
\item[$(a)$] $u_n^\ast (\phi \oplus_{s_1,s_2} \psi)(a) u_n - \phi(a) \in B$ for all $a\in A$ and all $n\in \mathbb N$,
\item[$(b)$] $\| u_n^\ast (\phi \oplus_{s_1,s_2} \psi)(a) u_n - \phi(a)\| \to 0$ for all $a\in A$.
\end{itemize}
Here $s_1,s_2\in \multialg{B}$ are $\mathcal O_2$-isometries.
\end{definition}

Note that absorption does not depend on the choice of $\mathcal O_2$-isometries as Cuntz sums are unique up to unitary equivalence.

By the following remark one may often reduce questions about absorption to the unital case.

\begin{remark}[Unitisations and absorption]\label{r:unitvsnonunitabs}
Let $A^\dagger$ denote the forced unitisation of the $C^\ast$-algebra $A$, i.e.~one adds a unit regardless if $A$ is unital or not.

Let $A$ be a separable $C^\ast$-algebra, let $B$ be a $\sigma$-unital, stable $C^\ast$-algebra and let $\phi, \psi \colon A \to \multialg{B}$ be $\ast$-homomorphisms. Let $\phi^\dagger , \psi^\dagger \colon A^\dagger \to \multialg{B}$ be the unitised $\ast$-homomorphisms of $\phi$ and $\psi$ respectively. If $u\in \multialg{B}$ is a unitary and $s_1,s_2\in \multialg{B}$ are $\mathcal O_2$-isometries, then
\begin{equation}
u^\ast (\phi^\dagger \oplus_{s_1,s_2} \psi^\dagger) (a+\mu 1_{A^\dagger}) u - \phi^\dagger(a+\mu 1_{A^\dagger}) = u^\ast (\phi \oplus_{s_1,s_2} \psi)(a) u - \phi(a)
\end{equation}
for all $a\in A$ and $\mu \in \mathbb C$. Hence $\phi$ absorbs $\psi$ if and only if $\phi^\dagger$ absorbs $\psi^\dagger$.
\end{remark}

The following abstract \emph{unital} absorption result of infinite repeats will be needed to prove a non-unital version.

\begin{proposition}[Cf.~\cite{Kasparov-Stinespring}, \cite{DadarlatEilers-classification}]\label{p:absunital}
Let $A$ be a separable, unital $C^\ast$-algebra, let $B$ be a $\sigma$-unital, stable $C^\ast$-algebra, let $\psi, \theta \colon A \to \multialg{B}$ be unital $\ast$-homo\-morphisms, and let $\theta_\infty$ be an infinite repeat of $\theta$. Suppose that $\theta$ approximately dominates (see Definition \ref{d:approxdom}) the c.p.~map
\begin{equation}
b^\ast \psi(-) b \colon A \to B
\end{equation}
for every $b\in B$. Then $\theta_\infty$ absorbs $\psi$.
\end{proposition}
\begin{proof}
By Remark \ref{r:infty+1} it suffices to show that $\theta_\infty$ absorbs an infinite repeat of $\psi$. By \cite[Theorem 2.13]{DadarlatEilers-classification} this is the case whenever the following is satisfied: for any $b\in B$ there is a bounded sequence $(x_n)_{n\in \mathbb N}$ in $B$ such that $x_n^\ast \theta_\infty(-) x_n$ converges point-norm to $b^\ast \psi (-) b$, and such that $\| c x_n\| \to 0$ for any $c\in B$. 

To check this, let $b\in B$. Fix $1_A \in \mathcal F_1 \subseteq \mathcal F_2\subseteq \dots$ finite sets such that $\bigcup \mathcal F_n$ is dense in $A$. For each $n\in \mathbb N$, we may find $d_1^{(n)},\dots, d_{m_n}^{(n)} \in B$ such that
\begin{equation}
\| b^\ast \psi(a) b - \sum_{k=1}^{m_n} d_k^{(n)\ast} \theta(a) d_k^{(n)} \| < 1/n , \qquad a\in \mathcal F_n.
\end{equation}
Let $(t_j)_{j\in \mathbb N}$ be a sequence of isometries in $\multialg{B}$ with orthogonal range projections such that $\sum_{j=1}^\infty t_j t_j^\ast = 1_{\multialg{B}}$ strictly. As infinite repeats are unique up to unitary equivalence, we may assume that $\theta_\infty = \sum_{j=1}^\infty t_j \theta(-) t_j^\ast$.

Let $x_n := \sum_{k=1}^{m_n} t_{n + k} d_{k}^{(n)}$. The sequence $(x_n)_{n\in \mathbb N}$ is bounded since
\begin{eqnarray}
\| x_n^\ast x_n \| &=& \| \sum_{k,l=1}^{m_n} d_k^{(n)\ast} t_{n+k}^\ast t_{n+l} d_l^{(n)} \| \nonumber\\
&=& \| \sum_{k=1}^{m_n} d_k^{(n)\ast} d_k^{(n)} \| \nonumber\\
&=& \| \sum_{k=1}^{m_n} d_k^{(n)\ast} \theta(1_A) d_k^{(n)} \| \nonumber\\
&\xrightarrow{n\to \infty}& \| b^\ast \psi(1_A) b\|.
\end{eqnarray}
Moreover, for any $a\in A$ we have
\begin{eqnarray}
x_n^\ast \theta_\infty(a) x_n &=& \sum_{k,l=1}^{m_n} \sum_{j=1}^\infty d_k^{(n)\ast} t_{n+k}^\ast t_j \theta(a) t_j^\ast t_{n+l} d_l^{(n)} \nonumber\\
&=& \sum_{k=1}^{m_n} d_k^{(n)\ast} \theta(a) d_k^{(n)} \nonumber\\
&\xrightarrow{n\to \infty}& b^\ast \psi(a) b.
\end{eqnarray}
Finally, let $c\in B$. Then
\begin{eqnarray}
\| c x_n\| &=& \| c  (1_{\multialg{B}} - \sum_{k=1}^n t_kt_k^\ast)  x_n \| \nonumber\\
& \leq & \| c (1_{\multialg{B}} - \sum_{k=1}^n t_kt_k^\ast)\| \|  x_n \| \nonumber\\
&\xrightarrow{n\to \infty}& 0.
\end{eqnarray}
Here we used that $t_j^\ast x_n = 0$ for $j=1,\dots, n$, the sequence $(x_n)_{n\in \mathbb N}$ is bounded, and and that $\sum_{k=1}^\infty t_k t_k^\ast = 1_{\multialg{B}}$ strictly.
\end{proof}

While the above proposition is proved in the unital case, it is often more desirable to obtain the results in the not necessarily unital case. When reducing the not necessarily unital case to the unital case, the $C^\ast$-algebra $A$ will often be replaced with its forced unitisation $A^\dagger$, see Remark \ref{r:unitvsnonunitabs}. The following is a useful trick for relating a c.p.~map $\rho \colon A^\dagger \to B$ to its restriction to $A$.

\begin{lemma}\label{l:unitisetrick}
Let $\rho \colon A^\dagger \to B$ be a c.p.~map, let $\pi \colon A^\dagger \to \mathbb C$ be the character for which $\pi(A) = \{0\}$, and let $(e_\lambda)_{\lambda \in \Lambda}$ be an approximate identity in $A$. Define the c.p.~maps $\rho_\lambda , \pi_\lambda \colon A^\dagger \to B$ for each $\lambda \in \Lambda$ by
\begin{equation}
\rho_\lambda = \rho|_A (e_\lambda(-)e_\lambda), \qquad \pi_\lambda = \rho(1_{A^\dagger} - e_\lambda^2) \pi(-).
\end{equation}
Then $\rho_\lambda + \pi_\lambda \to \rho$ in point-norm.
\end{lemma}
\begin{proof}
For $a\in A$ and $\mu \in \mathbb C$ we have
\begin{eqnarray}
&& \rho_\lambda(a+\mu 1_{A^\dagger}) + \pi_\lambda(a+\mu 1_{A^\dagger}) \nonumber\\
&=& \rho(e_\lambda a e_\lambda) + \mu \rho( e_\lambda^2) + \mu \rho(1_{A^\dagger}  - e_\lambda^2) \nonumber\\
&=& \rho(e_\lambda a e_\lambda + \mu 1_{A^\dagger}) \nonumber\\
&\to& \rho( a + \mu 1_{A^\dagger}),
\end{eqnarray}
which is what we wanted to prove.
\end{proof}

Although the following lemma will not be applied until much later, it easily illustrates how the above trick works.

\begin{lemma}\label{l:unitisednuc}
A c.p.~map $\rho \colon A^\dagger \to B$ is nuclear if and only if the restriction $\rho|_A$ is nuclear.
\end{lemma}
\begin{proof}
``Only if'' is obvious so we prove ``if''. Suppose $\rho|_A$ is nuclear, and define $\rho_\lambda$ and $\pi_\lambda$ as in Lemma \ref{l:unitisetrick}. As the maps $\rho|_A$ and $\pi$ are nuclear, so are $\rho_\lambda$ and $\pi_\lambda$. As sums of nuclear maps are again nuclear, $\rho_\lambda + \pi_\lambda$ is nuclear for each $\lambda$. As the set of nuclear maps is point norm closed, $\rho$ is nuclear by Lemma \ref{l:unitisetrick}.
\end{proof}

As a consequence of Proposition \ref{p:absunital} above one obtains the following non-unital analogue of the result. In this version it is crucial, yet somewhat subtle, that $\theta$ takes values in $B$ (not in $\multialg{B}$) and that $B$ is stable. For instance, if $A$, $B$ and $\theta$ were unital, but $\psi$ was non-unital, the criteria of Proposition \ref{p:absnonunital} could be true but $\theta_\infty$ could never absorb $\psi$ since $\theta_\infty$ would be unital and thus can only absorb unital $\ast$-homomorphisms.

Similar results appear in the literature where both $A$ and $B$ are assumed to be unital, such as \cite[Theorem 2.22]{DadarlatEilers-classification}, but usually such results are not stated in as abstract a form as the following result.

\begin{proposition}\label{p:absnonunital}
Let $A$ be a separable $C^\ast$-algebra, let $B$ be a $\sigma$-unital, stable $C^\ast$-algebra, let $\psi \colon A \to \multialg{B}$ be a $\ast$-homomorpshim, let $\theta \colon A \to B$ be a $\ast$-homomorphism, and let $\theta_\infty \colon A \to \multialg{B}$ be an infinite repeat of $\theta$. Suppose that $\theta$ approximately dominates (see Definition \ref{d:approxdom}) the c.p.~map 
\begin{equation}
b^\ast \psi(-) b \colon A \to B
\end{equation}
for every $b\in B$. Then $\theta_\infty$ absorbs $\psi$.
\end{proposition}
\begin{proof}
Let $(\theta_\infty)^\dagger, \psi^\dagger \colon A^\dagger \to \multialg{B}$ be the forced unitised $\ast$-homo\-morphisms. As observed in Remark \ref{r:unitvsnonunitabs}, $\theta_\infty$ absorbs $\psi$ if and only if $(\theta_\infty)^\dagger$ absorbs $\psi^\dagger$. Note that if $\theta^\dagger \colon A^\dagger \to \multialg{B}$ is the unitisation then $(\theta^\dagger)_\infty = (\theta_\infty)^\dagger$ where we use the same sequence of isometries to define the two infinite repeats. Hence it suffices to check that $\theta^\dagger$ and $\psi^\dagger$ these satisfy the condition of Proposition \ref{p:absunital}.

Fix $b\in B$ so that we wish to check that $\theta^\dagger$ approximately dominates $b^\ast \psi^\dagger(-)b$. Let $\pi \colon A^\dagger \to \mathbb C$ be the character vanishing on $A$, and let $(e_n)_{n\in \mathbb N}$ be an approximate identity in $A$. Define the c.p.~maps $A^\dagger \to B$
\begin{equation}
\psi_n := b^\ast \psi(e_n (-) e_n) b, \qquad \pi_n := b^\ast(1_{\multialg{B}} - e_n^2)b \pi(-).
\end{equation}
By Lemma \ref{l:unitisetrick} it follows that
\begin{equation}
\psi_n + \pi_n \to b^\ast \psi^\dagger(-) b
\end{equation}
in point-norm. As the set of c.p.~maps which are approximately dominated by $\phi$ is a point-norm closed, convex cone, see Remark \ref{r:approxdom}, it suffices to show that $\theta^\dagger$ approximately dominates $\psi_n$ and $\pi_n$ for each $n\in \mathbb N$, so fix $n\in \mathbb N$.

Let $\mathcal F\subset A^\dagger$ and $\epsilon>0$. As $\theta$ approximately dominates $b^\ast \psi(-) b$, we may find $d_1,\dots,d_m\in B$ such that
\begin{eqnarray}
&& \| b^\ast \psi(e_n x e_n) b - \sum_{k=1}^m d_k^\ast \theta(e_n x e_n) d_k \| \nonumber\\
& =& \| \psi_n(x) - \sum_{k=1}^m d_k^\ast \theta(e_n) \theta^\dagger(x) \theta(e_n) d_k\| \nonumber\\
& <& \epsilon,
\end{eqnarray}
for $x\in \mathcal F$. Hence $\theta^\dagger$ approximately dominates $\psi_n$.

As $B$ is stable we may pick a sequence $(t_j)_{j\in \mathbb N}$ of isometries in $\multialg{B}$ with orthogonal range projections such that $\sum_{j=1}^\infty t_j t_j^\ast =1_{\multialg{B}}$ in the strict topology. As $\theta(A) \subseteq B$ it follows that
\begin{equation}
t_j^\ast \theta^\dagger(x) t_j \xrightarrow{j\to \infty} 1_{\multialg{B}} \pi(x), \qquad x\in A^\dagger.
\end{equation}
Letting $c:= (b^\ast(1_{\multialg{B}} - e_n^2)b)^{1/2}$ it follows that
\begin{equation}
ct_j^\ast \theta^\dagger(x) t_j c \xrightarrow{j\to \infty}  b^\ast(1_{\multialg{B}} - e_n^2)b \pi(x) = \pi_n(x)
\end{equation}
for all $x\in A^\dagger$. Hence $\theta^\dagger$ approximately dominates $\pi_n$ which finishes the proof.
\end{proof}


\subsection{Nuclear absorption}

\begin{definition}\label{d:weaklynuc}
Let $A$ and $B$ be $C^\ast$-algebras. A c.p.~map $\rho \colon A \to \multialg{B}$ is called \emph{weakly nuclear} if $b^\ast \rho (-) b \colon A \to B$ is nuclear for every $b\in B$.
\end{definition}

In the literature, for instance in \cite{Skandalis-KKnuc} and \cite{DadarlatEilers-classification}, one often considers the notion of \emph{strictly nuclear} contractive c.p.~maps $\rho \colon A \to \multialg{B}$, i.e.~$\rho$ is a point-strict limit of maps factoring via contractive c.p.~maps through matrix algebras.

By the following folklore result, this is equivalent to being weakly nuclear.

\begin{proposition}
Let $A$ and $B$ be $C^\ast$-algebras, and let $\rho \colon A \to \multialg{B}$ be a contractive c.p.~map. The following are equivalent:
\begin{itemize}
\item[$(i)$] $\rho$ is weakly nuclear;
\item[$(ii)$] $\rho$ is strictly nuclear, i.e.~$\rho$ is a point-strict limit of c.p.~maps factoring via contractive c.p.~maps through matrix algebras;
\item[$(iii)$] $\rho$ is a point-strict limit of nuclear maps.
\end{itemize}
\end{proposition}
\begin{proof}
$(i)\Rightarrow (ii)$: If $\rho$ is weakly nuclear and if $(e_\lambda)$ is an approximate identity in $B$, then $(e_\lambda \rho(-) e_\lambda)$ is a net of contractive nuclear maps converging point-strictly to $\rho$. In fact, for any $a\in A$ and $b\in B$ we have
\begin{equation}
\| (\rho(a) - e_\lambda \rho(a) e_\lambda) b \| \leq \| \rho(a)b - e_\lambda \rho(a) b \| + \| e_\lambda\| \| \rho(a)b - \rho(a) e_\lambda b \| \to 0.
\end{equation}
Similarly $\| b(\rho(a) - e_\lambda \rho(a) e_\lambda)\| \to 0$. As each $e_\lambda \rho(-) e_\lambda$ is contractive and nuclear, it point-norm approximately factors via contractive c.p.~maps through matrix algebras, see Lemma \ref{l:nuccontractive}, so it easily follows that $\rho$ is the point-strict limit of c.p.~maps factoring via contractive c.p.~maps through matrix algebras.

$(ii) \Rightarrow (iii)$: This is obvious.

$(iii) \Rightarrow (i)$: Suppose $(\rho_\lambda \colon A \to \multialg{B})_\lambda$ is a net of nuclear maps converging point-strictly to $\rho$. For any $a\in A$ and $b\in B$ we have
\begin{equation}
\| b^\ast \rho(a) b - b^\ast \rho_\lambda(a) b \| \leq \| b\| \| (\rho(a)  - \rho_\lambda(a)) b\| \to 0.
\end{equation}
Hence the net $(b^\ast \phi_\lambda(-) b \colon A \to B)_\lambda$ of nuclear maps converges point-norm to $b^\ast \phi(-) b$ which is therefore nuclear.
\end{proof}

\begin{definition}
Let $A$ be a separable $C^\ast$-algebra and let $B$ be a $\sigma$-unital, stable $C^\ast$-algebra. A $\ast$-homomorphism $\phi \colon A \to \multialg{B}$ is called \emph{nuclearly absorbing} if it absorbs any weakly nuclear $\ast$-homomorphism $A \to \multialg{B}$.
\end{definition}

\begin{remark}[Unital nuclear absorption]
It is obvious that a unital $\ast$-homo\-morphism $\phi \colon A \to \multialg{B}$ can never absorb a non-unital $\ast$-homo\-morphism, in particular, a unital $\ast$-homomorphism cannot be nuclearly absorbing as it will not absorb the zero $\ast$-homomorphism (which is obviously weakly nuclear).
One therefore says that a unital $\ast$-homomorphism $\phi \colon A \to \multialg{B}$ is \emph{unitally nuclearly absorbing} if it absorbs any unital, weakly nuclear $\ast$-homo\-morphism $A \to \multialg{B}$. 

Using Remark \ref{r:unitvsnonunitabs} and Lemma \ref{l:unitisednuc} it is not hard to see  that $\phi \colon A \to \multialg{B}$ is nuclearly absorbing if and only if the (forced) unitisation $\phi^\dagger \colon A^\dagger \to \multialg{B}$ is unitally nuclearly absorbing.

Although the unital case is very important for studying non-stable $\Ext$-theory, for instance \cite{ElliottKucerovsky-extensions}, it will not play a significant role in this paper.
\end{remark}

The following main theorem of this section -- which is a non-unital version of \cite[Theorem 2.22]{DadarlatEilers-classification} -- is an easy corollary of previous results. As was also noted before Proposition \ref{p:absnonunital}, it is crucial both that $B$ is stable and that $\theta(A) \subseteq B$ for the following to hold. 

\begin{theorem}\label{t:fullnucabs}
Let $A$ be a separable $C^\ast$-algebra, let $B$ be a $\sigma$-unital, stable $C^\ast$-algebra, and suppose that $\theta \colon A \to B$ is a full $\ast$-homomorphism. Then any infinite repeat $\theta_\infty \colon A \to \multialg{B}$ of $\theta$ is nuclearly absorbing.

If, in addition, $\theta$ is nuclear (in which case $A$ must be exact) then $\theta_\infty$ is weakly nuclear and nuclearly absorbing.
\end{theorem}
\begin{proof}
The first part follows immediately from Propositions \ref{p:fulldom} and \ref{p:absnonunital}.

For the ``in addition'' part, we may assume that $\theta_\infty = \sum_{j=1}^\infty t_j \theta(-) t_j^\ast$, where $(t_j)_{j\in \mathbb N}$ is a sequence of isometries in $\multialg{B}$ with orthogonal range projections such that $\sum_{j=1}^\infty t_jt_j^\ast = 1_{\multialg{B}}$. Given $b\in B$, the c.p.~map $b^\ast \theta_\infty(-)b$ is the point-norm limit of the sequence $(\sum_{j=1}^N b^\ast t_j \theta(-) t_j^\ast b)_{N\in \mathbb N}$. As $\theta$ is nuclear, each map $\sum_{j=1}^N b^\ast t_j \theta(-) t_j^\ast b$ is nuclear and thus $b^\ast \theta_\infty(-) b$ is nuclear. Hence $\theta_\infty$ is weakly nuclear.
\end{proof}


\section{Asymptotic intertwining}

A celebrated way of classifying $C^\ast$-algebras is by using an intertwining argument a la Elliott \cite{Elliott-AFclass}. This argument is also used to \emph{lift} isomorphisms of a given invariant to isomorphisms of the $C^\ast$-algebras (see Remark \ref{r:apint} for a minor mistake in the literature in this context).

In this section it is shown that if $\phi_0 \colon A \to B$ and $\psi_0 \colon B \to A$ are $\ast$-homomorphisms of separable $C^\ast$-algebras such that $\psi_0 \circ \phi_0 \sim_{\asu} \id_A$ and $\phi_0 \circ \psi_0 \sim_{\asu} \id_B$, then there is an isomorphism $\phi \colon A \xrightarrow \cong B$ which is \emph{homotopic} to $\phi_0$. Moreover, this homotopy may be chosen to be very well-behaved (for instance, it will be ideal-preserving). This is used in Theorems \ref{t:KP} and \ref{t:nonsimpleclass} to show that the (ideal-related) $KK$-equivalence lifts to an isomorphism of $C^\ast$-algebras.

Given a unitary $u$, let $\Ad u = u^\ast (-) u$ be the induced inner automorphism.\footnote{This definition is only recalled to emphasise that $\Ad u$ denotes $u^\ast(-) u$ and \emph{not} $u(-)u^\ast$.}

\begin{remark}[On approximate intertwining]\label{r:approxintertwining}
To set the reader up for the asymptotic intertwining, the following details on the more classical approximate intertwining a la Elliott are recalled.

Suppose that $A$ and $B$ are separable $C^\ast$-algebras, that 
\begin{equation}
\phi_0 \colon A \to B , \qquad \psi_0 \colon B \to A
\end{equation}
are $\ast$-homomorphisms, and that $(u_n)_{n\in \mathbb N}$ and $(v_n)_{n\in \mathbb N}$ are sequences of unitaries in $\multialg{A}$ and $\multialg{B}$ respectively, such that
\begin{equation}
\lim_{n\to \infty} \| u_n^\ast \psi_0(\phi_0(a)) u_n  - a \| = 0, \qquad \lim_{n\to \infty} \| v_n^\ast \phi_0(\psi_0(b)) v_n  - b \| = 0,
\end{equation}
for all $a\in A$ and $b\in B$. We may find subsequences $(u_{r(n)})$ and $(v_{s(n)})$ (defined recursively), such that if we define
\begin{equation}\label{eq:apintunitaries}
U_n = u_{r(n)} u_{r(n-1)} \cdots u_{r(1)}, \qquad V_n = v_{s(n)} \cdots u_{s(1)},
\end{equation}
with $U_0 = 1_{\multialg{A}}$ and $V_0 = 1_{\multialg{B}}$, as well as
\begin{equation}
 \phi_n = \Ad V_n\circ \phi_0 \circ \Ad U_n^\ast, \qquad \psi_n = \Ad U_{n} \circ \psi_0 \circ \Ad V_{n-1}^\ast,
\end{equation}
for $n\in \mathbb N$, we get a diagram
\begin{equation}\label{eq:approxint}
\xymatrix{
A \ar[rr]^{\id_A} \ar[dr]_{\phi_0} && A \ar[rr]^{\id_A} \ar[dr]_{\phi_1} && A \ar[rr]^{\id_A} \ar[dr]_{\phi_2} && \dots \ar[r] & A \ar[d]_{\phi}^\cong  \\
& B \ar[rr]_{\id_B} \ar[ur]^{\psi_1} && B \ar[rr]_{\id_B} \ar[ur]^{\psi_2} && B \ar[r]_{\id_B} & \dots \ar[r] & B
}
\end{equation}
which is an approximate intertwining as in \cite[Definition 2.3.1]{Rordam-book-classification}. Thus $\phi_n$ converges point-norm to an \emph{isomorphism} $\phi$, and $\psi_n$ converges point-norm to $\psi = \phi^{-1}$, i.e.
\begin{eqnarray}
\lim_{n\to \infty} \|\phi(a)  - (\Ad V_n \circ \phi_0 \circ \Ad U_n^\ast)(a)\| &=& 0 , \\
 \lim_{n\to \infty} \| \psi(b) - (\Ad U_{n} \circ \psi_0 \circ \Ad V_{n-1}^\ast )(b) \| &=& 0,
\end{eqnarray}
for all $a\in A$ and $b\in B$. Note that this does \emph{not} (at least a priori) imply that $\phi$ and $\phi_0$ are approximately unitarily equivalent (see Remark \ref{r:apint}).
\end{remark}

\begin{remark}\label{r:apint}
In \cite[Corollary 2.3.3 and 2.3.4]{Rordam-book-classification} on approximate intertwining, it is claimed that $\phi_0$ and $\phi$ in Remark \ref{r:approxintertwining} are approximately unitarily equivalent. This is true if there are unitaries $\widetilde U_n \in \multialg{B}$ such that $\phi_0 \circ \Ad U_n^\ast = \Ad \widetilde U_n^\ast \circ \phi_0$, and is therefore in particular true if all unitaries are in the minimal unitisations, or if $\phi_0$ is non-degenerate. However, the result may fail in general. 

The mistake appears on Rørdam's list of mistakes in the book.\footnote{\verb+http://www.math.ku.dk/~rordam/Encyclopaedia.html+}

This is related to the subtle annoyance of approximate unitary equivalence (with multiplier unitaries), that it is not preserved by composition. In fact, if $\phi_1, \phi_2 \colon A \to B$ and $\psi \colon B \to C$ are $\ast$-homomorphisms for which $\phi_1$ and $\phi_2$ are approximately unitarily equivalent, then it does \emph{not} follow that $\psi \circ \phi_1$ and $\psi \circ \phi_2$ are approximately unitarily equivalent. However, approximate Murray--von Neumann equivalence \emph{is} preserved by compositions.
\end{remark}

In this paper, whenever one obtains uniqueness of $\ast$-homomorphism with unitaries, the unitaries are (a priori) only in the multiplier algebra and not necessarily in the minimal unitisation as above. Hence the following proposition is included in order to remedy this.

\begin{proposition}\label{p:apintMvN}
Let $A$ and $B$ be separable $C^\ast$-algebras, and suppose that $\phi_0 \colon A \to B$ and $\psi_0 \colon B \to A$ are $\ast$-homomorphisms such that $\psi_0 \circ \phi_0 \sim_{\au} \id_A$  and $\phi_0 \circ \psi_0 \sim_{\au} \id_B$ (with multiplier unitaries). Then there is an isomorphism $\phi \colon A \xrightarrow \cong B$ such that $\phi_0 \sim_{\aMvN} \phi$ and $\psi_0 \sim_{\aMvN} \phi^{-1}$.
\end{proposition}
\begin{proof}
We adopt the notation from Remark \ref{r:approxintertwining}. Let $(e_n)$ be an approximate identity in $A$, and let $w_n = V_n^\ast \phi_0(U_n e_n)$. We claim that $(w_n)$ implements an approximate Murray--von Neumann equivalence of $\phi_0$ and $\phi$. Similarly, one shows that $\psi_0$ and $\phi^{-1}$ are approximate Murray--von Neumann equivalent. 

Clearly one has $w_n \phi_0(a) w_n^\ast \to \phi(a)$ for any $a\in A$, so it remains to check $w_n^\ast \phi(a) w_n \to \phi_0(a)$ for any $a\in A$. Note that
\begin{eqnarray}
w_n^\ast \phi_n(a) w_n &=& \phi_0(e_n U_n^\ast) V_n^\ast V_n \phi_0(U_n a U_n^\ast) V_n^\ast V_n \phi_0(U_n e_n) \\
&=& \phi_0(e_n a e_n).\label{eq:apintMvN}
\end{eqnarray}
Given $a\in A$ and $\epsilon>0$, pick $N\in \mathbb N$ such that $\| \phi_n(a) - \phi(a)\| < \epsilon/2$ and $\|e_n a e_n - a\| < \epsilon /2$ for $n\geq N$. Then, as $\| w_n \| \leq 1$ for all $n$, we have
\begin{eqnarray}
\| w_n^\ast \phi(a) w_n - \phi_0(a) \| &<& \| w_n^\ast \phi_n(a) w_n - \phi_0(a)\| + \epsilon/2 \\
&\stackrel{\eqref{eq:apintMvN}}{=}& \| \phi_0(e_n a e_n -a) \| + \epsilon/2 \\
&<& \epsilon,
\end{eqnarray}
for $n\geq N$. Hence $\phi \sim_{\aMvN} \phi_0$, and similarly $\phi^{-1} \sim_{\aMvN} \psi_0$.
\end{proof}

\begin{remark}
The following asymptotic intertwining -- Lemma \ref{l:asint} -- looks very technical; here is the main idea: If $(u_t)_{t\in \mathbb R_+}$ and $(v_t)_{t\in \mathbb R_+}$ implement asymptotic unitary equivalences $\psi_0 \circ \phi_0 \sim_\asu \id_A$ and $\phi_0 \circ \psi_0 \sim_\asu \id_B$ respectively, then follow the strategy of Remark \ref{r:approxintertwining} and construct $s, r \colon \mathbb N \to \mathbb N$. The goal will be to construct a (well-behaved) homotopy from $\Ad V_{n} \circ \phi_0 \circ \Ad U_{n}^\ast$ (as defined in Remark \ref{r:approxintertwining}) to
\begin{equation}
\Ad V_{n+1} \circ \phi_0 \circ \Ad U_{n+1}^\ast = \Ad V_n \circ \Ad v_{s(n+1)} \circ \phi_0 \circ \Ad u_{r(n+1)}^\ast \circ \Ad U_n^\ast.
\end{equation}
If done in a suitably nice way, this will induce a (well-behaved) homotopy from $\phi_0$ to the isomorphism $\phi = \lim_{n\to \infty} \Ad V_n \circ \phi_0 \Ad U_n^\ast$.

The idea for constructing this homotopy is very easy: a simple observation shows that
\begin{equation}
\lim_{t\to \infty} \| (\Ad v_{s(n) + t} \circ \phi_0 \circ \Ad u_{r(n) + t} ^\ast )(a) - \phi_0(a) \| = 0, \qquad a\in A,
\end{equation}
so $\Phi \colon A \to C([0,\infty], B)$ given by
\begin{equation}
\Phi(-)(t) = \left\{ \begin{array}{ll} \Ad V_n \circ \Ad v_{s(n+1) +t } \circ \phi_0 \circ \Ad u_{r(n+1)+t}^\ast \circ \Ad U_n^\ast, & t\in [0,\infty) \\
\Ad V_n \circ \phi_0 \circ \Ad U_n^\ast , & t=  \infty \end{array} \right.
\end{equation}
is the desired homotopy.
\end{remark}

\begin{lemma}\label{l:asint}
Let $A$ and $B$ be separable $C^\ast$-algebras and suppose that 
\begin{equation}
\phi_0 \colon A \to B, \qquad \psi_0 \colon B \to A
\end{equation}
are $\ast$-homomorphisms such that $\psi_0 \circ \phi_0 \sim_{\asu} \id_A$ and $\phi_0 \circ \psi_0 \sim_{\asu} \id_B$. Then there exist an isomorphism $\phi \colon A \xrightarrow \cong B$ and (\emph{not necessarily continuous!}) families of unitaries $(U_t)_{t\in \mathbb R_+}$ and $(V_t)_{t\in \mathbb R_+}$ in $\multialg{A}$ and $\multialg{B}$ resepctively, with $U_0 = 1_{\multialg{A}}$ and $V_0 = 1_{\multialg{B}}$, such that
\begin{equation}\label{eq:asintermaps}
\mathbb R_+ \ni t\mapsto (\Ad V_t \circ \phi_0 \circ \Ad U_t^\ast)(a), \qquad [1, \infty) \ni t\mapsto (\Ad U_{t} \circ \psi_0 \circ \Ad V_{t-1}^\ast)(b)
\end{equation}
are continuous and converge to $\phi(a)$ and $\phi^{-1}(b)$ respectively for all $a\in A$ and $b\in B$.
\end{lemma}
\begin{proof}
Let $(u_t)_{t\in \mathbb R_+}$ and $(v_t)_{t\in \mathbb R_+}$ be continuous unitary paths in $\multialg{A}$ and $\multialg{B}$ respectively such that
\begin{equation}
\lim_{t\to \infty} \| u_t^\ast \psi_0(\phi_0(a)) u_t  - a \| = 0, \qquad \lim_{n\to \infty} \| v_t^\ast \phi_0(\psi_0(b)) v_t  - b \| = 0,
\end{equation}
for all $a\in A$ and $b\in B$. The first part of the proof is very much just a usual intertwining argument. The details will be filled in since some of the estimates below will be needed.\footnote{I could had easily waved my hands, and said something like ``the unitaries in the classical intertwining, may be chosen with these additional properties'', but I choose to fill in all the details for completion.}

Let $\mathcal F_1' \subset \mathcal F_2' \subset \dots \subset A$ and $\mathcal G_1' \subset \mathcal G_2' \subset \dots \subset B$ be nested sequences of finite sets such that $\overline{\bigcup \mathcal F_n'} =A$ and $\overline{\bigcup \mathcal G_n'}=B$. We construct $\mathcal F_n$, $\mathcal G_n$, $U_n$, $V_n$, $r(n)$ and $s(n)$ recursively. Let $\mathcal F_1 = \mathcal F_1'$ and pick $1\leq r(1) \in \mathbb R_+$ such that
\begin{equation}
\| u_t^\ast \psi_0(\phi_0(x)) u_t - x \| \leq 2^{-1} , \qquad t\geq r(1), \; x\in \mathcal F_1.
\end{equation}
Let $U_1 := u_{r(1)}$. Let $\mathcal G_1 = \phi_0(\mathcal F_1) \cup \mathcal G_1'$, pick $1 \leq s(1)\in \mathbb R_+$ such that
\begin{equation}
\| v_t^\ast \phi_0(\psi_0(y)) v_t - y \| \leq 2^{-1} , \qquad t\geq s(1), \; y\in \mathcal G_1,
\end{equation}
and define $V_1 := v_{s(1)}$.

Having constructed $\mathcal F_{n-1}$, $\mathcal G_{n-1}$, $U_{n-1}$, $V_{n-1}$, $r(n-1)$ and $s(n-1)$, we construct the next step as follows: Let 
\begin{equation}
\mathcal F_n = u_{n-1} \mathcal F_{n-1} u_{n-1}^\ast \cup  \psi_0(\mathcal G_{n-1}) \cup \mathcal F_n' \cup U_{n-1} \mathcal F_{n-1} U_{n-1}^\ast,
\end{equation}
pick $\max(n, r(n-1)) \leq r(n) \in \mathbb R_+$, such that
\begin{equation}\label{eq:apintF}
\| u_t^\ast \psi_0(\phi_0(a)) u_t - a \| \leq 2^{-n} , \qquad t\geq r(n), \; a\in \mathcal F_n,\footnote{In the classical intertwining, this norm estimate is only chosen for $t=r(n)$. Here it is needed for \emph{all} $t\geq r(n)$, which is of course possible.}
\end{equation}
and let $U_n = u_{r(n)} U_{n-1}$. Similarly, let
\begin{equation}
\mathcal G_n = v_{s(n-1)} \mathcal G_{n-1} v_{s(n-1)}^\ast \cup \phi_0(\mathcal F_n) \cup \mathcal  G_n' \cup V_{n-1} \mathcal G_{n-1} V_{n-1}^\ast,
\end{equation}
pick $\max(n, s(n-1)) \leq s(n) \in \mathbb R_+$, such that
\begin{equation}\label{eq:apintG}
\| v_t^\ast \phi_0(\psi_0(b)) v_t - b \| \leq 2^{-n} , \qquad t\geq s(n), \; b\in \mathcal G_n,
\end{equation}
and let $V_n = v_{s(n)} V_{n-1}$. Note that we picked $r(n),s(n)\geq n$, to ensure that $r(n)$ and $s(n)$ both tend to $\infty$. As in the usual approximate intertwining, let
\begin{equation}
 \phi_n = \Ad V_n\circ \phi_0 \circ \Ad U_n^\ast, \qquad \psi_n = \Ad U_{n} \circ \psi_0 \circ \Ad V_{n-1}^\ast.
\end{equation}
These induce an approximate intertwining as \eqref{eq:approxint}, see \cite[Definition 2.3.1]{Rordam-book-classification},\footnote{With the notation in \cite[Definition 2.3.1]{Rordam-book-classification} one has $\alpha_n = \id_A$ and $\beta_n = \id_B$.} so by \cite[Proposition 2.3.2]{Rordam-book-classification} it follows that $\phi_n$ converges point-norm to an isomorphism $\phi \colon A \xrightarrow \cong B$ and $\psi_n$ converges point-norm to $\phi^{-1}$.

Fix (necessarily orientation reversing) homeomorphisms
\begin{equation}
h_n \colon (n,n+1] \to [r(n+1), \infty), \qquad k_n \colon (n,n+1] \to [s(n+1), \infty),
\end{equation}
for all $n\in \mathbb N_0$, so in particular $h_n(n+1) = r(n+1)$ and $k_n(n+1) = s(n+1)$. We define $U_t$ and $V_t$ for all $t\in \mathbb R_+$, by letting $U_0=1_{\multialg{A}}$, $V_0 = 1_{\multialg{B}}$, and
\begin{equation}
U_{t}  := u_{h_n(t)} U_n, \qquad V_{t} := v_{k_n(t)} V_n
\end{equation}
for all $n\in \mathbb N_0$ and $t\in (n,n+1]$. This is well-defined, since 
\begin{equation}
U_{n+1} = u_{r(n+1)} U_n = u_{h_n(n+1)} U_n
\end{equation}
for all $n\in \mathbb N_0$, and similarly for the $V_n$'s. Let 
\begin{equation}
\phi_t := \Ad V_t \circ \phi_0 \circ \Ad U_t^\ast
\end{equation}
for $t\in \mathbb R_+$, and 
\begin{equation}
\psi_t := \Ad U_t \circ \psi_0 \circ \Ad V_{t-1}^\ast
\end{equation}
for $t\in [1,\infty)$.

Given $a\in A$, we check that $t \mapsto \phi_t(a)$ is continuous on $\mathbb R_+$. A similar argument shows that $t\mapsto \psi_t(b)$ is continuous on $t\in [1,\infty)$, so this is omitted. Clearly $t \mapsto U_t$ and $t\mapsto V_t$ are continuous when restricted to intervals $(n,n+1]$. Hence $t \mapsto \phi_t(a)$ is continuous in any point $\mathbb R_+ \setminus \mathbb N_0$, and is left continuous at all points in $\mathbb N$. Hence, it suffices to show that $t \mapsto \phi_t(a)$ is right continuous at any point $n$ in $\mathbb N_0$. Note that
\begin{eqnarray}
\psi_0(\phi_0(U_n a U_n^\ast) ) &=& \lim_{t\to \infty} u_t U_n a U_n^\ast u_t^\ast \nonumber\\
&=& \lim_{t\to n^+} u_{h_n(t)} U_n a U_n^\ast u_{h_n(t)}^\ast \nonumber\\
&=& \lim_{t\to n^+} U_t a U_t^\ast
\end{eqnarray}
and similarly
\begin{eqnarray}
&& V_n^\ast \phi_0 (U_n a U_n^\ast) V_n \nonumber\\
&=& \lim_{t\to \infty} V_n^\ast v_t^\ast \phi_0(\psi_0( \phi_0(U_n a U_n^\ast))) v_t V_n \nonumber\\
&=& \lim_{t\to n^+} V_n^\ast v_{k_n(t)}^\ast \phi_0(\psi_0(\phi_0( U_n a U_n^\ast )))v_{k_n(t)}^\ast V_n \nonumber\\
&=& \lim_{t \to n^+} V_t^\ast \phi_0 (\psi_0(\phi_0(U_n a U_n^\ast))) V_t.
\end{eqnarray}
It easily follows (since $\Ad V_t \circ \phi_0$ is contractive for each $t$), that
\begin{eqnarray}
\lim_{t\to n^+} \phi_t(a) &=& \lim_{t\to n^+} ( \Ad V_t \circ \phi_0 \circ \Ad U_t^\ast)(a) \nonumber\\
&=& \lim_{t\to n^+} (\Ad V_t \circ \phi_0 \circ \psi_0 \circ \phi_0\circ \Ad U_n^\ast)(a) \nonumber\\
&=& (\Ad V_n \circ \phi_0 \circ \Ad U_n^\ast)(a) \nonumber\\
&=& \phi_n(a).
\end{eqnarray}
Hence $t \mapsto \phi_t(a)$ is right continuous in $n$, and thus continuous on all $\mathbb R_+$. It remains to check that $\phi_t(a) \to \phi(a)$ and $\psi_t(b) \to \psi(b)$ for all $a\in A$ and $b\in B$. Again, we only check the former, as the latter is shown in the exact same way.

Fix $a\in A$ and $\epsilon >0$. Pick $N \in \mathbb N$ with $2^{-N} < \epsilon /5$ such that there is an $a'\in \mathcal F_{{N}-1}$ with $\| a - a'\| < \epsilon/5$, and such that $\| \phi(a') - \phi_n(a') \| < \epsilon/5$ for any $n\geq N$. We will check that $\| \phi_t(a) - \phi(a) \| < \epsilon$ for any $t\geq N$. This will finish the proof. Fix $n$ such that $t\in (n,n+1]$ and note that $n\geq N$. In the following, the notation $c \approx_\delta d$ means $\| c - d\| \leq \delta$. Thus
\begin{eqnarray}
&& \phi_t(a) \nonumber\\
&\approx_{\epsilon/5}& \phi_t(a') \nonumber\\
&=\quad \; &  ( \Ad V_n \circ \Ad v_{k_n(t)} \circ \phi_0\circ \Ad u_{h_n(t)}^\ast \circ \Ad U_n^\ast)( a')\nonumber\\
&\approx_{\epsilon/5}& (\Ad V_n \circ \Ad v_{k_n(t)}  \circ  \phi_0 \circ \psi_0 \circ \phi_0 \circ \Ad U_n^\ast) (a') \label{eq:Vnvkntphi0} \\
&\approx_{\epsilon/5}& ( \Ad V_n \circ \phi_0 \circ \Ad U_n^\ast) (a') \label{eq:Vnphi0Un}\\
&=\quad \;& \phi_n(a') \nonumber\\
& \approx_{\epsilon/5}& \phi(a') \nonumber\\
&\approx_{\epsilon/5}& \phi(a).
\end{eqnarray}
At \eqref{eq:Vnvkntphi0} we used that $a' \in \mathcal F_{N - 1} \subseteq \mathcal F_{n-1}$, so $x = U_n a' U_n^\ast \in \mathcal F_n$ (by construction of $\mathcal F_n$), and thus \eqref{eq:Vnvkntphi0} follows from \eqref{eq:apintF} and the choice of $N$. 

Similarly, at \eqref{eq:Vnphi0Un} we used (since $U_n a' U_n^\ast \in \mathcal F_n$) that $y = \phi_0(U_n a' U_n^\ast) \in \mathcal G_n$, so \eqref{eq:Vnphi0Un} follows from \eqref{eq:apintG} and the choice of $N$. This finishes the proof.
\end{proof}

\begin{proposition}\label{p:asint}
Let $A$ and $B$ be separable $C^\ast$-algebras, and suppose that $\phi_0 \colon A \to B$ and $\psi_0 \colon B \to A$ are $\ast$-homomorphisms such that $\psi_0 \circ \phi_0 \sim_{\asu} \id_A$ and $\phi_0 \circ \psi_0 \sim_{\asu} \id_B$. Then there exist an isomorphism $\phi \colon A \xrightarrow \cong B$, and a homotopy $(\phi_s)_{s\in [0,1]}$ from $\phi_0$ to $\phi$, such that $\phi_s \sim_{\aMvN} \phi_t$ for all $s,t\in [0,1]$.
\end{proposition}
\begin{proof}
For convenience, we replace $[0,1]$ with $[0,\infty]$. Let $\phi$, $(U_t)_{t\in \mathbb R_+}$ and $(V_t)_{t\in \mathbb R_+}$ be given as in Lemma \ref{l:asint}. Let
\begin{equation}
\phi_t = \left\{ \begin{array}{ll}
\Ad V_t \circ \phi_0 \circ \Ad U_t^\ast & \textrm{ for } t\in [0,\infty),  \\
\phi & \textrm{ for } t=\infty.
\end{array} \right.
\end{equation}
By Lemma \ref{l:asint}, this gives a well-defined homotopy from $\phi_0$ to $\phi$. 

To show $\phi_s \sim_{\aMvN} \phi_t$ for all $s,t\in [0,\infty]$, it is enough to show $\phi_0 \sim_{\aMvN} \phi_t$ for $t\in (0,\infty]$.

If $t\in (0,\infty)$, $(e_n)$ is an approximate identity in $A$, and $a_n = \phi_0(a_n U_t^\ast) V_t$, then 
\begin{equation}
a_n^\ast \phi_0(-) a_n = \phi_t(e_n(-)e_n) \to \phi_t, \qquad a_n \phi_t(-) a_n^\ast = \phi_0(e_n (-) e_n) \to \phi_0,
\end{equation}
point-norm, so $\phi_0 \sim_{\aMvN} \phi_t$.

If $t=\infty$, an argument identical to that in the proof of Proposition \ref{p:apintMvN}, shows that $\phi_0$ and $\phi$ are approximately\footnote{\emph{Not} asymptotically, since the maps $t \mapsto U_t$ and $t\mapsto V_t$ are not necessarily continuous.} Murray--von Neumann equivalent.
\end{proof}

\begin{question}
Is it possible to pick an isomorphism $\phi$ in Proposition \ref{p:asint}, such that $\phi$ and $\phi_0$ are asymptotically Murray--von Neumann equivalent?
\end{question}

\begin{remark}
Note that in Proposition \ref{p:asint}, one does \emph{not} get that $\psi_0$ is homotopic to $\phi^{-1}$. However, one does get that $\Ad U_1 \circ \psi_0$ \emph{is} homotopic to $\phi^{-1}$. So if one can choose a path of unitaries $(u_t)_{t\in \mathbb R_+}$ in $\multialg{A}$ which implements $\psi_0 \circ \phi_0 \sim_\asu \id_A$ and which satisfies $u_0 = 1_{\multialg{A}}$, then $\psi_0$ and $\phi^{-1}$ are also homotopic. This is for instance always the case if $A$ is stable, since the unitary group $\multialg{A}$ is then path-connected by \cite{CuntzHigson-Kuipersthm}.
\end{remark}

\begin{remark}
As observed in \cite[Remark 3.14]{Gabe-O2class}, if one lets $\phi \colon \mathcal O_2 \to \mathcal O_2 \otimes \mathcal K$ and $\psi \colon \mathcal O_2 \otimes \mathcal K \to \mathcal O_2$ be embeddings, then $\psi \circ \phi \sim_{\asMvN} \id_{\mathcal O_2}$ and $\phi \circ \psi \sim_{\asMvN} \id_{\mathcal O_2 \otimes \mathcal K}$. Thus, asymptotic (and approximate) Murray--von Neumann equivalence is not strong enough of an equivalence relation to get classification up to isomorphism. However, by \cite[Corollary 3.13]{Gabe-O2class}, one does obtain classification up to \emph{stable} isomorphism.
\end{remark}


\section{A unitary path and some key lemmas}\label{s:unitary}

This section is dedicated to showing that there is unitary path in the unitisation $(\mathcal O_2 \otimes \mathcal K)^\sim$ with some very desirable properties. Constructing this path is very elementary and it is the new key ingredient in proving the main theorems of this paper. The only properties of $\mathcal O_2$ that are used are that $M_3(\mathcal O_2) \cong \mathcal O_2$, that any two non-trivial projections\footnote{By non-trivial, I mean that they are non-zero and not the unit.} in $\mathcal O_2$ are unitarily equivalent, and that the unitary group of $\mathcal O_2$ is path-connected.

\begin{lemma}\label{l:unitarypath}
There exists a continuous path of unitaries $(u_t)_{t\in \mathbb R_+}$ in $(\mathcal O_2 \otimes \mathcal K)^\sim$ with $u_0 = 1_{(\mathcal O_2 \otimes \mathcal K)^\sim}$, such that the following hold:
\begin{itemize}
\item[$(a)$] $(u_t^\ast (1_{\mathcal O_2}\otimes e_{1,1}) u_t)_{t\in \mathbb R_+}$ is a continuous (not necessarily increasing) approximate identity of projections for $\mathcal O_2\otimes \mathcal K$;
\item[$(b)$] for any $x\in \mathcal O_2\otimes \mathcal K$, $u_t x$ converges in norm to an element in $\mathcal O_2\otimes \mathcal K$ as $t\to \infty$.
\end{itemize}
\end{lemma}
\begin{proof}
We construct unitary paths $v_{n,t}$ for $n\in \mathbb N_0$, $t\in [n,n+1]$ with $v_{n,n}=1_{(\mathcal O_2 \otimes \mathcal K)^\sim}$, and let $u_t = u_n v_{n,t}$ for $t\in [n,n+1]$ (defined recursively). In particular, $u_0 = v_{0,0} = 1_{(\mathcal O_2 \otimes \mathcal K)^\sim}$.

Recall the following facts about $\mathcal O_2$, which may be obtained using results from \cite{Cuntz-K-theoryI}: whenever $p,q\in \mathcal O_2$ are projections such that $0<p,q<1_{\mathcal O_2}$, then $p$ and $q$ are Murray--von Neumann equivalent, and $1_{\mathcal O_2}-p$ and $1_{\mathcal O_2}-q$ are Murray--von Neumann equivalent. Thus there is a unitary $u\in \mathcal O_2$, such that $u^\ast p u = q$. Moreover, as the unitary group of $\mathcal O_2$ is connected we may pick a continuous path $[0,1] \ni t \mapsto r_t\in \mathcal O_2$ of unitaries with $r_0 = 1_{\mathcal O_2}$ and $r_1 = u$. Fix $n\in \mathbb N_0$. As
\begin{equation}
\left( \sum_{k=n+1}^{n+3} 1_{\mathcal O_2} \otimes e_{k,k} \right) \left( \mathcal O_2 \otimes \mathcal K \right) \left( \sum_{k=n+1}^{n+3} 1_{\mathcal O_2} \otimes e_{k,k} \right) \cong M_3(\mathcal O_2) \cong \mathcal O_2,
\end{equation}
we may find a norm continuous path $[n,n+1] \ni t \mapsto r_{n,t} \in \mathcal O_2 \otimes \mathcal K$, such that
\begin{equation}
r_{n,n} = r_{n,t}^\ast r_{n,t} = r_{n,t} r_{n,t}^\ast = \sum_{k=n+1}^{n+3} 1_{\mathcal O_2} \otimes e_{k,k}, \quad \text{ for all }t\in [n,n+1],
\end{equation}
and such that
\begin{equation}\label{eq:O2stuff}
r_{n,n+1}^\ast ( 1_{\mathcal O_2} \otimes e_{n+1,n+1}) r_{n,n+1} = 1_{\mathcal O_2} \otimes (e_{n+1,n+1} + e_{n+2,n+2}).
\end{equation}
Let $v_{n,t} = r_{n,t} + \left( 1_{(\mathcal O_2 \otimes \mathcal K)^\sim} - \sum_{k=n+1}^{n+3} 1_{\mathcal O_2} \otimes e_{k,k} \right)$ for $t\in [n,n+1]$. Clearly $v_{n,t}$ is a unitary for each $n\in \mathbb N_0$, $t\in [n,n+1]$ and $v_{n,n} = 1_{(\mathcal O_2 \otimes \mathcal K)^\sim}$. Thus 
\begin{equation}
u_t := u_n v_{n,t} = v_{0,1} v_{1,2} \cdots v_{n-1,n} v_{n,t}
\end{equation}
for $t\in [n,n+1]$ defines a continuous path of unitaries $(u_t)_{t\in \mathbb R_+}$ in $(\mathcal O_2 \otimes \mathcal K)^\sim$. 

An easy induction argument shows that $u_n^\ast (1_{\mathcal O_2}\otimes e_{1,1}) u_n = \sum_{k=1}^{n+1} 1_{\mathcal O_2} \otimes e_{k,k}$ for $n\in \mathbb N_0$. This is obvious for $n=0$, and by induction
\begin{eqnarray}
u_n^\ast (1_{\mathcal O_2}\otimes e_{1,1}) u_n &=& v_{n-1,n}^\ast u_{n-1}^\ast (1_{\mathcal O_2} \otimes e_{1,1}) u_{n-1} v_{n-1,n} \nonumber\\
&=& v_{n-1,n}^\ast \left( \sum_{k=1}^n 1_{\mathcal O_2} \otimes e_{k,k} \right) v_{n-1,n} \nonumber\\
&=& \sum_{k=1}^{n-1} 1_{\mathcal O_2} \otimes e_{k,k} + r_{n-1,n}^\ast (1_{\mathcal O_2} \otimes e_{n,n}) r_{n-1,n} \nonumber\\
&\stackrel{\eqref{eq:O2stuff}}{=}& \sum_{k=1}^{n+1} 1_{\mathcal O_2} \otimes e_{k,k}.\label{eq:une11un}
\end{eqnarray}

We prove $(a)$ and $(b)$.

$(a)$: Clearly each $u_t^\ast (1_{\mathcal O_2}\otimes e_{1,1}) u_t$ is a projection. For $t\in [n,n+1]$ one has
\begin{eqnarray}
u_t^\ast(1_{\mathcal O_2}\otimes e_{1,1}) u_t &=& v_{n,t}^\ast u_{n}^\ast (1_{\mathcal O_2}\otimes e_{1,1}) u_n v_{n,t} \nonumber\\
&\stackrel{\eqref{eq:une11un}}{=}& v_{n,t}^\ast \left( \sum_{k=1}^{n+1} 1_{\mathcal O_2}\otimes e_{k,k} \right) v_{n,t} \nonumber\\
&= & \sum_{k=1}^n 1_{\mathcal O_2} \otimes e_{k,k} + r_{n,t}^\ast (1_{\mathcal O_2} \otimes e_{n+1,n+1}) r_{n,t} \nonumber\\
& \geq& \sum_{k=1}^n 1_{\mathcal O_2} \otimes e_{k,k}.
\end{eqnarray}
This implies that $u_t^\ast(1_{\mathcal O_2}\otimes e_{1,1} ) u_t$ is a not necessarily increasing approximate identity.

$(b)$: Let $x\in \mathcal O_2\otimes \mathcal K$, and fix $\epsilon >0$. Pick $n\in \mathbb N$ and $x_0\in \mathcal O_2\otimes M_n \subseteq \mathcal O_2 \otimes \mathcal K$ such that $\| x - x_0\| <\epsilon/2$. For any $t\geq n$ we have $v_{\lfloor t \rfloor, t} x_0 = x_0$, and also, $v_{n+j, n+j+1} x_0 = x_0$ for any $j\geq 0$, so
\begin{equation}
u_t x_0 = v_{0,1} \dots v_{n-1,n} v_{n,n+1} \dots v_{\lfloor t \rfloor, t} x_0 = v_{0,1} \dots v_{n-1,n} x_0
\end{equation}
is constant for $t\geq n$. Hence for $s,t \geq n$ we get
\begin{equation}
\| u_s x - u_t x\| < \| u_s x_0 - u_t x_0\| + \epsilon = \epsilon,
\end{equation}
and therefore $(u_t x)_{t\in \mathbb R_+}$ converges to some element. 
\end{proof}

As opposed to \emph{being} an approximate identity in a $C^\ast$-algebra it is convenient to consider nets of multipliers which \emph{act} as approximate identities.

More precisely, let $D$ be a $C^\ast$-algebra. A net $(e_\lambda)_{\lambda \in \Lambda}$ of positive contractions in $\multialg{D}$ is said to \emph{act as an approximate identity} (on $D$) if $\lim_{\lambda} \| e_\lambda d - d\| = 0$ for all $d\in D$. The net is not assumed to be increasing.

\begin{corollary}\label{c:unitarypath}
Let $B$ be a $C^\ast$-algebra such that there is a unital embedding $\iota \colon \mathcal O_2 \hookrightarrow \multialg{B}$ and let $\overline \iota \colon (\mathcal O_2 \otimes \mathcal K)^\sim \hookrightarrow \multialg{B \otimes \mathcal K}$ be the induced unital embedding.  There is a continuous path of unitaries $(v_t)_{t\in \mathbb R_+}$ in $\overline \iota( (\mathcal O_2 \otimes \mathcal K)^\sim)$ with $v_0 = 1_{\multialg{B\otimes \mathcal K}}$, such that the following hold:
\begin{itemize}
\item[$(a)$] $(v_t^\ast ( 1_{\multialg{B}} \otimes e_{1,1}) v_t)_{t\in \mathbb R_+}$ acts as an approximate identity on $B \otimes \mathcal K$;
\item[$(b)$] $v_t b v_t^\ast$ converges in norm to an element in $B\otimes \mathcal K$ as $t\to \infty$, for any $b\in B \otimes \mathcal K$.
\end{itemize}
\end{corollary}
\begin{proof}
Let $(u_t)_{t\in \mathbb R_+}$ be a unitary path as in Lemma \ref{l:unitarypath}, and let $v_t := \overline \iota(u_t)$. Clearly $t\mapsto v_t$ is continuous and $v_0 = 1_{\multialg{B\otimes \mathcal K}}$, so only $(a)$ and $(b)$ remains to be checked. Since $\iota(\mathcal O_2) \otimes \mathcal K$ is a non-degenerate $C^\ast$-subalgebra of $\multialg{B} \otimes \mathcal K$, part $(a)$ easily follows.

For $(b)$, it suffices to check the condition for $b= b' \otimes yz$ with $b'\in B$ and $y,z\in \mathcal K$. Then
\begin{equation}
b = \overline \iota(1_{\mathcal O_2} \otimes y) (b\otimes 1_{\multialg{\mathcal K}} )\overline{\iota}(1_{\mathcal O_2} \otimes z),
\end{equation}
so $v_t \overline{\iota}(1_{\mathcal O_2}\otimes y) = \overline{\iota}( u_t (1_{\mathcal O_2}\otimes y))$ converges in norm to an element $y' \in \overline{\iota}(\mathcal O_2 \otimes \mathcal K) \subseteq \multialg{B} \otimes \mathcal K$, and similarly $\overline{\iota}(1_{\mathcal O_2}\otimes z) v_t^\ast$ converges to $z' \in  \multialg{B} \otimes \mathcal K$. Thus $v_t b v_t^\ast$ converges in norm to $y' (b'\otimes 1_{\multialg{\mathcal K}}) z' \in B\otimes \mathcal K$.
\end{proof}

The following is the key lemma for lifting a $KK$-element (even in the ideal-related setting) to a $\ast$-homomorphism. Note that if $\phi, \theta \colon A\to B$ are $\ast$-homomorphisms, then $\phi(-) \otimes e_{1,1} + \theta(-) \otimes (1_{\multialg{\mathcal K}} - e_{1,1}) \colon A \to \multialg{B\otimes \mathcal K}$ is the diagonal $\ast$-homomorphism $\phi \oplus \theta \oplus \theta \oplus \cdots$, see Remark \ref{r:infrep}.

\begin{lemma}[Key lemma for existence]\label{l:keyexistence}
Let $A$ and $B$ be $C^\ast$-algebras, let $\theta \colon A \to B$ be a $\ast$-homomorphism, and suppose that there is a unital embedding $\mathcal O_2 \hookrightarrow \multialg{B} \cap \theta(A)'$. Let $\theta_\infty = \theta \otimes 1_{\multialg{\mathcal K}} \colon A \to \multialg{B \otimes \mathcal K}$. Then there is a norm-continuous path $(v_t)_{t\in \mathbb R_+}$ of unitaries in $\multialg{B\otimes \mathcal K} \cap \theta_\infty(A)'$ with $v_0 = 1_{\multialg{B\otimes \mathcal K}}$, which has the following property:

Suppose that $\psi \colon A \to \multialg{B \otimes \mathcal K}$ is a $\ast$-homomorphism such that $\psi(a) - \theta_\infty(a)\in B \otimes \mathcal K$ for all $a\in A$. Then there exists a $\ast$-homomorphism $\phi \colon A \to B$ such that $v_t \psi(-) v_t^\ast$ converges point-norm to the $\ast$-homomorphism 
\begin{equation}
\phi \otimes e_{1,1} + \theta \otimes (1_{\multialg{\mathcal K}} - e_{1,1}) \colon A \to \multialg{B \otimes \mathcal K}.
\end{equation}
\end{lemma}
\begin{proof}
Fix a unital embedding $\iota \colon \mathcal O_2 \hookrightarrow \multialg{B} \cap \theta(A)'$ and let $\overline \iota \colon (\mathcal O_2 \otimes \mathcal K)^\sim \hookrightarrow \multialg{B \otimes \mathcal K}$ be the induced unital $\ast$-homomorphism. 
Let $(v_t)_{t\in \mathbb R_+}$ be a unitary path in $\overline{\iota}((\mathcal O_2 \otimes \mathcal K)^\sim)$ as in Corollary \ref{c:unitarypath}. As $\theta_\infty = \theta \otimes 1_{\multialg{\mathcal K}}$, it easily follows that
\begin{equation}
\overline \iota ((\mathcal O_2 \otimes \mathcal K)^\sim) \subseteq \multialg{B \otimes \mathcal K} \cap \theta_\infty(A)'.
\end{equation}
For any $a\in A$ we get that
\begin{equation}
v_t \psi(a) v_t^\ast = v_t (\psi(a) - \theta_\infty(a)) v_t^\ast + v_t \theta_\infty(a) v_t^\ast = v_t (\psi(a) - \theta_\infty(a)) v_t^\ast + \theta_\infty(a)
\end{equation}
converges in norm for $t\to \infty$ by Corollary \ref{c:unitarypath}$(b)$. Thus $v_t \psi(-) v_t^\ast$ converges point-norm to a $\ast$-homomorphism $\psi_0 \colon A \to \multialg{B\otimes \mathcal K}$. In particular
\begin{equation}\label{eq:limithom}
\psi_0(a) = \lim_{t\to \infty} v_t \psi(a) v_t^\ast = \lim_{t\to \infty} v_t (\psi(a) - \theta_\infty(a)) v_t^\ast + \theta_\infty(a).
\end{equation}
In the following we let $e_{1,1}^\perp := 1_{\multialg{\mathcal K}} - e_{1,1}$. For any $b\in B\otimes \mathcal K$ we have
\begin{eqnarray}
&& \lim_{t\to \infty} \| (1_{\multialg{B}} \otimes e_{1,1}^\perp) v_t b v_t^\ast \| \nonumber \\
&=& \lim_{t\to \infty} \| v_t^\ast (1_{\multialg{B}}\otimes e_{1,1}^\perp) v_t b\| \nonumber\\
&=& \lim_{t\to \infty} \| b - v_t^\ast (1_{\multialg{B}}\otimes e_{1,1})v_t b \| \nonumber\\
&=& 0,
\end{eqnarray}
where the last equality follows from Corollary \ref{c:unitarypath}$(a)$. Thus
\begin{eqnarray}\label{eq:compcorner}
&& (1_{\multialg{B}} \otimes e_{1,1}^\perp) \psi_0(a) \nonumber\\
&\stackrel{\eqref{eq:limithom}}{=}& \lim_{t\to \infty} (1 _{\multialg{B}}\otimes e_{1,1}^\perp) v_t (\psi(a) - \theta_\infty) v_t^\ast + (1_{\multialg{B}} \otimes e_{1,1}^\perp)\theta_\infty(a) \nonumber\\
&=& \theta(a) \otimes e_{1,1}^\perp.\label{eq:compcorner}
\end{eqnarray}
By symmetry, we have
\begin{equation}
(1_{\multialg{B}} \otimes e_{1,1}^\perp) \psi_0(a) = \theta(a) \otimes e_{1,1}^{\perp} = \psi_0(a) (1_{\multialg{B}}\otimes e_{1,1}^\perp).
\end{equation}
Hence $1_{\multialg{B}}\otimes e_{1,1}$ commutes with $\psi_0(A)$ so we obtain a $\ast$-homomorphism 
\begin{equation}
\phi_0 = (1_{\multialg{B}}\otimes e_{1,1}) \psi_0(-) (1_{\multialg{B}}\otimes e_{1,1}) \colon A \to \multialg{B}\otimes e_{1,1}.
\end{equation}
However, we have
\begin{eqnarray}
&& (1_{\multialg{B}}\otimes e_{1,1}) \psi_0(a) \nonumber\\ 
& \stackrel{\eqref{eq:limithom}}{=} & \lim_{t\to \infty} (1_{\multialg{B}}\otimes e_{1,1}) v_t (\psi(a)- \theta_\infty(a)) v_t^\ast + \theta(a) \otimes e_{1,1} \nonumber\\
& \in& B\otimes \mathcal K,
\end{eqnarray}
so $\phi_0$ factors through $B \otimes e_{1,1}$. Let $\phi\colon A \to B$ be the corestriction of $\phi_0$. Then
\begin{eqnarray}
\lim_{t\to \infty} v_t \psi(a) v_t^\ast &=& \psi_0(a) \nonumber \\
&=& (1_{\multialg{B}}\otimes e_{1,1}) \psi_0(a) + (1_{\multialg{B}} \otimes e_{1,1}^\perp) \psi_0(a) \nonumber\\
&\stackrel{\eqref{eq:compcorner}}{=}& \phi(a) \otimes e_{1,1} + \theta(a) \otimes (1_{\multialg{\mathcal K}}- e_{1,1}), 
\end{eqnarray}
for all $a\in A$, which finishes the proof.
\end{proof}

The following will be the key lemma for proving uniqueness results. It shows how one goes from a stable uniqueness result a la Dadarlat--Eilers \cite[Theorems 3.8 and 3.10]{DadarlatEilers-asymptotic} to a stable uniqueness results where one stabilises with a smaller $\ast$-homomorphism. 
In the presence of strong 
$\mathcal O_\infty$-stability one even obtains a uniqueness result on the nose. 


\begin{lemma}[Key lemma for uniqueness]\label{l:keyuniqueness}
Let $A$ and $B$ be $C^\ast$-algebras with $A$ separable, and let $\phi, \psi, \theta \colon A \to B$ be $\ast$-homomorphisms. Suppose that there is a unital embedding $\iota \colon \mathcal O_2 \hookrightarrow \multialg{B} \cap \theta(A)'$. 
If there is a continuous unitary path $(w_t)_{t\in \mathbb R_+}$ in $(B\otimes \mathcal K)^\sim$ such that
\begin{equation}\label{eq:DElemma}
\lim_{t\to \infty}\| w_t ( \phi(a) \otimes e_{1,1} + \theta(a) \otimes e_{1,1}^\perp) w_t^\ast - \psi(a) \otimes e_{1,1} - \theta(a) \otimes e_{1,1}^\perp \| = 0
\end{equation}
for all $a\in A$, then $\phi \oplus \theta \sim_\asu \psi \oplus \theta$ considered as maps $A \to M_2(B)$. Here $e_{1,1}^\perp = 1_{\multialg{\mathcal K}} - e_{1,1} \in \multialg{\mathcal K}$.

In addition, if $\phi$ and $\psi$ are both strongly $\mathcal O_\infty$-stable and approximately dominate $\theta$ then $\phi \sim_\asMvN \psi$.
%
%
\end{lemma}
\begin{proof}
%
Let $\mathcal K_1 := e_{1,1}^\perp \mathcal K  e_{1,1}^\perp$ which is isomorphic to $\mathcal K$. In particular, $1_{\multialg{\mathcal K_1}} = e_{1,1}^\perp$. Let $\overline{\iota} \colon (\mathcal O_2 \otimes \mathcal K_1)^\sim \to \multialg{B \otimes \mathcal K_1} \subseteq \multialg{B\otimes \mathcal K}$ be the $\ast$-homomorphism induced by $\iota$. Let $(u_t)_{t\in \mathbb R_+}$ be a continuous unitary path in $(\mathcal O_2 \otimes \mathcal K_1)^\sim$ satisfying $(a)$ and $(b)$ of Lemma \ref{l:unitarypath} (with $e_{2,2}$ in place of $e_{1,1}$), and let $v_t := 1_{\multialg{B}} \otimes e_{1,1} + \overline{\iota}(u_t)$ be the induced unitary path in $\multialg{B\otimes \mathcal K}$.

Note that
\begin{eqnarray}
&&  (1_{\multialg{B}} \otimes e_{1,1}) (b \otimes e_{1,1} + \theta(a) \otimes 1_{\multialg{\mathcal K_1}}) \nonumber\\
&=& b \otimes e_{1,1} \nonumber\\
&=&  (b \otimes e_{1,1} + \theta(a) \otimes 1_{\multialg{\mathcal K_1}}) (1_{\multialg{B}} \otimes e_{1,1})
\end{eqnarray}
for all $a\in A$ and $b\in B$. Clearly $\overline{\iota}(u_t)$ commutes with $\theta(A) \otimes 1_{\multialg{\mathcal K_1}}$, and $\overline{\iota}(u_t)$ annihilates $B \otimes e_{1,1}$.

As $\overline{\iota}(\mathcal O_2 \otimes \mathcal K_1) = \iota(\mathcal O_2) \otimes \mathcal K_1 \subseteq \multialg{B} \otimes \mathcal K_1$ is a non-degenerate $C^\ast$-algebra, it follows from property $(a)$ of Lemma \ref{l:unitarypath}, that $\overline{\iota}(u_t (1_{\mathcal O_2} \otimes e_{2,2}) u_t^\ast)$ acts as an approximate identity on the corner $B \otimes \mathcal K_1$ of $B \otimes \mathcal K$. To ease notation, let $p_2 := 1_{\multialg{B}} \otimes (e_{1,1} + e_{2,2}) \in \multialg{B\otimes \mathcal K}$. Then
\begin{equation}
v_t p_2 v_t^\ast = 1_{\multialg{B}} \otimes e_{1,1} + \overline{\iota}(u_t (1_{\mathcal O_2} \otimes e_{2,2}) u_t^\ast)
\end{equation}
acts as an approximate identity on $B\otimes \mathcal K$. Thus, as $w_t \in (B\otimes \mathcal K)^\sim$ is continuous we may assume (by possibly reparametrising $v_t$) that 
\begin{equation}\label{eq:DEcommute}
\lim_{t\to \infty} \| [v_t p_2 v_t^\ast, w_t]\| = 0.
\end{equation}
Let $v$ and $w$ in $\multialg{B \otimes \mathcal K}_\as$ be the elements induced by $(v_t)_{t\in \mathbb R_+}$ and $(w_t)_{t\in \mathbb R_+}$ respectively. To see that $\phi \oplus \theta \sim_{\asu} \psi \oplus \theta$ it suffices to find a partial isometry $V$ in $\multialg{B\otimes \mathcal K}_\as$ such that $VV^\ast = V^\ast V = p_2$, and 
\begin{equation}
V (\phi(a) \otimes e_{1,1} + \theta(a) \otimes e_{2,2}) V^\ast =  \psi(a) \otimes e_{1,1} + \theta(a) \otimes e_{2,2}
\end{equation}
for all $a\in A$. We will check that
\begin{equation}
V := p_2 v^\ast w v p_2 \in \multialg{B \otimes \mathcal K}_\as
\end{equation}
does the trick. As $v$ and $w$ are unitaries, we get
\begin{equation}
VV^\ast = p_2 v^\ast w v  p_2 v^\ast w^\ast v  p_2 \stackrel{\eqref{eq:DEcommute}}{=}  p_2 v^\ast w w^\ast v p_2 = p_2,
\end{equation}
and similarly $V^\ast V = p_2$. For any $a\in A$ we have
\begin{eqnarray}
&& V ( \phi (a) \otimes e_{1,1} + \theta(a) \otimes e_{2,2}) V^\ast  \nonumber\\
&=& p_2 v^\ast w v p_2 ( \phi (a) \otimes e_{1,1} + \theta(a) \otimes e_{2,2}) p_2 v^\ast w^\ast v p_2 \nonumber\\
&= & p_2 v^\ast w v ( \phi (a) \otimes e_{1,1} + \theta(a) \otimes 1_{\multialg{\mathcal K_1}}) p_2 v^\ast w^\ast v  p_2 \nonumber\\
&\stackrel{(\ast)}{=}&  p_2 v^\ast w ( \phi (a) \otimes e_{1,1} + \theta(a) \otimes 1_{\multialg{\mathcal K_1}}) v p_2 v^\ast w^\ast v  p_2 \nonumber\\
&\stackrel{\eqref{eq:DEcommute}}{=} & p_2 v^\ast w ( \phi (a) \otimes e_{1,1} + \theta(a) \otimes 1_{\multialg{\mathcal K_1}}) w^\ast v p_2 \nonumber\\
&\stackrel{\eqref{eq:DElemma}}{=}&  p_2 v^\ast ( \psi (a) \otimes e_{1,1} + \theta(a) \otimes 1_{\multialg{\mathcal K_1}}) v  p_2 \nonumber\\
&\stackrel{(\ast)}{=}& p_2 ( \psi (a) \otimes e_{1,1} + \theta(a) \otimes 1_{\multialg{\mathcal K_1}})  p_2 \nonumber\\
&=& \psi(a) \otimes e_{1,1} + \theta(a) \otimes e_{2,2}.
\end{eqnarray}
At the equations labeled with $(\ast)$, it was used that $v$ commutes with $b \otimes e_{1,1} + \theta(a) \otimes 1_{\multialg{\mathcal K_1}}$, as was shown earlier in the proof. Thus $\phi \oplus \theta \sim_{\asu} \psi \oplus \theta$.

For the ``in addition'' part, note that the existence of a unital embedding $\mathcal O_2 \to \multialg{B} \cap \theta(A)'$ implies that $\theta$ is strongly $\mathcal O_2$-stable. Therefore $\phi \sim_\asMvN \psi$ by combining Propositions \ref{p:piabsorbing} and \ref{p:MvNeq}.
\end{proof}


\section{The Kirchberg--Phillips Theorem}\label{s:KP}

This section contains the proofs of Theorems \ref{t:existsimple}, \ref{t:uniquesimple}, and \ref{t:KP}. Note that Theorem \ref{t:KPUCT} -- classification of Kirchberg algebras satisfying the UCT -- is an easy corollary of Theorem \ref{t:KP} as explained in the introduction, so I consider the proof of Theorem \ref{t:KPUCT} as complete.

\subsection{KK-preliminaries}\label{ss:KKprel}

For the construction of groups $KK(A,B)$ by Kasparov and Skandalis' variation $KK_\nuc(A,B)$ the reader is referred to \cite{Blackadar-book-K-theory} and \cite{Skandalis-KKnuc} respectively. The constructions can also be obtained from Section \ref{s:KK} as a special case. Two pictures of $KK_\nuc$ will be emphasised: the Fredholm picture and the Cuntz pair picture. For this, fix a separable $C^\ast$-algebra $A$, and a $\sigma$-unital $C^\ast$-algebra $B$. All Hilbert modules $E$ are assumed to be right modules and countably generated. 

For Hilbert modules $E,F$, let $\mathcal B(E,F)$ be the space of adjointable operators $E\to F$, and let $\mathcal K(E,F)$ be the closed linear span of rank one operators $E\to F$. I will refer to $\mathcal K(E,F)$ as the \emph{``compacts''}.\footnote{I write ``compacts'' instead of compacts to emphasise that ``compact'' operators are not actually compact in the usual sense.}

\begin{remark}[The Fredholm picture]
If $E$ is a Hilbert $B$-module, then a $\ast$-homo\-morphism $\psi \colon A \to \mathcal B(E)$ is called \emph{weakly nuclear} if $\langle \xi, \psi(-) \xi \rangle_E \colon A \to B$ is a nuclear map for any $\xi \in E$. In the case $E=B$, one has $\mathcal B(E) = \multialg{B}$  and so this definition agrees with Definition \ref{d:weaklynuc}.

A triple $(\psi_0, \psi_1, v)$ is called a \emph{weakly nuclear cycle} if $\psi_i \colon A \to \mathcal B(E_i)$ are weakly nuclear $\ast$-homomorphisms with $E_i$ Hilbert $B$-modules, and $v\in \mathcal B(E_0,E_1)$ satisfies that
\begin{equation}
v\psi_0(a) - \psi_1(a) v, \qquad \psi_0(a) (v^\ast v - 1_{\mathcal B(E_0)}), \qquad \psi_1(a) (vv^\ast - 1_{\mathcal B(E_1)}) ,
\end{equation}
are ``compact'' for all $a\in A$. If these are all zero for every $a\in A$ then $(\psi_0, \psi_1 , v)$ is called \emph{degenerate}. One may add two weakly nuclear cycles by taking their direct sum in the obvious way.

If $(\psi_0, \psi_1,v)$ is a weakly nuclear cycle and $u\in \mathcal B(E_0', E_0)$ is a unitary, then $(\Ad u \circ \psi_0, \psi_1, v u)$ is a weakly nuclear cycle, and the two cycles $(\psi_0,\psi_1, v)$ and $(\Ad u \circ \psi_0, \psi_1, vu)$ are said to be \emph{unitarily equivalent}. Similarly, one defines unitary equivalence for a unitary $u \in \mathcal B(E_1, E_1')$.

Say that two weakly nuclear cycles $(\psi_0, \psi_1 , v_0)$ and $(\psi_0, \psi_1, v_1)$ are \emph{operator homotopic} if there is a continuous path $(v_s)_{s\in [0,1]}$ from $v_0$ to $v_1$ such that each $(\psi_0, \psi_1, v_s)$ is weakly nuclear cycle.

Then $KK_\nuc(A,B)$ is (canonically isomorphic to) the set of weakly nuclear cycles modulo the equivalence relation generated by unitary equivalence, addition of degenerate weakly nuclear cycles, and operator homotopy. The isomorphism is given by taking a weakly nuclear cycle $(\psi_0, \psi_1, v)$ to the Kasparov module $\left( E_0 \oplus E_1^\op, \psi_0\oplus \psi_1, \left(\begin{array}{cc} 0 & v^\ast \\ v & 0 \end{array}\right) \right)$.\footnote{Here $E_0$ and $E_1$ have the trivial $\mathbb Z/2\mathbb Z$-grading where every element has degree $0$, whereas $E_1^\op$ is $E_1$ with the opposite grading, i.e.~every element in $E_1^\op$ has degree 1.}

One may construct $KK(A,B)$ in the exact same way by removing the words ``weakly nuclear'' everywhere. Hence it obviously follows that there is a canonical homomorphism
\begin{equation}
KK_\nuc(A,B) \to KK(A,B),
\end{equation}
and that this is an isomorphism if either $A$ or $B$ is nuclear (as any cycle will automatically be weakly nuclear).\footnote{It is a somewhat common misconception that $KK_\nuc(A, B) \to KK(A,B)$ is injective, or equivalently, that $KK_\nuc(A,B)$ is the subgroup of $KK(A,B)$ generated by weakly nuclear cycles. Though I know of no example where this is not true, there is a priori no reason to believe that this is the case.}
\end{remark}

\begin{remark}[The Cuntz pair picture]
A pair $(\psi_0, \psi_1)$ of $\ast$-homomorphism $A \to \multialg{B\otimes \mathcal K}$ is called a \emph{Cuntz pair} if $\psi_0(a) - \psi_1(a)\in B \otimes \mathcal K$ for all $a\in A$. A Cuntz pair $(\psi_0, \psi_1)$ is called \emph{weakly nuclear} if $\psi_0$ and $\psi_1$ are both weakly nuclear.\footnote{Note that this is the case exactly when $(\psi_0, \psi_1, 1_{\multialg{B\otimes \mathcal K}})$ is a weakly nuclear cycle. To make sense of this, use the Hilbert $B$-modules $E_0 = E_1 = \ell^2(\mathbb N) \otimes B$.} Say that two weakly nuclear Cuntz pairs $(\phi_0, \phi_1)$ and $(\psi_0, \psi_1)$ are \emph{homotopic} if there is a family $(\eta_0^{(s)}, \eta_1^{(s)})_{s\in [0,1]}$ of weakly nuclear Cuntz pairs such that 
\begin{equation}
(\eta_0^{(0)}, \eta_1^{(0)}) = (\phi_0, \phi_1), \qquad (\eta_0^{(1)} , \eta_1^{(1)}) = (\psi_0, \psi_1),
\end{equation}
the map $[0,1] \ni s \mapsto \eta_i^{(s)}(a)$ is strictly continuous for $i=0,1$ and $a\in A$, and $[0,1] \ni s\mapsto \eta_0^{(s)}(a) - \eta_1^{(s)}(a)$ is norm-continuous for any $a\in A$. 

One can form sums of weakly nuclear Cuntz pairs by 
\begin{equation}
(\phi_0, \phi_1) \oplus_{s_1,s_2} (\psi_0, \psi_1) = (\phi_0 \oplus_{s_1,s_2} \psi_0, \phi_1 \oplus_{s_1,s_2} \psi_1),
\end{equation}
 where $s_1,s_2\in \multialg{B\otimes \mathcal K}$ are $\mathcal O_2$-isometries. As any two Cuntz sums are unitarily equivalent, and as the unitary group of $\multialg{B\otimes \mathcal K}$ is path-connected by \cite{CuntzHigson-Kuipersthm}, sums of weakly nuclear Cuntz pairs are unique up to homotopy.

The map $(\psi_0, \psi_1) \mapsto (\psi_0, \psi_1, 1_{\multialg{B\otimes \mathcal K}})$ induces an isomorphism of homotopy classes of weakly nuclear Cuntz pairs and $KK_\nuc(A,B)$ (in the Fredholm picture).

A nuclear $\ast$-homomorphism $\phi \colon A \to B \otimes \mathcal K$ induces an element $KK_\nuc(\phi)$ in $KK_\nuc(A,B)$ via the weakly nuclear Cuntz pair $(\phi, 0)$. If $\psi \colon A \to B \otimes \mathcal K$ is another nuclear $\ast$-homo\-morphism then the Cuntz pair $(\phi, \psi)$ induces the element $KK_\nuc(\phi) - KK_\nuc(\psi)$.

Finally, suppose $\theta \colon A \to B \otimes \mathcal K$ is a nuclear $\ast$-homomorphism such that there exist $\mathcal O_2$-isometries $s_1,s_2 \in \multialg{B \otimes \mathcal K} \cap \theta(A)'$. Then $KK_\nuc(\theta) + KK_\nuc(\theta)$ is represented by the weakly nuclear Cuntz pair
\begin{equation}
(\theta, 0) \oplus_{s_1,s_2} (\theta, 0 ) = (s_1 \theta(-) s_1^\ast + s_2 \theta(-)s_2^\ast, 0) = ((s_1s_1^\ast + s_2 s_2^\ast)\theta(-), 0) = (\theta, 0).
\end{equation}
Hence $KK_\nuc(\theta) + KK_\nuc(\theta) = KK_\nuc(\theta)$ which implies that $KK_\nuc(\theta)=0$ since $KK_\nuc(A,B)$ is a group.
\end{remark}

\begin{remark}[The Kasparov product]
For separable $C^\ast$-algebras $A,B$ and $C$ there is a canonical homomorphism
\begin{equation}
\circ \colon KK(B,C) \otimes KK(A,B) \to KK(A,C)
\end{equation}
called the \emph{Kasparov product} which satisfies
\begin{equation}
KK(\psi) \circ KK(\phi) = KK(\psi \circ \phi)
\end{equation}
whenever $\phi \colon A\to B$ and $\psi \colon B \to C$ are $\ast$-homomorphisms.\footnote{It is somewhat unconventional to denote the Kasparov product by the symbol $\circ$, although it appears like this in \cite[Section 18.1]{Blackadar-book-K-theory}. Often one uses $\times$ to denote the Kasparov product, in which case the product is usually reversed, so that $KK(\phi) \times KK(\psi) = KK(\psi \circ \phi)$. I write the Kasparov product $\circ$ in the same order as one composes maps.} Similarly, there are Kasparov products
\begin{eqnarray}
\circ \colon KK_\nuc(B,C) \otimes KK(A,B) &\to& KK_\nuc(A,C), \\
\circ \colon KK(B,C) \otimes KK_\nuc(A,B) &\to& KK_\nuc(A,C),
\end{eqnarray}
which are also well-behaved with respect to composition of morphisms when one is nuclear. Unfortunately, there is (at least to my knowledge) no way of describing the Kasparov product using either the Fredholm picture or the Cuntz pair picture. The reader is referred to \cite{Blackadar-book-K-theory} and \cite{Skandalis-KKnuc} for details. Alternatively, everything will be reproved in a much more general setting in Section \ref{s:KK}.
\end{remark}

The following elementary lemma will be used.

\begin{lemma}\label{l:asMvNKKnuc}
Let $A$ be a separable $C^\ast$-algebra, let $B$ be a $\sigma$-unital $C^\ast$-algebra, and let $\phi,\psi \colon A\to B$ be nuclear $\ast$-homomorphisms. If $\phi \sim_\asMvN \psi$ then $KK_\nuc(\phi) = KK_\nuc(\psi)$.
\end{lemma}
\begin{proof}
As the corner inclusion $\id_B \oplus 0 \colon B \to M_2(B)$ induces an isomorphism $KK_\nuc(A, B) \cong KK_\nuc(A, M_2(B))$, it is enough to show that $KK_\nuc(\phi\oplus 0) = KK_\nuc(\psi \oplus 0)$. By Proposition \ref{p:MvNeq}, $\phi \oplus 0$ and $\psi \oplus 0$ are asymptotically unitarily equvalent, say by a unitary path $(u_s)_{s\in [0,1)} \in \multialg{M_2(B)}$. Letting $\eta_s := \Ad u_s \circ (\phi \oplus 0)$ for $s\in [0,1)$, and $\eta_1 = \psi \oplus 0$, one obtains a point-wise nuclear homotopy from $\Ad u_0 \circ (\phi\oplus 0)$ to $\psi \oplus 0$, and thus $KK_\nuc(\phi \oplus 0) = KK_\nuc(\Ad u_0 \circ (\phi \oplus 0)) = KK_\nuc(\psi \oplus 0)$.
\end{proof}

\subsection{Proofs of Theorems \ref{t:existsimple}, \ref{t:uniquesimple} and \ref{t:KP}}

Before proving the main results, a few lemmas will be required. The first uses Kirchberg's celebrated $\mathcal O_2$-embedding theorem, see \cite{Kirchberg-ICM} or \cite{KirchbergPhillips-embedding}.

\begin{lemma}\label{l:fullO2map}
Let $A$ be a separable, exact $C^\ast$-algebra, and let $B$ be a $\sigma$-unital $C^\ast$-algebra which contains a full, properly infinite projection. Then $B$ contains a $\sigma$-unital, stable, full, hereditary $C^\ast$-subalgebra, and there exists full, nuclear $\ast$-homomorphism $\theta \colon A \to B\otimes \mathcal K$ such that $\mathcal O_2$ embeds unitally in $\multialg{B\otimes \mathcal K} \cap \theta(A)'$.
\end{lemma}
\begin{proof}
As $B$ contains a full, properly infinite projection $p$, there is a full embedding obtained as a composition $\mathcal O_2 \hookrightarrow \mathcal O_\infty \hookrightarrow pBp \subseteq B$, where we used that $pBp$ contains a unital copy of $\mathcal O_\infty$, see \cite[Proposition 4.2.3]{Rordam-book-classification}. Let $\eta$ denote the composition
\begin{equation}
\mathcal O_2 \otimes \mathcal O_2 \otimes \mathcal K \hookrightarrow \mathcal O_2 \hookrightarrow B
\end{equation}
for some embedding $\mathcal O_2 \otimes \mathcal O_2 \otimes \mathcal K \hookrightarrow \mathcal O_2$ which exists by the $\mathcal O_2$-embedding theorem \cite{KirchbergPhillips-embedding}. Let $B_0$ denote the hereditary $C^\ast$-subalgebra of $B$ generated by the image of $\eta$. Clearly $B_0$ is a $\sigma$-unital, full, hereditary $C^\ast$-subalgebra, and it is also stable (see e.g.~\cite[Proposition 4.4]{HjelmborgRordam-stability}), completing the first part of the proof.

Brown's stable isomorphism theorem \cite{Brown-stableiso} implies that $B_0 \cong B\otimes \mathcal K$. So we may replace $B\otimes \mathcal K$ with $B_0$. Let $\theta$ denote the composition
\begin{equation}
A \xrightarrow j \mathcal O_2 \xrightarrow{1_{\mathcal O_2} \otimes \id_{\mathcal O_2} \otimes e_{1,1}} \mathcal O_2 \otimes \mathcal O_2 \otimes \mathcal K \xrightarrow \eta B_0,
\end{equation}
where $j \colon A \hookrightarrow \mathcal O_2$ is an embedding, whose existence again uses the $\mathcal O_2$-embedding theorem. Clearly $\theta$ is full and nuclear.  The composition
\begin{equation}
\mathcal O_2 \xrightarrow{\id_{\mathcal O_2} \otimes 1_{\multialg{\mathcal O_2 \otimes \mathcal K}}} \multialg{\mathcal O_2 \otimes \mathcal O_2 \otimes \mathcal K} \xrightarrow{\multialg{\eta}} \multialg{B_0}
\end{equation}
gives a unital embedding of $\mathcal O_2$ in $\multialg{B_0} \cap \theta(A)'$.
\end{proof}

\begin{lemma}\label{l:absCuntzpair}
Let $A$ be a separable $C^\ast$-algebra, let $B$ be a $\sigma$-unital $C^\ast$-algebra, and let $\Theta \colon A \to \multialg{B\otimes \mathcal K}$ be a weakly nuclear, nuclearly absorbing representation. Then any element $x\in KK_\nuc(A,B)$ is represented by a weakly nuclear Cuntz pair of the form $(\psi, \Theta)$.
\end{lemma}
\begin{proof}
The proof is a chain of standard reductions. In the following,  let $\mathcal H_B := \ell^2(\mathbb N)\otimes B$ and  identify $\multialg{B\otimes \mathcal K}$ with $\mathcal B(\mathcal H_B)$.

Represent $x$ by a weakly nuclear Cuntz pair $(\psi_0,\psi_1)$. Then the weakly nuclear cycle $(\psi_0,\psi_1,1_{\mathcal B(\mathcal H_B)})$ represents $x$. Let 
\begin{equation}
\Psi' = \psi_0 \oplus \psi_1\oplus \psi_0 \oplus \psi_1 \oplus \cdots \colon A \to \mathcal B(E)
\end{equation}
where $E= \bigoplus_{\mathbb N} \mathcal H_B$. Note that $\Psi'$ is weakly nuclear. Then $(\psi_0\oplus \Psi' , \psi_1 \oplus \Psi', 1_{\mathcal B(\mathcal H_B)} \oplus 1_{\mathcal B(E)})$ also represents $x$, as $(\Psi',\Psi', 1_{\mathcal B(E)})$ is degenerate. Let $w$ be a unitary in $\mathcal B(\mathcal H_B \oplus E) = \mathcal B(\mathcal H_B \oplus \bigoplus_{\mathbb N} \mathcal H_B)$ which permutes the direct summands, so that $\Ad w \circ \psi_0 \oplus \Psi' = \psi_1 \oplus \Psi'$. As $(\psi_0 \oplus \Psi', \psi_1 \oplus \Psi', 1_{\mathcal B(\mathcal H_B)}\oplus 1_{\mathcal B(E)})$ and $(\psi_1 \oplus \Psi', \psi_1 \oplus \Psi', w)$ are unitarily equivalent, the latter represents $x$. Let $\Psi'' = \psi_1 \oplus \Psi'$, so that $(\Psi'', \Psi'', w)$ represents $x$. Thus $(\Psi'' \oplus \Theta, \Psi'' \oplus \Theta, w \oplus 1_{\mathcal B(\mathcal H_B)})$ also represents $x$.

As $\Theta$ is weakly nuclear and nuclearly absorbing we may find a unitary $u \in \mathcal B(\mathcal H_B, \mathcal H_B\oplus E \oplus \mathcal H_B)$ such that $\Ad u \circ (\Psi'' \oplus \Theta)(a) - \Theta(a)$ is ``compact'' for all $a\in A$. Let $\Psi = \Ad u \circ (\Psi'' \oplus \Theta)$ and $v = u^\ast (w \oplus 1_{\mathcal B(\mathcal H_B)}) u$. Note that $v$ is a unitary. Since $(\Psi, \Psi, v)$ and $(\Psi'' \oplus \Theta , \Psi'' \oplus \Theta , w \oplus 1_{\mathcal B(\mathcal H_B)})$ are unitarily equivalent, the former represents $x$.

Clearly $(\Psi \oplus \Theta, \Psi \oplus \Theta, v \oplus 1_{\mathcal B(\mathcal H_B)})$ represents $x$, since $(\Theta, \Theta, 1_{\mathcal B(\mathcal H_B)})$ is degenerate. Let $R_t = \left( \begin{array}{cc} \cos(t\pi/2) & \sin(t\pi /2) \\ - \sin(t\pi /2) & \cos(t\pi /2) \end{array}\right) \in M_2(\mathcal B(\mathcal H_B)) = \mathcal B(\mathcal H_B \oplus \mathcal H_B)$ for $t\in [0,1]$ be the usual path of $2\times 2$ unitary rotation matrices. As $\Psi$ and $\Theta$ agree modulo the ``compacts'', $(\Psi \oplus \Theta)(a)$ is of the form $(\Theta \oplus \Theta)(a)$ modulo the ``compacts''. As $R_t$ has scalar entries, it follows that $R_t$ and $(\Psi \oplus \Theta)(a)$ commute modulo the ``compacts''. It follows that $(\Psi \oplus \Theta, \Psi \oplus \Theta, R_t^\ast (v \oplus 1_{\mathcal B(\mathcal H_B)}) R_t)$ defines an operator homotopy from $(\Psi \oplus \Theta, \Psi \oplus \Theta, v \oplus 1_{\mathcal B(\mathcal H_B)})$ to $(\Psi \oplus \Theta, \Psi \oplus \Theta, 1_{\mathcal B(\mathcal H_B)} \oplus v) = (\Psi, \Psi, 1_{\mathcal B(\mathcal H_B)}) \oplus (\Theta, \Theta, v)$, so the latter represents $x$. As $(\Psi, \Psi, 1_{\mathcal B(\mathcal H_B)})$ is degenerate it follows that $(\Theta, \Theta, v)$ represents $x$.

Finally, let $\psi = \Ad v \circ \Theta$, which is a weakly nuclear $\ast$-homomorphism since $v$ is a unitary. Then $(\Theta, \Theta, v)$ and $(\psi, \Theta, 1_{\mathcal B(\mathcal H_B)})$ are unitarily equivalent, so $x$ is represented by the weakly nuclear Cuntz pair $(\psi, \Theta)$.
\end{proof}

With these ingredients, the proof of Theorem \ref{t:existsimple} -- the main existence part in the Kirchberg--Phillips classification theorem -- is ready to be handled.

\begin{proof}[Proof of Theorem \ref{t:existsimple}]
As $B$ contains a full, properly infinite projection, there is a $\sigma$-unital, stable, full, hereditary $C^\ast$-subalgebra $B_0 \subseteq B$ by Lemma \ref{l:fullO2map}. As the inclusion $\iota \colon B_0 \hookrightarrow B$ induces an isomorphism
\begin{equation}
\iota_\ast \colon KK_\nuc(A,B_0 ) \xrightarrow \cong KK_\nuc(A,B),
\end{equation}
there is a unique $y\in KK_\nuc(A,B_0)$ such that $\iota_\ast( y) = x$. Hence we may assume that $B$ is stable.

By Lemma \ref{l:fullO2map}, there is a full, nuclear $\ast$-homomorphism $\theta \colon A \to B$ such that $\mathcal O_2$ embeds unitally in $\multialg{B} \cap \theta(A)'$. Let $s_1,s_2,\dots \in \multialg{B}$ be isometries with mutually orthogonal range projections such that $\sum_{k=1}^\infty s_k s_k^\ast =1_{\multialg{B}}$. By Theorem \ref{t:fullnucabs}, the infinite repeat $\theta_\infty = \sum_{k=1}^\infty s_k \theta(-) s_k^\ast$ is weakly nuclear and nuclearly absorbing. By Lemma \ref{l:absCuntzpair}, $x\in KK_\nuc(A,B)$ is represented by a weakly nuclear Cuntz pair of the form $(\psi , \theta_\infty)$. Let $(u_t)_{t\in \mathbb [0,1)}$ be a unitary path in $\multialg{B} \cap \theta_\infty(A)'$ as given by Lemma \ref{l:keyexistence},\footnote{As in Remark \ref{r:infrep}, we use that $\Omega \colon B \otimes \mathcal K \to B$ given on elementary tensors by $b\otimes e_{i,j} = s_i b s_j^\ast$ (once extended to multiplier algebras) maps $\theta \otimes 1_{\multialg{\mathcal K}}$ to $\theta_\infty$, and $\phi(a) \otimes e_{1,1} + \theta(a) \otimes (1_{\multialg{\mathcal K}} - e_{1,1})$ to $s_1 \phi(a) s_1^\ast + \sum_{k=2}^\infty s_k \theta(a) s_k^\ast$.} and let $\phi \colon A \to B$ be the $\ast$-homomorphism such that $u_t \psi(-) u_t^\ast$ converges point-norm to $\phi_0 := s_1\phi(-)s_1^\ast + \sum_{k=2}^\infty s_k \theta(-) s_k^\ast$. As $\phi_0$ is weakly nuclear it follows that $\phi = s_1^\ast \phi_0(-) s_1$ is weakly nuclear and thus nuclear since it takes values in $B$ (as opposed to $\multialg{B}$). As $B$ is stable and $1_{\multialg{B}} - s_1s_1^\ast \geq s_2s_2^\ast$, we may fix an isometry $s_0 \in \multialg{B}$ with $s_0s_0^\ast = 1_{\multialg{B}} - s_1 s_1^\ast$ so that $s_1,s_0$ are $\mathcal O_2$-isometries.\footnote{Here the following well-known fact is used: if $B$ is stable and $p\in \multialg{B}$ is a projection so that $1_{\multialg{B}}\preceq p$, then $1_{\multialg{B}} \sim p$. To see this, using that $1_{\multialg{B}}$ is properly infinite, one gets $p\oplus p \leq 1_{\multialg{B}} \oplus 1_{\multialg{B}} \preceq 1_{\multialg{B}} \preceq p$. From \cite[Proposition 1.1.2]{Rordam-book-classification} it follows that $p$ is properly infinite, and $p$ is full since $1_{\multialg{B}}\preceq p$. As $K_0(\multialg{B}) = 0$ by \cite[Proposition 12.2.1]{Blackadar-book-K-theory}, it follows from \cite[Proposition 4.1.4]{Rordam-book-classification} that all properly infinite full projections in $\multialg{B}$ are equivalent, and in particular $p\sim 1_{\multialg{B}}$.} Define $\theta_0 := \sum_{k=2}^\infty s_0^\ast s_k \theta(-) s_k^\ast s_0$ (which is also an infinite repeat of $\theta$). Then $\phi_0$ can be expressed as the Cuntz sum $\phi \oplus_{s_1,s_0} \theta_0$ and similarly $\theta_\infty = \theta \oplus_{s_1,s_0} \theta_0$.

We obtain a homotopy of weakly nuclear Cuntz pairs, from $(\psi, \theta_\infty)$ to $(\phi_0 , \theta_\infty) = (\phi, \theta) \oplus_{s_1,s_0} (\theta_0, \theta_0)$, given by
\begin{equation}
(\Ad u_t \circ \psi, \Ad u_t \circ \theta_\infty) = (\Ad u_t \circ \psi, \theta_\infty)
\end{equation}
for $t\in (0,1)$. Thus $x$ is represented by $(\phi, \theta) \oplus_{s_1,s_0} (\theta_0, \theta_0)$. As $(\theta_0, \theta_0)$ is degenerate, $x$ is represented by $(\phi, \theta)$. Hence $x = KK_\nuc(\phi) - KK_\nuc(\theta)$. As $\mathcal O_2$ embeds unitally in $\multialg{B} \cap \theta(A)'$ it follows that $KK_\nuc(\theta) = 0$. Therefore $x = KK_\nuc(\phi)$. Note that any Cuntz sum $\phi \oplus_{s_1,s_0}\theta$ is a \emph{full}, nuclear $\ast$-homomorphism. Since $\theta$ is full and $\phi$ is nuclear, $\theta$ approximately dominates $\phi$ by Proposition \ref{p:fulldom}. As $\theta$ is strongly $\mathcal O_\infty$-stable by Proposition \ref{p:Oinftyfactor}, as it factors through $\mathcal O_2$, it follows from Proposition \ref{p:piabsorbing}$(a)$ that $\phi \oplus_{s_1,s_0} \theta$ is strongly $\mathcal O_\infty$-stable. Replacing $\phi$ with $\phi \oplus_{s_1,s_0} \theta$ we thus have a full, nuclear, strongly $\mathcal O_\infty$-stable $\ast$-homomorphism, such that $KK_\nuc(\phi) = x$.

It remains to prove  the ``moreover'' part, so from now on we assume that $A$ and $B$ are unital.

``Only if'': For any unital $C^\ast$-algebra $C$ the element $[1_C]_0 \in K_0(C) = KK(\mathbb C, C)$ is induced by the unique unital $\ast$-homomorphism $\eta_C \colon \mathbb C \to C$. Hence if $\phi \colon A \to B$ is unital and $KK_\nuc(\phi) = x$ then 
\begin{equation}
\Gamma_0(x)([1_A]_0) = KK_\nuc(\phi) \circ KK(\eta_A) = KK_\nuc(\phi \circ \eta_A) = KK_\nuc(\eta_B) = [1_B]_0.
\end{equation}
``If'': By the not necessarily unital version of the theorem we may find a full, nuclear, strongly $\mathcal O_\infty$-stable $\ast$-homo\-morphism $\phi_0 \colon A \to B$ such that $KK_\nuc(\phi_0) = x$. The strong $\mathcal O_\infty$-stability of $\phi_0$ implies that $\phi_0(1_A)$ is a properly infinite projection, see Remark \ref{r:fullpropinfproj}. As $[\phi_0(1_A)]_0 = \Gamma_0(x)([1_A]_0) = [1_B]_0$ and as both $\phi_0(1_A)$ and $1_B$ are properly infinite,\footnote{If $B$ is unital and contains a full, properly infinite projection, then $1_B$ is properly infinite.} full projections, it follows from a result of Cuntz \cite{Cuntz-K-theoryI} that $\phi_0(1_A)$ and $1_B$ are Murray--von Neumann equivalent. So we may pick $v \in B$ an isometry such that $vv^\ast = \phi_0(1_A)$ and define $\phi := v^\ast \phi_0(-) v \colon A \to B$. Then $\phi$ is a full, nuclear, unital, strongly $\mathcal O_\infty$-stable $\ast$-homomorphism, and as we clearly have $v\phi(-)v^\ast = \phi_0$ it follows that $\phi \sim_\asMvN \phi_0$ (implemented by the constant path $v$). Hence by Lemma \ref{l:asMvNKKnuc}
\begin{equation}
KK_\nuc(\phi) = KK_\nuc(\phi_0) = x.\qedhere
\end{equation}
\end{proof}

In the proof of Theorem \ref{t:uniquesimple} presented below, the following stable uniqueness theorem of Dadarlat and Eilers \cite{DadarlatEilers-asymptotic} will be needed. 

\begin{theorem}[{\cite[Theorem 3.10]{DadarlatEilers-asymptotic}}]\label{t:DE}
Let $A$ be a separable, exact $C^\ast$-algebra, let $B$ be a $\sigma$-unital, stable $C^\ast$-algebra, let $\phi, \psi \colon A \to \multialg{B}$ be weakly nuclear $\ast$-homomorphisms for which $\phi(a)- \psi(a) \in B$ for all $a\in A$, and suppose that $\theta \colon A \to B$ is a full, nuclear $\ast$-homomorphism. If $[\phi, \psi]$ vanishes in $KK_\nuc(A,B)$, then there is a continuous path $(u_t)_{t\in \mathbb R_+}$ of unitaries in the unitisation $\widetilde B$ such that
\begin{equation}\label{eq:DE}
 \lim_{t\to \infty} \| u_t (\phi \oplus_{s_1,s_2} \theta_\infty)(a) u_t^\ast - (\psi \oplus_{s_1,s_2} \theta_\infty)(a) \| =0
\end{equation}
for all $a\in A$. Here $s_1,s_2\in \multialg{B}$ are $\mathcal O_2$-isometries and $\theta_\infty$ is an infinite repeat of $\theta$.
\end{theorem}

\begin{proof}[Proof of Theorem \ref{t:uniquesimple}]
$(ii)\Rightarrow (i)$ follows from Lemma \ref{l:asMvNKKnuc}, and $(ii)\Leftrightarrow (iii)$ is Proposition \ref{p:MvNvsue}, where one uses Remark \ref{r:fullpropinfproj} to see that $B$ contains a full projection.

$(i)\Rightarrow (ii)$: Suppose $KK_\nuc(\phi) = KK_\nuc(\psi)$. By Proposition \ref{p:MvNeq}, it suffices to show that $\phi \otimes e_{1,1}, \psi \otimes e_{1,1} \colon A \to B\otimes \mathcal K$ are asymptotically Murray--von Neumann equivalent, so we may assume without loss of generality that $B$ is stable. Lemma \ref{l:fullO2map} produces a full, nuclear $\ast$-homo\-morphism $\theta \colon A \to B$ such that $\mathcal O_2$ embeds unitally in $\multialg{B} \cap \theta(A)'$. Proposition \ref{p:fulldom} implies that $\phi$ and $\psi$ approximately dominate $\theta$.  By Theorem \ref{t:fullnucabs} it follows that any infinite repeat of $\theta$ is nuclearly absorbing. Thus by Theorem \ref{t:DE} and Lemma \ref{l:keyuniqueness} -- the key lemma for uniqueness --  it follows that $\phi \sim_\asMvN \psi$.
\end{proof}


As a consequence of Theorems \ref{t:existsimple} and \ref{t:uniquesimple} one obtains the proof of Theorem \ref{t:KP}.

\begin{proof}[Proof of Theorem \ref{t:KP}]
$(a)$: By Theorem \ref{t:existsimple} we find full, nuclear, strongly $\mathcal O_\infty$-stable $\ast$-homo\-morphisms $\phi_0 \colon A \to B$ and $\psi_0 \colon B \to A$ such that $KK(\phi_0) = x$ and $KK(\psi_0) = x^{-1}$. Note that $KK(\psi_0 \circ \phi_0) = KK(\id_A)$ and $KK(\phi_0 \circ \psi_0) = KK(\id_B)$. Hence by Theorem \ref{t:uniquesimple} it follows that $\psi_0 \circ \phi_0 \sim_\asu \id_A$ and $\phi_0 \circ \psi_0 \sim_\asu \id_B$. By Proposition \ref{p:asint}\footnote{Note that an approximate intertwining argument a la Elliott (see \cite[Corollary 2.3.4]{Rordam-book-classification}) implies that there is an isomorphism $\phi \colon A \xrightarrow \cong B$ satisfying $\phi_\ast = (\phi_0)_\ast = \Gamma(x)$. However, this isn't quite good enough since we want $KK(\phi) = x$.} it follows that we may pick an isomorphism $\phi \colon A \xrightarrow \cong B$ which is homotopic to $\phi_0$, and thus $KK(\phi) = KK(\phi_0) = x$.

$(b)$: This is proved exactly as above, but by using the unital versions of Theorems \ref{t:existsimple} and \ref{t:uniquesimple}.
\end{proof}

\subsection{Kirchberg's and Phillips' existence and uniqueness results}\label{ss:classical}

Variations of Theorems \ref{t:existsimple} and \ref{t:uniquesimple} are also the key steps in Phillips' \cite{Phillips-classification} and Kirchberg's \cite{Kirchberg-simple} approaches to classification, see also \cite[Theorems 8.2.1 and 8.3.3]{Rordam-book-classification}. The results presented in this paper are more general than either approach, as will be explained below. 

In Phillips' approach, the domain $A$ is always separable, nuclear, simple and unital, and the target is $B\otimes \mathcal O_\infty \otimes \mathcal K$ for a separable, unital $C^\ast$-algebra $B$. In this case, any $\ast$-homomorphism $\phi \colon A \to B\otimes \mathcal O_\infty \otimes \mathcal K$ is always nuclear and strongly $\mathcal O_\infty$-stable, and $KK(A,B) = KK_\nuc(A,B)$, so Phillips' main results are immediate corollaries of Theorems \ref{t:existsimple} and \ref{t:uniquesimple}. 

Kirchberg's main results are more technical, as well as more general. His domain $A$ is separable, exact, unital (and usually stabilised, i.e.~one considers $A\otimes \mathcal K$ instead),  and his targets are $B\otimes \mathcal K$ where $B$ is unital and properly infinite.\footnote{Kirchberg also assumes that $B$ contains $\mathcal O_2$ unitally, but this is redundant when stabilising due to Brown's stable isomorphism theorem \cite{Brown-stableiso}.} The existence part of Kirchberg's result says that every element in $KK_\nuc(A,B)$ lifts to a nuclear $\ast$-homomorphism $A\otimes \mathcal K \to B\otimes \mathcal K$ and is therefore covered by Theorem \ref{t:existsimple}. 

For uniqueness, Kirchberg uses unitary homotopy as his equivalence relation on $\ast$-homo\-morphisms. Two $\ast$-homo\-morphisms $\phi,\psi \colon A \to B$ are \emph{unitarily homotopic}, written $\phi \sim_\mathrm{uh} \psi$, if there is a strictly continuous path $(u_t)_{t\in \mathbb R_+}$ of unitaries in $\multialg{B}$ such that $\lim_{t\to \infty} u_t \phi(a) u_t^\ast  = \psi(a)$ for all $a\in A$. Clearly $\phi \sim_\asu \psi$ implies $\phi \sim_\mathrm{uh} \psi$, and if $\phi \sim_\mathrm{uh} \psi$ is implemented by $(u_t)_{t\in \mathbb R_+}$, then $\Ad u_0 \circ \phi$ is homotopic to $\psi$ by a path implemented by unitary conjugation. Hence if $\phi,\psi$ are nuclear and $\phi \sim_\mathrm{uh} \psi$, then $KK_\nuc(\phi) = KK_\nuc(\psi)$. In the cases where Theorem \ref{t:uniquesimple} is applicable with stable target, it therefore follows that $\phi \sim_\asu \psi \Leftrightarrow \phi \sim_{\mathrm{uh}} \psi \Leftrightarrow KK_\nuc(\phi) = KK_\nuc(\psi)$.

Kirchberg shows that nuclear $\ast$-homomorphisms $\phi,\psi \colon A\otimes \mathcal K \to B\otimes \mathcal K$ have the same $KK_\nuc$-class exactly when $\phi\oplus \theta \sim_\mathrm{uh} \psi \oplus \theta$, where $\theta$ is a suitably chosen full (necessarily nuclear) $\ast$-homomorphism factoring through $\mathcal O_2$. By Propositions \ref{p:fulldom} and \ref{p:piabsorbing}$(a)$, it follows that $\phi\oplus \theta$ and $\psi \oplus \theta$ are nuclear, strongly $\mathcal O_\infty$-stable, and full, and they have the same $KK_\nuc$-classes as $\phi$ and $\psi$ respectively since $KK_\nuc(\theta) = 0$. Thus Theorem \ref{t:uniquesimple} implies
\begin{equation}
KK_\nuc(\phi) = KK_\nuc(\psi) \quad \Leftrightarrow \quad  \phi\oplus \theta \sim_\asu \psi \oplus \theta \quad  \Leftrightarrow \quad  \phi\oplus \theta \sim_\mathrm{uh} \psi \oplus \theta.
\end{equation}
Kirchberg also shows that if $B$ in addition is simple and purely infinite, and if $\phi,\psi \colon A\otimes \mathcal K \to B\otimes \mathcal K$ are injective, nuclear $\ast$-homomorphisms, then $KK_\nuc(\phi) = KK_\nuc(\psi)$ if and only if $\phi \sim_{\mathrm{uh}} \psi$. Simplicity of $B$ and injectivity of the maps imply that they are full. By Corollary \ref{c:nucintosimplepi} (proved in the following section), it follows that $\phi, \psi$ are strongly $\mathcal O_\infty$-stable, and hence this result is also covered by Theorem \ref{t:uniquesimple}.


\subsection{Approximate equivalence}\label{ss:KLnuc}

\begin{remark}[KL-theory]
In \cite{Dadarlat-KKtop}, Dadarlat equips $KK(A,B)$ (for $A$ and $B$ both separable) with a (not necessarily Hausdorff) topology. By \cite[Theorem 3.5]{Dadarlat-KKtop}, this topology can be described as the unique first countable topology on $KK(A,B)$ such that a sequence $(x_n)_{n\in \mathbb N}$ converges to $x\in KK(A,B)$ if and only if there is an element $y \in KK(A, C(\widetilde{\mathbb N}, B))$ with $(\ev_n)_\ast(y) = x_n$ for $n\in \mathbb N$ and $(\ev_\infty)_\ast(y) = x$. Here $\widetilde{\mathbb N} = \mathbb N \cup \{\infty\}$ denotes the one-point compactification of $\mathbb N$. One defines
\begin{equation}
KL(A,B) := KK(A,B) / \overline{\{0\}}.
\end{equation}
The group $KL(A,B)$ was first defined by Rørdam in \cite{Rordam-classsimple} whenever $A$ satisfies the UCT. Dadarlat's definition above generalises this to the non-UCT case. 

Doing the same in $KK_\nuc$ when $A$ is separable and $B$ is $\sigma$-unital, say that a sequence $(x_n)_{n\in \mathbb N}$ converges to $x$ in $KK_\nuc(A,B)$ if there is an element $y \in KK_\nuc(A, C(\widetilde{\mathbb N}, B))$ with $(\ev_n)_\ast(y) = x_n$ for $n\in \mathbb N$ and $(\ev_\infty)_\ast(y) = x$. Define
\begin{equation}
KL_\nuc(A,B) := KK_\nuc(A,B) / \overline{\{0\}}.
\end{equation}
\end{remark}

\begin{proposition}\label{p:aMvNKL}
Let $A$ be a separable $C^\ast$-algebra, let $B$ be a $\sigma$-unital $C^\ast$-algebra and suppose that $\phi, \psi \colon A \to B$ are nuclear $\ast$-homomorphisms. If $\phi \sim_\aMvN \psi$ then $KL_\nuc(\phi) = KL_\nuc(\psi)$.
\end{proposition}
\begin{proof}
As the map $\id_B \otimes e_{1,1} \colon B \to B\otimes \mathcal K$ induces a topological isomorphism $KK_\nuc(A,B) \cong KK_\nuc(A, B \otimes \mathcal K)$, it descends to an isomorphism $KL_\nuc(A,B) \cong KL_\nuc(A,B\otimes \mathcal K)$. Hence we may assume that $B$ is stable. By Proposition \ref{p:MvNvsue} it follows that $\phi \sim_\au \psi$, so pick a sequence $(u_n)_{n\in \mathbb N}$ of unitaries in $\multialg{B}$ such that $u_n^\ast \phi(-) u_n \to \psi$. This gives a $\ast$-homomorphism
\begin{equation}
\Phi \colon A \to C(\widetilde{\mathbb N}, B), \qquad \Phi(a)(n) = \left\{ \begin{array}{ll} u_n^\ast \phi(a) u_n, & n\in \mathbb N \\ \psi(a) , & n = \infty \end{array} \right.
\end{equation}
for $a\in A$. The map $\Phi$ is nuclear by Lemma \ref{l:XnucC(Y)}\footnote{An alternative direct proof can be sketched as follows: one can use that on a finite set $\mathcal F \subseteq A$ and up to a tolerance $\epsilon >0$,  $\Phi(a)(n)$ for $a\in \mathcal F$ is constant up to $\epsilon$ for $n$ sufficiently large. Hence $\Phi$ approximately decomposes as a direct sum of nuclear maps $A \to C(\{1,\dots, n\}, B)$ and $\ev_\infty \circ \Phi \colon A \to B \subseteq C(\{ n+1, \dots, \infty\}, B)$. Nuclearity of $\Phi$ easily follows.} and thus induces an element $KK_\nuc(\Phi) \in KK_\nuc(A, C(\widetilde{\mathbb N}, B))$ such that $(\ev_n)_\ast(KK_\nuc(\Phi)) = KK_\nuc(\phi)$ for $n\in \mathbb N$, and $(\ev_\infty)_\ast(KK_\nuc(\Phi)) = KK_\nuc(\psi)$. Hence we get $KL_\nuc(\phi) = KL_\nuc(\psi)$.
\end{proof}

From Theorems \ref{t:existsimple} and \ref{t:uniquesimple} one gets the following approximate uniqueness theorem. In order to get it for (not necessarily strongly) $\mathcal O_\infty$-stable maps, a McDuff type theorem \cite[Corollary 4.5]{Gabe-O2class} will be applied.

\begin{theorem}[Approximate uniqueness -- full case]\label{t:approxuniquesimple}
Let $A$ be a separable $C^\ast$-algebra, and let $B$ be a $\sigma$-unital $C^\ast$-algebra. Suppose that  $\phi, \psi\colon A \to B$ are nuclear, $\mathcal O_\infty$-stable, full $\ast$-homomorphisms. The following are equivalent:
\begin{itemize}
\item[$(i)$] $KL_\nuc(\phi) = KL_\nuc(\psi)$; 
\item[$(ii)$] $\phi$ and $\psi$ are approximately Murray--von Neumann equivalent.
\end{itemize}
If either $B$ is stable, or if $A,B,\phi$ and $\psi$ are all unital, then $(i)$ and $(ii)$ are equivalent to
\begin{itemize}
\item[$(iii)$] $\phi$ and $\psi$ are approximately unitarily equivalent (with unitaries in the minimal unitisation).
\end{itemize}
\end{theorem}
\begin{proof}
$(ii)\Rightarrow (i)$ is Proposition \ref{p:aMvNKL}. For proving $(i) \Rightarrow (ii)$, assume that $KL_\nuc(\phi) = KL_\nuc(\psi)$. By Proposition \ref{p:MvNeq}, it suffices to show that $\phi\otimes e_{1,1}, \psi\otimes e_{1,1} \colon A \to B\otimes \mathcal K$ are approximately unitarily equivalent. Note that $KL_\nuc(\phi \otimes e_{1,1}) = KL_{\nuc}(\psi \otimes e_{1,1})$ and that $\phi\otimes e_{1,1}$ and $\psi\otimes e_{1,1}$ are nuclear, $\mathcal O_\infty$-stable and full. Hence we may instead assume that $B$ is stable. By \cite[Corollary 4.5]{Gabe-O2class}, $\phi$ and $\psi$ are approximate Murray--von Neumann equivalent to $\ast$-homomorphisms factoring through $A \otimes \mathcal O_\infty$, so these maps are strongly $\mathcal O_\infty$-stable by Proposition \ref{p:Oinftyfactor}. By replacing $\phi$ and $\psi$ with such maps,  we may assume that $\phi$ and $\psi$ are strongly $\mathcal O_\infty$-stable.

Let $\mathcal F\subset A$ be finite and $\epsilon>0$. Pick $y\in KK_\nuc(A, C(\widetilde{\mathbb N},B))$ such that $(\ev_n)_\ast(y) = KK_\nuc(\phi)$ for $n\in \mathbb N$ and $(\ev_\infty)_\ast(y) = KK_\nuc(\psi)$. Note that the existence a full, nuclear, $\mathcal O_\infty$-stable map implies that $A$ is exact (see Remark \ref{r:nucemb}) and that $B$ contains a full, properly infinite projection (see Remark \ref{r:fullpropinfproj}). Hence $C(\widetilde{\mathbb N}, B)$ also contains a full properly infinite projection, so by Theorem \ref{t:existsimple}, there is a full, nuclear, strongly $\mathcal O_\infty$-stable $\ast$-homomorphism $\Phi \colon A \to C(\widetilde{\mathbb N},B)$ so that $KK_\nuc(\Phi) = y$. Pick $n\in \mathbb N$ such that 
\begin{equation}
\| \ev_n(\Phi(a)) - \ev_{\infty}(\Phi(a))\| < \epsilon, \qquad \textrm{for all }a\in \mathcal F.
\end{equation}
By Theorem \ref{t:uniquesimple}, it follows that $\phi \sim_{\asu} \ev_n \circ \Phi$ and $\psi \sim_{\asu} \ev_\infty\circ \Phi$. It easily follows that we may find a unitary $u\in \multialg{B}$ such that $\| u^\ast \phi(a) u - \psi(a)\| < \epsilon$ for all $a\in \mathcal F$. This completes the proof of $(i) \Rightarrow (ii)$. The part involving $(iii)$ follows immediately from Proposition \ref{p:MvNvsue}.
\end{proof}

If $A$ satisfies the UCT, Dadarlat and Loring proved \cite{DadarlatLoring-UMCT} that there is a universal \emph{multi}coefficient theorem (UMCT), given by the short exact sequence
\begin{equation}
\mathrm{PExt}(K_{1-\ast}(A), K_\ast(B)) \rightarrowtail KK(A,B) \twoheadrightarrow \Hom_\Lambda(\underline{K}(A), \underline{K}(B)).
\end{equation}
Here $\mathrm{PExt}(K_{1-\ast}(A), K_\ast(B))$ is the subgroup of $\Ext(K_{1-\ast}(A), K_\ast(B))$ consisting of equivalence classes of pure extensions, and $\underline{K}(C)$ is the \emph{total $K$-theory}, which is a module over a certain ring $\Lambda$. I refer the reader to \cite{DadarlatLoring-UMCT} for the relevant details.

By \cite[Theorem 4.1]{Dadarlat-KKtop} it follows that if $A$ and $B$ are separable, and $A$ satisfies the UCT, then there is a natural isomorphism of topological groups
\begin{equation}
KL(A,B) \xrightarrow \cong \Hom_\Lambda(\underline{K}(A), \underline{K}(B))
\end{equation}
where the latter group is equipped with the topology of point-wise convergence.


As the topologies on $KK(A,B)$ and $KK_\nuc(A,B)$ only depend on $A$ up to $KK$-equivalence (by the above definition), it follows that
\begin{equation}
KL_\nuc(A,B) \cong KL(A,B)
\end{equation}
whenever $A$ is $KK$-equivalent to a nuclear $C^\ast$-algebra. In particular, as $A$ satisfies the UCT (and is thus $KK$-equivalent to a commutative $C^\ast$-algebra) one obtains natural isomorphisms
\begin{equation}\label{eq:KLnucHom}
KL_\nuc(A,B) \cong KL(A,B) \cong \Hom_\Lambda(\underline{K}(A), \underline{K}(B))
\end{equation}
for $B$ separable by \cite[Theorem 4.1]{Dadarlat-KKtop}. Separability plays no important role here, only the existence of an absorbing representation $A \to \multialg{B\otimes \mathcal K}$, which always exists when $A$ is separable and nuclear and $B$ is $\sigma$-unital by \cite[Theorem 6]{Kasparov-Stinespring}. Hence \eqref{eq:KLnucHom} holds whenever $A$ is separable and satisfies the UCT, and $B$ is $\sigma$-unital.

From Theorems \ref{t:existsimple} and \ref{t:approxuniquesimple}, together with the result \eqref{eq:KLnucHom}, one immediately gets the following classification result for nuclear, $\mathcal O_\infty$-stable, full $\ast$-homo\-morphisms using the $K$-theoretic invariant $\underline K$. This was first proved by Lin in \cite[Theorem 4.10]{Lin-stableapproxuniqueness} in the special case of Kirchberg algebras satisfying the UCT.

\begin{theorem}
Let $A$ be a separable, exact $C^\ast$-algebra satisfying the UCT, and let $B$ be a $\sigma$-unital $C^\ast$-algebra containing a properly infinite, full projection. Then the nuclear, $\mathcal O_\infty$-stable, full $\ast$-homomorphisms $A\to B$ are parametrised up to approximate Murray--von Neumann equivalence by morphisms $\underline K(A) \to \underline K(B)$.
\end{theorem}

In the following, let $\underline K_u(A) = (\underline K(A), [1_A]_0)$ for a unital $C^\ast$-algebra $A$, so that a morphism $\alpha \colon \underline K_u(A) \to \underline K_u(B)$ is a $\Lambda$-homomorphism for which the induced map $\alpha_0 \colon K_0(A) \to K_0(B)$ satisfies $\alpha_0([1_A]_0) = [1_B]_0$. Then, exactly as the theorem above, one can classify unital maps up to approximate unitary equivalence.

\begin{theorem}
Let $A$ be a separable, exact, unital $C^\ast$-algebra satisfying the UCT, and let $B$ be a unital, properly infinite $C^\ast$-algebra. Then the nuclear, $\mathcal O_\infty$-stable, full, unital $\ast$-homomorphisms $A\to B$ are parametrised up to approximate unitary equivalence by morphisms $\underline K_u(A) \to \underline K_u(B)$.
\end{theorem}



\section{Strongly $\mathcal O_\infty$-stable $\ast$-homomorphisms}\label{s:stronglyOinfty}

The following section is a slight detour from the main classification theorem of the paper, and is strictly speaking not needed for proving this result. Corollary \ref{c:nucintosimplepi} was used in Subsection \ref{ss:classical} to show that Theorems \ref{t:existsimple} and \ref{t:uniquesimple} are in fact more general than all the classification results obtained by Kirchberg in \cite{Kirchberg-ICM}.

In this section, sufficient criteria are given for when an $\mathcal O_\infty$-stable map is strongly $\mathcal O_\infty$-stable. 
The main tool is to show that if a map $\phi$ is $\mathcal O_\infty$-stable, and induces a sequential relative commutant which is $K_1$-injective, then $\phi$ is strongly $\mathcal O_\infty$-stable. In particular, a positive solution to the open problem of whether every properly infinite, unital $C^\ast$-algebra is $K_1$-injective, would imply that every $\mathcal O_\infty$-stable map is strongly $\mathcal O_\infty$-stable.
This will be applied, using a results of Kirchberg and Rørdam from \cite{KirchbergRordam-absorbingOinfty} to show that any nuclear $\ast$-homomorphism into a strongly purely infinite $C^\ast$-algebra is strongly $\mathcal O_\infty$-stable.

First some notation and preliminary observations will be set up. The thing to keep in mind in the following construction is that $[0,\infty)$ can be constructed by gluing the sequence intervals $[n,n+1]$ together. This will be used to write path algebras and their relative commutants as pull-backs of certain related sequence algebras. 
Fix a $C^\ast$-algebra $B$, and let $IB := C([0,1],B)$. There are canonical $\ast$-homo\-morphisms
\begin{equation}
\ev_{\mathbb N} \colon B_\as \to B_\infty, \qquad \ev_I \colon  B_\as \to (IB)_\infty
\end{equation}
induced by $f \mapsto (f(n))_{n\in \mathbb N}$ and $f \mapsto (f_n)_{n\in \mathbb N}$ for $f\in C_b(\mathbb R_+, B)$ respectively, where $f_n \in IB$ is given by $f_n(s) = f(s+n)$ for $s\in [0,1]$. Similarly, let 
\begin{equation}
\sigma \colon B_\infty \to B_\infty , \qquad \ev_s \colon (IB)_\infty \to B_\infty
\end{equation}
 be induced by the shift map $(b_n)_{n\in \mathbb N} \mapsto (b_{n+1})_{n\in \mathbb N}$ and the evaluation map $(g_n)_{n\in \mathbb N} \mapsto (g_n(s))_{n\in \mathbb N}$ for $s\in [0,1]$ respectively. Letting $SB:= C((0,1), B)$, it is easy to see that one obtains the following commutative diagram
\begin{equation}\label{eq:asseqpullback}
\xymatrix{
0 \ar[r] & (SB)_\infty \ar@{=}[d] \ar[r]^j & B_\as \ar[d]^{\ev_I} \ar[r]^{\ev_{\mathbb N}} & B_\infty \ar[d]^{\id \oplus \sigma} \ar[r] & 0 \\
0 \ar[r] & (SB)_\infty \ar[r] & (IB)_\infty \ar[r]^{\ev_0 \oplus \ev_1 \quad} & B_\infty \oplus B_\infty \ar[r] & 0
}
\end{equation}
with exact rows, where $j$ is induced by the map $\hat{j} \colon \prod_{\mathbb N} SB \to C_b(\mathbb R_+, B)$ given by $\hat{j}((f_n)_{n\in \mathbb N})(t) = f_n(t-n)$ whenever $t\in [n,n+1]$ for $(f_n)_\in \prod_{\mathbb N}SB$. In particular, the right square above is a pull-back square.

For a $\ast$-homomorphism $\phi \colon A \to B$, one considers $\phi(A)$ as a subalgebra of $B_\as$, $B_\infty$ and $(IB)_\infty$ in the usual way.

\begin{lemma}
With notation as above, the diagram \eqref{eq:asseqpullback} induces the commutative diagram
\begin{equation}\label{eq:relcompullback}
\xymatrix{
0 \ar[r] & \frac{(SB)_\infty \cap \phi(A)'}{(SB)_\infty \cap \Ann \phi(A)} \ar@{=}[d] \ar[r]^{\overline j} & \frac{B_\as \cap \phi(A)'}{\Ann \phi(A)} \ar[d]^{\overline{\ev}_I} \ar[r]^{\overline{\ev}_{\mathbb N}} & \frac{B_\infty \cap \phi(A)'}{\Ann \phi(A)} \ar[d]^{\id \oplus \overline{\sigma}} \ar[r] & 0  \\
0 \ar[r] & \frac{(SB)_\infty \cap \phi(A)'}{(SB)_\infty \cap \Ann \phi(A)} \ar[r] & \frac{(IB)_\infty \cap \phi(A)'}{\Ann \phi(A)} \ar[r]^{\overline{\ev}_0 \oplus \overline{\ev}_1 \qquad} & \frac{B_\infty \cap \phi(A)'}{\Ann \phi(A)} \oplus \frac{B_\infty \cap \phi(A)'}{\Ann \phi(A)} \ar[r] & 0
}
\end{equation}
with exact rows. In particular, the right square above is a pull-back square.
\end{lemma}
\begin{proof}
It is obvious that the all maps in \eqref{eq:asseqpullback} preserve commutativity with elements $\phi(a)$ for all $a\in A$ (as these elements are constant sequences/paths), and that they map annihilators of $\phi(A)$ to annihilators of $\phi(A)$. Hence all maps in \eqref{eq:relcompullback} are well-defined. The only thing not obvious about exactness of the rows, is that $\overline{\ev}_{\mathbb N}$ and $\overline{\ev}_0 \oplus \overline{\ev}_1$ are surjective. I present a proof that $\overline{\ev}_{\mathbb N}$ is surjective, the map $\overline{\ev}_0 \oplus \overline{\ev}_1$ is surjective by a similar argument.

Let $x\in \frac{B_\infty \cap \phi(A)'}{\Ann \phi(A)}$ be any element. Let $(b_n)_{n\in \mathbb N} \in \prod_{\mathbb N} B$ be a lift of $x$, and define $f\in C_b(\mathbb R_+, B)$ by interpolation $f(t) = (t + 1 - n) b_n + (t - n) b_{n+1}$ for $t\in [n,n+1]$. As $\lim_{n\to\infty} \|[ b_n , \phi(a) ] \| = 0$ for every $a\in A$, it follows that $\lim_{t\to\infty}\| [ f(t), \phi(a)]\| = 0$ for every $a\in A$. Hence $f$ induces an element $\overline{f}$ in $\frac{B_\as \cap \phi(A)'}{\Ann \phi(A)}$, and $\overline{\ev}_{\mathbb N}(\overline{f}) = x$ by construction.
\end{proof}

Recall that a unital $C^\ast$-algebra $D$ is \emph{$K_1$-injective} if whenever $u\in D$ is a unitary with trivial $K_1$-class then $u$ is homotopic to $1_D$ in the unitary group of $D$. Using results from \cite{BlanchardRohdeRordam-K1inj}, one obtains the following.

\begin{proposition}\label{p:K1injstrongOinfty}
Let $A$ and $B$ be $C^\ast$-algebras with $A$ separable, and let $\phi \colon A \to B$ be an $\mathcal O_\infty$-stable $\ast$-homomorphism. If the sequential relative commutant $\frac{B_\infty \cap \phi(A)'}{\Ann \phi(A)}$ is $K_1$-injective, then $\phi$ is strongly $\mathcal O_\infty$-stable.
\end{proposition}
\begin{proof}
As observed in Remark \ref{r:Oinftypictures}, strong $\mathcal O_\infty$-stability of $\phi$ means that the asymptotic relative commutant $\frac{B_\as\cap \phi(A)'}{\Ann \phi(A)}$ is properly infinite, so we prove this. Since the right square of \eqref{eq:relcompullback} is a pull-back square with the map $\overline{\ev}_0 \oplus \overline{\ev}_1$ surjective, it follows from \cite[Proposition 2.7]{BlanchardRohdeRordam-K1inj}\footnote{To apply \cite[Proposition 2.7]{BlanchardRohdeRordam-K1inj}, one would a priori need the map $\id \oplus \overline{\sigma}$ to also be surjective. However, the proof only needs surjectivity of one of the maps, see \cite[Lemma 6.2]{GabeRuiz-unitalExt}.} that a sufficient condition is that $\frac{B_\infty \cap \phi(A)'}{\Ann \phi(A)}$ and $\frac{(IB)_\infty \cap \phi(A)'}{\Ann \phi(A)}$ are properly infinite, and $\frac{B_\infty \cap \phi(A)'}{\Ann \phi(A)} \oplus \frac{B_\infty \cap \phi(A)'}{\Ann \phi(A)}$ is $K_1$-injective. The latter condition is one of our hypothesis (as direct sums of $K_1$-injective $C^\ast$-algebras are clearly $K_1$-injective), and $\frac{B_\infty \cap \phi(A)'}{\Ann \phi(A)}$ is properly infinite by $\mathcal O_\infty$-stability of $\phi$. It is easy to see that $\mathcal O_\infty$-stability of $\phi$ implies that $\frac{(IB)_\infty \cap \phi(A)'}{\Ann \phi(A)}$ is properly infinite. Alternatively, one can use that it is the sequential relative commutant of the composition $A\xrightarrow \phi B \xrightarrow{\mathrm{constant}} IB$, and that this composition is $\mathcal O_\infty$-stable by \cite[Lemma 3.20$(i)$]{Gabe-O2class}. 
\end{proof}

\begin{remark}
It is an open problem whether every unital, properly infinite $C^\ast$-algebra is $K_1$-injective. As a $\ast$-homomorphism $\phi \colon A \to B$ is $\mathcal O_\infty$-stable if and only if the unital $C^\ast$-algebra $\frac{B_\infty \cap \phi(A)'}{\Ann \phi(A)}$ is properly infinite, an affirmative answer to this open problem would imply that every $\mathcal O_\infty$-stable $\ast$-homomorphism is strongly $\mathcal O_\infty$-stable.
\end{remark}

Following \cite[Remark 5.10]{KirchbergRordam-absorbingOinfty}, a $C^\ast$-algebra $B$ is \emph{strongly purely infinite} if for all $b_1,b_2\in B$ positive, and $\epsilon >0$, there are $s_1,s_2\in B$ such that
\begin{equation}
\max_{i=1,2}\| s_i^\ast b_i s_i - b_i \| \leq \epsilon, \quad \textrm{and} \quad \| s_2^\ast b_2^{1/2} b_1^{1/2} s_1 \| \leq \epsilon.
\end{equation}

The following was essentially proved by Kirchberg and Rørdam in \cite{KirchbergRordam-absorbingOinfty}, albeit the domain $A$ was assumed to be nuclear as opposed to exact, and it was essentially how they proved that separable, nuclear, strongly purely infinite $C^\ast$-algebras are $\mathcal O_\infty$-stable. The proposition will be improved in Theorem \ref{t:nucintospi} below.

\begin{proposition}[Kirchberg--Rørdam]\label{p:KRspi}
Let $A$ be a separable, exact $C^\ast$-algebra, and let $B$ be a strongly purely infinite $C^\ast$-algebra. Then every nuclear $\ast$-homo\-morphism $\phi \colon A \to B$ is $\mathcal O_\infty$-stable.
\end{proposition}
\begin{proof}
The proof will first be carried out assuming that $B$ is also stable. In particular, $\mathcal O_2$ embeds unitally into $\multialg{B}$.

Let $\mathscr C_0 \subseteq \CP(A, B_\infty)$ be the set of c.p.~maps $\rho$ which are approximately dominated by $\phi$, and for which $C^\ast(\rho(A))$ is commutative. Let $\mathscr C\subseteq \CP(A,B_\infty)$ be the set of c.p.~maps of the form
\begin{equation}\label{eq:coccmap}
A \ni a \mapsto \sum_{i,j=1}^n y_i^\ast \rho(x_i^\ast a x_j) y_j  \in B_\infty
\end{equation}
for $\rho \in \mathscr C_0$, $n\in \mathbb N$, $x_1,\dots, x_n\in \multialg{A}$ and $y_1,\dots, y_n\in \multialg{B_\infty}$,
and let $\overline{\mathscr C}$ denote the point-norm closure of $\mathscr C$. By \cite[Proposition 7.14$(i)$ and Lemma 7.19]{KirchbergRordam-absorbingOinfty}, every map in $\overline{\mathscr C}$ is approximately 1-dominated by $\phi$. 

By \cite[Lemma 7.16]{KirchbergRordam-absorbingOinfty}, $\mathscr C$ is an operator convex cone in the sense of \cite[Definition 4.1]{KirchbergRordam-zero}, and thus $\overline{\mathscr C}$ is a point-norm closed operator convex cone. Note that every map in $\mathscr C_0$ is nuclear, and thus so is every map of the form \eqref{eq:coccmap}. In particular, $\overline{\mathscr C} \subseteq \CP_\nuc(A, B_\infty)$. By \cite[Proposition 7.13]{KirchbergRordam-absorbingOinfty}, there is for every $a\in A_+$ a map $\rho \in \mathscr C_0$ such that $\phi(a) = \rho(a)$. A minor modification of \cite[Proposition 4.2]{KirchbergRordam-zero},\footnote{Where one moves the nuclearity assumption of $A$ to assuming that all the maps are nuclear.} see \cite[Theorem 2.5]{Gabe-cplifting}, implies that $\phi\in \overline{\mathscr C}$.  

Let $s_1,s_2 \in \multialg{B}$ be $\mathcal O_2$-isometries, and $\phi \oplus_{s_1,s_2} \phi := s_1 \phi(-)s_1^\ast + s_2\phi(-)s_2^\ast$. Then $\phi \oplus_{s_1,s_2} \phi \in \overline{\mathscr C}$ by \cite[Lemma 7.16$(i)$]{KirchbergRordam-absorbingOinfty}, and thus this map is approximately $1$-dominated by $\phi$. Hence \cite[Lemma 7.3]{KirchbergRordam-absorbingOinfty} provides $b\in B_\infty$ (which we may assume to be a contraction, otherwise do a standard cutting down argument with an approximate identity in $A$), such that $b^\ast \phi(-) b = \phi(-) \oplus_{s_1,s_2} \phi(-)$. Let $t_i := b s_i \in B_\infty$. Since $t_1^\ast \phi(-) t_1 = t_2^\ast \phi(-) t_2 = \phi$, it follows from Lemma \ref{l:conjhom} that $t_1,t_2 \in B_\infty \cap \phi(A)'$. Continuing to apply Lemma \ref{l:conjhom}, one obtains
\begin{equation}
t_i^\ast t_j \phi(a) = t_i^\ast \phi(a) t_j = s_i^\ast (s_1 \phi(a) s_1^\ast + s_2 \phi(a) s_2^\ast) s_j = \delta_{i,j} \phi(a)
\end{equation}
for all $a\in A$, and $i,j=1,2$. Hence $t_1$ and $t_2$ induce isometries in $\frac{B_\infty \cap \phi(A)'}{\Ann \phi(A)}$ with orthogonal range projections, so $\frac{B_\infty \cap \phi(A)'}{\Ann \phi(A)}$ is properly infinite, or equivalently, $\phi$ is $\mathcal O_\infty$-stable. This completes the case where $B$ is stable.

Now, if $B$ is not stable, it follows from \cite[Proposition 5.11]{KirchbergRordam-absorbingOinfty} that $B\otimes \mathcal K$ is strongly purely infinite, and clearly $\phi \otimes e_{1,1} \colon A \to B\otimes \mathcal K$ is nuclear. Hence $\phi \otimes e_{1,1}$ is $\mathcal O_\infty$-stable by what was proved above, and thus $\phi$ is $\mathcal O_\infty$-stable by Lemma \ref{l:relcombasic}$(a)$.
\end{proof}

The following lemma provides a way for concluding nuclearity of $\ast$-homo\-morphisms out of tensor products when one of the tensors is nuclear. 

\begin{lemma}\label{l:nucoutoftensor}
Let $A,B$ and $C$ be $C^\ast$-algebras with $C$ nuclear, and let $\psi \colon A \otimes C \to B$ be a $\ast$-homomorphism. If $(e_\lambda)_{\lambda \in \Lambda}$ is an approximate identity in $C$, and if $\psi(- \otimes c_\lambda) \colon A \to B$ is nuclear for each $\lambda \in \Lambda$, then $\psi$ is nuclear.
\end{lemma}
\begin{proof}
Let $\psi_A \colon A \to B^{\ast \ast}$ be a point-ultraweak accumulation point of $\psi(- \otimes e_\lambda)$, and similarly define $\psi_C \colon C \to B^{\ast \ast}$. Then $\psi_A$ and $\psi_C$ are $\ast$-homo\-morphisms with commuting images and $\psi_A$ is weakly nuclear as it is a limit of nuclear maps. By \cite[Corollary 2.8]{Gabe-qdexact}, $\psi$ is nuclear.
\end{proof}

Consequently, one obtains the following lemma which seems interesting in its own right.

\begin{lemma}\label{l:sepOinftystable}
Let $A$ be a separable, exact $C^\ast$-algebra, let $B$ be a strongly purely infinite $C^\ast$-algebra, and let $\phi \colon A \to B$ be a nuclear $\ast$-homomorphism. For every separable, nuclear $C^\ast$-subalgebra $C\subseteq \frac{B_\infty \cap \phi(A)'}{\Ann \phi(A)}$, there exists a unital embedding 
\begin{equation}\label{eq:Oinftyindoublecom}
\mathcal O_\infty \hookrightarrow \frac{B_\infty \cap \phi(A)'}{\Ann \phi(A)} \cap C'.
\end{equation}
\end{lemma}
\begin{proof}
By replacing $C$ with its minimal unitisation (which is also nuclear), one may assume that $C$ is a unital $C^\ast$-subalgebra of $\frac{B_\infty \cap \phi(A)'}{\Ann \phi(A)}$.
Let $\psi \colon A \otimes C \to B_\infty$ be the $\ast$-homomorphism given on elementary tensors by $\psi(a\otimes c) = \phi(a) \overline{c}$ for $a\in A$ and $c\in C$, where $\overline{c} \in B_\infty \cap \phi(A)'$ is a lift of $c$.
As $\psi(-\otimes 1_C)= \phi$ is nuclear, Lemma \ref{l:nucoutoftensor} implies that $\psi$ is nuclear. As $B_\infty$ is strongly purely infinite by \cite[Proposition 5.12]{KirchbergRordam-absorbingOinfty}, it follows that $\psi$ is $\mathcal O_\infty$-stable by Proposition \ref{p:KRspi}. Hence there exist sequences $t_i^{(1)}, t_i^{(2)}, \dots \in B_\infty$ for $i=1,2$ such that
\begin{equation}
\lim_{n\to \infty} \|[ \psi(x), t_i^{(n)}] \| = 0, \qquad \lim_{n\to \infty} \| ((t_i^{(n)})^\ast t_j^{(n)} - \delta_{i,j}) \psi(x) \| = 0
\end{equation}
for all $x\in A\otimes C$. By a standard diagonal argument for sequence algebras, one finds $t_1,t_2 \in B_\infty \cap \psi(A\otimes C)'$ such that
\begin{equation}
t_i^\ast t_j \psi(x) = \delta_{i,j} \psi(x), \qquad i,j = 1,2,\, x\in A\otimes C.
\end{equation}
Hence $t_1$ and $t_2$ induce isometries in $\frac{B_\infty \cap \phi(A)'}{\Ann \phi(A)} \cap C'$ with mutually orthogonal range projections. Consequently, this $C^\ast$-algebra contains a unital copy of $\mathcal O_\infty$.
\end{proof}

With this, one can improve on Proposition \ref{p:KRspi} as follows.

\begin{theorem}\label{t:nucintospi}
Let $A$ be a separable, exact $C^\ast$-algebra, and let $B$ be a strongly purely infinite $C^\ast$-algebra. Then every nuclear $\ast$-homomorphism $\phi \colon A \to B$ is strongly $\mathcal O_\infty$-stable.
\end{theorem}
\begin{proof}
By Propositions \ref{p:K1injstrongOinfty} and \ref{p:KRspi}, it suffices to prove that $\frac{B_\infty \cap \phi(A)'}{\Ann \phi(A)}$ is $K_1$-injective. Let $u\in \frac{B_\infty \cap \phi(A)'}{\Ann \phi(A)}$ be a unitary with trivial $K_1$-class. As $C^\ast(u)$ is nuclear Lemma \ref{l:sepOinftystable} provides the existence of a unital copy of $\mathcal O_\infty$ in $\frac{B_\infty \cap \phi(A)'}{\Ann \phi(A)}$ which commutes with $u$. By \cite[Lemma 2.4(ii)]{BlanchardRohdeRordam-K1inj}, $u$ is homotopic to the unit in the unitary group of $\frac{B_\infty \cap \phi(A)'}{\Ann \phi(A)}$.
\end{proof}

Using that simple, purely infinite $C^\ast$-algebras are strongly purely infinite by \cite[Corollary 6.9]{KirchbergRordam-absorbingOinfty}, one obtains the following.

\begin{corollary}\label{c:nucintosimplepi}
Let $A$ be a separable, exact $C^\ast$-algebra, and let $B$ be a simple, purely infinite $C^\ast$-algebra. Then every nuclear $\ast$-homomorphism $\phi \colon A \to B$ is strongly $\mathcal O_\infty$-stable.
\end{corollary}


\section{Ideals and actions of topological spaces}

The rest of this paper is about generalising the Kirchberg--Phillips theorem to the non-simple case. This entails keeping track of the ideal structure of the $C^\ast$-algebras as well as how these interact. 

I should emphasise that all the key methods used in the simple case -- in particular, the methods developed in Sections \ref{s:PIhom} and \ref{s:unitary} -- are also the key ingredients in the non-simple case. 

The situation in the non-simple case is much more complex than the simple case and needs a generalised version of $KK$-theory to work. For instance, one can construct separable, nuclear, $\mathcal O_\infty$-stable $C^\ast$-algebras $A$ and $B$ which both have exactly one non-trivial ideal $I_A$ and $I_B$ respectively, such that $I_A \sim_{KK} I_B$ and $A \sim_{KK} B$, but for which $A$ and $B$ are not stably isomorphic. For instance, the two (non-isomorphic) six-term exact sequences
\begin{equation}
\xymatrix{
\mathbb Z \ar[r]^{\id} & \mathbb Z \ar[r]^0 & \mathbb Z \ar[d]^\id & \mathbb Z \ar[r]^{0} & \mathbb Z \ar[r]^\id & \mathbb Z \ar[d]^0 \\
\mathbb Z \ar[u]^0 & \mathbb Z \ar[l]_{\id} & \mathbb Z \ar[l]_0 & \mathbb Z \ar[u]^\id & \mathbb Z \ar[l]_0 & \mathbb Z \ar[l]_\id
}
\end{equation}
arise as the $K$-theory of separable, nuclear, stable, $\mathcal O_\infty$-stable $C^\ast$-algebras $A$ and $B$ respectively with unique non-trivial ideals $I_A$ and $I_B$ respectively, so that $I_A,I_B,A/I_A$ and $B/I_B$ satisfy the UCT, see \cite[Proposition 5.4]{Rordam-classsixterm}. By the Kirchberg--Phillips Theorem (Theorem \ref{t:KPUCT}) $I_A \cong I_B$ and $A/I_A \cong B/I_B$. Moreover, $A$ and $B$ satisfy the UCT by the 2-out-of-3 property, and thus $A\sim_{KK} B$ since $A$ and $B$ have isomorphic $K$-theory. However, stably isomorphic $C^\ast$-algebras with a unique non-trivial ideal induce isomorphic six-term exact sequences in $K$-theory, and thus $A$ and $B$ are not stably isomorphic.
The generalised version of $KK$-theory needed for classification will be studied in the Section \ref{s:KK}.

In this section the focus is on how to incorporate the ideal structure of $C^\ast$-algebras in a way that makes non-simple $C^\ast$-algebras amenable for classification.

\subsection{Ideal lattices}

\emph{All ideals are assumed to be two-sided and closed.}

Recall that a \emph{complete lattice} is a partially ordered set in which every subset has a supremum and an infimum. Note that any complete lattice $\mathcal L$ has a largest element $\sup \mathcal L = \inf \emptyset$ and a smallest element $\inf \mathcal L = \sup \emptyset$.

\begin{definition}\label{d:lattice}
Let $\mathcal L$ be a complete lattice, and $I,J \in \mathcal L$. Say that $I$ is \emph{compactly contained} in $J$, written $I\Subset J$, if whenever $(I_\lambda)$ is a family in $\mathcal L$ for which $J \leq \sup I_\lambda$, then there are finitely many $I_{\lambda_1}, \dots, I_{\lambda_n}$ such that $I \leq \sup I_{\lambda_k}$.

A map $\Phi \colon \mathcal L \to \mathcal L'$ of complete lattices is called a \emph{$\Cu$-morphism} if it preserves suprema and compact containment.\footnote{See \cite[Section 2.2]{Gabe-O2class} for motivation for why this name makes sense.}

Whenever $\Phi , \Psi \colon \mathcal L \to \mathcal L'$ are maps of partially ordered sets, write $\Phi \leq \Psi$ whenever $\Phi(I) \leq \Psi(I)$ for all $I \in \mathcal L$.
\end{definition}

Recall that the ideal lattice $\mathcal I(A)$ of a $C^\ast$-algebra $A$ -- the set of all two-sided, closed ideals in $A$ -- is a complete lattice with suprema and infima of a non-empty subset $\mathcal S \subseteq \mathcal I(A)$ given by
\begin{equation}
\sup \mathcal S = \overline{\sum_{I\in \mathcal S} I}, \qquad \inf \mathcal S = \bigcap_{I\in \mathcal S} I.
\end{equation}

By convention, $\sup \emptyset = 0$ and $\inf \emptyset = A$ in $\mathcal I(A)$.
As $\mathcal I(A)$ is a complete lattice, there is a notion of compact containment of ideals in $C^\ast$-algebras which played a major role in \cite{Gabe-O2class}.

It is important to consider the ideal lattice as an invariant which is also defined for maps as in the following definition.

\begin{definition}
For any completely positive map $\phi \colon A \to B$, let 
\begin{equation}
\mathcal I(\phi) \colon \mathcal I(A) \to \mathcal I(B), \qquad \mathcal I(\phi)(I) = \overline{B \phi(I) B}
\end{equation}
for $I\in \mathcal I(A)$.
\end{definition}

\begin{remark}
It was shown in \cite[Lemma 2.12]{Gabe-O2class}, that $\mathcal I(\phi)$ always preserves suprema, and that $\mathcal I(\phi)$ is a $\Cu$-morphism provided $\phi$ is a $\ast$-homomorphism.
\end{remark}

\begin{remark}
Suppose that $\phi \colon A \to B$ is a $\ast$-homomorphism. Above one associates a map $\mathcal I(\phi) \colon \mathcal I(A) \to \mathcal I(B)$ in a covariant way. This was done in \cite{Gabe-O2class} so that one could make use of compact containment of ideals.

In the work of Kirchberg \cite{Kirchberg-non-simple} he instead considers the pre-image map $\phi^{-1} \colon \mathcal I(B) \to \mathcal I(A)$ which is a contravariant approach. 

One can show that these two notions are each others duals in the sense that there is a natural one-to-one correspondence between $\Cu$-morphisms $\mathcal I(A) \to \mathcal I(B)$, and maps $\mathcal I(B) \to \mathcal I(A)$ which are monotone continuous, i.e.~maps that preserve infima and increasing suprema. This one-to-one correspondence takes $\mathcal I(\phi)$ to $\phi^{-1}$.

This is because $(\phi^{-1}, \mathcal I(\phi))$ is a \emph{Galois connection}, see \cite[Section O-3]{GHKLMS-book}, i.e.~both $\phi^{-1}$ and $\mathcal I(\phi)$ are order preserving, and whenver $I\in \mathcal I(A)$ and $J \in \mathcal I(B)$ one has 
\begin{equation}
\mathcal I(\phi)(I) \subseteq J  \quad \textrm{if and only if} \quad I \subseteq \phi^{-1}(J).
\end{equation}

As the duality is not needed in this paper, I leave it to the reader to confirm this at their own will.\footnote{It is straight forward to verify; alternatively it can be deduced by combining several results from \cite[Section O-3]{GHKLMS-book}.}
\end{remark}

It was observed \cite[Remark 2.11]{Gabe-O2class}, that $\mathcal I$ is \emph{not} functorial on the category of $C^\ast$-algebras with c.p.~maps as morphisms.\footnote{For instance, if $\phi \colon \mathbb C \to M_2(\mathbb C)$ is the embedding to the $(1,1)$-corner, and $\psi \colon M_2(\mathbb C) \to \mathbb C$ is the compression to the $(2,2)$-corner, then $\mathcal I(\psi \circ \phi) \neq \mathcal I(\psi) \circ \mathcal I(\phi)$.} However, this annoyance goes away if we restrict our attention to $\ast$-homomorphisms.\footnote{Or, more generally, if we only consider c.p.~order zero maps.} The following proposition was proved in \cite{Gabe-O2class}, and as it is so fundamental, it will usually not be mentioned.

\begin{proposition}[Cf.~\cite{Gabe-O2class} Proposition 2.15]
$\mathcal I$ is a covariant functor from the category of $C^\ast$-algebras with $\ast$-homomorphisms to the category of complete lattices with $\Cu$-morphisms.
\end{proposition}

A special case of Definition \ref{d:lattice} says that for c.p.~maps $\phi, \psi \colon A \to B$ we write $\mathcal I(\phi) \leq \mathcal I(\psi)$ if $\mathcal I(\phi)(I) \subseteq \mathcal I(\psi)(I)$ for all $I \in \mathcal I(A)$.

\begin{theorem}[{\cite[Theorem 3.3]{Gabe-O2class}}]\label{t:approxdom}
Let $A$ and $B$ be $C^\ast$-algebras with $A$ exact, and let $\phi , \rho \colon A \to B$ be nuclear maps with $\phi$ a $\ast$-homomorphism. The following are equivalent:
\begin{itemize}
\item[$(i)$] $\phi$ approximately dominates $\rho$ (see Definition \ref{d:approxdom});
\item[$(ii)$] $\mathcal I(\rho) \leq \mathcal I(\phi)$;
\item[$(iii)$] $\rho(a) \in \overline{B \phi(a) B}$ for any positive $a\in A$.
\end{itemize}
\end{theorem}

The above theorem will be used in place of Proposition \ref{p:fulldom} for full maps, which stated that any full $\ast$-homomorphism approximately dominates any nuclear map.


\subsection{Actions of topological spaces on $C^\ast$-algebras}

\begin{definition}
Let $X$ be a topological space. An \emph{action} of $X$ on a $C^\ast$-algebra $A$ is an order preserving map
\begin{equation}
\Phi_A \colon \mathcal O(X) \to \mathcal I(A)
\end{equation}
where $\mathcal O(X)$ denotes the complete lattice of open subsets of $X$.\footnote{Suprema in $\mathcal O(X)$ is given by taking unions, and infima is given by taking the interior of the intersection, i.e.~if $(U_\lambda)_{\lambda \in \Lambda}$ is a collection in $\mathcal O(X)$ then $\inf_{\lambda \in \Lambda} U_\lambda = (\bigcap_{\lambda \in \Lambda} U_\lambda)^\circ$.} The pair $(A, \Phi_A)$ is called an \emph{$X$-$C^\ast$-algebra}. Often $\Phi_A$ is omitted from the notation by defining $A(U) := \Phi_A(U)$ for all $U\in \mathcal O(X)$.

A map $\phi \colon A \to B$ between $X$-$C^\ast$-algebras is called \emph{$X$-equivariant} (or $\Phi_A$-$\Phi_B$-equivariant) if 
\begin{equation}
\phi(A(U)) \subseteq B(U) \qquad \textrm{for all }U\in \mathcal O(X),
\end{equation}
or equivalently, if $\phi(\Phi_A(U)) \subseteq \Phi_B(U)$ for all $U\in \mathcal O(X)$.
\end{definition}

Clearly a c.p.~map $\phi \colon A \to B$ is $X$-equivariant if and only if $\mathcal I(\phi) \circ \Phi_A \leq \Phi_B$.

The advantage of considering $X$-$C^\ast$-algebras is that one can consider the category of all (separable) $X$-$C^\ast$-algebras with $X$-equivariant $\ast$-homo\-morphisms as morphisms. In this sense, $KK(X)$ -- which will be defined in the Section \ref{s:KK} -- becomes a functor from this category to some target category which turns out to be triangulated, see \cite[Proposition 3.11]{MeyerNest-bootstrap}.

Often it is an advantage only to consider $X$-$C^\ast$-algebras for which the action has certain additional properties, such as actions that preserve (increasing) suprema and/or infima. This motivates the following definition.

\begin{definition}\label{d:Xalg}
Let $X$ be a topological space, let $(A, \Phi_A)$ be an $X$-$C^\ast$-algebra. Then $(A, \Phi_A)$ is said to be
\begin{itemize}
\item \emph{(fintely) lower semicontinuous} if $\Phi_A$ preserves (finite) infima;
\item \emph{(finitely) upper semicontinuous} if $\Phi_A$ preserves (finite) suprema;
\item \emph{monotone upper semicontinuous} if $\Phi_A$ preserves suprema of non-empty, upwards directed sets;
\item \emph{(monotone) continuous} if it is both lower semicontinuous and (monotone) upper semicontinuous;
\item \emph{$X$-compact} if $\Phi_A$ preserves compact containment;
\item \emph{tight} if $\Phi_A$ is an order isomorphism.
\end{itemize}
\end{definition}

Recall that in any complete lattice $\inf \emptyset$ is the largest element and $\sup \emptyset$ is the smallest element. Hence if $A$ is an $X$-$C^\ast$-algebra which is (finitely) lower semicontinuous then $A(X) = A(\inf \emptyset) = \inf \emptyset = A$. Similarly, if $A$ is (finitely) upper semicontinuous then $A(\emptyset) = 0$ (where $\emptyset \in \mathcal O(X)$ is the smallest element).

\begin{remark}
All but one of the above definitions of actions have appeared previously in the literature, in particular in \cite{Kirchberg-non-simple}. The only exception is the notion of an \emph{$X$-compact} action which is a new concept.
\end{remark}

\begin{example}[Canonical tight action]\label{ex:tightaction}
As is customary, let $\Prim A$ denote the \emph{primitive ideal space} of the $C^\ast$-algebra $A$, i.e.~the set of all primitive ideals\footnote{An ideal is \emph{primitive} if it is the kernel of an irreducible representation.} of $A$ equipped with the Jacobsen topology.
Recall -- see for instance \cite[Theorem 4.1.3]{Pedersen-book-automorphism} -- that for a $C^\ast$-algebra $A$ there is an order iso\-morphism $\I_A \colon \mathcal O(\Prim A) \xrightarrow \cong \mathcal I(A)$ given by
\begin{equation}
\I_A(V) = \bigcap_{\mathfrak p \in \Prim A \setminus V} \mathfrak p, \qquad V\in \mathcal O(\Prim A).
\end{equation}
Hence $(A, \I_A)$ is a tight $\Prim A$-$C^\ast$-algebra.
\end{example}

\begin{notation}
The action $\I_A \colon \mathcal O(\Prim A) \to \mathcal I(A)$ from Example \ref{ex:tightaction} will be referred to as the \emph{canonical tight action} of $A$.
\end{notation}

\begin{example}[Ordinary $C^\ast$-algebras]\label{ex:onepointXalg}
Let $\{\star\}$ be a one-point topological space, and let $A$ be a $\{\star\}$-$C^\ast$-algebra. Then $A(\emptyset) = 0$ and $A(\{\star\}) = A$ if and only if $A$ is continuous. In particular, the category of $C^\ast$-algebras is isomorphic to the category of continuous $\{\star\}$-$C^\ast$-algebras.

If $A$ is a continuous $\{\star\}$-$C^\ast$-algebra then
\begin{itemize}
\item $A$ is tight if and only if the $C^\ast$-algebra $A$ is simple;
\item $A$ is $\{\star\}$-compact if and only if the primitive ideal space $\Prim A$ is compact.
\end{itemize}
\end{example}

\begin{example}[$C_0(X)$-algebras]\label{ex:C(X)alg}
Let $X$ be a locally compact Hausdorff space. A \emph{$C_0(X)$-algebra} is a $C^\ast$-algebra $A$ together with a $\ast$-homomorphism $\psi_A \colon C_0(X) \to \mathscr Z\multialg{A}$  (the center of the multiplier algebra)
such that 
\begin{equation}\label{eq:psiC(X)A}
\overline{\psi_A(C_0(X)) A} = A.\footnote{There is some disagreement in the literature whether or not one only wants to consider the case where $\psi_A$ is injective. However, this additional criteria rules out important special cases such as the skyscraper $C_0(X)$-algebras, obtained by letting $\psi_A$ be a composition $C_0(X) \xrightarrow{\ev_x} \mathbb C \to \mathscr Z\multialg{A}$ of a point-evaluation and the canonical unital embedding of $\mathbb C$.}
\end{equation} 
The map $\psi_A$ induces an action of $X$ on $A$, namely $\Phi_A \colon \mathcal O(X) \to \mathcal I(A)$ given by
\begin{equation}
\Phi_A(U) = \overline{\psi_A(C_0(U)) A}, \qquad U\in \mathcal O(X).
\end{equation}
The $C_0(X)$-algebra $A$ is called \emph{continuous} if the map 
\begin{equation}
X \ni x \mapsto \| a + \Phi_A(X \setminus \{x\})\|_{A/\Phi_A(X \setminus \{x\})}
\end{equation}
 is continuous for every $a\in A$.
It turns out that the assignment $\psi_A \mapsto \Phi_A$ described above is a one-to-one correspondence between $\ast$-homomorphisms $\psi_A \colon C_0(X) \to \mathscr Z\multialg{A}$ satisfying \eqref{eq:psiC(X)A}, 
and actions of $X$ on $A$ which are upper semicontinuous and finitely lower semicontinuous, see \cite[Sections 2.1 and 2.2]{MeyerNest-bootstrap}.
In this way, continuous $C_0(X)$-algebras correspond exactly to continuous $X$-$C^\ast$-algebras.
\end{example}

\begin{example}[$C^\ast$-algebras over topological spaces]\label{ex:CoverX}
Let $X$ be a topological space. A \emph{$C^\ast$-algebra over $X$} is a $C^\ast$-algebra $A$ together with a continuous map $\psi_A \colon \Prim A \to X$. This induces an action of $X$ on $A$, namely $\Phi_A \colon \mathcal O(X) \to \mathcal I(A)$ given by
\begin{equation}
\Phi_A (U) = \I_A (\psi_A^{-1}(U)), \qquad U \in \mathcal O(X),
\end{equation}
where $\I_A \colon \mathcal O(\Prim A) \xrightarrow \cong \mathcal I(A)$ is the canonical tight action (Example \ref{ex:tightaction}).

It turns out that $(A, \Phi_A)$ is always upper semicontinuous and fintely lower semicontinuous.
If $X$ is sober,\footnote{A topological space $X$ is called \emph{sober} if the map $X \to \mathcal O(X)$ given by $x\mapsto X \setminus \overline{\{x\}}$ is injective and maps onto the set of all prime open subsets of $X$. An open subset $U \subseteq X$ is \emph{prime} if whenever $V,W \in \mathcal O(X)$ are such that $V \cap W \subseteq U$ then $V \subseteq U$ or $W \subseteq U$. For any topological space $X$ there is a sober space $\hat X$ such that $\mathcal O(X) \cong \mathcal O(\hat X)$, so it is essentially no loss of generality to assume that $X$ is sober, see \cite[Section 2.5]{MeyerNest-bootstrap}.} which is essentially no loss of generality, then the above construction $\psi_A \mapsto \Phi_A$ is a one-to-one correspondence between continuous maps $\Prim A \to X$ and actions $\mathcal O(X) \to \mathcal I(A)$ which are upper semicontinuous and finitely lower semicontinuous, see \cite[Lemma 2.25]{MeyerNest-bootstrap}.
\end{example}

\subsection{$X$-fullness}

Recall that a $\ast$-homomorphism $\phi \colon A \to B$ is full if $\phi(a)$ is full in $B$ for every non-zero $a\in A$. Essentially, this means that $\phi$ is as large as possible in an ideal-related sense. The same phenomena will be studied for $X$-equivariant maps in the sense that $X$-fullness means that $\phi$ is as (ideal-related) large as possible, provided one assumes $\phi$ is $X$-equivariant.

For a subset $Y$ of a topological space $X$, let $Y^\circ$ denote the interior of $Y$.

\begin{lemma}\label{l:dualaction}
Let $(A, \Phi_A)$ be a lower semicontinuous $X$-$C^\ast$-algebra. There is a well-defined order preserving map $\Psi_A \colon \mathcal I(A) \to \mathcal O(X)$ given by
\begin{equation}\label{eq:dualaction}
\Psi_A (I) = \bigg( \bigcap_{\substack{V \in \mathcal O(X) \\ I \subseteq \Phi_A(V)}} V \bigg)^\circ , \qquad I\in \mathcal I(A).
\end{equation}
The maps $\Phi_A$ and $\Psi_A$ satisfy
\begin{equation}\label{eq:dualGalois}
I \subseteq \Phi_A(U) \qquad \textrm{if and only if} \qquad \Psi_A(I) \subseteq U
\end{equation}
for all $I\in \mathcal I(A)$ and $U \in \mathcal O(X)$.\footnote{This means that $(\Phi_A, \Psi_A)$ is a Galois connection, see \cite[Section O-3]{GHKLMS-book}. In particular, it follows from \cite[Corollary O-3.5]{GHKLMS-book} that there exists a map $\Psi_A$ satisfying \eqref{eq:dualGalois} if and only if $(A, \Phi_A)$ is lower semicontinuous.}
 In particular,
\begin{equation}\label{eq:dualeq}
I \subseteq \Phi_A(\Psi_A(I)) \qquad \textrm{and} \qquad \Psi_A (\Phi_A(U)) \subseteq U
\end{equation}
for all $I\in \mathcal I(A)$ and $U \in \mathcal O(X)$.
\end{lemma}
\begin{proof}
Note that $\inf \emptyset = X$ in the complete lattice $\mathcal O(X)$, and that $\inf \emptyset = A$ in the complete lattice $\mathcal I(A)$.
Hence as $\Phi_A$ preserves infima of the empty set, one has $\Phi_A(X) = A$. It easily follows that $\Psi_A$ is well-defined since the index of the intersection in \eqref{eq:dualaction} always contains $X\in \mathcal O(X)$ and is therefore never empty. Clearly $\Psi_A$ is order preserving.

If $I \subseteq \Phi_A(U)$ then $\Psi_A(I) \subseteq U$ by the definition of $\Psi_A$. Conversely, suppose $\Psi_A(I) \subseteq U$. As $\Phi_A$ preserves infima,\footnote{This means that $\Phi_A((\bigcap \mathcal U)^\circ) = \bigcap \Phi_A(\mathcal U)$ for a subset $\mathcal U \subseteq \mathcal O(X)$.} it follows that
\begin{equation}
\Phi_A(\Psi_A(I)) = \bigcap_{\substack{V \in \mathcal O(X) \\ I \subseteq \Phi_A(V)}} \Phi_A(V).
\end{equation}
As the right hand side above contains $I$ by definition, one gets 
\begin{equation}
I \subseteq \Phi_A(\Psi_A(I)) \subseteq \Phi_A(U)
\end{equation}
since $\Phi_A$ is order preserving.

``In particular'' follows from \eqref{eq:dualGalois} by considering the cases $I= \Phi_A(U)$ and $\Psi_A(I) = U$.
\end{proof}

\begin{definition}
Let $(A,\Phi_A)$ be a lower semicontinuous $X$-$C^\ast$-algebra. The \emph{dual action} of $\Phi_A$ (or the dual action of the $X$-$C^\ast$-algebra $(A, \Phi_A)$) is the map $\Psi_A \colon \mathcal I(A) \to \mathcal O(X)$ defined in \eqref{eq:dualaction}.
\end{definition}

\begin{lemma}\label{l:Cuaction}
Let $(A,\Phi_A)$ be a lower semicontinuous $X$-$C^\ast$-algebra with dual action $\Psi_A$, and let $(B, \Phi_B)$ be an $X$-$C^\ast$-algebra.
\begin{itemize}
\item[$(a)$] A c.p.~map $\phi \colon A \to B$ is $X$-equivariant if and only if $\mathcal I(\phi) \leq \Phi_B \circ \Psi_A$.
\item[$(b)$] $\Psi_A$ preserves suprema. In particular, $\Phi_B \circ \Psi_A \colon \mathcal I(A) \to \mathcal I(B)$ preserves suprema whenever $(B, \Phi_B)$ is upper semicontinuous.
\item[$(c)$] If $(A, \Phi_A)$ is monotone continuous then $\Psi_A$ preserves compact containment. In particular, $\Phi_B \circ \Psi_A \colon \mathcal I(A) \to \mathcal I(B)$ is a $\Cu$-morphism whenever $(B, \Phi_B)$ is $X$-compact and upper semicontinuous (in addition to $(A, \Phi_A)$ being monotone continuous).
\end{itemize}
\end{lemma}
\begin{proof}
$(a)$: If $\phi$ is $X$-equivariant and $I\in \mathcal I(A)$ then
\begin{equation}
\phi(I) \stackrel{\eqref{eq:dualeq}}{\subseteq} \phi(\Phi_A(\Psi_A(I))) \subseteq \Phi_B(\Psi_A(I)),
\end{equation}
so $\mathcal I(\phi) \leq \Phi_B \circ \Psi_A$. Conversely, if $\mathcal I(\phi) \leq \Phi_B \circ \Psi_A$ and $U\in \mathcal O(X)$ then
\begin{equation}
\phi(\Phi_A(U)) \subseteq \Phi_B (\Psi_A( \Phi_A(U))) \stackrel{\eqref{eq:dualeq}}{\subseteq} \Phi_B(U),
\end{equation}
so $\phi$ is $X$-equivariant.

$(b)$: Let $S \subseteq \mathcal I(A)$ be a (possibly empty) subset, let $I := \sup S$ and $U := \sup \Psi_A(S)$. As $\Psi_A(J) \subseteq \Psi_A(I)$ for all $J  \in S$, one gets $U \subseteq \Psi_A(I)$, and it remains to show the reverse inclusion. Since $\Psi_A(J) \subseteq U$ for all $J \in S$ one gets $J \subseteq \Phi_A(U)$ for all $J\in S$ by \eqref{eq:dualGalois}. Hence $I \subseteq \Phi_A(U)$ so by \eqref{eq:dualGalois} one gets $\Psi_A(I) \subseteq U$.

$(c)$: Suppose $I \Subset J$ in $\mathcal I(A)$ and let $(U_\lambda)_{\lambda \in \Lambda}$ be an increasing net in $\mathcal O(X)$ such that $\Psi_A(J) \subseteq \sup U_\lambda$. By \eqref{eq:dualGalois} one has 
\begin{equation}
J \subseteq \Phi_A(\sup U_\lambda) = \sup \Phi_A(U_\lambda)
\end{equation}
where the second equality follows from monotone upper semicontinuity of $(A, \Phi_A)$. As $I\Subset J$ there is $\lambda\in \Lambda$ such that $I \subseteq \Phi_A(U_\lambda)$. By \eqref{eq:dualGalois} one has $\Psi_A(I) \subseteq U_\lambda$, so $\Psi_A(I) \Subset \Psi_A(J)$. Thus $\Psi_A$ preserves compact containment.
\end{proof}

If $\phi \colon A \to B$ is an $X$-equivariant c.p.~map then the largest $\mathcal I(\phi)$ can possibly be in an ideal-related sense is $\Phi_B \circ \Psi_A$ by Lemma \ref{l:Cuaction}$(a)$. This motivates the following definition.

\begin{definition}
Let $(A, \Phi_A)$ and $(B, \Phi_B)$ be $X$-$C^\ast$-algebras with $(A,\Phi_A)$ lower semicontinuous, and let $\Psi_A$ be the dual action of $\Phi_A$. A c.p.~map $\phi \colon A \to B$ is said to be \emph{$X$-full} if $\mathcal I(\phi) = \Phi_B \circ \Psi_A$.
\end{definition}

\begin{remark}
Suppose $A$ and $B$ are $X$-$C^\ast$-algebras with $A$ lower semicontinuous and dual action $\Psi_A \colon \mathcal I(A) \to \mathcal O(X)$. A c.p.~map $\phi \colon A \to B$ is $X$-full if and only if $\phi(a)$ is full in $B(\Psi_A(\overline{AaA}))$ for any $a\in A$.
This was how $X$-fullness was (essentially) defined in \cite{GabeRuiz-absrep}, and how $X$-fullness was described in the introduction. In fact, by definition of $\Psi_A$, $U=\Psi_A(\overline{AaA})$ is the smallest open subset of $X$ such that $a\in A(U)$.
\end{remark}

\begin{example}
Let $\{\star\}$ be a one-point topological space, and let $A$ and $B$ be continuous $\{ \star\}$-$C^\ast$-algebras, cf.~Example \ref{ex:onepointXalg}. Then a $\ast$-homo\-morphism $\phi \colon A \to B$ is full if and only if it is $\{\star\}$-full.
\end{example}

\begin{example}
Let $(A, \Phi_A)$ be an $X$-$C^\ast$-algebra. Then $(A, \Phi_A)$ is tight if and only if $(A,\Phi_A)$ is lower semicontinuous, $\Phi_A$ is injective and $\id_A$ is $X$-full.
\end{example}

\begin{example}
Let $\phi \colon A \to B$ be a c.p.~map, let $\I_A \colon \mathcal O(\Prim A) \to \mathcal I(A)$ be the canonical tight action, and let $\Phi_B := \mathcal I(\phi) \circ \I_A$. Then $(A, \I_A)$ and $(B, \Phi_B)$ are $\Prim A$-$C^\ast$-algebras, and $\phi$ is $\Prim A$-full.

Similarly, if $\phi$ is a $\ast$-homomorphism, one can let $\I_B \colon \mathcal O(\Prim B) \to \mathcal I(B)$ be the canonical tight action, and $\Phi_A := \phi^{-1} \circ \I_B$. Then $(A, \Phi_A)$ and $(B, \I_B)$ are $\Prim B$-$C^\ast$-algebras and $\phi$ is $\Prim B$-full.
\end{example}

While the following is not needed in this paper, it shows that there often exist $X$-full c.p.~maps. It is an immediate consequence of Lemma \ref{l:Cuaction}$(b)$ and \cite[Proposition 5.5]{Gabe-O2class}. Results of this form were one of the key ingredients in \cite{Gabe-cplifting}.

\begin{proposition}
Let $A$ and $B$ be separable $X$-$C^\ast$-algebras for which $A$ is lower semicontinuous, and $B$ is nuclear and upper semicontinuous. Then there exists an $X$-full c.p.~map $A\to B$.
\end{proposition}

The following is an immediate consequence of Theorem \ref{t:approxdom} and Lemma \ref{l:Cuaction} $(a)$.

\begin{corollary}\label{c:Xfulldom}
Let $A$ and $B$ be $X$-$C^\ast$-algebras with $A$ exact and lower semicontinuous, suppose that $\phi \colon A \to B$ is an $X$-full, nuclear $\ast$-homo\-morphism, and that $\rho\colon A \to B$ is an $X$-equivariant, nuclear c.p.~map. Then $\phi$ approximately dominates $\rho$.
\end{corollary}


\subsection{Tensor products}

Let $X$ be a topological space, let $A$ be an $X$-$C^\ast$-algebra, and let $D$ be a $C^\ast$-algebra. The tensor products $A \otimes D$ (spatial tensor product) and $A \otimes_{\max{}} D$ (maximal tensor product) have canonical actions of $X$ given by
\begin{equation}
(A\otimes D)(U) = A(U) \otimes D, \qquad (A \otimes_{\max{}} D) (U) = A(U) \otimes_{\max{}} D
\end{equation}
for $U\in \mathcal O(X)$.

\begin{remark}
Unless otherwise stated, if $A$ is an $X$-$C^\ast$-algebra and $D$ is a $C^\ast$-algebra, then $A\otimes D$ and $A\otimes_{\max{}} D$ are implicitly assumed to be $X$-$C^\ast$-algebras equipped with the action of $X$ defined above.
\end{remark}

If $B$ is an $X$-$C^\ast$-algebra and $Y$ is a locally compact, Hausdorff space, then $C_0(Y,B)$ is an $X$-$C^\ast$-algebra via the canonical identification $C_0(Y, B) \cong C_0(Y) \otimes B$.

\begin{remark}[Homotopies]
The above construction allows one to consider \emph{homotopies} of $X$-equivariant $\ast$-homomorphisms: Two $X$-equivariant $\ast$-homomorphisms $\phi_0,\phi_1 \colon A \to B$ are \emph{homotopic} if (by definition) there is an $X$-equivariant  $\ast$-homo\-morphism $\Phi \colon A \to C([0,1], B)$ such that $\phi_0 = \ev_0 \circ \Phi$ and $\phi_1 = \ev_1 \circ \Phi$.

In this sense, the cone $C_0((0,1], B)$ of an $X$-$C^\ast$-algebra $B$ is always contractible (i.e.~homotopic to zero) as an $X$-$C^\ast$-algebra. 

However, there might be other actions of $X$ on $C_0((0,1], B)$ such that the corresponding $X$-$C^\ast$-algebra is not contractible. For instance, $C_0((0,1])$ with the canonical tight action of $(0,1]$ is not contractible as $(0,1]$-$C^\ast$-algebras.
\end{remark}


\subsection{$X$-nuclear maps}

Just as in the classical case of absorption from Section \ref{s:absrep}, one needs nuclearity in order to properly study absorption. In this subsection the focus is on this ideal-related version of nuclearity.

\begin{definition}\label{d:Xnucabs}
Let $A$ and $B$ be $X$-$C^\ast$-algebras.
A c.p.~map $\eta\colon A \to B$ is called \emph{$X$-nuclear} (or $X$-residually nuclear, or $\Phi_A$-$\Phi_B$-residually nuclear) if $\eta$ is $X$-equivariant and the induced map
\begin{equation}
[\eta]_U \colon A/A(U) \to B/B(U)
\end{equation}
is nuclear for each $U\in \mathcal O(X)$.
\end{definition}

As quotients of nuclear $C^\ast$-algebras are nuclear \cite[Corollary 9.4.4]{BrownOzawa-book-approx}, it follows that if $A$ or $B$ is a nuclear $X$-$C^\ast$-algebra then any $X$-equivariant map $\phi \colon A \to B$ is $X$-nuclear.

In the case where $A$ is exact, a map $\phi \colon A \to B$ being $X$-nuclear is equivalent (up to minor assumptions on the actions) to being nuclear and $X$-equivariant, as witnessed by the following lemma. 

\begin{lemma}\label{l:nucquotient}
Let $A$ and $B$ be $X$-$C^\ast$-algebras for which $A$ is exact, and suppose that $B(\emptyset) =0$. Then a c.p.~map $\eta \colon A \to B$ is $X$-nuclear if and only if it is nuclear and $X$-equivariant.
\end{lemma}
\begin{proof}
If $\eta$ is $X$-nuclear it is $X$-equivariant by definition. Moreover, as $\eta(A(\emptyset)) \subseteq B(\emptyset) = 0$, it follows that $\eta$ is the composition $A \to A/A(\emptyset) \xrightarrow{[\eta]_\emptyset} B/B(\emptyset) = B$.
As $[\eta]_\emptyset$ is nuclear, so is $\eta$.

For the converse, suppose that $\eta$ is nuclear and $X$-equivariant, and let $U\in \mathcal O(X)$. The composition  $A \xrightarrow \eta B \twoheadrightarrow B/B(U)$ is nuclear by nuclearity of $\eta$, and $A(U)$ is contained in its kernel since $\eta$ is $X$-equivariant. Hence by \cite[Proposition 3.2]{Dadarlat-qdmorphisms}\footnote{Which uses the deep result that exact $C^\ast$-algebras are locally reflexive, and that local reflexivity passes to quotients, see \cite[Section 9]{BrownOzawa-book-approx}.} it follows that the induced map $A/A(U) \to B/B(U)$ is nuclear. So $\eta$ is $X$-nuclear.
\end{proof}

The following shows that the criteria $B(\emptyset) = 0$ in the above lemma is always satisfied in the cases we are interested in.

\begin{remark}\label{r:Bempty}
Let $(A,\Phi_A)$ and $(B,\Phi_B)$ be $X$-$C^\ast$-algebras with $(A,\Phi_A)$ lower semicontinuous and suppose that there exists an $X$-full c.p.~map $\phi \colon A \to B$. Then $\Phi_B(\emptyset) = 0$.

In fact, let $\Psi_A\colon \mathcal I(A) \to \mathcal O(X)$ be the dual action of $\Phi_A$. As $0 \subseteq \Phi_A(\emptyset)$ it follows from \eqref{eq:dualGalois} that $\Psi_A(0) = \emptyset$. By $X$-fullness one has $\mathcal I(\phi) = \Phi_B \circ \Psi_A$, so
\begin{equation}
\Phi_B(\emptyset) = \Psi_B(\Psi_A(0)) = \mathcal I(\phi)(0) = 0.
\end{equation}
\end{remark}

The following lemma will be useful.

\begin{lemma}\label{l:XnucC(Y)}
Let $A$ and $B$ be $X$-$C^\ast$-algebras, and let $Y$ be a locally compact, Hausdorff space. Then a c.p.~map $\rho \colon A \to C_0(Y, B)$ is $X$-equivariant (respectively $X$-nuclear) if and only if $\ev_y \circ \rho \colon A \to B$ is $X$-equivariant (respectively $X$-nuclear) for every $y\in Y$.
\end{lemma}
\begin{proof}
We first show that a map $\rho \colon A \to C_0(Y, B)$ is nuclear if and only if $\ev_y \circ \rho$ is nuclear for every $y\in Y$.
If $\rho$ is nuclear then clearly $\ev_y \circ \rho$ is nuclear for every $y\in Y$ (see Observation \ref{o:nuccomp}). To prove the converse, suppose that each $\ev_y \circ \rho$ is nuclear. We will use the tensor product characterisation of nuclear maps, Proposition \ref{p:nuctensor}, to prove that $\rho$ is nuclear, so fix a $C^\ast$-algebra $C$ and an element $d\in A \otimes_{\max{}} C$ which vanishes in $A\otimes C$. We should check that $(\rho \otimes_{\max{}} \id_C)(d) =0$ in $C_0(Y, B) \otimes_{\max{}} C$. By identifying $C_0(Y, B)\otimes_{\max{}} C$ and $C_0(Y, B\otimes_{\max{}} C)$ in the canonical way, it suffices to check that $\ev_y((\rho \otimes_{\max{}} \id_C)(d)) =0$ for every $y\in Y$. However, we have
\begin{equation}
\ev_y((\rho \otimes_{\max{}} \id_C)(d)) = ((\ev_y \circ \rho) \otimes_{\max{}} \id_C)(d) =0
\end{equation}
for all $y\in Y$ since $\ev_y \circ \rho$ is nuclear (Proposition \ref{p:nuctensor}), and hence $\rho$ is nuclear.

Now, the $X$-equivariant case is obvious, and the $X$-nuclear case is a consequence of what was proved above.
\end{proof}

\begin{remark}
Given Lemma \ref{l:nucquotient}, the reader may be wondering why $X$-nuclearity is not simply defined as an $X$-equivariant map which is nuclear, since this is equivalent when the domain is exact which is all we will eventually care about. The reason is that I do not believe that this is the ``correct'' definition when the domain is non-exact.

In fact, the role of nuclearity and $X$-nuclearity is essentially about determining approximate domination of maps, cf.~Proposition \ref{p:fulldom} and Theorem \ref{t:approxdom}. While Theorem \ref{t:approxdom} determines approximate domination of nuclear maps with exact domain, Proposition \ref{p:fulldom} had the advantage that any (not necessarily nuclear) full $\ast$-homomorphism out of a (not necessarily exact) $C^\ast$-algebra approximately dominates any nuclear map. The same holds in the $X$-equivariant case: any $X$-full $\ast$-homomorphism approximately dominates any $X$-nuclear map. However, my current proof of this is quite involved and I have therefore chosen not to include it in the paper as it is not needed.
\end{remark}


\section{Absorbing representations revisited}

In this section the theory of absorbing representations from Section \ref{s:absrep} will be generalised to $X$-$C^\ast$-algebras. It will be convenient to work with Hilbert $C^\ast$-modules instead of just $C^\ast$-algebras and their multiplier algebras, as Hilbert modules play an important role in $KK$-theory.

Throughout this section all Hilbert modules are assumed to be \emph{right} Hilbert modules, and the inner product is conjugate linear in the first variable and linear in the second variable, in contrast to what is customary for Hilbert spaces.
I refer the reader to \cite{Lance-book-Hilbertmodules} for the basics of Hilbert modules.

For a Hilbert $B$-module $E$, $\mathcal B_B(E)$ denotes the $C^\ast$-algebra of adjointable operators on $E$, and $\mathcal K_B(E)$ denotes the ``compact'' operators on $E$, i.e.~the closed linear span of all rank one operators. Often $\mathcal B(E)$ and $\mathcal K(E)$ will be written instead.

\begin{definition}\label{d:weaklyXnuc}
Let $A$ and $B$ be $X$-$C^\ast$-algebras, and let $E$ be a right Hilbert $B$-module.
A $\ast$-homomorphism $\phi \colon A \to \mathcal B(E)$ is called \emph{weakly $X$-equivariant} (respectively \emph{weakly $X$-nuclear}) if the c.p.~map 
\begin{equation}\label{eq:weaknucHilbert}
A \ni a \mapsto \langle \xi, \phi(a) \xi \rangle_E \in B
\end{equation}
is $X$-equivariant (respectively $X$-nuclear) for every $\xi\in E$.
\end{definition}

\begin{remark}
If $E = B$ then one has $\mathcal B(E) = \multialg{B}$. Hence a $\ast$-homomorphism $\phi \colon A \to \multialg{B}$ is weakly $X$-equivariant (respectively weakly $X$-nuclear) if and only if $b^\ast \phi(-) b \colon A \to B$ is $X$-equivariant (respectively $X$-nuclear) for all $b\in B$.
\end{remark}

\begin{remark}
A $\ast$-homomorphism $\phi \colon A \to \mathcal B(E)$ is weakly $X$-equivariant if and only if $\phi(A(U)) E \subseteq E B(U)$ for all $U\in \mathcal O(X)$. Hence the above definition coincides with what Kirchberg calls ``$X$-equivariant'' in \cite[Definition 4.1]{Kirchberg-non-simple}.

The ``if'' is trivial, and for ``only if'' suppose $\phi$ is weakly $X$-equivariant. Let $a\in A(U)$ and $\xi \in E$. By $X$-equivariance of $\langle\xi, \phi(-)\xi\rangle_E$, 
\begin{equation}
\langle \phi(a) \xi , \phi(a) \xi\rangle_E = \langle \xi, \phi(a^\ast a)\xi\rangle_E \in B(U).
\end{equation}
Let $(e_\lambda)$ be an approximate identity in $B(U)$. Then
\begin{equation}
\| \phi(a) \xi - \phi(a)\xi e_{\lambda}\|^2_E = (1_{\widetilde B} - e_\lambda) \langle \phi(a) \xi, \phi(a) \xi\rangle_E (1_{\widetilde B} - e_\lambda) \to 0
\end{equation}
so $\phi(a)\xi = \lim_{\lambda} \phi(a) \xi e_\lambda \in \overline{EB(U)}$. By Cohen's factoristion theorem (see \cite[Theorem 4.6.4]{BrownOzawa-book-approx}), $\overline{EB(U)} = EB(U)$.
\end{remark}

\begin{remark}\label{r:Skandalis}
In \cite{Skandalis-KKnuc}, Skandalis considers $\ast$-homo\-morphisms $\phi \colon A \to \mathcal B(E)$ such that
\begin{equation}\label{eq:weaknucHilbertS}
A \ni a \mapsto (\langle \xi_i, \phi(a) \xi_j\rangle_E)_{i,j=1}^n \in M_n(B) 
\end{equation}
is nuclear for all $n\in \mathbb N$ and $\xi_1,\dots, \xi_n \in E$. By Lemma \ref{l:nucmatrix} below, this is equivalent to the map \eqref{eq:weaknucHilbert} being nuclear for all $\xi\in E$, i.e.~one only needs to verify the case $n=1$. So Skandalis' definition is a special case of the Definition \ref{d:weaklyXnuc} in the case where $X= \{\star\}$ is a one-point space, and $A$ and $B$ are continuous $\{\star\}$-$C^\ast$-algebras, see Example \ref{ex:onepointXalg}. This implies that $KK_\nuc(\{\star\}; A, B) = KK_\nuc(A,B)$ ($KK_\nuc(X)$ is defined in the following section).
\end{remark}

\begin{lemma}\label{l:nucmatrix}
Let $\rho = (\rho_{i,j})_{i,j=1}^n \colon A \to M_n(B)$ be a c.p.~map. Then $\rho$ is nuclear if and only if $\rho_{i,i}\colon A \to B$ is nuclear for each $i=1,\dots, n$.
\end{lemma}
\begin{proof}
``Only if'' is obvious. To prove ``if'', assume $\rho_{i,i}$ is nuclear for $i=1,\dots, n$. To show that $\rho$ is nuclear we will use the tensor product characterisation of nuclear maps, Proposition \ref{p:nuctensor}, so let a $C^\ast$-algebra $C$ be given, and $x\in A\otimes_{\max{}} C$ be a positive element vanishing in $A \otimes C$. We should show that $(\rho \otimes_{\max{}} \id_C)(x) =0$ in $M_n(B) \otimes_{\max{}} C = M_n(B \otimes_{\max{}} C)$. As $(\rho \otimes_{\max{}} \id_C)(x)$ is positive, it is zero if all its diagonal entries are zero (when considered as an $n\times n$-matrix). These entries are exactly $(\rho_{i,i} \otimes_{\max{}} \id_C)(x)$ for $i=1,\dots, n$ and these vanish as each $\rho_{i,i}$ is nuclear (use Proposition \ref{p:nuctensor} again). Hence $\rho$ is nuclear.
\end{proof}

A useful tool for showing that a representation $\phi \colon A \to \mathcal B(E)$ is weakly $X$-nuclear will be Lemma \ref{l:densespan}, which shows that one only has to check that \eqref{eq:weaknucHilbert} is $X$-nuclear for $\xi$ that span a dense subset of $E$. 
This will play an important role when showing that the Kasparov product is well-defined in the $X$-equivariant and $X$-nuclear setting. To prove this, the following lemma is needed.

\begin{lemma}\label{l:nucsum}
Let $\phi, \psi \colon A \to B$ be c.p.~maps. If $\phi + \psi$ is nuclear, then $\phi$ and $\psi$ are nuclear.

In particular, if $\phi+\psi$ is $X$-equivariant (respectively $X$-nuclear), then $\phi$ and $\psi$ are $X$-equivariant (respectively $X$-nuclear).
\end{lemma}
\begin{proof}
We use the tensor product characterisation of nuclear maps, Proposition \ref{p:nuctensor}, so let $D$ be any $C^\ast$-algebra, and $y \in A \otimes_{\max{}} D$ be a positive element such that $y$ vanishes in $A \otimes D$. Nuclearity entails $((\phi + \psi)\otimes \id_D) (y) =0$, so 
\begin{equation}
0 \leq (\phi \otimes \id_D) (y) \leq (\phi \otimes \id_D) (y) + ( \psi \otimes \id_D) (y) = ((\phi + \psi)\otimes \id_D )(y) = 0,
\end{equation}
so $(\phi \otimes \id_D)(y)=0$. Hence $\phi$, and similarly $\psi$, is nuclear.

``In particular'': If $\phi+ \psi$ is $X$-equivariant, let $U\in \mathcal O(X)$, and $a\in A(U)$ be positive. We have $0 \leq \phi(a) \leq \phi(a) + \psi(a) \in B(U)$, and since $B(U)$ is hereditary, $\phi(a) \in B(U)$. Hence $\phi$, and similarly $\psi$, is $X$-equivariant.

If $\phi + \psi$ is $X$-nuclear, then $[\phi + \psi]_U = [\phi]_U + [\psi]_U$ is nuclear for each $U\in \mathcal O(X)$. Hence $[\phi]_U$ and $[\psi]_U$ are nuclear by what we proved above, so $\phi$ and $\psi$ are $X$-nuclear.
\end{proof}

\begin{notation}\label{n:Xcone}
Recall that $\CP(A,B)$ denotes the convex cone of completely positive maps from $A$ to $B$ equipped with the point-norm topology. If $A$ and $B$ are $X$-$C^\ast$-algebras we define
\begin{equation}
\CP(X; A, B) := \{ \phi \in \CP(A,B) : \phi \textrm{ is $X$-equivariant}\}
\end{equation}
 and similarly 
\begin{equation}
\CP_\nuc(X;A, B) := \{ \phi \in \CP(A,B) : \phi \textrm{ is $X$-nuclear}\}.
\end{equation}
These are both easily seen to be point-norm closed subcones of $\CP(A,B)$.
\end{notation}

The following observation will help simplify some proofs, as one will only have to prove things once.

\begin{observation}
A $\ast$-homomorphism $A \to \mathcal B(E)$  is weakly $X$-equivariant (respectively weakly $X$-nuclear) exactly when $\langle \xi , \phi(-) \xi \rangle_E \in \mathscr C$ for all  $\xi \in E$, where $\mathscr C = \CP(X; A, B)$ (respectively $\mathscr C = \CP_\nuc(X; A, B)$).
\end{observation}

\begin{remark}\label{r:Xnuccone}
Lemma \ref{l:nucsum} means that $\CP(X;A ,B)$ and $\CP_\nuc(X; A, B)$ are \emph{hereditary} in the sense that if $\phi, \psi \in \CP(A,B)$ satisfy $\phi + \psi \in \mathscr C$, then $\phi, \psi \in \mathscr C$ where $\mathscr C$ denotes either $\CP(X; A, B)$ or $\CP_\nuc(X; A, B)$.\footnote{$\CP(X; A, B)$ and $\CP_\nuc(X;A,B)$ are point-norm closed, operator convex cones in the sense of \cite[Definition 4.1]{KirchbergRordam-zero}. All results and constructions in this -- as well as the following -- section have analogues where $\mathscr C \subseteq \CP(A,B)$ is such a point-norm closed, operator convex cone. A crucial and non-trivial fact is that such cones are hereditary, which can be proved by the same methods as the proof of \cite[Proposition 4.2]{KirchbergRordam-zero}, but as this will not be needed the details are omitted.}
\end{remark}

That the cones $\CP(X;A,B)$ and $\CP_\nuc(X;A,B)$ are hereditary, is the crucial ingredient which implies the following important lemma.

\begin{lemma}\label{l:densespan}
Let $\phi \colon A \to \mathcal B(E)$ be a $\ast$-homomorphism, and let $S \subseteq E$ be a subset such that $\overline{\mathrm{span}}\, S = E$. Then $\phi$ is weakly $X$-equivariant (respectively weakly $X$-nuclear) if and only if $\langle \xi, \phi(-) \xi \rangle_E$ is $X$-equivariant (respectively $X$-nuclear) for all $\xi\in S$.
\end{lemma}
\begin{proof}
``Only if'' is trivial, so we only prove ``if''. Let $\mathscr C$ denote either $\CP(X; A, B)$ or $\CP_\nuc(X; A, B)$, depending on whether we are proving the $X$-equivariant or the $X$-nuclear case. Suppose that $\xi_1, \xi_2\in E$ are such that $\langle \xi_i , \phi(-) \xi_i \rangle_E\in \mathscr C$. Then 
\begin{equation}
2 \langle \xi_1 , \phi(-) \xi_1\rangle_E + 2\langle \xi_2 , \phi(-) \xi_2 \rangle_E \in \mathscr C
\end{equation}
since $\mathscr C$ is a convex cone. As
\begin{eqnarray}
&& \langle \xi_1 + \xi_2, \phi(-) (\xi_1 + \xi_2) \rangle_E + \langle \xi_1 - \xi_2 , \phi(-) (\xi_1 - \xi_2) \rangle_E \nonumber\\
&=& 2 \langle \xi_1 , \phi(-) \xi_1\rangle_E + 2\langle \xi_2 , \phi(-) \xi_2 \rangle_E,
\end{eqnarray}
it follows that $\langle \xi_1 + \xi_2, \phi(-) (\xi_1 + \xi_2) \rangle_E  \in \mathscr C$ since $\mathscr C$ is hereditary by Remark \ref{r:Xnuccone}.

In particular, if $\lambda \in \mathbb C$ and $\xi \in S$, then 
\begin{equation}
\langle \lambda \xi, \phi(-) \lambda \xi\rangle_E = |\lambda|^2 \langle \xi, \phi(-) \xi\rangle_E \in \mathscr C,
\end{equation}
so by iterating the above result, it follows that
\begin{equation}
\left\langle \sum_{i=1}^n \lambda_i \xi_i , \phi(-) \sum_{j=1}^n \lambda_j \xi_j \right\rangle_E \in \mathscr C,
\end{equation}
for $n\in \mathbb N$, $\lambda_1, \dots,\lambda_n \in \mathbb C$ and $\xi_1, \dots, \xi_n \in S$. This shows that $\langle \xi, \phi(-) \xi\rangle_E \in \mathscr C$ for all $\xi \in \mathrm{span}\, S$.

Let $\xi \in E$ and $(\xi_n)_{n\in \mathbb N}$ be a sequence in $\mathrm{span}\, S$ converging to $\xi$. Then
\begin{eqnarray}
&& \| \langle \xi_n , \phi(a) \xi_n \rangle_E - \langle \xi, \phi(a) \xi \rangle_E \| \nonumber\\
&\leq & \| \langle \xi_n - \xi, \phi(a) \xi_n \rangle_E\| + \|\langle \xi, \phi(a) (\xi_n - \xi) \rangle_E \| \nonumber\\
&\leq& \| \xi_n - \xi\| \| \xi_n\| \| a \| + \| \xi_n - \xi\| \| \xi\| \| a \| \nonumber\\
&\to & 0
\end{eqnarray}
for all $a\in A$. So the c.p.~maps $\langle \xi_n , \phi(-) \xi_n \rangle_E \in \mathscr C$ converge point-norm to $\langle \xi, \phi(-) \xi \rangle_E$. As $\mathscr C$ is point-norm closed, it follows that $\langle \xi, \phi(-) \xi \rangle_E \in \mathscr C$. 
\end{proof}

The following is an immediate consequence.

\begin{corollary}\label{c:dirsumrep}
If $\gamma_n \colon A \to \mathcal B(E_n)$ are $\ast$-homomorphisms which are weakly $X$-equivariant (respectively weakly $X$-nuclear) for $n\in \mathbb N$, then the diagonal sum
\begin{equation}
\gamma \colon A \to \mathcal B( \bigoplus_{n\in \mathbb N} E_n) , \qquad \gamma(a) ((\xi_n)_{n\in \mathbb N}) = (\gamma_n(a)\xi_n)_{n\in \mathbb N}
\end{equation}
is weakly $X$-equivariant (respectively weakly $X$-nuclear).
\end{corollary}

Recall that $\mathcal K := \mathcal K(\ell^2(\mathbb N))$. The following additional corollary will be recorded for later use. In the following, $\multialg{B\otimes \mathcal K}$ and $\mathcal B_B(\ell^2(\mathbb N) \otimes B)$ will be identified in the canonical way.

\begin{corollary}\label{c:densespan}
Let $A$ and $B$ be $X$-$C^\ast$-algebras, and let 
\begin{equation}
\phi \colon A\to \multialg{B\otimes \mathcal K} = \mathcal B_B(\ell^2(\mathbb N) \otimes B)
\end{equation}
be a $\ast$-homomorphism. Then $\phi$ is weakly $X$-equivariant in the multiplier algebra sense (i.e.~$x^\ast \phi(-) x$ is $X$-equivariant for all $x\in B\otimes \mathcal K$) if and only if $\phi$ is weakly $X$-equivariant  in the Hilbert module sense (i.e.~the c.p.~map $\langle \xi, \phi(-) \xi\rangle_{\ell^2(\mathbb N) \otimes B}$ is $X$-equivariant for every $\xi \in \ell^2(\mathbb N) \otimes B$). The same is true if one replaces ``$X$-equivariant'' with ``$X$-nuclear''.
\end{corollary}

At first sight, the conclusion of the above result may seem obvious. The main difficulty is that $x^\ast \phi(-) x$ takes values in $B\otimes \mathcal K$ while $\langle \xi, \phi(-) \xi\rangle_B$ takes values in $B$. The result will be important when considering the Cuntz pair picture of $KK$-theory.

\begin{proof}
We only do the $X$-nuclear case. By applying Lemma \ref{l:densespan} to $B\otimes \mathcal K$ as a Hilbert $(B \otimes \mathcal K)$-module, and the subset $S = \{ b \otimes e_{i,j} : b\in B,  i, j \in \mathbb N\}$, it follows that $\phi$ is weakly $X$-nuclear in the multiplier algebra sense, if and only if $(b^\ast\otimes e_{j,i}) \phi(-) (b\otimes e_{i,j})$ is $X$-nuclear for all $b\in B$ and $i,j\in \mathbb N$. As
\begin{equation}
(b^\ast\otimes e_{j,i}) \phi(-) (b\otimes e_{i,j}) = (1_{\multialg{B}}  \otimes e_{j,i}) (b^\ast \otimes e_{i,i}) \phi(-) (b \otimes e_{i,i})(1_{\multialg{B}}  \otimes e_{i,j}),
\end{equation}
it easily follows that this in turn is equivalent to 
\begin{equation}
(b^\ast \otimes e_{i,i}) \phi(-) (b\otimes e_{i,i}) = \langle (\delta_i\otimes b) , \phi(-)  (\delta_i\otimes b)\rangle_{\ell^2(\mathbb N)\otimes B}
\end{equation}
 being $X$-nuclear for all $b\in B$ and $i\in \mathbb N$, where $\delta_i \in \ell^2(\mathbb N)$ is the characteristic function on $\{i\} \subseteq \mathbb N$. By Corollary \ref{c:dirsumrep} (after identifying $\ell^2(\mathbb N)\otimes B$ and $\bigoplus_{i\in \mathbb N} B$ canonically), this is equivalent to $\phi$ being weakly $X$-nuclear in the Hilbert module sense.
\end{proof}

\begin{definition}
Let $A$ be a separable $X$-$C^\ast$-algebra, and let $B$ be a $\sigma$-unital, stable $X$-$C^\ast$-algebra. A $\ast$-homomorphism $\phi \colon A \to \multialg{B}$ is called \emph{$X$-nuclearly absorbing} if it absorbs (see Definition \ref{d:abs}) any weakly $X$-nuclear $\ast$-homomorphism $A \to \multialg{B}$.
\end{definition}

The following theorem will be used in the same way as Theorem \ref{t:fullnucabs} was used in the case of full maps.

\begin{theorem}\label{t:infrepXabs}
Let $X$ be a topological space, and let $A$ and $B$ be $X$-$C^\ast$-algebras for which $A$ is separable, exact, and lower semicontinuous, and $B$ is $\sigma$-unital and stable. Suppose that $\theta \colon A \to B$ is an $X$-full, nuclear $\ast$-homomorphism. Then any infinite repeat $\theta_\infty \colon A \to \multialg{B}$ of $\theta$ is weakly $X$-nuclear and $X$-nuclearly absorbing.
\end{theorem}
\begin{proof}
By Remark \ref{r:Bempty} one has $B(\emptyset) = 0$. Hence Lemma \ref{l:nucquotient} implies that a c.p.~map $\eta \colon A \to B$ is $X$-nuclear if and only if it is $X$-equivariant and nuclear. Thus $\theta_\infty$ is weakly $X$-nuclear by Corollary \ref{c:dirsumrep} and Remark \ref{r:infrep}.

It remains to show that $\theta_\infty$ is $X$-nuclearly absorbing, so fix a weakly $X$-nuclear $\ast$-homomorphism $\psi \colon A \to \multialg{B}$. We wish to show that $\theta_\infty$ absorbs $\psi$. By Proposition \ref{p:absnonunital} it suffices to show that $\theta$ approximately dominates $b^\ast \psi(-) b \colon A \to B$ for every $b\in B$. Fix such $b\in B$.

As $\psi$ is weakly $X$-nuclear the map $b^\ast \psi(-) b$ is $X$-nuclear and thus nuclear and $X$-equivariant by Lemma \ref{l:nucquotient}. As $\theta$ is $X$-full it follows from Lemma \ref{l:Cuaction}$(a)$ that $\mathcal I(b^\ast \psi(-) b) \leq \mathcal I(\phi)$. Finally, as both $\theta$ and $b^\ast \psi(-) b$ are nuclear, it follows from Theorem \ref{t:approxdom} that $\theta$ approximately dominates $b^\ast \psi(-) b$, thus finishing the proof.
\end{proof}


\section{Ideal-related $KK$-theory}\label{s:KK}

In this section, Kirchberg's ideal-related $KK$-theory for $X$-$C^\ast$-algebras from \cite{Kirchberg-non-simple} will be constructed from the bottom up. This will be done both in an $X$-equivariant version which generalises constructions such as Kasparov's $\mathscr RKK(X;A,B)$ \cite{Kasparov-eqKKNovikov} for $C(X)$-algebras, as well as a $X$-nuclear version which generalises Skandalis' $KK_\nuc(A,B)$ \cite{Skandalis-KKnuc}.

The reader is expected to be familiar with the basics of $KK$-theory, see \cite{Blackadar-book-K-theory} and \cite{JensenThomsen-book-KK-theory}.

The pictures of $KK$-theory that will be used for classification -- the Cuntz pair picture and the Fredholm picture -- are considered in detail in Sections \ref{ss:CuntzPairs} and \ref{ss:Fredholm} respectively.

\subsection{The basics}

All Hilbert $C^\ast$-modules are \emph{right} Hilbert $C^\ast$-modules.
Recall that a \emph{($\mathbb Z/2\mathbb Z$-)graded $C^\ast$-algebra} is a $C^\ast$-algebra $A$ together with an automorphism $\beta_A$ such that $\beta_A \circ \beta_A = \id_A$. A graded $C^\ast$-algebra $A$ is \emph{trivially graded} if $\beta_A = \id_A$.
A c.p.~map $\phi \colon A \to B$ between graded $C^\ast$-algebras is \emph{graded} if $\phi \circ \beta_A = \beta_B \circ \phi$.
A (two-sided, closed) ideal $I$ in a graded $C^\ast$-algebra $A$ is called \emph{graded} if $\beta_A(I) = I$.

For graded $C^\ast$-algebras $A$ and $B$, a \emph{Kasparov $A$-$B$-module} is a triple $(E,\phi, F)$ where $E$ is a countably generated, graded Hilbert $B$-module, 
\begin{equation}
\phi \colon A \to \mathcal B(E)
\end{equation}
is a graded $\ast$-homomorphism, and $F\in \mathcal B(E)$ is an element of odd degree, such that 
\begin{equation}
[F, \phi(A)] \subseteq \mathcal K(E),\footnote{Recall that one uses \emph{graded} commutators, so that $[a,b] = ab - (-1)^{\partial a \cdot \partial b} ba$ for homogeneous elements $a,b$.} \quad  (F^2 - 1_{\mathcal B(E)})\phi(A) \subseteq \mathcal K(E), \quad (F-F^\ast) \phi(A) \subseteq \mathcal K(E).
\end{equation}
A Kasparov module $(E,\phi, F)$ is \emph{degenerate} if
\begin{equation}
[F, \phi(A)]=\{0\} , \quad (F^2 - 1_{\mathcal B(E)})\phi(A) = \{0\}  , \quad (F-F^\ast) \phi(A) =\{0\}.
\end{equation}

\begin{definition}
Let $X$ be a topological space. A \emph{graded $X$-$C^\ast$-algebra} is a graded $C^\ast$-algebra $A$ with an action of $X$, such that $A(U)$ is a graded ideal for every $U\in \mathcal O(X)$.
\end{definition}

\emph{For the rest of this section, let $X$ be a topological space, and let $A$, $B$, and $C$ be graded $X$-$C^\ast$-algebras.}

\begin{definition}
Let $A$ and $B$ be graded $X$-$C^\ast$-algebras. A Kasparov $A$-$B$-module $(E, \phi , F)$ is said to be \emph{$X$-equivariant} (respectively \emph{$X$-nuclear}) if $\phi$ is weakly $X$-equivariant (respectively weakly $X$-nuclear), see Definition \ref{d:weaklyXnuc}.
\end{definition}

\begin{remark}[Direct sums]
By Corollary \ref{c:dirsumrep} it follows that (finite) direct sums of $X$-equivariant (respectively $X$-nuclear) Kasparov modules are again $X$-equivariant (respectively $X$-nuclear).

Also, it follows that (countable) infinite direct sums of degenerate $X$-equivariant (respectively $X$-nuclear) Kasparov modules are again (necessarily degenerate) $X$-equivariant (respectively $X$-nuclear) Kasparov modules.
\end{remark}

\begin{notation}
Given a graded $X$-$C^\ast$-algebra $B$, let $IB:= C([0,1], B)$ be the induced graded $X$-$C^\ast$-algebra with grading automorphism $\beta_{IB}(f)(t) = \beta_B(f(t))$ for $f\in IB$ and $t\in [0,1]$, and action given by 
\begin{equation}
IB(U) = C([0,1], B(U)) \qquad \textrm{for } U\in \mathcal O(X).
\end{equation}
For $t\in [0,1]$, let $\ev_t \colon IB \to B$ denote the evaluation map at $t$.
\end{notation}

For any graded Hilbert $IB$-module $E$ and any $t\in [0,1]$, there is an induced graded Hilbert $B$-module $E_t$ given by completing the pre-Hilbert module
\begin{equation}
E/\{ \xi \in E : \ev_t(\langle \xi, \xi \rangle) = 0\},
\end{equation}
with inner product given by $\langle [\xi] , [\eta] \rangle_{E_t} = \ev_t(\langle \xi , \eta\rangle_E)$. 

There is an isomorphism $E\hat \otimes_{\ev_t} B \xrightarrow \cong E_t$ given on elementary tensors by $\xi \hat \otimes b \mapsto [\xi] b$ for $\xi \in E$ and $b\in B$.

There is a $\ast$-homomorphism $(\ev_t)_\ast \colon \mathcal B(E) \to \mathcal B(E_t)$, which takes an operator $T\in \mathcal B(E)$ and maps it to the operator $[\xi] \mapsto [T\xi]$ for $\xi \in E$.

\begin{lemma}\label{l:weakhtpy}
If $E$ is a graded Hilbert $IB$-module, then a $\ast$-homomorphism $\phi \colon A \to \mathcal B(E)$ is weakly $X$-equivariant (respectively weakly $X$-nuclear), if and only if $\phi_t := (\ev_t)_\ast \circ \phi \colon A \to \mathcal B(E_t)$ is weakly $X$-equivariant (respectively weakly $X$-nuclear) for all $t\in [0,1]$.
\end{lemma}
\begin{proof}
We do the $X$-nuclear case, the $X$-equivariant case is identical.
Suppose that $\phi_t$ is weakly $X$-nuclear for each $t\in [0,1]$. For every $\xi \in E$ we get that
\begin{equation}
a\mapsto \ev_t(\langle \xi , \phi(a) \xi \rangle_E) = \langle [\xi] , \phi_t(a) [\xi] \rangle_{E_t}
\end{equation}
is $X$-nuclear for each $t\in [0,1]$. By Lemma \ref{l:XnucC(Y)}, $\langle \xi , \phi(-) \xi \rangle_E$ is $X$-nuclear.

Conversely, suppose $\phi$ is weakly $X$-nuclear and fix $t\in [0,1]$. As $\{ [\xi] : \xi \in E\} \subseteq E_t$ is dense, it suffices by Lemma \ref{l:densespan} to show that $\langle [\xi] , \phi_t(-) [\xi] \rangle_{E_t} = \ev_t(\langle \xi , \phi(-) \xi\rangle_E)$ is $X$-nuclear for all $\xi \in E$. This again follows from Lemma \ref{l:XnucC(Y)} since $\phi$ is weakly $X$-nuclear.
\end{proof}

The usual notions of homotopy, operator homotopy, etc.~also carries over to the $X$-equivariant and $X$-nuclear setting, simply by assuming that all Kasparov modules involved are $X$-equivariant or $X$-nuclear. The details are filled in here.

\begin{definition}
Let $A$ and $B$ be graded $X$-$C^\ast$-algebras, and let $\mathcal E_0 = (E_0,\phi_0, F_0)$ and $\mathcal E_1 = (E_1,\phi_1,F_1)$ be $X$-equivariant (respectively $X$-nuclear) Kasparov $A$-$B$-modules.
\begin{itemize}
\item Say that $\mathcal E_0$ and $\mathcal E_1$ are \emph{unitarily equivalent} if there is a unitary $u\in \mathcal B(E_0,E_1)$ of degree 0, such that $u^\ast \phi_1(-) u = \phi_0$ and $u^\ast F_1 u = F_0$. 

Write $\mathcal E_0 \approx_u \mathcal E_1$ when $\mathcal E_0$ and $\mathcal E_1$ are unitarily equivalent.

\item A \emph{homotopy} from $\mathcal E_0$ to $\mathcal E_1$ is an $X$-equivariant (respectively $X$-nuclear) Kasparov $A$-$IB$-module $(E_I, \phi_I, F_I)$, such that
\begin{equation}
\mathcal E_0 \approx_u ((E_I)_0, (\phi_I)_0, (F_I)_0), \qquad  \mathcal E_1 \approx_u ((E_I)_1, (\phi_I)_1, (F_I)_1).
\end{equation}

Write $\mathcal E_0 \approx_\h \mathcal E_1$ if there is a homotopy from $\mathcal E_0$ to $\mathcal E_1$.

\item Say that $\mathcal E_0$ and $\mathcal E_1$ are \emph{homotopic}, if there are $X$-equivariant (respectively $X$-nuclear) Kasparov modules $\mathcal F_1,\dots, \mathcal F_n$, such that
\begin{equation}
\mathcal E_0 \approx_\h \mathcal F_1 \approx_\h \dots \approx_\h \mathcal F_n \approx_\h \mathcal E_1.
\end{equation}

Write $\mathcal E_0 \sim_{\h} \mathcal E_1$ when $\mathcal E_0$ and $\mathcal E_1$ are homotopic.

\item Say that $\mathcal E_0$ and $\mathcal E_1$ are \emph{operator homotopic} if $E_0= E_1$, $\phi_0 = \phi_1$ and there is a norm-continuous path $[0,1] \ni t \mapsto G_t$ with $G_0=F_0$, $G_1=F_1$ and $(E_0 , \phi_0 , G_t)$ is a Kasparov module for each $t$.

\item Write $\mathcal E_0 \approx_\oh \mathcal E_1$ if there are operator homotopic modules $\mathcal F_0$ and $\mathcal F_1$ such that $\mathcal E_i \approx_u \mathcal F_i$ for $i=0,1$.

\item Write $\mathcal E_0 \sim_\oh \mathcal E_1$ if there are $X$-equivariant (respectively $X$-nuclear), degenerate Kasparov $A$-$B$-modules $\mathcal D_0$ and $\mathcal D_1$ such that
\begin{equation}
\mathcal E_0 \oplus \mathcal D_0 \approx_\oh \mathcal E_1 \oplus \mathcal D_1.
\end{equation}
\end{itemize}
\end{definition}

\begin{remark}
The notations ``$\approx_\h$'', ``$\sim_\h$'', and ``$\sim_\oh$'' do not contain the information that everything is done in the $X$-equivariant (respectively $X$-nuclear) case. This will always be implied.

The relations ``$\approx_u$'', and ``$\approx_\oh$'' do not depend on any $X$-equivariant (respectively $X$-nuclear) structure, as this additional structure is automatically preserved.
\end{remark}

\begin{remark}
One can show that $\approx_\h$ is actually transitive and that therefore $\approx_\h$ and $\sim_\h$ are the same equivalence relation. While this is claimed without proof in \cite{Kasparov-KKExt} and \cite{Blackadar-book-K-theory}, I do not know of any other reference in the classical case other than the very recent \cite[Proposition A.1]{KumjianPaskSims-gradedKtheory}. The same proof carries over to the $X$-equivariant (respectively $X$-nuclear) setting, but in order to keep this as self-contained as possible, the transitive closure $\sim_\h$ will simply be used instead.
\end{remark}

The following is a consequence of Lemma \ref{l:weakhtpy}.

\begin{corollary}\label{c:deghtpy}
Degenerate $X$-equivariant (respectively $X$-nuclear) Kasparov modules are homotopic in the $X$-equivariant (respectively $X$-nuclear) sense to the Kasparov module $(0,0,0)$.
\end{corollary}
\begin{proof}
We only do the $X$-nuclear case. Let $(E,\phi,F)$ be a degenerate $X$-nuclear Kasparov module, and let 
\begin{equation}
(C_0(0,1] \hat \otimes E, 1_{\mathcal B(C_0(0,1])} \hat \otimes \phi, 1_{\mathcal B(C_0(0,1])} \hat \otimes F)
\end{equation}
be the induced (degenerate) $A$-$IB$-module, which we should prove is $X$-nuclear. By Lemma \ref{l:weakhtpy} it suffices to show that $(1_{\mathcal B(C_0(0,1])} \hat \otimes \phi)_t$ is weakly $X$-nuclear for each $t\in [0,1]$. For $t=0$ one has $(1_{\mathcal B(C_0(0,1])} \hat \otimes \phi)_0 = 0$ and if $t\in (0,1]$ then $(1_{\mathcal B(C_0(0,1])} \hat \otimes \phi)_t = \phi$ which in both cases are weakly $X$-nuclear by assumption, so $1_{\mathcal B(C_0(0,1])}\hat \otimes \phi$ is weakly $X$-nuclear.
\end{proof}

\begin{corollary}\label{c:ohstronger}
$\sim_\oh$ is a stronger equivalence relation than $\sim_\h$ on $X$-equivariant (respectively $X$-nuclear) Kasparov modules.
\end{corollary}
\begin{proof}
By Lemma \ref{l:weakhtpy}, $\approx_\oh$ is stronger than $\sim_\h$ by considering the obvious homotopy. Hence the result follows from Corollary \ref{c:deghtpy}.
\end{proof}

Once it is shown that the Kasparov product is well-defined, it will follow as in the classical case that $\sim_\h$ and $\sim_\oh$ are the same equivalence relation on $X$-equivariant (respectively $X$-nuclear) Kasparov modules, whenever $A$ is separable and $B$ is $\sigma$-unital.

\begin{definition} 
Let $A$ and $B$ be graded $X$-$C^\ast$-algebras. 
\begin{itemize}
\item Let $KK(X; A,B)$ (respectively $KK_\nuc(X; A,B)$) denote the semigroup of $\sim_\h$-equivalence classes of $X$-equivariant (respectively $X$-nuclear) Kasparov $A$-$B$-modules. The semigroup structure comes from direct sums.

For an $X$-equivariant (respectively $X$-nuclear) Kasparov $A$-$B$-module $(E,\phi, F)$, write $[E, \phi, F]$ for its induced equivalence class.

\item If $\phi \colon A \to B$ is a graded, $X$-equivariant $\ast$-homomorphism, let
\begin{equation}
KK(X; \phi) := [ B, \phi, 0] \in KK(X; A, B).
\end{equation}
Similarly, if $\phi \colon A \to B$ is a graded, $X$-nuclear $\ast$-homomorphism, let
\begin{equation}
KK_\nuc(X; \phi) := [B , \phi , 0] \in KK_\nuc(X; A, B).
\end{equation}
\end{itemize}
\end{definition}

\begin{lemma}
$KK(X; A,B)$ and $KK_\nuc(X; A, B)$ are abelian groups.
\end{lemma}
\begin{proof}
Only $KK(X; A, B)$ will be shown to be an abelian group, the proof $KK_\nuc(X; A, B)$ is obtained by replacing the word ``$X$-equivariant'' with ``$X$-nuclear''.

$KK(X; A, B)$ is clearly an abelian semigroup, as $\mathcal E \oplus \mathcal E' \approx_u \mathcal E' \oplus \mathcal E$, and a monoid with identity element $[0,0,0]$. If $(E,\phi, F)$ is an $X$-equivariant Kasparov module then $(E^\op, \phi \circ \beta_A, -F)$\footnote{Recall that $E^\op$ is equal to $E$ as Hilbert modules but with the opposite grading, i.e.~if $\xi\in E$ is homogeneous of degree $i$ for $i\in \{0,1\}$, then $\xi$ has degree $1-i$ in $E^\op$.} is also an $X$-equivariant Kasparov module since $A(U) = \beta_A(A(U))$ for all $U\in \mathcal O(X)$. Thus
\begin{equation}
(E, \phi, F) \oplus (E^\op, \phi \circ \beta_A, -F) \approx_\oh (E\oplus E^\op , \phi \oplus (\phi \circ \beta_A), \left( \begin{array}{cc} 0 & 1_{E} \\ 1_{E} & 0 \end{array} \right) ),
\end{equation}
via the operator homotopy induced by $t \mapsto \left( \begin{array}{cc} F \cos t & \sin t \\ \sin t & - F\cos t \end{array} \right)$. The latter $X$-equivariant Kasparov module above is degenerate, so $[E^\op , \phi \circ \beta_A, -F]$ is the inverse of $[E, \phi, F]$ by Lemma \ref{c:deghtpy}.
\end{proof}

\begin{example}[The classical cases]
Considering a one-point topological space $X= \{\star\}$, and continuous $\{\star\}$-$C^\ast$-algebras $A$ and $B$, one may simply think of $A$ and $B$ as $C^\ast$-algebras without an action, see Example \ref{ex:onepointXalg}. One clearly has $KK(\{\star\}; A, B) = KK(A,B)$.

Moreover, by Remark \ref{r:Skandalis} one has $KK_\nuc(\{\star\}; A, B) = KK_\nuc(A,B)$ as defined by Skandalis in \cite{Skandalis-KKnuc}.
\end{example}

\begin{example}[$C_0(X)$-algebras]
If $X$ is a locally compact Hausdorff space, and $A$ and $B$ are continuous $X$-$C^\ast$-algebras, then one may think of $A$ and $B$ as continuous $C_0(X)$-algebras, see Example \ref{ex:C(X)alg}. It can easily be checked that $KK(X; A, B) = \mathscr RKK(X; A, B)$ as defined by Kasparov in \cite{Kasparov-eqKKNovikov}.
\end{example}

\begin{observation}\label{o:KK(X)hom}
There are canonical homomorphisms 
\begin{equation}
KK_\nuc(X; A, B) \to KK(X; A, B) \to KK(A,B)
\end{equation}
given by forgetting that the modules are $X$-nuclear or $X$-equivariant.
In particular, any element $x\in KK_\nuc(X; A, B)$ induces a homomorphism $\Gamma_0(x) \colon K_0(A) \to K_0(B)$ such that if $\phi \colon A \to B$ is an $X$-nuclear $\ast$-homo\-morphism then
\begin{equation}
\Gamma_0(KK_\nuc(X;\phi)) = \phi_0 \colon K_0(A) \to K_0(B).
\end{equation}
\end{observation}


\subsection{The Kasparov product}

The \emph{Kasparov product} of a Kasparov $A$-$B$-module $(E_1, \phi_1, F_1)$ and a $B$-$C$-module $(E_2, \phi_2, F_2)$ is a Kasparov $A$-$C$-module $(E_1 \hat \otimes_{\phi_2} E_2, \phi_1 \hat \otimes 1_{\mathcal B(E_2)} , F)$ where $F$ satisfies 
\begin{itemize}
\item[(a)] $(\phi_1(a) \hat \otimes 1_{\mathcal B(E_2)})[F_1 \otimes 1_{\mathcal B(E_2)}, F] (\phi_1(a) \hat \otimes 1_{\mathcal B(E_2)})^\ast$ is positive modulo the ``compacts'' 
$\mathcal K(E_1 \hat \otimes_{\phi_2} E_2)$ for every $a\in A$; and
\item[(b)] $[\widetilde T_\xi, F_2\oplus F] \in \mathcal K(E_2 \oplus (E_1 \hat \otimes_{\phi_2} E_2))$ for every $\xi \in E_1$ where  $\widetilde T_\xi = \left( \begin{array}{cc} 0 & (\xi \hat \otimes -)^\ast \\ \xi \hat \otimes  - & 0 \end{array}\right) \in \mathcal B(E_2 \oplus (E_1 \hat \otimes_{\phi_2} E_2))$.
\end{itemize}
Such an $F$ exists and is unique up to operator homotopy, provided $A$ is separable, and $B$ and $C$ are $\sigma$-unital, see \cite[Section 18.4]{Blackadar-book-K-theory}. Hence in the $X$-equivariant and $X$-nuclear case, all one has to do is check that $\phi_1 \hat \otimes 1_{\mathcal B(E_2)}$ is in fact weakly $X$-equivariant or weakly $X$-nuclear in order to get a well-defined Kasparov product.

\begin{lemma}\label{l:tensormodule}
Let $A$, $B$, and $C$ be graded $X$-$C^\ast$-algebras, let $E_1$ and $E_2$ be graded Hilbert $B$- and $C$-modules respectively, and let $\phi_1 \colon A \to \mathcal B(E_1)$ and $\phi_2 \colon B \to \mathcal B(E_2)$ be graded, weakly $X$-equivariant representations. Then
\begin{equation}
\phi_1 \hat \otimes 1_{\mathcal B(E_2)} \colon A \to \mathcal B(E_1 \hat \otimes_{\phi_2} E_2)
\end{equation}
is weakly $X$-equivariant. Moreover, if one of $\phi_1$ and $\phi_2$ is weakly $X$-nuclear, then $\phi_1 \hat \otimes 1_{\mathcal B(E_2)}$ is weakly $X$-nuclear.
\end{lemma}
\begin{proof}
As $\overline{\mathrm{span}} \{ \xi_1 \otimes \xi_2 : \xi_1\in E_1, \xi_2 \in E_2\} = E_1 \hat \otimes_{\phi_2} E_2$,  Lemma \ref{l:densespan} implies that it suffices to show that the map
\begin{equation}
a \mapsto \langle \xi_1 \otimes \xi_2 , (\phi_1 \hat \otimes 1_{\mathcal B(E_2)})(a) (\xi_1 \otimes \xi_2) \rangle_{E_1 \hat \otimes_{\phi_2} E_2}
\end{equation}
is $X$-equivariant, or $X$-nuclear if one of $\phi_1$ and $\phi_2$ is weakly $X$-nuclear, for all $\xi_1 \in E_1$ and $\xi_2 \in E_2$. Fix $\xi_i \in E_i$ and let $\psi_i = \langle \xi_i, \phi_i(-) \xi_i\rangle_{E_i}$ for $i=1,2$. Then $\psi_i$ is $X$-equivariant by weak $X$-equivariance of $\phi_i$, and $\psi_i$ is $X$-nuclear if $\phi_i$ is weakly $X$-nuclear. One has
\begin{eqnarray}
&& \langle \xi_1 \otimes \xi_2 , (\phi_1 \hat \otimes 1_{\mathcal B(E_2)})(a) (\xi_1 \otimes \xi_2) \rangle_{E_1 \hat \otimes_{\phi_2} E_2} \nonumber\\
&=& \langle \xi_2, \phi_2(\langle \xi_1, \phi_1(a) \xi_1\rangle_{E_1}) \xi_2 \rangle_{E_2} \nonumber\\
&=& \psi_2 (\psi_1 ( a)).
\end{eqnarray}
As the composition of two $X$-equivariant maps is clearly $X$-equivariant, and the composition of an $X$-equivariant map and an $X$-nuclear map is $X$-nuclear, see Observation \ref{o:nuccomp}, the result follows.
\end{proof}

In the proof of the above lemma, it was important that one could check weak $X$-nuclearity only on elementary tensors. This relied on the non-trivial fact Lemma \ref{l:nucsum}; that the set of $X$-nuclear maps is hereditary, see also Remark \ref{r:Xnuccone}.


\begin{proposition}\label{p:Kasparovprod}
Let $\mathcal E_1 = (E_1, \phi_1, F_1)$ be an $X$-equivariant Kasparov $A$-$B$-module, and let $\mathcal E_2 = (E_2,\phi_2,F_2)$ be an $X$-equivariant Kasparov $B$-$C$-module. If $A$ is separable, and $B$ and $C$ are $\sigma$-unital, then there is a Kasparov product $\mathcal E_{12} = (E_1 \hat \otimes_{\phi_2} E_2, \phi_1 \hat \otimes 1_{\mathcal B(E_2)}, F)$ of $\mathcal E_1$ and $\mathcal E_2$, unique up to operator homotopy, and every such Kasparov module is $X$-equivariant.

Moreover, if one of $\mathcal E_1$ and $\mathcal E_2$ is $X$-nuclear, then so is $\mathcal E_{12}$.
\end{proposition}
\begin{proof}
This follows immediately from Lemma \ref{l:tensormodule} and \cite[Theorem 2.2.8]{JensenThomsen-book-KK-theory}.
\end{proof}

\begin{lemma}\label{l:prodoh}
Let $A,B$, and $C$ be graded $X$-$C^\ast$-algebras for which $A$ is separable, and $B$ and $C$ are $\sigma$-unital, let $\mathcal E_1$ and $\mathcal E_1'$ be $X$-equivariant Kasparov $A$-$B$-modules, and let $\mathcal E_2$ and $\mathcal E_2'$ be $X$-equivariant Kasparov $B$-$C$-modules. Let $\mathcal E_{12}$ be a Kasparov product of $\mathcal E_1$ and $\mathcal E_2$, and let $\mathcal E_{12}'$ be a Kasparov product of $\mathcal E_1'$ and $\mathcal E_2'$. Suppose that $\mathcal E_1 \sim_\oh \mathcal E_1'$ and $\mathcal E_2 \sim_\oh \mathcal E_2'$ (both in the $X$-equivariant sense). The following hold.
\begin{itemize}
\item[(1)] $\mathcal E_{12} \sim_\oh \mathcal E_{12}'$ (in the $X$-equivariant sense).
\item[(2)] If $i\in \{1,2\}$ is such that $\mathcal E_i$ and $\mathcal E_i'$ are $X$-nuclear and $\mathcal E_i \sim_\oh \mathcal E_i'$ (in the $X$-nuclear sense), then $\mathcal E_{12} \sim_\oh \mathcal E_{12}'$ (in the $X$-nuclear sense).
\end{itemize}
\end{lemma}
\begin{proof}
This follows immediately from Proposition \ref{p:Kasparovprod} and \cite[Lemmas 2.2.9--14]{JensenThomsen-book-KK-theory}.
\end{proof}

Recall that we may form a \emph{graded spatial (a.k.a.~minimal) tensor product} $A \hat \otimes C$ of graded $C^\ast$-algebras, which has grading operator $\beta_A \hat \otimes \beta_C$. Also, if $E_A$ and $E_C$ are graded Hilbert $A$- and $C$-modules respectively, we may form their \emph{exterior tensor product} $E_A \hat \otimes E_C$ which is a graded Hilbert $A \hat \otimes C$-module.

If $A$ is a graded $C^\ast$-algebra (with no given action of $X$) and $C$ is a graded $X$-$C^\ast$-algebras then $A\hat \otimes C$ is a graded $X$-$C^\ast$-algebras with the action
\begin{equation}
(A\hat \otimes C)(U) = A \hat \otimes C(U), \qquad U \in \mathcal O(X).
\end{equation}

When $A,B$, and $C$ are graded $C^\ast$-algebras, and $\mathcal E = (E,\phi, F)$ is a Kasparov $A$-$B$-module, one may form the exterior tensor product
\begin{equation}
\mathcal E  \otimes C := ( E \hat \otimes C , \phi  \hat \otimes \id_C \colon A\hat \otimes C \to \mathcal B(E \hat \otimes C), F\hat \otimes 1_{\mathcal B(C)}),
\end{equation}
which is a Kasparov $A\hat \otimes C$-$B\hat \otimes C$-module.

Clearly the external tensor product $- \otimes C$ of Kasparov modules preserves direct sums (up to unitary equivalence), operator homotopies, and degeneracy. Note that none of these conditions rely on $X$-equivariance or $X$-nuclearity.

\begin{lemma}\label{l:exttensorbasic}
Let $A$ and $B$ be graded $C^\ast$-algebras (with no given action of $X$), let $C$ be a graded $X$-$C^\ast$-algebra, and let $\mathcal E = (E, \phi, F)$ be a Kasparov $A$-$B$-module. Then $\mathcal E  \otimes C$ is an $X$-equivariant Kasparov $(A\hat \otimes C)$-$(B\hat \otimes C)$-module. In particular, if $\mathcal E_0\sim_{\oh} \mathcal E_1$ (as ordinary Kasparov $A$-$B$-modules), then $\mathcal E_0  \otimes C \sim_{\oh} \mathcal E_1  \otimes C$ in the $X$-equivariant sense.
\end{lemma}
\begin{proof}
As $\overline{\mathrm{span}} \{ \xi \otimes c : \xi \in E , c \in C \textrm{ are homogeneous}\} = E  \hat \otimes C$, Lemma \ref{l:densespan} implies that it suffices to show that
\begin{equation}
A \hat \otimes C \ni x \mapsto \langle \xi \otimes c , (\phi \hat \otimes \id_C)(x) (\xi \otimes c)\rangle_{E \hat \otimes C} \in B \hat \otimes C
\end{equation}
is $X$-equivariant for all homogeneous $\xi \in E$ and $c\in C$. Denote this map by $\phi_0$ for a given $\xi$ and $c$. Hence, given $U\in \mathcal O(X)$ we should show that $\phi_0(A \hat \otimes C(U)) \subseteq B \hat \otimes C(U)$. Let $a\in A$ and $d\in C(U)$ be homogeneous. Then
\begin{equation}
\phi_0(a \otimes d) = \langle \xi \otimes c , (\phi(a) \otimes d) (\xi \otimes c)\rangle_{E \hat \otimes C} = \pm \langle \xi , \phi(a) \xi\rangle_E  \otimes c^\ast dc
\end{equation}
where the sign depends on the degrees of $\xi, a,c$, and $d$ (see \cite[14.4.4]{Blackadar-book-K-theory}, although the exact computation is not needed). No matter the sign we have $\phi_0(a\otimes d) \in B\hat \otimes C(U)$. As $A \hat \otimes C(U)$ is densely spanned by its elementary tensors of homogeneous elements, it follows that $\phi_0(A \hat \otimes C(U)) \subseteq B \hat \otimes C(U)$. Hence $\mathcal E  \otimes C$ is $X$-equivariant.

The ``in particular'' part follows immediately since $- \otimes C$ preserves direct sums, operator homotopies, and takes degenerate modules to degenerate modules which are $X$-equivariant by what we already proved.
\end{proof}

The following holds as in the classical case.

\begin{proposition}\label{p:ohvsh}
Let $A$ and $B$ be graded $X$-$C^\ast$-algebras with $A$ separable and $B$ $\sigma$-unital. Then the equivalence relations $\sim_\h$ and $\sim_\oh$ agree on $X$-equivariant (respectively $X$-nuclear) Kasparov $A$-$B$-modules.
\end{proposition}
\begin{proof}
As $\sim_\oh$ is stronger than $\sim_\h$ by Corollary \ref{c:ohstronger} we only need to prove the converse. We only do the $X$-nuclear case.
Let $\mathcal E_0$ and $\mathcal E_1$ be $X$-nuclear Kasparov modules such that $\mathcal E_0 \sim_\h \mathcal E_1$. We should prove that $\mathcal E_0 \sim_\oh \mathcal E_1$. As $\sim_\h$ is the transitive closure of $\approx_\h$, and as $\sim_\oh$ is an equivalence relation, we may assume without loss of generality that $\mathcal E_0 \approx_\h \mathcal E_1$. Fix $\mathcal E$ an $X$-nuclear Kasparov module which is a homotopy from $\mathcal E_0$ to $\mathcal E_1$.

By \cite[Lemma 18.5.1]{Blackadar-book-K-theory}, $(C[0,1], \ev_0, 0) \sim_\oh (C[0,1], \ev_1, 0)$ in the classical sense as Kasparov $C[0,1]$-$\mathbb C$-modules. As $(C[0,1], \ev_i , 0) \otimes B$ and $(IB, \ev_i, 0)$ are unitarily equivalent for $i=0,1$, Lemma \ref{l:exttensorbasic} implies that $(IB, \ev_0,0) \sim_\oh (IB, \ev_1, 0)$ in the $X$-equivariant sense. As $\mathcal E_i$ is a Kasparov product of $\mathcal E$ and $(IB, \ev_i, 0)$ for $i=0,1$, it follows from Lemma \ref{l:prodoh} that $\mathcal E_0 \sim_{\oh} \mathcal E_1$ in the $X$-nuclear sense.
\end{proof}

\begin{theorem}[The Kasparov product]\label{t:Kaspprod}
Let $X$ be a topological space,  and let $A$, $B$, and $C$ be separable, graded $X$-$C^\ast$-algebras. The Kasparov product induces bilinear products
\begin{equation}
\begin{array}{rrclcl}
\circ \colon & KK(X; B, C) &\otimes& KK(X; A, B) & \to & KK(X; A, C) \\
\circ \colon & KK_\nuc(X; B, C) &\otimes& KK(X; A, B) & \to & KK_\nuc(X; A, C) \\
\circ \colon & KK(X; B, C) &\otimes& KK_\nuc (X; A, B) & \to & KK_\nuc(X; A, C) \\
\circ \colon & KK_\nuc(X; B, C) &\otimes& KK_\nuc(X; A, B) & \to & KK_\nuc(X; A, C).
\end{array}
\end{equation}
These bilinear products are all associative in the obvious sense.
\end{theorem}
\begin{proof}
This is an immediate consequence of Lemma \ref{l:prodoh}, Proposition \ref{p:ohvsh}, and \cite[Theorem 18.6.1]{Blackadar-book-K-theory}\footnote{The proof of this result is exactly showing associativity of the Kasparov product up to operator homotopy.} for associativity.
\end{proof}

At first it might seem weird that taking the product of a $KK(X)$-element and a $KK_\nuc(X)$-element produces a $KK_\nuc(X)$-element. This is essentially because the composition of two c.p.~maps, one of which is nuclear, is again a nuclear c.p.~map.

\begin{remark}
As in the classical case it follows that
\begin{equation}
KK(X; \psi) \circ KK(X; \phi) = KK(X; \psi \circ \phi)
\end{equation}
for graded, $X$-equivariant $\ast$-homomorphisms. If one of these $\ast$-homo\-morph\-isms is $X$-nuclear, say $\phi$, then
\begin{equation}
KK(X; \psi) \circ KK_\nuc(X; \phi) = KK_\nuc(X; \psi \circ \phi),
\end{equation}
and similarly if $\psi$ is $X$-nuclear.
\end{remark}

\begin{lemma}\label{l:productid}
Let $A$ and $B$ be graded $X$-$C^\ast$-algebras for which $A$ is separable and $B$ is $\sigma$-unital. Let $\mathcal E$ be an $X$-equivariant (respectively $X$-nuclear) Kasparov $A$-$B$-module. Then the Kasparov product of $\mathcal E$ with either the $X$-equivariant $A$-$A$-module $(A,\id_A,0)$ or the $X$-equivariant $B$-$B$-module $(B,\id_B,0)$, is $X$-equivariantly (respectively $X$-nuclearly) homotopic to $\mathcal E$.
\end{lemma}
\begin{proof}
The case with $(B,\id_B,0)$ is trivial since $\mathcal E$ is itself a Kasparov product of $(B,\id_B,0)$ and $\mathcal E$. For the other product, let $\mathcal E = (E, \phi, F)$. In (the proof of) \cite[Proposition 18.3.6]{Blackadar-book-K-theory}, a Kasparov $A$-$IB$-module is constructed of the form $\overline{\mathcal E} = (\overline{E_0}, \psi, G)$ such that $\overline{\mathcal E}_0$ is operator homotopic to $\mathcal E$ where $\overline{E_0} = \{ f\in C([0,1], E) : f(1) \in \overline{\phi(A)E}\}$ and $\psi$ is the map which is point-wise $\phi$. By Lemma \ref{l:weakhtpy} it follows that $\overline{\mathcal E}$ is $X$-equivariant (respectively $X$-nuclear), and thus by Lemma \ref{l:prodoh} we may replace $\mathcal E$ with $\overline{\mathcal E}_1$ and therefore assume without loss of generality that $\overline{\phi(A)E} = E$. But in this case $\mathcal E$ is itself a Kasparov product of $\mathcal E$ and $(A,\id_A,0)$.
\end{proof}

As a technical device which will be used later, the following result about full, hereditary $C^\ast$-subalgebras is proved.

\begin{proposition}\label{p:KKXfullher}
Let $A$ be a separable $X$-$C^\ast$-algebra, and let $B_0$ and $B$ be $\sigma$-unital $X$-$C^\ast$-algebras for which $B_0 \subseteq B$ is a full, hereditary $C^\ast$-subalgebra. If $B_0(U)$ is full in $B(U)$ for all $U\in \mathcal O(X)$ then the inclusion $\iota \colon B_0 \hookrightarrow B$ induces isomorphisms
\begin{equation}
KK(X; A, B_0) \xrightarrow \cong KK(X; A, B), \quad KK_\nuc(X; A, B_0) \xrightarrow \cong KK_\nuc(X; A, B).
\end{equation}
\end{proposition}
\begin{proof}
We do the $X$-nuclear case; the $X$-equivariant case is identical. Let $\mathcal E = (\overline{B B_0}, \id_B \colon B \to B = \mathcal K_{B_0}(\overline{B B_0}), 0)$ be the induced Kasparov $B$-$B_0$-module. Taking the Kasparov product with $\mathcal E$ induces a homomorphism 
\begin{equation}
[\mathcal E] \circ - \colon KK_\nuc(X; A ,B) \to KK_\nuc(X; A , B_0).
\end{equation}
As the Kasparov product is associative up to $\approx_\oh$, even when only the first $C^\ast$-algebra is separable assuming the relevant Kasparov products exist (see \cite[Lemma 22]{Skandalis-RemarksKK}), it suffices by Lemma \ref{l:productid} to show that $(B_0,\id_{B_0},0)$ and $(B,\id_B,0)$ are Kasparov products of $\mathcal E$ and $(\iota) = (B, \iota \colon B_0 \to B = \mathcal K_B(B), 0)$ up to $X$-equivariant homotopy. Clearly $(B,\id_B,0)$ is a Kasparov product $(\iota) \circ \mathcal E$. The other Kasparov product, $\mathcal E \circ (\iota)$, also exists and the canonical choice is unitarily equivalent to $(\overline{B B_0} , \iota \colon B_0 \to B = \mathcal K_{B_0}(\overline{B B_0}), 0)$. Let $E= \{ f\in C([0,1], \overline{BB_0}) : f(0) \in B_0\}$ as a Hilbert $IB_0$-module. Then the Kasparov $B_0$-$IB_0$-module $(E, B_0 \to \mathcal K(E), 0)$ defines an $X$-equivariant homotopy from $(\overline{B B_0} , \iota \colon B_0 \to B = \mathcal K_{B_0}(\overline{B B_0}), 0)$ to $(B_0, \id_{B_0}, 0)$.
\end{proof}

\begin{corollary}\label{c:KKXstable}
If $p\in \mathcal K$ is a rank one projection then the embedding $\id_B \otimes p \colon B \hookrightarrow B\otimes \mathcal K$ induces an isomorphism
\begin{equation}
KK(X; A, B) \xrightarrow \cong KK(X; A, B\otimes \mathcal K),
\end{equation}
and similarly for $KK_\nuc(X)$.
\end{corollary}

\begin{corollary}\label{c:KKXasMvN}
If $\phi, \psi \colon A \to B$ are $X$-equivariant (respectively $X$-nuc\-lear) $\ast$-homo\-morphisms and $\phi \sim_\asMvN \psi$, then
\begin{equation}
KK(X; \phi ) = KK(X; \psi) \qquad (\textrm{respectively } KK_\nuc(X; \phi) = KK_\nuc(X; \psi)).
\end{equation}
\end{corollary}
\begin{proof}
We only do the $X$-nuclear version. By Proposition \ref{p:MvNeq}, there is a norm-continuous unitary path $(u_t)_{t\in [0,1)}$ in $\multialg{B\otimes \mathcal K}$ such that $\Ad u_t \circ (\phi\otimes e_{1,1}) \xrightarrow{t\to 1} \psi \otimes e_{1,1}$ in the point-norm topology. As the unitary group of $\multialg{B\otimes \mathcal K}$ is connected by \cite{CuntzHigson-Kuipersthm}, we may assume that $u_0 = 1$. Let $\Phi \colon A \to C([0,1], B \otimes \mathcal K)$ be given by $\Phi(a)(t) = \Ad u_t \circ (\phi \otimes e_{1,1})$ for $t\in [0,1)$ and $\Phi(a)(1) = \psi \otimes e_{1,1}$. By Lemma \ref{l:XnucC(Y)}, $\Phi$ is $X$-nuclear and thus $KK_\nuc(X; \phi\otimes e_{1,1}) = KK_\nuc(X; \psi \otimes e_{1,1})$. By Corollary \ref{c:KKXstable}, it follows that $KK_\nuc(X; \phi) = KK_\nuc(X; \psi)$.
\end{proof}


\subsection{The Cuntz pair picture}\label{ss:CuntzPairs}

As in Subsection \ref{ss:KKprel}, two pictures of $KK$-theory will be considered: the Cuntz pair picture and the Fredholm picture. For these pictures to make sense, the $C^\ast$-algebras $A$ and $B$ will always be assumed to be trivially graded.

Since $\multialg{B \otimes \mathcal K} \cong \mathcal B_B(\ell^2(\mathbb N)\otimes B)$ canonically, the two will be identified without mentioning.

Recall that an \emph{($A$-$B$-)Cuntz pair} $(\psi_0, \psi_1)$ consists of a pair of $\ast$-homo\-morphisms $\psi_0,\psi_1 \colon A \to \multialg{B\otimes \mathcal K}$ for which $\psi_0(a) - \psi_1(a) \in B\otimes \mathcal K$ for all $a\in A$. 

A \emph{homotopy} of two $A$-$B$-Cuntz pairs $(\phi_0, \phi_1)$ and $(\psi_0, \psi_1)$, is a family $(\eta_0^{(s)}, \eta_1^{(s)})_{s\in [0,1]}$ of Cuntz pairs, such that 
\begin{equation}\label{eq:CPhtpy}
(\eta_0^{(0)}, \eta_1^{(0)}) = (\phi_0, \phi_1), \qquad (\eta_0^{(1)} , \eta_1^{(1)}) = (\psi_0, \psi_1),
\end{equation}
$s \mapsto \eta_i^{(s)}(a)$ is strictly continuous for $i=0,1$ and $a\in A$, and $s\mapsto \eta_0^{(s)}(a) - \eta_1^{(s)}(a)$ is norm-continuous for any $a\in A$.\footnote{The point is that $\multialg{I B \otimes \mathcal K}$ is canonically isomorphic to the set of strictly continuous, bounded functions from $[0,1]$ to $\multialg{B\otimes \mathcal K}$, and hence a homotopy of $A$-$B$-Cuntz pairs is the same as an $A$-$IB$-Cuntz pair. Here, as usual, $IB = C([0,1], B)$.}

\begin{definition}
The Cuntz pair $(\psi_0,\psi_1)$ is called \emph{weakly $X$-equivariant} (respectively weakly $X$-nuclear), if both $\ast$-homomorphisms $\psi_0$ and $\psi_1$ are weakly $X$-equivariant (respectively weakly $X$-nuclear).

Weakly $X$-equivariant (respectively weakly $X$-nuclear) Cuntz pairs $(\phi_0, \phi_1)$ and $(\psi_0, \psi_1)$ are said to be \emph{homotopic} if there is a homotopy $(\eta_0^{(s)}, \eta_1^{(s)})_{s\in [0,1]}$ of Cuntz pairs satisfying \eqref{eq:CPhtpy}, and for which $\eta_i^{(s)}$ is weakly $X$-equivariant (resp. weakly $X$-nuclear) for $i=0,1$ and $s\in [0,1]$.
\end{definition}

One can form sums of weakly $X$-equivariant/nuclear Cuntz pairs by 
\begin{equation}
(\phi_0, \phi_1) \oplus_{s_1,s_2} (\psi_0, \psi_1) = (\phi_0 \oplus_{s_1,s_2} \psi_0, \phi_1 \oplus_{s_1,s_2} \psi_1),
\end{equation}
 where $s_1,s_2\in \multialg{B\otimes \mathcal K}$ are $\mathcal O_2$-isometries. As any two Cuntz sums are unitarily equivalent, and as the unitary group of $\multialg{B\otimes \mathcal K}$ is path-connected, sums of weakly nuclear Cuntz pairs are unique up to homotopy.

The set of homotopy classes of weakly $X$-equivariant (respectively $X$-nuclear) $A$-$B$-Cuntz pairs is an abelian group with the addition defined by Cuntz sums as above. Denote the equivalence class of $(\phi_0,\phi_1)$ by $[\phi_0, \phi_1]$.

Let $\mathcal H_B := \ell^2(\mathbb N) \otimes B$ and identify $\multialg{B\otimes \mathcal K}$ with $\mathcal B(\mathcal H_B)$ in the standard way. Moreover, recall that $\mathcal H_B$ has the grading where every element has degree $0$, and $\mathcal H_B^\op$ is equal to $\mathcal H_B$ as Hilbert $B$-modules, but every element in $\mathcal H_B^\op$ has degree $1$.

\begin{proposition}[The Cuntz pair picture]\label{p:KKCP}
Let $X$ be a topological space, and let $A$ and $B$ be $X$-$C^\ast$-algebras for which $A$ is separable and $B$ is $\sigma$-unital. The assignment
\begin{equation}\label{eq:CPmodule}
(\phi_0,\phi_1) \mapsto \mathcal E_{(\phi_0,\phi_1)} := \left( \mathcal H_B \oplus \mathcal H_B^\op, \phi_0 \oplus \phi_1, \left( \begin{array}{cc} 0 & 1_{\mathcal H_B} \\ 1_{\mathcal H_B} & 0 \end{array} \right) \right)
\end{equation}
from the collection of $A$-$B$-Cuntz pairs to Kasparov $A$-$B$-modules
induces an isomorphism between the group of homotopy classes of weakly $X$-equivariant (respectively weakly $X$-nuclear) $A$-$B$-Cuntz pairs, and the Kasparov group $KK(X; A, B)$ (respectively $KK_\nuc(X; A, B)$).
\end{proposition}
\begin{proof}
As the proof very close to the similar proof in the classical case, the proof will only be sketched, and only in the $X$-nuclear case.

By Corollary \ref{c:densespan},  $\mathcal E_{(\phi_0,\phi_1)}$ is $X$-nuclear whenever $(\phi_0,\phi_1)$ is weakly $X$-nuclear. A homotopy of weakly $X$-nuclear $A$-$B$-Cuntz pairs induces an $A$-$IB$-Cuntz pair which is weakly $X$-nuclear by Lemma \ref{l:XnucC(Y)}. So by applying the map \eqref{eq:CPmodule} one obtains a homotopy of $X$-nuclear Kasparov $A$-$B$-modules, which implies that the map $[\phi_0, \phi_1] \mapsto [\mathcal E_{(\phi_0,\phi_1)}]$ is well-defined. It is routine to show that it takes Cuntz sums to diagonal sums, and clearly $[0,0] \mapsto [0]$, so $[\phi_0,\phi_1] \mapsto [\mathcal E_{(\phi_0,\phi_1)}]$ is a well-defined group homomorphism into $KK_\nuc(X; A, B)$.

Arguing exactly as \cite[Section 17.6]{Blackadar-book-K-theory} it follows that the map $[\phi_0, \phi_1] \mapsto [\mathcal E_{(\phi_0,\phi_1)}]$ is surjective. Similarly, by applying these arguments to an $X$-nuclear homotopy from $\mathcal E_{(\phi_0, \phi_1)}$ to $\mathcal E_{(\psi_0,\psi_1)}$, one may lift such a homotopy of $X$-nuclear Kasparov modules, to a homotopy of weakly $X$-nuclear Cuntz pairs. The endpoints of this homotopy are unitarily equivalent to the weakly $X$-nuclear Cuntz pairs $(\phi_0, \phi_1) \oplus_{s_1,s_2} (0,0)$ and $(\psi_0,\psi_1)\oplus_{s_1,s_2} (0,0)$ respectively. As these are homotopic to $(\phi_0,\phi_1)$ and $(\psi_0,\psi_1)$ respectively, it follows that $[\phi_0,\phi_1] \mapsto [\mathcal E_{(\phi_0,\phi_1)}]$ is injective and thus induces a group isomorphism.
\end{proof}


\subsection{The Fredholm picture}\label{ss:Fredholm}

In this  subsection, a Fredholm type picture of $KK(X)$ and $KK_\nuc(X)$ is obtained, cf.~\cite[Section 2.1]{Higson-charKK}.

Let $A$ be a separable $X$-$C^\ast$-algebra, $B$ be a $\sigma$-unital $X$-$C^\ast$-algebra (these are considered as trivially graded). An \emph{($A$-$B$-)cycle} is a triple 
\begin{equation}
(\phi_0 \colon A \to \mathcal B(E_0), \phi_1 \colon A \to \mathcal B(E_1), u),
\end{equation}
usually just written $(\phi_0, \phi_1, u)$, where $E_0$ and $E_1$ are countably generated Hilbert $B$-modules, $\phi_0$ and $\phi_1$ are $\ast$-homomorphisms, and $u\in \mathcal B(E_0,E_1)$ satisfies
\begin{equation}\label{eq:Frintertwiner}
u \phi_0(a) - \phi_1(a) u \in \mathcal K(E_0, E_1)
\end{equation}
\begin{equation}\label{eq:Fredholmunitary}
(u^\ast u - 1_{\mathcal B(E_0)})\phi_0(a) \in \mathcal K(E_0) , \quad (uu^\ast - 1_{\mathcal B(E_1)}) \phi_1(a) \in \mathcal K(E_1)
\end{equation}
for all $a\in A$. A cycle $(\phi_0,\phi_1,u)$ is called \emph{degenerate} if
\begin{equation}\label{eq:Fredholmdeg}
u\phi_0(a) - \phi_1(a) u = 0, \quad (u^\ast u-1_{\mathcal B(E_0)}) \phi_0(a) =0, \quad (uu^\ast - 1_{\mathcal B(E_1)}) \phi_1(a) =0,
\end{equation}
for all $a\in A$. Two cycles
\begin{equation}\label{eq:twoCcycles}
(\phi_0 \colon A \to \mathcal B(E_0), \phi_1 \colon A \to \mathcal B(E_1), u), \; \; (\phi_0' \colon A \to \mathcal B(E_0'), \phi_1' \colon A \to \mathcal B(E_1'), u'),
\end{equation}
are \emph{unitarily equivalent} if there are unitaries $v_i \in \mathcal B(E_i, E_i')$ for $i=0,1$, such that
\begin{equation}
\phi_0(a) v_0 = v_0 \phi_0'(a), \quad \phi_1(a) v_1 = v_1 \phi_1'(a) , \quad v_1 u = u' v_0.
\end{equation}
for all $a\in A$. 

Say that two cycles as in \eqref{eq:twoCcycles} are \emph{operator homotopic} if $\phi_i = \phi_i'$ (and thus also $E_i = E_i'$), for $i=0,1$, and if there is a norm continuous path $[0,1] \ni t \mapsto v_t \in \mathcal B(E_0,E_1)$ such that $v_0 = v$, $v_1 = v'$, and each $(\phi_0,\phi_1, v_t)$ is a cycle.

Addition of two cycles as in \eqref{eq:twoCcycles} is defined by direct sums as
\begin{equation}\label{eq:dirsumcycles}
(\phi_0 \oplus \phi_0' \colon A \to \mathcal B(E_0\oplus E_0'), \phi_1 \oplus \phi_1' \colon A \to \mathcal B(E_1\oplus E_1'), u \oplus u').
\end{equation}

\begin{definition}
Say that an $A$-$B$-cycle $(\phi_0,\phi_1,v)$ is \emph{weakly $X$-equivariant} (respectively \emph{weakly $X$-nuclear}) if $\phi_0$ and $\phi_1$ are weakly $X$-equivariant (respectively weakly $X$-nuclear).

Let $\sim_\oh$ denote the equivalence relation on weakly $X$-equivariant (resp. weakly $X$-nuclear) $A$-$B$-cycles generated by unitary equivalence, operator homotopy, and addition of degenerate, weakly $X$-equivariant (respectively weakly $X$-nuclear) $A$-$B$-cycles.
\end{definition}

The set of $\sim_\oh$-equivalence classes of weakly $X$-equivariant (respectively $X$-nuclear) $A$-$B$-cycles is an abelian group where addition is as in \eqref{eq:dirsumcycles}.

\begin{proposition}[The Fredholm picture]\label{p:KKFredholm}
Let $X$ be a topological space, and let $A$ and $B$ be $X$-$C^\ast$-algebras for which $A$ is separable and $B$ is $\sigma$-unital. The assignment
\begin{equation}\label{eq:Fredholmmodule}
(\phi_0,\phi_1,u) \mapsto \mathcal E_{(\phi_0,\phi_1,u)} := \left(E_0 \oplus E_1^{\op} , \phi_0 \oplus \phi_1 , \left( \begin{array}{cc} 0 & u^\ast \\ u & 0 \end{array} \right) \right)
\end{equation}
between the collections of $A$-$B$-cycle and Kasparov $A$-$B$-modules
induces an isomorphism between the group of $\sim_\oh$-equivalence classes of weakly $X$-equi\-variant (respectively weakly $X$-nuclear) $A$-$B$-cycles, and the Kasparov group $KK(X; A, B)$ (respectively $KK_\nuc(X; A, B)$).
\end{proposition}
\begin{proof}
As with the proof of the Cuntz pair picture, this is very close to the similar proof in the classical case, so it will only be sketched, and only in the $X$-nuclear case.

By definition, $\mathcal E_{(\phi_0,\phi_1,u)}$ is $X$-nuclear whenever $(\phi_0,\phi_1, u)$ is weakly $X$-nuclear. As the assignment \eqref{eq:Fredholmmodule} preserves direct sums (up to unitary equivalence), unitary equivalence, operator homotopy, and degeneracy, it follows that it induces a well-defined homomorphism into $KK_\nuc(X; A, B)$.

For surjectivity, the Cuntz pair picture (Proposition \ref{p:KKCP}) implies that for every element in $KK_\nuc(X; A, B)$ there is a weakly $X$-nuclear Cuntz pair $(\phi_0,\phi_1)$ such that $\mathcal E_{(\phi_0,\phi_1)}$ represents the element. We consider $\phi_0$ and $\phi_1$ as maps to $\mathcal B(\mathcal H_B)$ where $\mathcal H_B = \ell^2(\mathbb N) \otimes B$. By Corollary \ref{c:densespan}, $(\phi_0,\phi_1,1_{\mathcal B(\mathcal H_B)})$ is a weakly $X$-nuclear cycle and $\mathcal E_{(\phi_0,\phi_1)} = \mathcal E_{(\phi_0,\phi_1,1_{\mathcal B(\mathcal H_B)})}$, so our map is surjective.

For injectivity, let $(\phi_0,\phi_1,u) $ and $(\psi_0, \psi_1, v)$ be weakly $X$-nuclear cycles for which $\mathcal E_0 := \mathcal E_{(\phi_0,\phi_1,u)}$ and $\mathcal E_1 := \mathcal E_{(\psi_0,\psi_1,v)}$ have the same class in $KK_\nuc(X; A, B)$. By Proposition \ref{p:ohvsh}, $\mathcal E_0 \sim_\oh \mathcal E_1$, so there are degenerate, $X$-nuclear Kasparov modules $\mathcal D_0$ and $\mathcal D_1$ such that $\mathcal E_0\oplus \mathcal D_0 \approx_\oh \mathcal E_1 \oplus \mathcal D_1$. Arguing as in \cite[Section 17.5]{Blackadar-book-K-theory},\footnote{The Fredholm picture in \cite[Section 17.5]{Blackadar-book-K-theory} is slightly different from the one given here. While it requires that $u$ in the cycle $(\phi_0,\phi_1, u)$ is a unitary, it also requires that $A$ is unital. The arguments given in \cite[Section 17.5]{Blackadar-book-K-theory} still carry through word for word, but produce an operator $u$ satisfying \eqref{eq:Fredholmunitary} instead of a unitary.} there are degenerate, $X$-nuclear Kasparov modules $\mathcal D_i'$ for $i=0,1$, such that $\mathcal D_i \oplus \mathcal D_i'$ is unitarily equivalent to a module of the form $\mathcal E_{(\eta_0^{(i)} , \eta_1^{(i)} , w^{(i)})}$ for a degenerate, weakly $X$-nuclear cycle $(\eta_0^{(i)} , \eta_1^{(i)} , w^{(i)})$. It is easy to see that the operator homotopy $\mathcal E_0\oplus \mathcal D_0 \oplus \mathcal D_0' \approx_\oh \mathcal E_1 \oplus \mathcal D_1 \oplus \mathcal D_1'$ induces an operator homotopy between
\begin{equation}
(\phi_0,\phi_1, u) \oplus (\eta_0^{(0)} , \eta_1^{(0)} , w^{(0)}) \quad \textrm{and} \quad (\psi_0,\psi_1, v) \oplus (\eta_0^{(1)} , \eta_1^{(1)} , w^{(1)}).
\end{equation}
Hence the assignment \eqref{eq:Fredholmmodule} induces an isomorphism as desired.
\end{proof}

If $(\phi_0,\phi_1)$ is an $A$-$B$-Cuntz pair, then $(\phi_0,\phi_1, 1_{\ell^2(\mathbb N)\otimes B})$ is an $A$-$B$-cycle. By Corollary \ref{c:densespan}, if $(\phi_0, \phi_1)$ is weakly $X$-equivariant or weakly $X$-nuclear, then so is $(\phi_0, \phi_1, 1_{\ell^2(\mathbb N) \otimes B})$. However, it is a very non-trivial fact (relying on homotopy and operator homotopy agreeing for Kasparov modules) that this assignment induces a well-defined map between homotopy class of Cuntz pairs and $\sim_\oh$-classes of cycles. It follows immediately from Propositions \ref{p:KKCP} and \ref{p:KKFredholm} that it is actually well-defined, and, moreover, an isomorphism.

\begin{corollary}\label{c:CPvsFredholm}
Let $X$ be a topological space, and let $A$ and $B$ be $X$-$C^\ast$-algebras for which $A$ is separable and $B$ is $\sigma$-unital. The assignment
\begin{equation}
(\phi_0, \phi_1) \mapsto (\phi_0,\phi_1, 1_{\ell^2(\mathbb N)\otimes B})
\end{equation}
between the collections of $A$-$B$-Cuntz pairs and $A$-$B$-cycles
induces an isomorphism between the group of homotopy classes of weakly $X$-equi\-variant (respectively weakly $X$-nuclear) $A$-$B$-Cuntz pairs, and the group of $\sim_\oh$-equivalence classes of weakly $X$-equi\-variant (respectively weakly $X$-nuclear) $A$-$B$-cycles.
\end{corollary}


\section{A stable uniqueness theorem}

The new proof of the Kirchberg--Phillips theorem presented in Section \ref{s:KP} required a stable uniqueness theorem due to Dadarlat and Eilers \cite[Theorem 3.10]{DadarlatEilers-asymptotic}. A similar stable uniqueness theorem for $KK_\nuc(X)$ will be used to prove the non-simple classification theorem. The goal of this section is to prove this stable uniqueness theorem. 

The proof presented here mimics the original proof, but one runs into trouble due to the original proof using several instances of unitisations. Resolving these unitisation issues becomes one of the main complications when compared to the original proof.

Consider the automorphism group of a $C^\ast$-algebra as equipped with the uniform topology.\footnote{In particular, a \emph{norm}-continuous unitary path $(u_t)$ in $\multialg{B}$ gives rise to a continuous path $(\Ad u_t)_{t\in \mathbb R_+}$ in $\Aut(B)$, whereas for a \emph{strictly} continuous unitary path $(v_t)$ in $\multialg{B}$, the path $(\Ad v_t)_{t\in \mathbb R_+}$ is not necessarily continuous! However, $(\Ad v_t)_{t\in \mathbb R_+}$ is continuous if one instead equips $\Aut(B)$ with the point-norm topology, but this will not be relevant here.} 

\begin{proposition}[{\cite[Proposition 2.15]{DadarlatEilers-asymptotic}}]\label{p:DEaut}
Let $D$ be a unital, separable $C^\ast$-algebra, and let $(\alpha_t)_{t\in \mathbb R_+}$ be a continuous path of automorphisms on $D$ with $\alpha_0 = \id_D$. Then there exists a continuous path $(v_t)_{t\in \mathbb R_+}$ of unitaries in $D$ with $v_0 = 1_D$ such that
\begin{equation}
\lim_{t\to \infty} \| \alpha_t(d) - \Ad v_t (d) \| = 0 , \qquad d\in D.
\end{equation}
\end{proposition}

The following is an extension of the above proposition, and the (unital version of the) proof is essentially contained in the proof of \cite[Proposition 3.6]{DadarlatEilers-asymptotic}. This version is easier to apply, when working with not necessarily separable and unital $C^\ast$-algebras.

\begin{corollary}\label{c:DEaut}
Let $E$ be a $C^\ast$-algebra, $A\subseteq E$ be a separable $C^\ast$-sub\-algebra, and $(\alpha_t)_{t\in \mathbb R_+}$ be a continuous path of automorphisms on $E$ with $\alpha_0 = \id_E$. Then there is a continuous path $(v_t)_{t\in \mathbb R_+}$ of unitaries in $\widetilde E$ with $v_0 = 1_{\widetilde E}$ such that
\begin{equation}
\lim_{t\to \infty} \| \alpha_t(a) - \Ad v_t (a) \| = 0 , \qquad \textrm{ for all }a\in A.
\end{equation}
\end{corollary}
\begin{proof}
If $E$ is non-unital then each $\alpha_t$ extends canonically to an automorphism $\widetilde \alpha_t$ on $\widetilde E$. For any $x\in E$ and $\mu \in \mathbb C$ we have 
\begin{eqnarray}
\| \widetilde \alpha_t(x + \mu1_{\widetilde E}) - \widetilde \alpha_s(x+\mu 1_{\widetilde E}) \| &=& \| \alpha_t(x) + \mu1_{\widetilde E} - \alpha_s(x)- \mu 1_{\widetilde E} \| \nonumber \\
&=&  \| \alpha_t(x) - \alpha_s(x)\|.
\end{eqnarray}
Moreover, if $\| x + \mu 1_{\widetilde E} \| \leq 1$, then $\| x\| \leq 2$ and $|\mu | \leq 1$. Thus it follows that $\| \widetilde \alpha_t - \widetilde \alpha_s\|_{\mathcal B(\widetilde E)} \leq 2 \| \alpha_t - \alpha_s\|_{\mathcal B(E)}$,\footnote{Here $\mathcal B(E)$ denotes the Banach algebra of bounded operators on $E$, not to be confused with the $C^\ast$-algebra of adjointable operators when $E$ is a Hilbert $C^\ast$-module. I hope this overlap of notation does not confuse the reader.} so $(\widetilde \alpha_t)_{t\in \mathbb R_+}$ is also continuous. Hence we may assume that $E$ is unital.

Let $D$ denote the unital $C^\ast$-subalgebra of $E$ generated by all subalgebras of the form
\begin{equation}
(\alpha_{t_1}^{j_1} \circ \dots \circ \alpha_{t_n}^{j_n} )(A)
\end{equation}
for all $n\in \mathbb N$, $j_i \in \mathbb Z$ and $t_i \in \mathbb R_+ \cap \mathbb Q$. Then $D$ is separable and unital, contains $A$, and each $\alpha_t$ restricts to an automorphism on $D$. Thus $(\alpha_{t}|_D)_{t\in \mathbb R_+}$ and $D$ satisfy the condition of Proposition \ref{p:DEaut}, so there exists a continuous path of unitaries $(v_t)_{t\in \mathbb R_+}$ in $D$, which are also unitaries in $E$, such that
\begin{equation}
\lim_{t\to \infty} \| \alpha_t(d) - \Ad v_t (d) \| = 0 , \qquad \textrm{ for all }d\in D.
\end{equation}
As $A\subseteq D$, this finishes the proof.
\end{proof}

Given $C^\ast$-algebras $A$ and $B$ and a $\ast$-homomorphism $\phi \colon A \to \multialg{B}$,  let $\dot \phi \colon A \to \corona{B}$ denote the induced $\ast$-homomorphism, and
\begin{equation}
D_\phi := \{ x\in \multialg{B} : [x, \phi(A)] \subseteq B\}, \qquad E_\phi := \phi(A) + B.
\end{equation}
Note that if $u\in D_\phi$ is a unitary then $\Ad u$ induces a (not necessarily inner) automorphism on $E_\phi$.
Also, one obtains a short exact sequence
\begin{equation}\label{eq:sesgamma}
0 \to B \to D_\phi \to \corona{B} \cap \dot \phi(A)' \to 0.
\end{equation}

If $\psi\colon A \to \multialg{B}$ is such that $\dot \psi = \dot \phi$, i.e.~$\phi(a) - \psi(a) \in B$ for all $a\in A$, then $D_\phi = D_{\psi}$, and $E_\phi = E_{\psi}$, so in particular, $\psi$ induces the exact same exact sequence as above.

\begin{definition}
Two $\ast$-homomorphisms $\phi, \psi \colon A \to \multialg{B}$ are \emph{properly asymptotically unitarily equivalent}, written $\phi \approxeq \psi$, if there is a continuous unitary path $(u_t)_{t\in \mathbb R_+}$ in $\widetilde B$ such that
\begin{equation}
\lim_{t\to \infty} \| u_t^\ast \phi(a) u_t - \psi(a) \| = 0, \qquad a\in A.
\end{equation}
Write $\phi \approxeq_0 \psi$ if it is possible to choose $u_0 = 1_{\widetilde B}$ above.
\end{definition}

Obviously $\phi(a) - \psi(a) \in B$ for all $a\in A$ provided $\phi \approxeq \psi$. The following is essentially proved in \cite{DadarlatEilers-asymptotic} and is how one obtains proper asymptotic unitary equivalence.

\begin{lemma}\label{l:propertrick}
Let $A$ be a separable $C^\ast$-algebra, and let $\phi, \psi \colon A \to \multialg{B}$ be $\ast$-homo\-morphisms such that $(\phi(A) + \mathbb C 1_{\multialg{B}}) \cap B = \{0\}$ and for which $\phi(a) - \psi(a) \in B$ for all $a\in A$. Suppose that $(u_t)_{t\in \mathbb R_+}$ is a norm-continuous unitary path in $D_\phi$ for which 
\begin{equation}
\lim_{t\to \infty} \| u_t^\ast \phi(a) u_t - \psi(a) \| = 0, \qquad a\in A.
\end{equation}
Then $\Ad u_0 \circ \phi \approxeq_0 \psi$. 

In particular, if $u_0 \in \widetilde B$ then $\phi \approxeq \psi$.
\end{lemma}
\begin{proof}
Let $\alpha_t = \Ad (u_0^\ast u_t) \in \Aut(E_\phi)$, and note that $\alpha_t \circ \phi$ converges point-norm to $\Ad u_0^\ast \circ \psi$. Clearly $(\alpha_t)_{t\in \mathbb R_+}$ is a continuous path in $\Aut (E_\phi)$ (with the uniform topology) with $\alpha_0 = \id_{E_\phi}$, so by Corollary \ref{c:DEaut} there is a continuous unitary path $(v_t)_{t\in \mathbb R_+}$ in $\widetilde E_\phi = E_\phi + \mathbb C 1_{\multialg{B}}$ with $v_0=1_{\multialg{B}}$ such that
\begin{equation}
\lim_{t\to \infty} \| \alpha_t(\phi(a)) - \Ad v_t(\phi(a)) \| = 0, \qquad a\in A.
\end{equation}
In particular, $\Ad v_t \circ \phi$ converges point-norm to $\Ad u_0^\ast \circ \psi$. Since $(\phi(A) + \mathbb C 1_{\multialg{B}}) \cap B = \{0\}$, there is a canonical identification $\phi(A) + \mathbb C 1_{\multialg{B}} \cong (E_\phi + \mathbb C1_{\multialg{B}}) /B$, so let $(x_t)_{t\in \mathbb R_+}$ be the unitary path in $\phi(A) + \mathbb C 1_{\multialg{B}}$ corresponding to $(v_t + B)_{t\in \mathbb R_+}$ in $(E_\phi + \mathbb C1_{\multialg{B}}) /B$. For $y_t := v_t - x_t \in B$, the path $(y_t)_{t\in \mathbb R_+}$ is clearly continuous and bounded. 
Note that
\begin{eqnarray}
\lim_{t\to \infty} \| x_t \phi(a) x_t^\ast - \phi(a) \| &=& \lim_{t\to \infty} \| v_t u_0 \psi (a) u_0^\ast v_t^\ast - \phi(a) + B\|_{(E_\phi + \mathbb C1_{\multialg{B}}) /B} \nonumber \\
&\leq& \lim_{t\to \infty} \| v_t^\ast \phi(a) v_t - \Ad u_0^\ast \circ \psi(a)\| \nonumber \\
&=& 0.
\end{eqnarray}
Letting $w_t =  x_t^\ast v_t  = 1_{\multialg{B}} +  x_t^\ast y_t$ produces a continuous unitary path $(w_t)_{t\in \mathbb R_+}$ in $\widetilde B$ for which $w_0 = 1_{\widetilde B}$. It follows that
\begin{eqnarray}
&& \| w_t^\ast \phi(a) w_t - \Ad u_0^\ast \circ \psi(a) \| \nonumber\\
 &\leq& \| v_t^\ast (x_t \phi(a) x_t^\ast - \phi(a)) v_t\| + \| v_t^\ast \phi(a) v_t - \Ad u_0^\ast \circ \psi(a)\| \nonumber\\
&\to& 0,
\end{eqnarray}
for all $a\in A$, so $\phi \approxeq_0 \Ad u_0^\ast \circ \psi$, or equivalently, $\Ad u_0 \circ \phi \approxeq_0 \psi$.
\end{proof}

It is finally time to take care of the unitisation issues mentioned at the beginning of the section.
For a topological space $X$, let $X^\dagger$ denote the (always non-Hausdorff) topological space $X \sqcup \{ \star\}$ with open sets $\mathcal O(X^\dagger) = \mathcal O(X) \sqcup \{ X^\dagger\}$. If $A$ is an $X$-$C^\ast$-algebra, then the forced unitisation $A^\dagger$ has a canonical action of $X^\dagger$ given by 
\begin{equation}
A^\dagger(U) = \left\{ \begin{array}{ll} 
A(U), & \textrm{ if }U\in \mathcal O(X) \; \; (\subseteq \mathcal O(X^\dagger)) \\
A^\dagger , & \textrm{ if } U = X^\dagger. \end{array} \right.
\end{equation}
Similarly, if $B$ is an $X$-$C^\ast$-algebra then there is a canonical action of $X^\dagger$ on $B$ which is the ordinary action for open subsets in $X \subseteq X^\dagger$, and $B(X^\dagger) = B$.

\begin{lemma}\label{l:Xnucunitise}
Let $A$ and $B$ be $X$-$C^\ast$-algebras and let $\phi \colon A^\dagger \to B$ be a c.p.~map. Then $\phi$ is $X^\dagger$-equivariant (respectively $X^\dagger$-nuclear) if and only if $\phi|_A$ is $X$-equivariant (respectively $X$-nuclear). 
\end{lemma}
\begin{proof}
Since $\phi(A^\dagger(X^\dagger)) = \phi(A^\dagger) \subseteq B = B(X^\dagger)$, it easily follows that $\phi$ is $X^\dagger$-equivariant if and only if $\phi|_A$ is $X$-equivariant. If $U\in \mathcal O(X)$, then
\begin{equation}
[\phi]_U \colon A^\dagger/A^\dagger(U) \to B/B(U), \qquad [\phi|_A]_U \colon A/A(U) \to B/B(U),
\end{equation}
the latter map being the restriction of $[\phi]_U$ to $A^\dagger/A^\dagger(U) \cong (A/A(U))^\dagger$. By Lemma \ref{l:unitisednuc} it follows that $[\phi]_U$ is nuclear if and only if $[\phi|_A]_U$ is nuclear, so $\phi$ is $X^\dagger$-nuclear if and only if $\phi|_A$ is $X$-nuclear. 
\end{proof}

\begin{proposition}\label{p:KKunitise}
Let $A$ be a separable $X$-$C^\ast$-algebra and let $B$ be a $\sigma$-unital $X$-$C^\ast$-algebra. Then the maps 
\begin{equation}
KK(X; A,B) \to KK(X^\dagger; A^\dagger, B), \quad KK_\nuc(X; A, B) \to KK_\nuc(X^\dagger ; A^\dagger , B) 
\end{equation}
given in the Cuntz pair picture by $[\phi, \psi] \mapsto [\phi^\dagger, \psi^\dagger]$ are well-defined, injective homomorphisms.
\end{proposition}
\begin{proof}
We only do the $X$-nuclear version. If $\phi$ is weakly $X$-nuclear then $\phi^\dagger$ is weakly $X^\dagger$-nuclear by Lemma \ref{l:Xnucunitise}. By unitising an $X$-nuclear homotopy between Cuntz pairs $(\phi_0,\psi_0)$ and $(\phi_1 , \psi_1)$, one thus obtains an $X^\dagger$-nuclear homotopy between $(\phi_0^\dagger, \psi_0^\dagger)$ and $(\phi_1^\dagger, \psi_1^\dagger)$. Hence the map is well-defined, and is clearly homomorphism.

As any $X^\dagger$-nuclear homotopy from $(\phi^\dagger, \psi^\dagger)$ to $(0,0)$ restricts to a $X$-nuclear homotopy from $(\phi, \psi)$ to $(0,0)$, the homomorphism is injective.
\end{proof}

If $(\phi_0,\phi_1, u)$ is an $A$-$B$-cycle then $(\phi_0^\dagger, \phi_1^\dagger, u)$ is an $A^\dagger$-$B$-cycle if and only if $u$ is a unitary modulo ``compacts''.\footnote{Since one adds the relations $u^\ast u -1_{\mathcal B(E_0)} \in \mathcal K(E_0)$ and $uu^\ast - 1_{\mathcal B(E_1)} \in \mathcal K(E_1)$ to \eqref{eq:Fredholmunitary}.} 
In particular, if $\gamma \colon A \to \multialg{B\otimes \mathcal K}$ is a $\ast$-homo\-morphism and $w\in D_\gamma$ is a unitary, then $(\gamma^\dagger, \gamma^\dagger, w)$ is an $A^\dagger$-$B$-cycle. By Corollary \ref{c:densespan} and Lemma \ref{l:Xnucunitise}, if $\gamma$ is weakly $X$-nuclear, then $(\gamma^\dagger, \gamma^\dagger , w) $ is a weakly $X^\dagger$-nuclear cycle.

\begin{lemma}\label{l:unitiseunitaries}
Suppose $\gamma \colon A \to \multialg{B\otimes \mathcal K}$ is a weakly $X$-nuclear $\ast$-homo\-morphism and that $w_1,w_2 \in D_\gamma$ are unitaries. If $[\gamma, \gamma, w_1] = [\gamma , \gamma , w_2]$ in $KK_\nuc(X; A, B)$ then $[\gamma^\dagger, \gamma^\dagger , w_1] = [\gamma^\dagger, \gamma^\dagger, w_2]$ in $KK_\nuc(X^\dagger, A^\dagger, B)$ (using the Fredholm picture, Proposition \ref{p:KKFredholm}).
\end{lemma}
\begin{proof}
The cycle $(\gamma, \gamma, w_i)$ is unitarily equivalent to $(\gamma, \Ad w_i \circ \gamma, 1_{\multialg{B\otimes \mathcal K}})$. Using Corollary \ref{c:CPvsFredholm} to view everything in the Cuntz pair picture (Proposition \ref{p:KKCP}), $[\gamma, \Ad w_1 \circ \gamma] = [\gamma , \Ad w_2 \circ \gamma]$ in $KK_\nuc(X; A, B)$. By Proposition \ref{p:KKunitise}, $[\gamma^\dagger, (\Ad w_1 \circ \gamma)^\dagger] = [\gamma^\dagger, (\Ad w_2 \circ \gamma)^\dagger]$ in $KK_\nuc(X^\dagger; A^\dagger, B)$. As $(\Ad w_i \circ \gamma)^\dagger = \Ad w_i \circ \gamma^\dagger$, and as the cycles $(\gamma^\dagger, \Ad w_i \circ \gamma^\dagger, 1_{\multialg{B\otimes \mathcal K}})$ and $(\gamma^\dagger, \gamma^\dagger, w_i)$ are unitarily equivalent, it follows from Corollary \ref{c:CPvsFredholm} that $[\gamma^\dagger, \gamma^\dagger, w_1] = [\gamma^\dagger, \gamma^\dagger, w_2]$ in $KK_\nuc(X^\dagger; A^\dagger, B)$. 
\end{proof}

The following lemma, which is an ideal-related version of \cite[Lemma 3.5]{DadarlatEilers-asymptotic}, will be needed. 

\begin{lemma}\label{l:DEoh}
Let $A$ be a separable $X$-$C^\ast$-algebra and let $B$ a $\sigma$-unital $X$-$C^\ast$-algebra. Let $\gamma \colon A \to \multialg{B\otimes \mathcal K}$ be a weakly $X$-nuclear $\ast$-homomorphism and let $w_1,w_2\in D_\gamma$ be unitaries. If $[\gamma, \gamma, w_1] = [\gamma, \gamma , w_2]$ in $KK_\nuc(X; A,B)$ then there is a unital, weakly $X^\dagger$-nuclear $\ast$-homomorphism $\Theta \colon A^\dagger \to \multialg{B\otimes \mathcal K}$ such that
\begin{equation}
(\gamma^\dagger \oplus \Theta, \gamma^\dagger \oplus \Theta , w_1 \oplus 1_{\multialg{B\otimes \mathcal K}}) \quad \textrm{and} \quad (\gamma^\dagger \oplus \Theta, \gamma^\dagger \oplus \Theta , w_2 \oplus 1_{\multialg{B\otimes \mathcal K}})
\end{equation}
are operator homotopic.
\end{lemma}
\begin{proof}
By Lemma \ref{l:unitiseunitaries}, $[\gamma^\dagger, \gamma^\dagger, w_1] = [\gamma^\dagger, \gamma^\dagger, w_2]$ in $KK_\nuc(X^\dagger; A^\dagger, B)$ (using the Fredholm picture). Hence there are degenerate, weakly $X^\dagger$-nuclear cycles 
\begin{equation}
(\theta_0^{(i)} \colon A^\dagger \to \mathcal B(E_0^{(i)}), \theta_1^{(i)} \colon A^\dagger \to \mathcal B(E_1^{(i)}), v^{(i)})
\end{equation}
 for $i=1,2$ such that
\begin{equation}\label{eq:Frohlem}
(\gamma^\dagger, \gamma^\dagger, w_1) \oplus (\theta_0^{(1)}, \theta_1^{(1)}, v^{(1)}) \quad \textrm{and} \quad (\gamma^\dagger, \gamma^\dagger, w_2) \oplus (\theta_0^{(2)}, \theta_1^{(2)}, v^{(2)})
\end{equation}
are unitarily equivalent to operator homotopic cycles. By otherwise cutting down with $\theta_j^{(i)}(1_{A^\dagger})$ all over, we may assume without loss of generality that each $\theta_j^{(i)}$ is a unital map and that each $v^{(i)}$ is a unitary for $i=1,2$ and $j=0,1$.\footnote{This only works because $(\theta_0^{(i)}, \theta_1^{(i)}, v^{(i)})$ is degenerate for $i=1,2$.}
 Define $E_j := \bigoplus_{\mathbb N} (E_j^{(1)} \oplus E_j^{(2)})$, $\Theta_j := \bigoplus_{\mathbb N} (\theta_j^{(1)} \oplus \theta_j^{(2)}) \colon A^\dagger \to \mathcal B(E_j)$, and $v = \bigoplus_{\mathbb N} (v^{(1)} \oplus v^{(2)}) \in \mathcal B(E_0,E_1)$, the latter being a unitary. As
\begin{equation}
(\theta_0^{(i)}, \theta_1^{(i)}, v^{(i)})\oplus (\Theta_0, \Theta_1 , v)  \quad \textrm{and} \quad (\Theta_0, \Theta_1 , v)
\end{equation}
are unitarily equivalent (by unitaries that permute the indices of the direct sums), it follows that
\begin{equation}
(\gamma^\dagger, \gamma^\dagger, w_1) \oplus (\Theta_0, \Theta_1 , v) \quad \textrm{and} \quad (\gamma^\dagger, \gamma^\dagger, w_2) \oplus (\Theta_0, \Theta_1 , v)
\end{equation}
are operator homotopic. Since $v$ is a unitary and $(\Theta_0, \Theta_1 , v)$ is degenerate, it follows that $\Theta_0 = \Ad v \circ \Theta_1$, and thus $(\Theta_0, \Theta_1 , v)$ is unitarily equivalent to $(\Theta_0, \Theta_0 , 1_{\mathcal B(E_0)})$. Let $\pi \colon A^\dagger \to \mathcal B(\ell^2(\mathbb N)\otimes B)$ be the unique unital $\ast$-homomorphism such that $\pi|_A = 0$. By Lemma \ref{l:Xnucunitise}, $\pi$ is weakly $X^\dagger$-nuclear. Using that $E_0 \oplus (\ell^2(\mathbb N)\otimes B) \cong \ell^2(\mathbb N)\otimes B$ by Kasparov's stabilisation theorem, let $\Theta \colon A^\dagger \to \mathcal B(\ell^2(\mathbb N)\otimes B) = \multialg{B\otimes \mathcal K}$ be a map unitarily equivalent to $\Theta_0 \oplus \pi$. As $(\Theta, \Theta, 1_{\multialg{B\otimes \mathcal K}})$ is unitarily equivalent to $(\Theta_0, \Theta_1, v)\oplus (\pi, \pi, 1_{\multialg{B\otimes \mathcal K}})$, it follows that
\begin{equation}
(\gamma^\dagger \oplus \Theta, \gamma^\dagger \oplus \Theta, w_1\oplus 1_{\multialg{B\otimes \mathcal K}}) \quad \textrm{and} \quad (\gamma^\dagger \oplus \Theta, \gamma^\dagger \oplus \Theta, w_2\oplus 1_{\multialg{B\otimes \mathcal K}})
\end{equation}
are operator homotopic.
\end{proof}

With the above ingredients, one can now prove an ideal-related version of the Dadarlat--Eilers stable uniqueness theorem from \cite{DadarlatEilers-asymptotic}.

\begin{theorem}\label{t:irDE}
Let $X$ be a topological space, let $A$ be a separable $X$-$C^\ast$-algebra, and let $B$ be a $\sigma$-unital $X$-$C^\ast$-algebra. Let $\phi, \psi \colon A \to \multialg{B \otimes \mathcal K}$ be weakly $X$-nuclear $\ast$-homo\-morphisms such that $\phi(a) - \psi(a) \in B \otimes \mathcal K$ for all $a\in A$. If $[\phi,\psi] = 0$ in $KK_\nuc(X; A, B)$ (using the Cuntz pair picture, Proposition \ref{p:KKCP}) then there exist a weakly $X$-nuclear $\ast$-homomorphism $\Theta\colon A \to \multialg{B\otimes \mathcal K}$ and a continuous path $(v_t)_{t\in \mathbb R_+}$ of unitaries in $(B \otimes \mathcal K)^\sim$ such that
\begin{equation}
\| v_t (\phi \oplus_{s_1,s_2} \Theta)(a) v_t^\ast - (\psi \oplus_{s_1,s_2} \Theta)(a) \| \to 0,
\end{equation}
for all $a\in A$. Here $s_1,s_2\in \multialg{B \otimes \mathcal K}$ are $\mathcal O_2$-isometries.
\end{theorem}
\begin{proof}
Let $\phi_\infty$ and $\psi_\infty$ be infinite repeats of $\phi$ and $\psi$, see Remark \ref{r:infrep}, and let $\sigma = \phi_\infty \oplus \psi_\infty$ be a Cuntz sum. Define $\phi_0 = \phi \oplus \sigma$ and $\psi_0 = \psi \oplus \sigma$ by Cuntz sums. Then $\psi_0 = \Ad u' \circ \phi_0$ for some unitary $u'$ which permutes the direct summands. Thus
\begin{eqnarray}
[\phi_0 , \phi_0, u'] &=& [\phi_0 , \psi_0 , 1_{\multialg{B\otimes \mathcal K}}] \nonumber\\
&=& [\phi, \psi, 1_{\multialg{B\otimes \mathcal K}}]\oplus [\sigma, \sigma ,1_{\multialg{B\otimes \mathcal K}}]  \nonumber \\
&=& 0 \nonumber\\
&=& [\phi_0, \phi_0, 1_{\multialg{B\otimes \mathcal K}}]
\end{eqnarray}
in $KK_\nuc(X; A, B)$ using the Fredholm picture, Proposition \ref{p:KKFredholm}, as well as using Corollary \ref{c:CPvsFredholm} to conclude that $[\phi,\psi,1_{\multialg{B\otimes \mathcal K}}]=0$. As $(\phi_0, \phi_0 , u')$ is a cycle, it follows that $[u', \phi_0(A)] \subseteq B\otimes \mathcal K$ by \eqref{eq:Frintertwiner}, so $u'$ is a unitary in $D_{\phi_0}$. By Lemma \ref{l:DEoh}, there is a unital, weakly $X^\dagger$-nuclear $\ast$-homomorphism $\Theta_0 \colon A^\dagger \to \multialg{B \otimes \mathcal K}$ such that
\begin{equation}
(\phi_0^\dagger \oplus \Theta_0, \phi_0^\dagger \oplus \Theta_0, u' \oplus 1_{\multialg{B\otimes \mathcal K}}) \quad \textrm{and} \quad (\phi_0^\dagger \oplus \Theta_0, \phi_0^\dagger \oplus \Theta_0, 1_{\multialg{B\otimes \mathcal K}} \oplus 1_{\multialg{B\otimes \mathcal K}})
\end{equation}
are operator homotopic. By replacing $\Theta_0$ with its infinite repeat, we may obtain the above with $\Theta_0$ being an infinite repeat. We make sense of the direct sum $\oplus$ above by a fixed Cuntz sum. Let $\Phi_0 := \phi_0^\dagger \oplus \Theta_0$ and $u= u' \oplus 1_{\multialg{B\otimes \mathcal K}}$. Let $(w_s)_{s\in [0,1]}$ be a norm-continuous path in $\multialg{B\otimes \mathcal K}$ which implements the above operator homotopy from $w_0 =1_{\multialg{B\otimes \mathcal K}}$ to $w_1 = u$. As $(\Phi_0, \Phi_0, w_s)$ is a cycle for every $s\in [0,1]$, and as $\Phi_0$ is unital, it follows from \eqref{eq:Frintertwiner} and \eqref{eq:Fredholmunitary} that the image in $\corona{B\otimes \mathcal K}$, $(\dot w_s)_{s\in [0,1]}$, is a continuous unitary path in $\corona{B\otimes \mathcal K} \cap \dot \Phi_0(A^\dagger)'$ which connects $1_{\corona{B\otimes \mathcal K}}$ and $\dot u$.\footnote{This is the subtle part where the unitality we fought for earlier in the section plays a crucial role. The arguments could had been run without any unitality assumptions, but then the unitary path would be in $\corona{B \otimes \mathcal K} \cap \dot \Phi_0(A)' / \Ann \dot \Phi_0(A)$. Considering this relative commutant, the final part of the proof would not be enough to obtain a unitary path in $(B \otimes \mathcal K)^\sim$.} Using the exact sequence \eqref{eq:sesgamma}, we may lift $(\dot w_s)_{s\in [0,1]}$ to a continuous unitary path $(u_s)_{s\in[0,1]}$ in $D_{\Phi_0}$ for which $u_1 = u$ and $u_0 \in (B\otimes \mathcal K)^\sim$.

Let $\Theta := (\sigma \oplus \Theta_0)|_A$, so that $\phi \oplus \Theta = \Phi_0|_A$. As $\psi \oplus \sigma = \psi_0 = \Ad u' \circ \phi \oplus \sigma$, it follows that $\Ad u \circ \psi \oplus \Theta = \phi \oplus \Theta$ (with respect to the same Cuntz sum $\oplus = \oplus_{s_1,s_2}$). Hence $(u_s)_{s\in [0,1)}$ is a continuous path of unitaries in $D_{\Phi_0}$ with $u_0 \in (B\otimes \mathcal K)^\sim$ such that
\begin{equation}
\lim_{s\to 1} \| u_s^\ast (\psi \oplus \Theta)(a) u_s - (\phi \oplus \Theta)(a)\| = 0.
\end{equation}
As $\phi$ and $\psi$ agree modulo $B \otimes \mathcal K$, and as $\Phi_0$ is the forced unitisation of $\phi \oplus \Theta$,  it follows that $D_{\psi \oplus \Theta} = D_{\phi \oplus \Theta} = D_{\Phi_0}$. As $\psi\oplus \Theta$ is unitarily equivalent to the infinite repeat $(\phi \oplus \psi)_\infty \oplus \Theta_0|_A$, it follows that $((\psi \oplus \Theta)(A) + 1_{\multialg{B\otimes \mathcal K}}\mathbb C ) \cap (B\otimes \mathcal K) = \{0\}$. Hence Lemma \ref{l:propertrick} implies that $\psi \oplus \Theta \approxeq \phi \oplus \Theta$.
\end{proof}

Using that infinite repeats of nuclear $X$-full $\ast$-homomorphisms are weakly $X$-nuclear and $X$-nuclearly absorbing by Theorem \ref{t:infrepXabs}, the following is a consequence of the above theorem.

\begin{theorem}\label{t:irDE2}
Let $X$ be a topological space, let $A$ be a separable, exact, lower semicontinuous $X$-$C^\ast$-algebra, and  let $B$ be a $\sigma$-unital, stable $X$-$C^\ast$-algebra. Let $\phi,\psi,\theta \colon A \to B$ be nuclear, $X$-equivariant $\ast$-homomorphisms and suppose that $\theta$ is $X$-full. If $KK_\nuc(X; \phi) = KK_\nuc(X; \psi)$ then there is a continuous unitary path $(v_t)_{t\in \mathbb R_+}$ in $\widetilde B$ such that
\begin{equation}
\lim_{t\to \infty} \| v_t ( \phi \oplus_{s_1, s_2} \theta_\infty)(a) v_t^\ast - (\psi \oplus_{s_1,s_2} \theta_\infty) (a) \| = 0
\end{equation}
for all $a\in A$. Here $s_1,s_2 \in \multialg{B}$ are $\mathcal O_2$-isometries, and $\theta_\infty$ is an infinite repeat of $\theta$.
\end{theorem}
\begin{proof}
By Theorem \ref{t:irDE} there is a weakly $X$-nuclear $\Theta \colon A \to \multialg{B}$ such that $\phi \oplus \Theta \approxeq \psi \oplus \Theta$. Hence $\phi \oplus \Theta \oplus \theta_\infty \approxeq \psi \oplus \Theta \oplus \theta_\infty$. By Theorem \ref{t:infrepXabs}, $\theta_\infty$ is weakly $X$-nuclear and $X$-nuclearly absorbing, and in particular it absorbs $\Theta$. As $\theta_\infty$ is an infinite repeat, and is thus unitarily equvalent to an infinite repeat of itself $(\theta_\infty)_\infty$, it follows from \cite[Lemmas 2.4 and 3.4]{DadarlatEilers-asymptotic} that $\phi \oplus \theta_\infty \approxeq \psi \oplus \theta_\infty$.
\end{proof}


\section{An ideal-related $\mathcal O_2$-embedding theorem}

The iconic $\mathcal O_2$-embedding theorem of Kirchberg \cite{Kirchberg-ICM} (see also \cite{KirchbergPhillips-embedding}) asserts that any separable, exact $C^\ast$-algebra admits an embedding into $\mathcal O_2$. This played a major role in the proofs of Theorems \ref{t:existsimple}, \ref{t:uniquesimple}, \ref{t:KP}, and \ref{t:KPUCT}. Similarly, an ideal-related version of this theorem will be used to prove Theorems \ref{t:irexistence}, \ref{t:iruniqueness}, and \ref{t:nonsimpleclass}.

Such an ideal-related $\mathcal O_2$-embedding theorem was also a main ingredient in Kirchberg's approach to Theorem \ref{t:nonsimpleclass}, see \cite[Hauptsatz 2.15]{Kirchberg-non-simple}, and a similar result has appeared in \cite[Theorem 6.1]{Gabe-O2class}.  The latter of these results assumed that the target $C^\ast$-algebra $B$ was separable, nuclear, and $\mathcal O_\infty$-stable, so that one could combine an application of Michael's selection theorem with a deep result of Kirchberg and Rørdam from \cite{KirchbergRordam-zero} to produce the desired map. While \cite[Theorem 6.1]{Gabe-O2class} would be strong enough to prove the main classification result, Theorem \ref{t:nonsimpleclass}, one would only be able to prove the existence result Theorem \ref{t:irexistence} under the additional assumptions that the target $B$ be separable and nuclear. To remove these assumptions, I present Theorem \ref{t:irO2}; a generalised version of the ideal-related $\mathcal O_2$-embedding theorem from \cite{Gabe-O2class}.

The proof does not require the use of Michael's selection theorem or the results of Kirchberg and Rørdam. Instead it uses a recent construction from \cite{BGSW-nucdim}, which despite its technical proof is a fairly elementary application of Voiculescu's quasidiagonality theorem \cite{Voiculescu-qdhtpy}. With that at hand, the proof from \cite{Gabe-O2class} carries over almost verbatim.

\begin{theorem}\label{t:irO2}
Let $A$ be a separable, exact $C^\ast$-algebra, let $B$ be an $\mathcal O_\infty$-stable $C^\ast$-algebra, and let $\Phi \colon \mathcal I(A) \to \mathcal I(B)$ be a $\Cu$-morphism (see Definition \ref{d:lattice}). Then there exists a nuclear, strongly $\mathcal O_2$-stable $\ast$-homomorphism $\phi \colon A \to B$ such that $\mathcal I(\phi) = \Phi$.
\end{theorem}

\begin{proof}
Arguing exactly as in the beginning of the proof of \cite[Theorem 6.1]{Gabe-O2class}, we may assume that $B$ is stable and $\mathcal O_2$-stable.\footnote{Essentially, by embedding $B \otimes \mathcal O_2 \otimes \mathcal K \hookrightarrow B\otimes \mathcal O_\infty \cong B$ and making sure that these maps do the right thing on ideals.} By \cite[Lemma 3.5]{BGSW-nucdim}, we find a $\ast$-homomorphism $\rho \colon C_0(0,1] \otimes A \to B_\infty$ such that
\begin{equation}\label{eq:IIJBinfty}
\mathcal I(\rho)(I \otimes J) = \overline{B_\infty \Phi(J) B_\infty}, \qquad I \in \mathcal I(C_0(0,1]) \setminus \{0\}, \quad J \in \mathcal I(A),
\end{equation}
and such that the completely positive order zero map $\rho(\id_{(0,1]} \otimes -)$ can be represented as a sequence of completely positive contractive maps (which would be $(\eta_n \circ \psi_n(\id_{(0,1]} \otimes -))_{n\in \mathbb N}$ in that lemma), each of which is nuclear as it factors through a finite dimensional $C^\ast$-algebra. By exactness of $A$, $\rho(\id_{(0,1]} \otimes -)$ is nuclear by \cite[Proposition 3.3]{Dadarlat-qdmorphisms}, and so $\rho$ is nuclear by \cite[Theorem 2.9]{Gabe-qdexact}. By \cite[Lemma 6.7]{Gabe-O2class}, $\rho$ is $\mathcal O_2$-stable, since we assumed $B$ is $\mathcal O_2$-stable.

Now the proof follows that of \cite[Theorem 6.1]{Gabe-O2class} almost word by word. Let $\Psi \colon C_0(0,1)\otimes A \to B_\infty$ be the restriction of $\rho$, which is nuclear, $\mathcal O_2$-stable, and for which $\mathcal I(\Psi)(I\otimes J) = \overline{B_\infty \Phi(J) B_\infty}$ for $I\neq 0$. Let $\alpha$ be an automorphism on $C_0(0,1)$ such that $C_0(0,1)\rtimes_\alpha \mathbb Z \cong C(\mathbb T)\otimes \mathcal K$, and let $\beta := \alpha \otimes \id_A$. Then $\Psi \circ \beta$ is also nuclear, $\mathcal O_2$-stable, and $\mathcal I(\Psi \circ \beta) = \mathcal I(\Psi)$ by \eqref{eq:IIJBinfty}. By \cite[Theorem 3.23]{Gabe-O2class}, $\Psi \circ \beta$ and $\Psi$ are approximately Murray--von Neumann equivalent. By a standard reindexing argument (cf.~\cite[Lemma 4.1]{Gabe-O2class}) there is a contraction $u\in B_\infty$ such that $u\Psi(-) u^\ast = \Psi \circ \beta$ and $u^\ast \Psi(\beta(-)) u = \Psi$. Hence there exists a $\ast$-homomorphism $\psi_0 \colon (C_0(0,1)\otimes A)\rtimes_\beta \mathbb Z \to B_\infty$ such that $\psi_0(x v^n) = \Psi(x) u^n$ for $x\in C_0(0,1)\otimes A$ and $n\in \mathbb N_0$ and $\psi_0(xv^{-n}) = \Psi(x) (u^\ast)^n$ if $n\in \mathbb N$.\footnote{Dilate $u$ to a unitary $w = \left( \begin{array}{cc} u & (1_{\widetilde B}-uu^\ast)^{1/2} \\ (1_{\widetilde B}-u^\ast u)^{1/2} & - u^\ast \end{array} \right) \in M_2(\widetilde B_\infty)$. This satisfies $w (\Psi \oplus 0) w^\ast = (\Psi \circ \beta)\oplus 0$ (one can use Lemma \ref{l:conjhom} to see this) so there is an induced $\ast$-homomorphism $(C_0(0,1)\otimes A)\rtimes_\beta \mathbb Z \to M_2(\widetilde B_\infty)$. Using Lemma \ref{l:conjhom} it is easy to check that it factors through the $(1,1)$-corner and induces a $\ast$-homomorphism $\psi_0$ as described.} Here $v$ denotes the canonical unitary in the crossed product. By \cite[Lemma 6.10]{Gabe-O2class},\footnote{This is proved in the unital case. One obtains the non-unital case by unitising everything, and using that the unitisation of nuclear maps are nuclear.} nuclearity of $\Psi$ entails nuclearity of $\psi_0$.

As $(C_0(0,1) \otimes A)\rtimes_\beta \mathbb Z \cong (C_0(0,1)\rtimes_\alpha \mathbb Z) \otimes A$ canonically, we identify these. Since $C_0(0,1)\rtimes_\alpha \mathbb Z \cong C(\mathbb T)\otimes \mathcal K$ contains a full projection $p$, let $\psi := \psi_0(p\otimes -) \colon A \to B_\infty$. Using \eqref{eq:IIJBinfty}, it is not hard to see\footnote{The proof of \cite[Theorem 6.1, bottom of page 39]{Gabe-O2class} contains the exact details for this proof, word by word.} that 
\begin{equation}\label{eq:IpsiJBinfty}
\mathcal I(\psi)(J) = \overline{B_\infty \Phi(J) B_\infty}, \qquad J\in \mathcal I(A).
\end{equation}
Consider the unitisation $\psi^\dagger \colon A^\dagger \to \multialg{B}_\infty$ of $\psi$ which is nuclear by Lemma \ref{l:unitisednuc}. 
 Let $\eta\colon \mathbb N\to \mathbb N$ be a map such that $\lim_{n\to \infty} \eta(n) = \infty$ and denote by $\eta^\ast \colon \multialg{B}_\infty \to \multialg{B}_\infty$ the induced $\ast$-homomorphism. As every ideal in the image of $\mathcal I(\psi)$ is generated by constant elements, \cite[Lemmas 6.6 and 6.7]{Gabe-O2class} implies that the maps $\psi^\dagger$ and $\eta^\ast \circ \psi^\dagger$ are nuclear, $\mathcal O_2$-stable, and agree when applying $\mathcal I$. Hence by \cite[Theorem 3.23]{Gabe-O2class}, $\psi^\dagger$ and $\eta^\ast \circ \psi^\dagger$ are approximately unitarily equivalent. By \cite[Theorem 4.3]{Gabe-O2class}, $\psi$ is approximately unitarily equivalent to a $\ast$-homomorphism $\phi \colon A \to B \subseteq B_\infty$. As $\mathcal I(\psi) = \mathcal I(\phi)$ (when $\phi$ is considered as a map into $B_\infty$), it follows from \eqref{eq:IpsiJBinfty} that $\mathcal I(\phi) = \Phi$ when $\phi$ is considered as a $\ast$-homomorphism into $B$. Moreover, $\phi$ is nuclear by nuclearity of $\psi$, and strongly $\mathcal O_2$-stable since it was assumed that $B$ is $\mathcal O_2$-stable.
\end{proof}

The following is an immediate consequence of the above theorem and Lemma \ref{l:Cuaction}$(c)$.

\begin{corollary}\label{c:irO2X}
Let $X$ be a topological space, let $A$ be a separable, exact, monotone continuous $X$-$C^\ast$-algebra, and let $B$ be an $\mathcal O_\infty$-stable, $X$-compact, upper semicontinuous $X$-$C^\ast$-algebra. Then there exists a nuclear, strongly $\mathcal O_2$-stable, $X$-full $\ast$-homomorphism $\phi \colon A \to B$. 
\end{corollary}


\section{The main theorems}

The end is nigh, and the classification of separable, nuclear, $\mathcal O_\infty$-stable $C^\ast$-algebras -- which are not necessarily simple -- is finally within reach. The proof will follow the same strategy as in Section \ref{s:KP} for the simple case.

\subsection{Existence of $\theta$}

In Theorems \ref{t:existsimple} and \ref{t:uniquesimple}, the target $C^\ast$-algebra $B$ was always assumed to contain a properly infinite, full projection. Such a projection ensures the existence of a full embedding $\mathcal O_2 \hookrightarrow B$, and combined with Kirchberg's $\mathcal O_2$-embedding theorem one could produce a $\ast$-homomorphism as the composition $A \hookrightarrow \mathcal O_2 \hookrightarrow B$. In Lemma \ref{l:fullO2map}, this idea was slightly modified to construct a full, nuclear $\ast$-homomorphism $\theta \colon A \to B \otimes \mathcal K$ such that $\mathcal O_2$-embeds unitally in $\multialg{B\otimes \mathcal K}\cap \theta(A)'$. The existence of such a map $\theta$ allows one to apply the results from Section \ref{s:unitary}. This was a key technical ingredient in the proofs of Theorems \ref{t:existsimple} and \ref{t:uniquesimple}. 

In this section some similar results are presented on how to produce $X$-full, nuclear $\ast$-homomorphisms $\theta \colon A \to \multialg{B\otimes \mathcal K}$ such that $\mathcal O_2$ embeds unitally in $\multialg{B\otimes \mathcal K}\cap \theta(A)'$. The main one of these constructions comes from the existence of an $\mathcal O_\infty$-stable $C^\ast$-subalgebra $D \subseteq B$ which induces the action of $X$ on $B$, see Lemma \ref{l:XO2map} for the precise statement. The existence of such a $C^\ast$-subalgebra $D\subseteq B$ is an ideal-related analogue of $B$ containing a properly infinite, full projection.

First some ideal-related versions of standard results will be proved. Kasparov's stabilisation theorem \cite[Theorem 2]{Kasparov-Stinespring} implies that if $B$ is a $\sigma$-unital, stable $C^\ast$-algebra, and if $B_0 \subseteq B$ is a $\sigma$-unital, hereditary $C^\ast$-subalgebra, then $B_0$ is isomorphic to a corner in $B$. Part $(a)$ below is an ideal-related version of a technical extension of this result. Part $(b)$ is an ideal-related version of Brown's stable isomorphism theorem \cite{Brown-stableiso}.

 Recall that if $\phi \colon A \to B$ is a $\ast$-homomorphism, then $\mathcal I(\phi) \colon \mathcal I(A) \to \mathcal I(B)$ is the map $I \mapsto \overline{B\phi(I)B}$.
Also, a $\ast$-homomorphism $\phi \colon D \to B$ is called \emph{extendible} if the hereditary $C^\ast$-subalgebra $\overline{\phi(D) B \phi(D)}$ is a corner $PBP$ for a (unique) projection $P\in \multialg{B}$. Note that $\phi$ extends canonically to a $\ast$-homomorphism $\multialg{\phi} \colon \multialg{D} \to \multialg{B}$ satisfying $\multialg{\phi}(1_{\multialg{D}}) = P$.

\begin{proposition}\label{p:Brown}
Let $B$ be a $\sigma$-unital, stable $C^\ast$-algebra. 
\begin{itemize}
\item[$(a)$] Suppose that $D$ is a $\sigma$-unital $C^\ast$-algebra and that $\eta \colon D \to B$ is a $\ast$-homo\-morphism. Then there exists an extendible $\ast$-homomomorphism $\kappa \colon D \to B$ such that $\mathcal I(\kappa) = \mathcal I(\eta)$ and for which $1_{\multialg{B}} - \multialg{\kappa}(1_{\multialg{D}})$ is Murray--von Neumann equivalent to $1_{\multialg{B}}$ in $\multialg{B}$. Also, if $\eta$ is nuclear, then $\kappa$ may be chosen nuclear.
\item[$(b)$] Suppose $B_0$ is a $\sigma$-unital, full, hereditary, stable $C^\ast$-subalgebra of $B$, and let $j \colon B_0 \hookrightarrow B$ denote the inclusion. Then there exists an isomorphism $\Omega \colon B_0 \xrightarrow \cong B$, such that $\mathcal I(\Omega) = \mathcal I(j)$.
\end{itemize}
\end{proposition}
\begin{proof}
The proofs of $(a)$ and $(b)$ only differ at the very end. Let $B_0 := \overline{\eta(D) B \eta(D)}$ (which we only assume is full and stable when considering part $(b)$ at the end of the proof). Consider the $C^\ast$-algebra $C= \left( \begin{array}{cc} B_0 & \overline{B_0 B} \\ \overline{BB_0} & B \end{array}\right)$ which is a subalgebra of $M_2(B)$. As $\overline{B_0 B} \oplus B \cong B$ as Hilbert $B$-modules by Kasparov's stabilisation theorem \cite[Theorem 2]{Kasparov-Stinespring}, it follows that $C \cong \mathcal K_B(\overline{B_0 B} \oplus B) \cong B$ is $\sigma$-unital and stable.

Let $\eta_0 \colon D \to B_0$ be the corestriction of $\eta$, which is non-degenerate and nuclear if $\eta$ is nuclear, and let $j\colon B_0 \to B$ be the inclusion, so that $\eta = j \circ \eta_0$. Consider the obvious embedding $j_{1,1} \colon B_0 \to C$ which is extendible, as well as $j_{2,2} \colon B \to C$, $i \colon C \to M_2(B)$, and $\id_B \otimes e_{k,k} \colon B \to M_2(B)$ for $k=1,2$ which are all embeddings of full, hereditary $C^\ast$-subalgebras, and are thus invertible once applying $\mathcal I$ (cf.~\cite[Proposition 4.1.10]{Pedersen-book-automorphism}). Moreover, as $\mathcal I(\id_B \otimes e_{1,1}) = \mathcal I(\id_B \otimes e_{2,2})$, $(\id_B \otimes e_{1,1}) \circ j = i \circ j_{1,1}$ and $i \circ j_{2,2} = \id_{B} \otimes e_{2,2}$, it follows that
\begin{eqnarray}
\mathcal I(j) &=& \mathcal I(\id_B \otimes e_{1,1})^{-1} \circ \mathcal I(i \circ j_{1,1}) \nonumber\\
&=& \mathcal I(j_{2,2})^{-1} \circ \mathcal I(i)^{-1} \circ \mathcal I(i) \circ \mathcal I(j_{1,1}) \nonumber\\
&=& \mathcal I(j_{2,2})^{-1} \circ \mathcal I(j_{1,1}). \label{eq:j11j22}
\end{eqnarray}

Note that $j_{2,2} \colon B \hookrightarrow C$ is the embedding of $B$ as a full corner in $C$, and let $p_{2,2} = \multialg{j_{2,2}} (1_{\multialg{B}})$. As both $B$ and $C$ are $\sigma$-unital and stable, it follows from \cite[Theorem 4.23]{Brown-semicontmultipliers} that there is an isometry $v$ in $\multialg{C}$ with $vv^\ast = p_{2,2}$. Clearly $\Omega_{2,2} := \Ad v \circ j_{2,2}$ is an isomorphism $B \xrightarrow \cong C$ such that $\mathcal I(\Omega_{2,2}) = \mathcal I(j_{2,2})$, so $\kappa := \Omega_{2,2}^{-1} \circ j_{1,1} \circ \eta_0$ is extendible, nuclear if $\eta$ is nuclear, and satisfies 
\begin{equation}
\mathcal I(\kappa)  = \mathcal I(j_{2,2})^{-1} \circ \mathcal I(j_{1,1}) \circ \mathcal I(\eta_0) \stackrel{\eqref{eq:j11j22}}{=} \mathcal I(j \circ \eta_0) = \mathcal I(\eta). 
\end{equation}
Moreover, as $\multialg{\Omega_{2,2}}(1_{\multialg{B}}- \multialg{\kappa}(1_{\multialg{D}})) = p_{2,2}$, which is equivalent to $1_{\multialg{C}}$ in $\multialg{C}$, it follows that $1_{\multialg{B}}-\multialg{\kappa}(1_{\multialg{D}})$ is equivalent to $1_{\multialg{B}}$ in $\multialg{B}$. This completes part $(a)$.

For part $(b)$, $B_0$ is assumed to be stable and full in $B$, so arguing as for $\Omega_{2,2}$ above, one constructs an isomorphism $\Omega_{1,1} \colon B_0 \xrightarrow{\cong} C$ such that $\mathcal I(\Omega_{1,1}) = \mathcal I(j_{1,1})$. Hence $\Omega := \Omega_{2,2}^{-1} \circ \Omega_{1,1}$ satisfies the conditions of $(b)$ by \eqref{eq:j11j22}.
\end{proof}

\begin{lemma}\label{l:genO2emb}
Let $D$ be a $\sigma$-unital, $\mathcal O_\infty$-stable $C^\ast$-algebra, let $B$ be a $\sigma$-unital, stable $C^\ast$-algebra, and let $\eta \colon D \to B$ be a $\ast$-homomorphism. Then there exists a $\ast$-homomorphism $\kappa \colon D \to B$ such that $\mathcal I(\kappa) = \mathcal I(\eta)$, and such that $\mathcal O_2$ embeds unitally into $\multialg{B} \cap \kappa(D)'$. Moreover, if $\eta$ is nuclear then $\kappa$ may be chosen nuclear.
\end{lemma}
\begin{proof}
As $D$ is $\mathcal O_\infty$-stable, and as $\mathcal O_\infty$ is strongly self-absorbing (see \cite{TomsWinter-ssa}), we may assume that $D = D_1 \otimes \mathcal O_\infty$ and that there exists an isomorphism $\delta \colon D \xrightarrow \cong D_1$ such that $\id_D$ and $\delta \otimes 1_{\mathcal O_\infty}$ are approximately unitarily equivalent.\footnote{In fact, we may assume $D=D_2 \otimes \mathcal O_\infty \otimes \mathcal O_\infty$. As $\mathcal O_\infty$ is strongly self-absorbing, there is an isomorphism $\phi \colon \mathcal O_\infty \xrightarrow \cong \mathcal O_\infty \otimes \mathcal O_\infty$ which is approximately unitarily equivalent to $\id_{\mathcal O_\infty} \otimes 1_{\mathcal O_\infty}$. Letting $D_1 = D_2 \otimes \mathcal O_\infty$, and $\delta = \id_{D_2} \otimes \phi^{-1}$ does the trick.} Hence $\mathcal I(\delta \otimes 1_{\mathcal O_\infty}) = \id_{\mathcal I(D)}$.
Fix a non-zero projection $p\in \mathcal O_\infty$ with $[p]_0 =0$ in $K_0(\mathcal O_\infty)$, and let $\mathcal O_\infty^\st := p \mathcal O_\infty p$. As $\mathcal O_\infty$ is simple, it follows that $\delta \otimes p \colon D \to D$ agrees with $\delta \otimes 1_{\mathcal O_\infty}$ when applying $\mathcal I$. Consider the $\ast$-homomorphism $\delta \otimes 1_{\mathcal O_\infty^\st} \colon D \to D_1 \otimes \mathcal O_\infty^\st$, the inclusion $\iota \colon D_1 \otimes \mathcal O_\infty^\st \hookrightarrow D_1 \otimes \mathcal O_\infty = D$, and note that $\delta \otimes p = \iota \circ (\delta \otimes 1_{\mathcal O_\infty^\st})$. Hence
\begin{equation}\label{eq:Iiotadelta}
\mathcal I(\iota) \circ \mathcal I(\delta \otimes 1_{\mathcal O_\infty^\st}) = \mathcal I(\delta \otimes p) = \mathcal I(\delta \otimes 1_{\mathcal O_\infty}) = \id_{\mathcal I(D)}.
\end{equation}

Apply Proposition \ref{p:Brown}$(a)$ to the $\ast$-homomorphism $\eta \circ \iota \colon D_1 \otimes \mathcal O_\infty^\st \to B$, and obtain an extendible $\ast$-homomorphism $\kappa_0 \colon D_1 \otimes \mathcal O_\infty^\st \to B$ such that $\mathcal I(\kappa_0) = \mathcal I(\eta) \circ \mathcal I(\iota)$ and  $1_{\multialg{B}} - \multialg{\kappa_0}(1_{\multialg{D_1\otimes \mathcal O_\infty^\st}}) \sim 1_{\multialg{B}}$ in $\multialg{B}$. If $\eta$ is nuclear then so is $\eta \circ \iota$, so that $\kappa_0$ could be chosen nuclear.
Define $\kappa := \kappa_0 \circ (\delta \otimes 1_{\mathcal O_\infty^\st}) \colon D \to B$ which is nuclear if $\kappa_0$ is nuclear. Then
\begin{equation}
\mathcal I(\kappa) = \mathcal I(\kappa_0) \circ \mathcal I(\delta \otimes 1_{\mathcal O_\infty^\st}) = \mathcal I(\eta) \circ \mathcal I(\iota) \circ \mathcal I(\delta \otimes 1_{\mathcal O_\infty^\st}) \stackrel{\eqref{eq:Iiotadelta}}{=} \mathcal I(\eta). 
\end{equation}
Fix $\mathcal O_2$-isometries $s_1,s_2\in \mathcal O_\infty^\st$. As $B$ is stable and $1_{\multialg{B}} - \multialg{\kappa}(1_{\multialg{D}}) \sim 1_{\multialg{B}}$, we may pick $t_1,t_2 \in \multialg{B}$ such that
\begin{equation}
t_1^\ast t_1 = t_2^\ast t_2 = t_1t_1^\ast + t_2 t_2^\ast = 1_{\multialg{B}}-\multialg{\kappa}(1_{\multialg{D}}).
\end{equation} 
Then $S_i := \multialg{\kappa_0}(1_{\multialg{D_1}} \otimes s_i) + t_i \in \multialg{B}$ for $i=1,2$ are $\mathcal O_2$-isometries that commute with the image of $\kappa$, so $\mathcal O_2$ embeds unitally in $\multialg{B} \cap \kappa(D)'$.
\end{proof}

\begin{lemma}\label{l:theta}
Let $X$ be a topological space, let $A$ be a separable, exact, lower semicontinuous $X$-$C^\ast$-algebra, and let $B$ be a $\sigma$-unital $X$-$C^\ast$-algebra. Suppose that there exists an $X$-full, nuclear, $\mathcal O_\infty$-stable $\ast$-homomorphism $\phi \colon A \to B$. Then there exists an $X$-full, nuclear $\ast$-homomorphism $\theta \colon A \to B\otimes \mathcal K$ such that $\mathcal O_2$ embeds unitally in $\multialg{B\otimes \mathcal K} \cap \theta(A)'$.
\end{lemma}
\begin{proof}
By the McDuff type property for $\mathcal O_\infty$-stable maps \cite[Corollary 4.5]{Gabe-O2class}, there exists a $\ast$-homomorphism $\psi \colon A \otimes \mathcal O_\infty \to B\otimes \mathcal K$ such that $\psi\circ(\id_A \otimes 1_{\mathcal O_\infty})$ and $\phi \otimes e_{1,1}$ are approximately Murray--von Neumann equivalent. As $\phi\otimes e_{1,1}$ is nuclear, so is $\psi\circ(\id_A \otimes 1_{\mathcal O_\infty})$ and thus by Lemma \ref{l:nucoutoftensor}, $\psi$ is nuclear. By Lemma \ref{l:genO2emb}, there is a nuclear $\ast$-homomorphism $\kappa \colon A \otimes \mathcal O_\infty \to B$ such that $\mathcal I(\kappa ) = \mathcal I(\psi)$ and for which $\mathcal O_2$ embeds unitally in $\multialg{B\otimes \mathcal K} \cap \kappa(A \otimes \mathcal O_\infty)'$. Letting $\theta = \kappa  \circ(\id_A \otimes 1_{\mathcal O_\infty})$ one has
\begin{equation}
\mathcal I(\theta) = \mathcal I(\psi) \circ \mathcal I(\id_A \otimes 1_{\mathcal O_\infty}) = \mathcal I(\phi \otimes e_{1,1})
\end{equation}
and therefore $\theta$ is $X$-full since $\phi\otimes e_{1,1}$ is clearly $X$-full.
\end{proof}

Consequently, and by using the ideal-related $\mathcal O_2$-embedding theorem from the previous section, one obtains the following $X$-equivariant version of Lemma \ref{l:fullO2map}. The condition of containing an $\mathcal O_\infty$-stable subalgebra as in the lemma should be considered as an ideal-related version of containing a properly infinite, full projection, as was assumed in Theorem \ref{t:existsimple}.

\begin{lemma}\label{l:XO2map}
Let $X$ be a topological space, let $A$ be a separable, exact, monotone continuous $X$-$C^\ast$-algebra, and let $B$ be a $\sigma$-unital $C^\ast$-algebra, containing an $\mathcal O_\infty$-stable $C^\ast$-subalgebra $D \subseteq B$. Suppose $D$ is equipped with an upper semicontinuous, $X$-compact action of $X$, and equip $B$ with the action of $X$ given by $B(U) := \overline{B D(U) B}$ for $U\in \mathcal O(X)$.
Then there exists a nuclear, $X$-full $\ast$-homomorphism $\theta \colon A \to B \otimes \mathcal K$ such that $\mathcal O_2$ embeds unitally in $\multialg{B \otimes \mathcal K} \cap \theta(A)'$.
\end{lemma}
\begin{proof}
By Corollary \ref{c:irO2X}, there is a nuclear $X$-full $\ast$-homomorphism $\phi \colon A \to D$. The composition $\phi \colon A \to D \subseteq B$ is a nuclear, $X$-full, $\mathcal O_\infty$-stable $\ast$-homo\-morphism, so $\theta$ as in the statement of the lemma exists by Lemma \ref{l:theta}.
\end{proof}

\subsection{Proving Theorems \ref{t:irexistence}, \ref{t:iruniqueness} and \ref{t:nonsimpleclass}}

The following is an $X$-equivariant analogue of Lemma \ref{l:absCuntzpair}, and the proof is identical.

\begin{lemma}\label{l:absCuntzpair2}
Let $A$ be a separable $X$-$C^\ast$-algebra, $B$ be a $\sigma$-unital, stable $X$-$C^\ast$-algebra, and suppose that $\Theta \colon A \to \multialg{B}$ is a weakly $X$-nuclear, $X$-nuclearly absorbing $\ast$-homomorphism. Any element $x\in KK_\nuc(X; A ,B)$ is represented by a weakly $X$-nuclear Cuntz pair of the form $(\psi, \Theta)$.
\end{lemma}
\begin{proof}
The proof of Lemma \ref{l:absCuntzpair} carries over verbatim, by replacing the word ``nuclear'' with the word ``$X$-nuclear''.
\end{proof}

The following is the main existence result. It is stated in a somewhat abstract way, but I emphasise that Lemma \ref{l:XO2map} can be used to produce a nuclear, $X$-full, $\mathcal O_\infty$-stable map $A \to B$ as in the statement of the proposition.

\begin{proposition}\label{p:existence}
Let $X$ be a topological space, let $A$ be a separable, exact, lower semicontinuous $X$-$C^\ast$-algebra, and let $B$ be a $\sigma$-unital $X$-$C^\ast$-algebra containing a $\sigma$-unital, stable, full, hereditary $C^\ast$-subalgebra. Suppose that there exists a nuclear, $X$-full, $\mathcal O_\infty$-stable $\ast$-homomorphism $A \to B$.

For any $x\in KK_\nuc(X; A,B)$ there exists a nuclear, strongly $\mathcal O_\infty$-stable, $X$-full $\ast$-homomorphism $\phi \colon A \to B$ such that $KK_\nuc(X; \phi) = x$.

Moreover, if $A$ and $B$ are unital then $\phi$ may be picked to also be unital if and only if the following conditions hold:
\begin{itemize}
\item[(1)] $B(U) = B$ for every $U\in \mathcal O(X)$ satisfying $A(U) = A$; and 
\item[(2)] $\Gamma_0(x)([1_A]_0) = [1_B]_0$ in $K_0(B)$.
\end{itemize}
\end{proposition}

\begin{proof}
Let $B_0\subseteq B$ be a $\sigma$-unital, stable, full, hereditary $C^\ast$-subalgebra of $B$. The inclusion $\iota \colon B_0 \hookrightarrow B$ induces an isomorphism $\mathcal I(\iota) \colon \mathcal I(B_0) \xrightarrow \cong \mathcal I(B)$, and thus induces an action of $X$ on $B_0$ which satisfies that $B_0(U)$ is full in $B(U)$ for all $U\in \mathcal O(X)$. By Proposition \ref{p:KKXfullher}, the inclusion $\iota$ induces an isomorphism $KK_\nuc(X; A , B_0) \cong KK_\nuc(X; A, B)$. Let $x_0\in KK_\nuc(X; A , B_0)$ be the element mapped to $x$ by this isomorphism. The strategy will be to find a nuclear, strongly $\mathcal O_\infty$-stable, $X$-full $\ast$-homomorphism $\phi_0 \colon A \to B_0$ such that $KK_\nuc(X; \phi_0) = x_0$. Then the composition $\phi := \iota \circ \phi_0 \colon A \to B$ is nuclear, strongly $\mathcal O_\infty$-stable (by Lemma \ref{l:relcombasic}$(c)$), $X$-full, and satisfies $KK_\nuc(X; \phi) = x$. 

The ideal-related version of Brown's stable isomorphism theorem, Proposition \ref{p:Brown}$(b)$, implies that $B_0= B_0 \otimes e_{1,1}$ and $B\otimes \mathcal K$ are isomorphic as $X$-$C^\ast$-algebras, so by Lemma \ref{l:theta} there exists a nuclear, $X$-full $\ast$-homomorphism $\theta \colon A \to B_0$ such that $\mathcal O_2$ embeds unitally into $\multialg{B_0}\cap \theta(A)'$.

By Theorem \ref{t:infrepXabs}, any infinite repeat $\theta_\infty = \sum_{i=1}^\infty s_i \theta(-) s_i^\ast$ of $\theta$ is weakly $X$-nuclear, and $X$-nuclearly absorbing. Let $s_0\in \multialg{B_0}$ be an isometry satisfying $s_0s_0^\ast = 1_{\multialg{B_0}}-s_1 s_1^\ast$ so that $s_1,s_0$ are $\mathcal O_2$-isometries. Lemma \ref{l:absCuntzpair2} provides the existence of a weakly $X$-nuclear Cuntz pair of the form $(\psi, \theta_\infty)$ which induces $x_0$. By Lemma \ref{l:keyexistence}, there is a continuous unitary path $(v_t)_{t\in \mathbb R_+}$ in $\multialg{B_0}\cap \theta_\infty(A)'$ with $v_0=1_{\multialg{B_0}}$, and a $\ast$-homomorphism $\phi_1 \colon A \to B_0$ such that $v_t \psi(-) v_t^\ast$ converges point-norm to $\phi_0 := s_1 \phi_1(-)s_1^\ast + \sum_{i=2}^\infty s_i \theta(-) s_i^\ast$. Let $\theta_0 := \sum_{i=2}^\infty s_0^\ast s_i \theta(-) s_i^\ast s_0$ so that $\phi_0 = \phi_1 \oplus_{s_1,s_0} \theta_0$ and $\theta_\infty = \theta\oplus_{s_1,s_0} \theta_0$. There is a homotopy of weakly $X$-nuclear Cuntz pairs from $(\psi , \theta_\infty)$ to $(\phi_1, \theta) \oplus_{s_1,s_0} (\theta_0, \theta_0)$ given via $(\Ad v_t\circ  \psi, \theta_\infty)$. As $(\theta_0, \theta_0)$ is zero homotopic, it follows that $(\phi_1, \theta)$ is a weakly $X$-nuclear Cuntz pair inducing $x_0$. As $\phi_1 = s_1^\ast \phi_0 (-) s_1$ is weakly $X$-nuclear and takes values in $B_0$, it follows that it is $X$-nuclear. Thus
\begin{equation}
x_0 = KK_\nuc(X; \phi_1) - KK_\nuc(X; \theta) = KK_\nuc(X; \phi_1).
\end{equation}
Here we used that $\mathcal O_2$ embeds unitally in $\multialg{B_0} \cap \theta(A)'$, which by the same argument as for ordinary $KK$-theory implies that $KK_\nuc(X; \theta) = 0$.

Let $\phi_2 = \phi_1 \oplus_{s_1,s_0} \theta$. As both $\phi_1$ and $\theta$ are $X$-nuclear, so is $\phi_2$. As $\theta$ is nuclear and $X$-full, and $\phi_1$ is nuclear and $X$-equivariant, it follows from Corollary \ref{c:Xfulldom} that $\theta$ approximately dominates $\phi_2$.  As $\theta$ is strongly $\mathcal O_\infty$-stable, Proposition \ref{p:piabsorbing}$(a)$ implies that $\phi_2$ is strongly $\mathcal O_\infty$-stable. As $\theta$ is nuclear and $X$-full, and as $\phi_1$ is nuclear and $X$-equivariant, it follows that $\phi_2$ is nuclear and $X$-full. Hence
\begin{equation}
KK_\nuc(X; \phi_2) = KK_\nuc(X; \phi_1) + KK_\nuc(X; \theta) = KK_\nuc(X; \phi_1 ) = x_0.
\end{equation}
Letting $\phi := \iota \circ \phi_2 \colon A \to B$, we have obtained a nuclear, strongly $\mathcal O_\infty$-stable, $X$-full $\ast$-homomorphism for which $KK_\nuc(X; \phi) = x$ as desired. This finishes the proof in the not necessarily unital case.

``Moreover'': Now suppose that $A$ and $B$ are unital. For ``only if'', suppose $\phi \colon A \to B$ is an $X$-full, unital $\ast$-homomorphism. If $U\in \mathcal O(X)$ is such that $A(U) = A$, then
\begin{equation}
1_B = \phi(1_A) \in \phi(A(U)) \subseteq B(U)
\end{equation}
which implies that $B(U) = B$, so $(1)$ holds. That $(2)$ holds follows from Observation \ref{o:KK(X)hom}.

For ``if'', suppose that $(1)$ and $(2)$ hold. Use the not necessarily unital part of the proposition to lift $x$ to a nuclear, strongly $\mathcal O_\infty$-stable, $X$-full $\ast$-homomorphism $\phi_0 \colon A \to B$. Let $\Phi_B$ denote the action of $B$, and let $\Psi_A$ denote the dual action of $A$. As $\phi_0$ is $X$-full, it follows that $\mathcal I(\phi_0) = \Phi_B \circ \Psi_A$.  By Lemma \ref{l:dualaction}, it follows that $U = \Psi_A(A) \in \mathcal O(X)$ satisfies $A(U) = A$. So $(1)$ implies that $\Phi_B \circ \Psi_A(A) = B$.   Hence $\overline{B \phi_0(A) B} = B$, and thus $\phi_0(1_A)$ is a full projection in $B$. As $\phi_0$ is strongly $\mathcal O_\infty$-stable, it follows that $\phi_0(1_A)$ is properly infinite by Remark \ref{r:fullpropinfproj}, and hence $1_B$ is also a full, properly infinite projection. By $(2)$ it follows that $[1_B]_0 = [\phi_0(1_A)]_0$ in $K_0(B)$, so by \cite{Cuntz-K-theoryI} there is an isometry $v\in B$ for which $vv^\ast = \phi_0(1_A)$. Now $\phi := v^\ast \phi_0(-) v \colon A \to B$ is a nuclear, $X$-full $\ast$-homomorphism with the same $KK_\nuc(X)$-class as $\phi_0$. Finally $\phi$ is strongly $\mathcal O_\infty$-stable by Lemma \ref{l:relcombasic}$(c)$. 
\end{proof}

Theorem \ref{t:irexistence} is an immediate corollary.

\begin{proof}[Proof of Theorem \ref{t:irexistence}]
Combine Proposition \ref{p:existence} with Lemma \ref{l:XO2map} in the special case $D=B$.
\end{proof}

\begin{proof}[Proof of Theorem \ref{t:iruniqueness}]
$(ii) \Rightarrow (i)$ is Corollary \ref{c:KKXasMvN} (using Lemma \ref{l:nucquotient} and Remark \ref{r:Bempty} to see that being $X$-nuclear is the same as being $X$-equivariant and nuclear). The equivalence $(ii) \Leftrightarrow (iii)$ is Proposition \ref{p:MvNvsue}. Only $(i) \Rightarrow (ii)$ remains to be proved, so assume that $KK_\nuc(X; \phi) = KK_\nuc(X; \psi)$. 

By Proposition \ref{p:MvNeq}, it suffices to prove that the maps $\phi\otimes e_{1,1}, \psi \otimes e_{1,1} \colon A \to B\otimes \mathcal K$ are asymptotically Murray--von Neumann equivalent. Let $\theta\colon A \to B\otimes \mathcal K$ be as in Lemma \ref{l:theta}. Since $\phi \otimes e_{1,1},\psi \otimes e_{1,1}$ and $\theta$ are nuclear and $X$-full, they approximately dominate each other by Corollary \ref{c:Xfulldom}. Thus, combining the stable uniqueness theorem, Theorem \ref{t:irDE2}, with the key lemma for uniqueness, Lemma \ref{l:keyuniqueness},
it follows that $\phi \otimes e_{1,1} \sim_\asMvN \psi \otimes e_{1,1}$.
\end{proof}

The following is a consequence of the uniqueness result above which does not initially use actions of topological spaces. Recall that for any $C^\ast$-algebra $A$, there is a canonical action $\I_A \colon \mathcal O(\Prim A) \to \mathcal I(A)$ which is an order isomorphism.

\begin{corollary}
Let $A$ be a separable, exact $C^\ast$-algebra, let $B$ be a $\sigma$-unital $C^\ast$-algebra, and let $\phi, \psi \colon A \to B$ be nuclear, strongly $\mathcal O_\infty$-stable $\ast$-homomorphisms. Then $\phi$ and $\psi$ are asymptotically Murray--von Neumann equivalent if and only if $\mathcal I(\phi) = \mathcal I(\psi)$ and for $X:= \Prim A$ the induced action $\Phi_B = \mathcal I(\phi) \circ \I_A \colon \mathcal O(X)  \to \mathcal I(B)$ one has 
\begin{equation}
KK_\nuc(X; \phi) = KK_\nuc(X; \psi) \textrm{ in } KK_\nuc(X; (A, \I_A), (B, \Phi_B)).
\end{equation}
\end{corollary}

The main classification theorem, Theorem \ref{t:nonsimpleclass}, is an easy consequence of Theorems \ref{t:irexistence} and \ref{t:iruniqueness}, using the asymptotic intertwining (Proposition \ref{p:asint}) to see that one may lift the ideal-related $KK$-equivalence to an isomorphism of the $C^\ast$-algebras. The details are presented below.

\begin{proof}[Proof of Theorem \ref{t:nonsimpleclass}]
$(a)$: By Theorem \ref{t:irexistence} it is possible to find $X$-full $\ast$-homo\-morphisms $\phi_0 \colon A \to B$ and $\psi_0 \colon B \to A$ such that $KK(X; \phi_0) = x$ and $KK(X; \psi_0) = x^{-1}$. As $\phi_0$ and $\psi_0$ are $X$-full so is $\psi_0 \circ \phi_0$, and as $A$ is tight, it follows that $\psi_0(\phi_0(A))$ is full in $A$. Hence both $\psi_0 \circ \phi_0$ and $\id_A$ are $X$-full, nuclear, strongly $\mathcal O_\infty$-stable $\ast$-homomorphisms with full images for which
\begin{equation}
KK(X; \psi_0 \circ \phi_0) = x^{-1} \circ x = KK(X; \id_A).
\end{equation}
By Theorem \ref{t:iruniqueness} $(i) \Rightarrow (iii)$ it follows that $\psi_0 \circ \phi_0 \sim_\asu \id_A$. Similarly, one obtains $\phi_0 \circ \psi_0 \sim_\asu \id_B$. By Proposition \ref{p:asint}, there exists an isomorphism $\phi \colon A \xrightarrow \cong B$ and a homotopy $(\phi_s)_{s\in [0,1]}$ from $\phi_0$ to $\phi$, such that $\phi_s \sim_\aMvN \phi_t$ for all $s,t\in [0,1]$. As approximate Murray--von Neumann equivalence preserves $X$-equivariance of $\ast$-homomorphisms, it follows that each $\phi_t$ is $X$-equivariant. By Lemma \ref{l:XnucC(Y)}, it follows that $\phi_0$ and $\phi$ are homotopic in the $X$-equivariant sense, and thus $KK(X; \phi_0) = KK(X; \phi)$. As $A$ and $B$ are both tight, any $X$-equivariant $\ast$-isomorphism is automatically an isomorphism of $X$-$C^\ast$-algebras, thus completing the proof of part $(a)$.

$(b)$: This is proved exactly as above but using the unital versions of Theorems \ref{t:irexistence} and \ref{t:iruniqueness}.
\end{proof}

\subsection{Approximate equivalence}

As when defining $KL_\nuc$ (see Section \ref{ss:KLnuc}) one can do the same for $KK_\nuc(X)$. I will avoid addressing whether the approach described below defines a topology as done in \cite{Dadarlat-KKtop}, and take a shortcut by simply defining $\overline{\{0\}}$.

 If $A$ is a separable $X$-$C^\ast$-algebra, and $B$ is a $\sigma$-unital $X$-$C^\ast$-algebra, let $\overline{\{0\}} \subseteq KK_\nuc(X; A,B)$ be the set of elements $x\in KK_\nuc(X; A, B)$, for which there exists $y\in KK_\nuc(X; A, C(\widetilde{\mathbb N}, B))$ such that $(\ev_n)_\ast(y) = 0$ for $n\in \mathbb N$, and $(\ev_\infty)_\ast(y) = x$. Clearly $\overline{\{0\}}$ is a subgroup of $KK_\nuc(X; A, B)$, so one may define
\begin{equation}
KL_\nuc(X; A, B) := KK(X; A, B)/\overline{\{0\}}.
\end{equation}

If $\phi \colon A \to B$ is an $X$-nuclear $\ast$-homomorphism, the induced element in $KL_\nuc(X)$ is denote by $KL_\nuc(X; \phi)$. It is easy to see that $KL_\nuc(X; \phi) = KL_\nuc(X; \psi)$ if and only if there exists $y\in KK_\nuc(X; A, C(\widetilde{\mathbb N}, B))$ such that $(\ev_n)_\ast(y) = KK_\nuc(X; \phi)$ for all $n\in \mathbb N$, and $(\ev_\infty)_\ast(y) = KK_\nuc(X; \psi)$. 

\begin{proposition}
Let $A$ be a separable $X$-$C^\ast$-algebra, let $B$ be a $\sigma$-unital $X$-$C^\ast$-algebra and suppose that $\phi, \psi \colon A \to B$ are $X$-nuclear $\ast$-homo\-morphisms. If $\phi \sim_\aMvN \psi$ then $KL_\nuc(X; \phi) = KL_\nuc(X; \psi)$. 
\end{proposition}
\begin{proof}
The proof is essentially identical to that of the non-$X$-equivariant version, Proposition \ref{p:aMvNKL}, and using Corollary \ref{c:KKXstable} to get $KK_\nuc(X; A, B) \cong KK_\nuc(X; A, B\otimes \mathcal K)$. The rest of the proof is omitted.
\end{proof}

The invariant $KL_\nuc(X)$ can be used to characterise approximate equivalence. As in Theorem \ref{t:approxuniquesimple}, this is done for $\mathcal O_\infty$-stable maps instead of strongly $\mathcal O_\infty$-stable maps.

\begin{theorem}[Approximate uniqueness]\label{t:approxuniqueness}
Let $X$ be a topological space, let $A$ be a separable, exact, lower semicontinuous $X$-$C^\ast$-algebra, and let $B$ be a $\sigma$-unital $X$-$C^\ast$-algebra. Suppose that $\phi, \psi \colon A \to B$ are $X$-full, nuclear,  $\mathcal O_\infty$-stable $\ast$-homo\-morphisms. The following are equivalent:
\begin{itemize}
\item[$(i)$] $KL_\nuc(X; \phi) = KL_\nuc(X; \psi)$;
\item[$(ii)$] $\phi$ and $\psi$ are approximately Murray--von Neumann equivalent.
\end{itemize}
Additionally, if either $B$ is stable, or if $A,B,\phi$ and $\psi$ are all unital, then $(i)$ and $(ii)$ are equivalent to
\begin{itemize}
\item[$(iii)$] $\phi$ and $\psi$ are approximately unitarily equivalent (with multiplier unitaries in the stable case).
\end{itemize} 
\end{theorem}
\begin{proof}
The proof is identical to that of Theorem \ref{t:approxuniquesimple}, but using Proposition \ref{p:existence} and Theorem \ref{t:iruniqueness} instead of Theorems \ref{t:existsimple} and \ref{t:uniquesimple}. The details are omitted.
\end{proof}

In the non-$X$-equivariant case, $KL_\nuc$ could be computed using the $K$-theoretic invariant $\underline K$ (assuming the UCT). Unfortunately there is no known $K$-theoretic invariant that computes $KL_\nuc(X)$ under UCT assumptions. This makes Theorem \ref{t:approxuniqueness} harder to apply than the non-$X$-equivariant analogue.

\newcommand{\etalchar}[1]{$^{#1}$}

\end{document}